\documentclass{nottsthesis}
\usepackage{ucs}
\usepackage[utf8x]{inputenc}
\usepackage[T1]{fontenc}
\usepackage[english]{babel}

\usepackage{amsmath}
\usepackage{amsfonts}
\usepackage{amssymb}
\usepackage{amsthm}
\usepackage[nottoc,notlof,notlot]{tocbibind} 
\usepackage{color}
\usepackage{url}
\usepackage{faktor}
\usepackage{array}
\usepackage{todonotes}
\usepackage{paralist} 
\usepackage{mathtools}

\newtheorem{theorem}{Theorem}[chapter]
\newtheorem{corollary}[theorem]{Corollary}
\newtheorem{lemma}[theorem]{Lemma}
\newtheorem{proposition}[theorem]{Proposition}
\newtheorem{example}[theorem]{Example}
\newtheorem{examples}[theorem]{Examples}
\newtheorem*{definition}{Definition}
\newtheorem*{remark}{Remark}
\newtheorem*{remarks}{Remarks}
\numberwithin{equation}{chapter}

\allowdisplaybreaks[4]

\thesisauthor{Christian Brown}
\thesistitle{Petit Algebras and their Automorphisms}
\degree{Doctor of Philosophy}
\degreedate{September 2017}
\school{Mathematical Sciences}

\begin{document}

\maketitle

\frontmatter

\begin{Acknowledgements}
I would like to express my sincere gratitude to my supervisor Susanne Pumpl{\"u}n for all her help, guidance and enthusiasm throughout the last three years. When times were difficult and progress was slow, she gave me the support I needed to carry on and this thesis would not have been possible without her. 

My thanks also go to Esther and my parents, for their continuing love and support.
\end{Acknowledgements}

\begin{Preface}
First introduced by Ore \cite{ore1933theory} in 1933, skew polynomial rings are one of the earliest examples in the theory of noncommutative algebra. A skew polynomial ring $R = D[t;\sigma,\delta]$ consists of a unital associative ring $D$, an injective endomorphism $\sigma$ of $D$, a left $\sigma$-derivation $\delta$ of $D$, and an indeterminate $t$ satisfying the commutation rule $ta = \sigma(a)t + \delta(a)$ for all $a \in D$. Since their introduction, skew polynomial rings have been extensively studied, and their properties are well understood (see for instance \cite[Chapter 2]{cohn1995skew} and \cite[Chapter 1]{jacobson1996finite}).

We now assume $D$ is a division ring. In this case, it is well-known $R$ possesses a right division algorithm, that is, for all $f(t), g(t) \in R$ with $f(t) \neq 0$ there exist unique $q(t), r(t) \in R$ with $\mathrm{deg}(r(t)) < \mathrm{deg}(f(t))$ such that $g(t) = q(t)f(t) + r(t)$. The existence of this right division algorithm allows us to construct a class of nonassociative algebras following a little known paper by Petit \cite{Petit1966-1967}: Let $f(t) \in R$ be of degree $m \geq 2$ and consider the additive subgroup $R_m = \{ g \in R \ \vert \ \mathrm{deg}(g) < m \}$ of $R$. Then $R_m$ together with the multiplication $a \circ b = ab \ \mathrm{mod}_r f$ yields a nonassociative algebra $S_f = (R_m, \circ)$ over $F = \{ c \in D \ \vert \ c \circ h = h \circ c \text{ for all } h \in S_f \}$. Here the juxtaposition $ab$ denotes multiplication in $R$, and $\mathrm{mod}_r f$ denotes the remainder after right division by $f(t)$. We call these algebras Petit algebras and also denote them $R/Rf$ when we wish to make clear the ring $R$ is used in their construction. After their introduction by Petit in 1967, these algebras were largely ignored until Wene \cite{wene2000finite} and more recently Lavrauw and Sheekey \cite{lavrauw2013semifields} studied them in the context of finite semifields. Earlier in 1906, the algebra $S_f$ with $f(t) = t^2-i \in \mathbb{C}[t;^{-}]$, $^{-}$ complex conjugation, appeared as the first example of a nonassociative division algebra in a paper by L.E. Dickson \cite{dickson1906commutative}. The structure of this thesis is as follows:

\vspace*{4mm}
In Chapter \ref{chapter:Preliminaries} we state the necessary definitions and notations. We describe the construction of Petit algebras and discuss how this relates to other known constructions of algebras. For example, if $D$ is a finite-dimensional central division algebra over a field $C$, $\sigma \vert_C$ is an automorphism of finite order $m$ and $f(t) = t^m - a \in D[t;\sigma]$, $a \in \mathrm{Fix}(\sigma)^{\times}$, is right invariant, then the associative algebra $S_f = D[t;\sigma]/D[t;\sigma]f$ is called a generalised cyclic algebra and denoted $(D,\sigma,a)$ \cite[\S 1.4]{jacobson1996finite}. This happens to be the quotient algebra.

A unital finite nonassociative division ring is called a semifield in the literature. It is well-known that every associative semifield is in fact a field, however, there are many examples of semifields which are not associative. These are called proper semifields. An important example of semifields which appear throughout this thesis are Jha-Johnson semifields (also called cyclic semifields) \cite{jha1989analog}, which generalise the Hughes-Kleinfeld and Sandler semifields. In a recent paper, Lavrauw and Sheekey proved that if $K$ is a finite field, $\sigma$ is an automorphism of $K$ and $f(t) \in K[t;\sigma]$ is irreducible, then $S_f$ is a Jha-Johnson semifield \cite[Theorem 15]{lavrauw2013semifields}. While each Jha-Johnson semifield is isotopic to some algebra $S_f$ \cite[Theorem 16]{lavrauw2013semifields}, it is not itself necessarily isomorphic to such an algebra $S_f$. In this thesis we will focus on those Jha-Johnson semifields which arise from Petit's algebra construction.

\vspace*{4mm}
In Chapter \ref{chapter:The Structure of Petit Algebras} we will move on to study the properties of Petit algebras, and prove results concerning their nuclei, center, eigenring, zero divisors and associativity. We pay particular attention to the question of when Petit algebras are division algebras, and show this is closely related to whether the polynomial $f(t) \in R$ used in their construction is irreducible. Indeed, we will prove $f(t)$ is irreducible if and only if $S_f$ is a (right) division algebra (Theorems \ref{thm:f(t) irreducible iff S_f right division} and \ref{thm:S_f_division_iff_irreducible}). We will also show that when $\sigma$ is not surjective and $f(t)$ is irreducible, then $S_f$ is a right but not left division algebra (Corollary \ref{cor:S_f right but not left division algebra} and Example \ref{example:S_f right but not left division algebra}).

\vspace*{4mm}
The connection between $f(t)$ being irreducible and $S_f$ being a (right) division algebra motivates the study of irreducibility criteria in skew polynomial rings in Chapter \ref{chapter:Irreducibility Criteria for Polynomials in a Skew Polynomial Ring}. The results we obtain, in conjunction with Theorems \ref{thm:f(t) irreducible iff S_f right division} and \ref{thm:S_f_division_iff_irreducible}, immediately yield criteria for some Petit algebras to be (right) division algebras. 

Irreducibility and factorisation in skew polynomial rings have been investigated before and algorithms for factoring skew polynomials over finite fields and $\mathbb{F}_q(y)$ have appeared already in \cite{caruso2012some}, \cite{caruso2017new}, \cite{giesbrecht1998factoring}, \cite{gomez2013computing} and \cite{leroy2012noncommutative}. We also mention the papers of Churchill and Zhang \cite{churchill2009irreducibility}, and Granja, Martinez and Rodriguez \cite{granja2014real}, which employ valuation theory to obtain an analogue of the Eisenstein criteria for skew polynomial rings. The methods we use in Chapter \ref{chapter:Irreducibility Criteria for Polynomials in a Skew Polynomial Ring}, however, are purely algebraic and build upon the ideas of Lam and Leroy \cite[Lemma 2.4]{lam1988vandermonde}, Jacobson \cite[\S 1.3]{jacobson1996finite} and Petit \cite[(17), (18), (19)]{Petit1966-1967}.

\vspace*{4mm}
In Chapter \ref{chapter:Isomorphisms Between some Petit Algebras} we investigate isomorphisms between Petit algebras $S_f$ and $S_g$, with $f(t), g(t) \in R = D[t;\sigma,\delta]$ and $\sigma$ an automorphism of $D$. We apply these results to study in detail the automorphism group of Petit algebras in Chapter \ref{chapter:Automorphisms of S_f}, focussing on the case where $S_f$ is a nonassociative cyclic algebra in Chapter \ref{chapter:Automorphisms of Nonassociative Cyclic Algebras}. Many of the results appearing in Chapter's \ref{chapter:Automorphisms of S_f} and \ref{chapter:Automorphisms of Nonassociative Cyclic Algebras} recently appeared in \cite{brownautomorphism2017}.

One of the main motivations for studying automorphisms of Petit algebras comes from the question how the automorphism groups of Jha-Johnson semifields look like. We are also motivated by a question by Hering \cite{hering1991fibrations}: Given a finite group $G$, does there exist a semifield such that $G$ is a subgroup of its automorphism group?

It is well-known that two semifields coordinatize the same Desarguesian projective plane if and only if they are isotopic, hence semifields are usually classified up to isotopy rather than up to isomorphism and in many cases their automorphism group is not known. We apply our results to obtain information on the automorphism groups of some Jha-Johnson semifields in Section \ref{section:Automorphisms of Jha-Johnson Semifields Obtained from Skew Polynomial Rings}, and in the special case where $S_f$ is a nonassociative cyclic algebra over a finite field in Section \ref{section:Automorphisms of Nonassociative Cyclic Algebras over Finite Fields}. In particular, we completely determine the automorphism group of a nonassociative cyclic algebra of prime degree over a finite field (Theorem \ref{thm:Automorphisms of nonassociative cyclic algebras over finite fields prime}): it is either a cyclic group, a dicyclic group, or the semidirect product of two cyclic groups.

\vspace*{4mm}
Next we look at a generalisation of Petit's algebra construction using the skew polynomial ring $S[t;\sigma,\delta]$, where $S$ is any associative unital ring, $\sigma$ is an injective endomorphism of $S$, and $\delta$ is a left $\sigma$-derivation of $S$. While $S[t;\sigma,\delta]$ is in general not right Euclidean (unless $S$ is a division ring), we are still able to right divide by polynomials $f(t) \in S[t;\sigma,\delta]$ whose leading coefficient is invertible (Theorem \ref{thm:generalised S_f euclidean division}). Therefore, when $f(t)$ has an invertible leading coefficient, it is possible to define the same algebra construction. We briefly study some of the properties of these algebras including their center, zero divisors and nuclei.

\vspace*{4mm}
Recall that a central simple algebra $A$ of degree $n$ over a field $F$ is a $G$-crossed product algebra if it contains a maximal subfield $M$ (i.e. $[M:F] = n$) that is a Galois extension of $F$ with Galois group $G$. Moreover, we say $A$ is a solvable $G$-crossed product algebra if $G = \mathrm{Gal}(M/F)$ is a solvable group. In Chapter \ref{chapter:G-Admissible Groups and Crossed Products}, we revisit a result on the structure of solvable crossed product algebras, due to both Petit \cite[\S 7]{Petit1966-1967} and a careful reading of Albert \cite[p.~186]{albert1939structure}. We write up a proof of Albert's result using generalised cyclic algebras following the approach of Petit. We note that none of Petit's results are proved in \cite[\S 7]{Petit1966-1967}. To do this we extend the definition of classical generalised cyclic algebras $(D,\sigma,d)$ to where $D$ need not be a division algebra. More specifically, we show a $G$-crossed product algebra is solvable if and only if it can be constructed as a finite chain of generalised cyclic algebras satisfying certain conditions.
We describe how the structure of the solvable group, that is its chain of normal subgroups, relates to the structure of the crossed product algebra. We also generalise \cite[\S 7]{Petit1966-1967} to central simple algebras which need neither be crossed product or division algebras.

We finish the Chapter by giving a recipe for constructing central division algebras containing a given finite abelian Galois field extension by forming a chain of generalised cyclic algebras. This generalises a result by Albert \cite[p.~186]{albert1939structure}, see also \cite[Theorem 2.9.55]{jacobson1996finite}, in which $G = \mathbb{Z}_2 \times \mathbb{Z}_2$.
\end{Preface}

\tableofcontents

\mainmatter

\reversemarginpar
\chapter{Preliminaries} \label{chapter:Preliminaries}
\section{Some Basic Definitions} \label{section:Some Basic Definitions}

Let $F$ be a field. An \textbf{algebra} $A$ over $F$ is an $F$-vector space together with a bilinear map $A \times A \rightarrow A, (x,y) \mapsto xy$, which we call the \textbf{multiplication} of $A$. $A$ is \textbf{associative} if the associative law $(xy)z = x(yz)$ holds for all $x,y,z \in A$. Our algebras are  \textbf{nonassociative} in the sense that we do not assume this law and we call algebras in which the associative law fails \textbf{not associative}. $A$ is called \textbf{unital} if it contains a multiplicative identity $1$. We assume throughout this thesis that our algebras are unital without explicitly saying so. Given an algebra $A$, the \textbf{opposite algebra} $A^{\mathrm{op}}$ is the algebra with the same elements and addition operator as $A$, but where multiplication is performed in the reverse order.

We say an algebra $0 \neq A$ is a \textbf{left} (resp. \textbf{right}) \textbf{division algebra}, if the left multiplication $L_a : x \mapsto ax$ (resp. the right multiplication $R_a : x \mapsto xa$) is bijective for all $0 \neq a \in A$ and $A$ is a \textbf{division algebra} if it is both a left and a right division algebra. A finite-dimensional algebra is a division algebra if and only if it has no non-trivial zero divisors \cite[p.~12]{schafer1966introduction}, that is $a b = 0$ implies $a = 0$ or $b = 0$ for all $a, b \in A$. Finite division algebras are also called \textbf{finite semifields} in the literature.

The \textbf{associator} of three elements of an algebra $A$ is defined to be $[x,y,z] = (xy)z-x(yz)$. We then define the \textbf{left nucleus} to be $\mathrm{Nuc}_l (A) = \{ x \in A \ \vert \ [x,A,A] = 0 \}$, the \textbf{middle nucleus} to be $\mathrm{Nuc}_m (A) = \{ x \in A \ \vert \ [A,x,A] = 0 \}$ and the \textbf{right nucleus} to be $\mathrm{Nuc}_r (A) = \{ x \in A \ \vert \ [A,A,x] = 0 \}$; their intersection $\mathrm{Nuc}(A) = \{ x \in A \ \vert \ [A,A,x] = [A,x,A] = [x,A,A] = 0 \}$ is the \textbf{nucleus} of $A$. $\mathrm{Nuc}(A)$, $\mathrm{Nuc}_l(A)$, $\mathrm{Nuc}_m(A)$ and $\mathrm{Nuc}_r(A)$ are all associative subalgebras of $A$. The \textbf{commutator} of $A$ is the set of all elements which commute with every other element,
$$\mathrm{Comm}(A) = \{ x \in A \ \vert \ xy = yx \text{ for all } y \in A \}.$$
The \textbf{center} of $A$ is $\mathrm{Cent}(A) = \mathrm{Nuc}(A) \cap \mathrm{Comm}(A).$

If $D$ is an associative division ring, an automorphism $\sigma: D \rightarrow D$ is called an \textbf{inner automorphism} if $\sigma = I_u : x \mapsto u x u^{-1}$ for some $u \in D^{\times}$. The \textbf{inner order} of an automorphism $\sigma$ of $D$, is the smallest positive integer $n$ such that $\sigma^n$ is an inner automorphism. If no such $n$ exists we say $\sigma$ has \textbf{infinite inner order}.

A \textbf{left (right) principal ideal domain} is a domain $R$ such that every left (right) ideal in $R$ is of the form $Rf$ ($fR$) for some $f \in R$. We say $R$ is a \textbf{principal ideal domain}, if it is both a left and a right principal ideal domain.

Let $R$ be a left principal ideal domain and $0 \neq a, b \in R$. Then there exists $d \in R$ such that $Ra + Rb = Rd$. This implies $a = c_1 d$ and $b = c_2 d$ for some $c_1, c_2 \in R$, so $d$ is a right factor of both $a$ and $b$. We denote this by writing $d \vert_r a$ and $d \vert_r b$. In addition, if $e \vert_r a$ and $e \vert_r b$, then $Ra \subset Re$, $Rb \subset Re$, hence $Rd \subset Re$ and so $e \vert_r d$. Therefore we call $d$ a \textbf{right greatest common divisor} of $a$ and $b$.
Furthermore, there exists $g \in R$ such that $Ra \cap Rb = Rg$. Then $a \vert_r g$ and $b \vert_r g$. Moreover, if $a \vert_r n$ and $b \vert_r n$ then $Rn \subset Rg$ and so $g \vert_r n$. We call $g$ the \textbf{least common left multiple} of $a$ and $b$.

The \textbf{field norm} $N_{K/F} : K \rightarrow F$ of a finite Galois field extension $K/F$ is given by
$$N_{K/F}(k) = \prod_{\sigma \in \mathrm{Gal}(K/F)} \sigma(k).$$
In particular, if $K/F$ is a cyclic Galois field extension of degree $m$ with Galois group generated by $\sigma$, then the field norm has the form
\begin{equation*}
N_{K/F}(k) = k \sigma(k) \cdots \sigma^{m-1}(k).
\end{equation*}


\section{Skew Polynomial Rings} \label{section:Skew Polynomial Rings}

Let $D$ be an associative ring, $\sigma$ be an injective endomorphism of $D$ and $\delta$ be a \textbf{left $\sigma$-derivation} of D, i.e. $\delta: D \rightarrow D$ is an additive map and satisfies 
$$\delta(ab) = \sigma(a) \delta(b) + \delta(a) b,$$
for all $a,b \in D$, in particular $\delta(1) = 0$. Furthermore, an easy induction yields
\begin{equation} \label{eqn:goodearl}
\delta(a^n) = \sum_{i=0}^{n-1} \sigma(a)^i \delta(a)a^{n-1-i},
\end{equation}
for all $a \in D$, $n \in \mathbb{N}$ \cite[Lemma 1.1]{goodearl1992prime}. The set $\mathrm{Const}(\delta) = \{ d \in D \ \vert \ \delta(d) = 0 \}$ of \textbf{$\delta$-constants} forms a subring of $D$, moreover if $D$ is a division ring, this is a division subring of $D$ \cite[p.~7]{jacobson1996finite}.

The following definition is due to Ore \cite{ore1933theory}:
\begin{definition}
The \textbf{skew polynomial ring} $R = D[t;\sigma,\delta]$ is the set of left polynomials $a_0 + a_1t + a_2 t^2 + \ldots + a_m t^m$ with $a_i \in D$, where addition is defined term-wise, and multiplication by $ta = \sigma(a)t + \delta(a)$.
\end{definition}

This multiplication makes $R$ into an associative ring \cite[p.~2-3]{jacobson1996finite}. If $\delta = 0$, then $D[t;\sigma,0] = D[t;\sigma]$ is called a \textbf{twisted polynomial ring}, and if $\sigma$ is the identity map, then $D[t;\mathrm{id},\delta] = D[t;\delta]$ is called a \textbf{differential polynomial ring}. For the special case that $\delta = 0$ and $\sigma = \mathrm{id}$, we obtain the usual left polynomial ring $D[t] = D[t;\mathrm{id},0]$. 
Skew polynomial rings are also called Ore extensions in the literature and their properties are well understood. For a thorough introduction to skew polynomial rings see for example \cite[Chapter 2]{cohn1995skew}, \cite[Chapter 1]{jacobson1996finite} and \cite{ore1933theory}. 

We briefly mention some definitions and properties of skew polynomials which will be useful to us: The associative and distributive laws in $R = D[t;\sigma,\delta]$ yield
\begin{equation} \label{eqn:mult in S_f 1}
t^n b = \sum_{j=0}^n S_{n,j}(b)t^j, \quad n \geq 0,
\end{equation}
and
\begin{equation} \label{eqn:mult in S_f 2}
(bt^n)(ct^m) = \sum_{j=0}^n b(S_{n,j}(c))t^{j+m},
\end{equation} 
for all $b, c \in D$, where the maps $S_{n,j}: D \rightarrow D$ are defined by the recursion formula
\begin{equation} \label{eqn:mult in S_f 3}
S_{n,j} = \delta S_{n-1,j} + \sigma S_{n-1,j-1},
\end{equation}
with $S_{0,0} = \mathrm{id}_D, S_{1,0} = \delta$, $S_{1,1} = \sigma$ and $S_{n,0} = \delta^n$ \cite[p.~2]{jacobson1996finite}. This means $S_{n,j}$ is the sum of all monomials in $\sigma$ and $\delta$ that are of degree $j$ in $\sigma$ and of degree $n-j$ in $\delta$. In particular $S_{n,n} = \sigma^n$ and if $\delta = 0$ then $S_{n,j} = 0$ for $n \neq j$.

We say $f(t) \in R$ is \textbf{right invariant} if $Rf$ is a two-sided ideal in $R$, $f(t)$ is \textbf{left invariant} if $fR$ is a two-sided ideal in $R$, and $f(t)$ is \textbf{invariant} if it is both right and left invariant. We define the \textbf{degree} of a polynomial $f(t) = a_m t^m + \ldots + a_1 t + a_0 \in R$ with $a_m \neq 0$ to be $\mathrm{deg}(f(t)) = m$ and $\mathrm{deg}(0) = - \infty$. We call $a_m$ the \textbf{leading coefficient} of $f(t)$. Then $\mathrm{deg}(g(t)h(t)) \leq \mathrm{deg}(g(t)) + \mathrm{deg}(h(t))$ for all $g(t), h(t) \in R$, with equality if $D$ is a domain. This implies that when $D$ is a domain, $D[t;\sigma,\delta]$ is also a domain.

Henceforth we assume $D$ is a division ring and remark that every endomorphism of $D$ is necessarily injective. Then $R = D[t;\sigma,\delta]$ is a left principal ideal domain and there is a \textbf{right division algorithm} in $R$ \cite[p.~3]{jacobson1996finite}. That is, for all $f(t), g(t) \in R$ with $f(t) \neq 0$ there exist unique $r(t), q(t) \in R$ with $\mathrm{deg}(r(t)) < \mathrm{deg}(f(t))$, such that $g(t) = q(t) f(t) + r(t)$. Here $r(t)$ is the remainder after right division by $f(t)$, and if $r(t) = 0$ we say $f(t)$ \textbf{right divides} $g(t)$ and write $f(t) \vert_r g(t)$. A polynomial $f(t) \in R$ is \textbf{irreducible} if it is not a unit and has no proper factors, i.e. there do not exist $g(t), h(t) \in R$ with $\mathrm{deg}(g(t)), \ \mathrm{deg}(h(t)) < \mathrm{deg}(f(t))$ such that $f(t) = g(t)h(t)$. Two non-zero $f(t), g(t) \in R$ are \textbf{similar} if there exist unique $h, q, u \in R$ such that $1 = hf + qg$ and $u'f = gu$ for some $u' \in R$. Notice if $f(t)$ is similar to $g(t)$ then $\mathrm{deg}(f(t)) = \mathrm{deg}(g(t))$ \cite[p.~14]{jacobson1996finite}.

If $\sigma$ is a ring automorphism, then $R$ is also a right principal ideal domain, (hence a principal ideal domain) \cite[Proposition 1.1.14]{jacobson1996finite}, and there exists a left division algorithm in $R$ \cite[p.~3 and Proposition 1.1.14]{jacobson1996finite}. \label{sigma automorphism right polynomial ring} In this case any right invariant polynomial $f(t)$ is invariant \cite[p.~6]{jacobson1996finite}, furthermore we can also view $R$ as the set of right polynomials with multiplication defined by $at = t \sigma^{-1}(a) - \delta(\sigma^{-1}(a))$ for all $a \in D$ \cite[(1.1.15)]{jacobson1996finite}.


\section{Petit's Algebra Construction} \label{section:Petit's Algebra Construction}

In this Section, we describe the construction of a family of nonassociative algebras $S_f$ built using skew polynomial rings. These algebras will be the focus of study of this thesis. They were first introduced in 1966 by Petit \cite{Petit1966-1967}, \cite{petit1968quasi}, and laregly ignored until Wene \cite{wene2000finite} and more recently Lavrauw and Sheekey \cite{lavrauw2013semifields} studied them in the context of semifields.

Let $D$ be an associative division ring with center $C$, $\sigma$ be an endomorphism of $D$ and $\delta$ be a left $\sigma$-derivation of $D$.

\begin{definition}[Petit \cite{Petit1966-1967}]
Let $f(t) \in R = D[t;\sigma,\delta]$ be of degree $m$ and $$R_m = \{ g \in R \ \vert \ \mathrm{deg}(g) < m \} .$$
Define a multiplication $\circ$ on $R_m$ by  $a \circ b = ab \ \mathrm{mod}_r f,$ where the juxtaposition $ab$ denotes multiplication in R, and $\mathrm{mod}_r f$ denotes the remainder after right division by $f(t)$. Then $S_f = (R_m, \circ)$ is a nonassociative algebra over $F = \{ c \in D \ \vert \ c \circ h = h \circ c \text{ for all } h \in S_f \}$. We also call the algebras $S_f$ \textbf{Petit algebras} and denote them by $R/Rf$ if we want to make it clear which ring $R$ is used in the construction.
\end{definition}

W.l.o.g., we may assume $f(t)$ is monic, since the algebras $S_f$ and $S_{df}$ are equal for all $d \in D^{\times}$. We obtain the following straightforward observations:

\begin{remarks}
\begin{itemize}
\item[(i)] If $f(t)$ is right invariant, then $S_f$ is the associative quotient algebra obtained by factoring out the two-sided ideal $Rf$.
\item[(ii)] If $\mathrm{deg}(g(t)) + \mathrm{deg}(h(t)) < m$, then the multiplication $g \circ h$ is the usual multiplication of polynomials in $R$.
\item[(iii)] If $\mathrm{deg}(f(t)) = 1$ then $R_1 = D$ and $S_f \cong D$. We will assume throughout this thesis that $\mathrm{deg}(f(t)) \geq 2$.
\end{itemize}
\end{remarks}

Note that $F$ is a subfield of $D$ \cite[(7)]{Petit1966-1967}. It is straightforward to see that $F = C \cap \mathrm{Fix}(\sigma) \cap \mathrm{Const}(\delta)$. Indeed, if $c \in F$ then $c \in D$ and $c \circ h = h \circ c$ for all $h \in S_f$, in particular this means $c \in C$. Furthermore, we have $c \circ t = t \circ c$ so that $ct = \sigma(c) t + \delta(c)$, hence $\sigma(c) = c$ and $\delta(c) = 0$ and thus $F \subset C \cap \mathrm{Fix}(\sigma) \cap \mathrm{Const}(\delta)$.

Conversely, if $c \in C \cap \mathrm{Fix}(\sigma) \cap \mathrm{Const}(\delta)$ and $h = \sum_{i=0}^{m-1} h_i t^i \in S_f$, then
\begin{align*}
hc &= \sum_{i=0}^{m-1} h_i t^i c = \sum_{i=0}^{m-1} h_i \sum_{j=0}^{i} S_{i,j}(c) t^j = \sum_{i=0}^{m-1} h_i c t^i = \sum_{i=0}^{m-1} c h_i t^i = ch,
\end{align*}
because $S_{i,j}(c) = 0$ for $i \neq j$ and $S_{i,i}(c) = \sigma^i(c) = c$. Therefore $c \in F$ as required. \label{page:F=Fix sigma} In the special case where $D$ is commutative and $\delta = 0$ then $F = \mathrm{Fix}(\sigma)$.

\begin{examples}
\begin{itemize}
\item[(i)] Let $f(t) = t^m - a \in D[t;\sigma]$ where $a \neq 0$. Then the multiplication in $S_f$ is given by
$$(bt^i) \circ (ct^j) = 
\begin{cases}
b \sigma^i (c) t^{i+j} & \text{ if } i + j < m, \\
b \sigma^i (c) \sigma^{i+j-m}(a) t^{i+j-m} & \text{ if } i + j \geq m,
\end{cases}$$
for all $b, c \in D$ and $i,j \in \{ 0, \ldots, m-1 \}$, then linearly extended.
\item[(ii)] Let $f(t) = t^2 - a_1 t - a_0 \in D[t;\sigma]$. Then multiplication in $S_f$ is given by 
\begin{align*}
(x + yt) \circ (u + vt) &= \big( x u + y \sigma(v) a_0 \big) + \big( x v + y \sigma(u) + y \sigma(v) a_1 \big) t,
\end{align*}
for all $x, y, u, v \in D$. By identifying $x + y t = (x,y)$ and $u + v t = (u,v)$, the multiplication in $S_f$ can also be written as
$$(x + yt) \circ (u + vt) = (x , y) \begin{pmatrix}
u & v \\
\sigma(v) a_0 & \sigma(u) + \sigma(v) a_1
\end{pmatrix}.$$
\item[(iii)] Let $f(t) = t^2 - a \in D[t;\sigma, \delta]$, then multiplication in $S_f$ is given by
$$(x + yt) \circ (u + vt) = (x , y) \begin{pmatrix}
u & v \\
\sigma(v) a + \delta(u) & \sigma(u) + \delta(v)
\end{pmatrix},$$
for all $x, y, u, v \in D$.
\end{itemize}
\end{examples}

When $\sigma$ is an automorphism there is also a left division algorithm in $R$ and can define a second algebra construction: Let $f(t) \in R$ be of degree $m$ and denote by $\mathrm{mod}_l f$ the remainder after left division by $f(t)$. Then $R_m$ together with the multiplication $a \prescript{}{f}\circ \ b = ab \ \mathrm{mod}_l f,$ becomes a nonassociative algebra $\prescript{}{f}S$ over $F$, also denoted $R/fR$. It suffices to study the algebras $S_f$, as every algebra $\prescript{}{f}S$ is the opposite algebra of some Petit algebra:

\begin{proposition} \label{prop:_fS opposite algebra of some S_g}
(\cite[(1)]{Petit1966-1967}).
Suppose $\sigma \in \mathrm{Aut}(D)$ and $f(t) \in D[t;\sigma,\delta]$. The canonical anti-isomorphism
$$\psi: D[t;\sigma,\delta] \rightarrow D^{\mathrm{op}}[t;\sigma^{-1},-\delta \sigma^{-1}], \ \ \sum_{i=0}^{n} a_i t^i \mapsto \sum_{i=0}^{n} \Big( \sum_{j=0}^{i} S_{n,j}(a_i) \Big) t^i,$$
between the skew polynomial rings $D[t;\sigma,\delta]$ and $D^{\mathrm{op}}[t;\sigma^{-1},-\delta \sigma^{-1}]$, induces an anti-isomorphism between $S_f = D[t;\sigma,\delta]/D[t;\sigma,\delta]f$, and
$$_{\psi(f)}S = D^{\mathrm{op}}[t;\sigma^{-1},-\delta \sigma^{-1}]/\psi(f)D^{\mathrm{op}}[t;\sigma^{-1},-\delta \sigma^{-1}].$$
\end{proposition}

\section{Relation of \texorpdfstring{$S_f$}{S\_f} to other Known Constructions} \label{section:Relation of S_f to Other Known Constructions}

We now show connections between Petit algebras and some other known constructions of algebras.

\subsection{Nonassociative Cyclic Algebras} \label{section:Nonassociative Cyclic Algebras}

Let $K/F$ be a cyclic Galois field extension of degree $m$ with $\mathrm{Gal}(K/F) = \langle \sigma \rangle$ and $f(t) = t^m - a \in K[t;\sigma]$. Then $$A = (K/F,\sigma,a) = K[t;\sigma]/K[t;\sigma]f(t)$$ is called a \textbf{nonassociative cyclic algebra of degree $m$} over $F$. The multiplication in $A$ is associative if and only if $a \in F$, in which case $A$ is a classical associative cyclic algebra over $F$. 

Nonassociative cyclic algebras were studied in detail by Steele in his Ph.D. thesis \cite{AndrewPhD}. We remark that our definition of nonassociative cyclic algebras yields the opposite algebras to the ones studied by Steele in \cite{AndrewPhD}. Moreover, if $K$ is a finite field, then $A$ is an example of a Sandler semifield \cite{sandler1962autotopism}. 

Nonassociative cyclic algebras of degree $2$ are \textbf{nonassociative quaternion algebras}. These algebras were first studied in 1935 by Dickson \cite{dickson1935linear} and subsequently by Althoen, Hansen and Kugler \cite{althoen1986c} over $\mathbb{R}$, however, the first systematic study was carried out by Waterhouse \cite{waterhouse}. 

Classical associative quaternion algebras of characteristic not $2$ are precisely associative cyclic algebras of degree $2$. That is, they have the form $(K/F,\sigma,a)$ where $K/F$ is a quadratic separable field extension with non-trivial automorphism $\sigma$, $\mathrm{char}(F) \neq 2$ and $a \in F^{\times}$. Thus the only difference in defining nonassociative quaternion algebras, is that the element $a$ belongs to the larger field $K$.

\subsection{(Nonassociative) Generalised Cyclic Algebras} \label{section:Generalised Cyclic Algebras}

Let $D$ be an associative division algebra of degree $n$ over its center $C$ and $\sigma$ be an automorphism of $D$ such that $\sigma \vert_C$ has finite order $m$ and fixed field $F = C \cap \mathrm{Fix}(\sigma)$. A \textbf{nonassociative generalised cyclic algebra} of degree $mn$, is an algebra $S_f = D[t;\sigma]/D[t;\sigma]f(t)$ over $F$ with $f(t) = t^m-a \in D[t;\sigma]$, $a \in D^{\times}$. We denote this algebra $(D,\sigma,a)$ and note that it has dimension $n^2 m^2$ over $F$. Note that when $D = K$ is a field, and $K/F$ is a cyclic field extension with Galois group generated by $\sigma$, we obtain the nonassociative cyclic algebra $(K/F,\sigma,a)$.

In the special case where $f(t) = t^m-a \in D[t;\sigma]$, $d \in F^{\times}$, then $f(t)$ is invariant and $(D,\sigma,a)$ is the associative \textbf{generalised cyclic algebra} defined by Jacobson \cite[p.~19]{jacobson1996finite}.

\subsection{(Nonassociative) Generalised Differential Algebras}

Let $C$ be a field of characteristic $p$ and $D$ be an associative central division algebra over $C$ of degree $n$. Suppose $\delta$ is a derivation of $D$ such that $\delta \vert_C$ is algebraic, that is there exists a $p$-polynomial
$$g(t) = t^{p^e} + c_1 t^{p^{e-1}} + \ldots + c_et \in F[t],$$
where $F = C \cap \mathrm{Const}(\delta)$ such that $g(\delta) = 0$. Suppose $g(t)$ is chosen with minimal $e$, $d \in D$ and $f(t) = g(t) - d \in D[t;\delta]$, then the algebra $S_f = D[t;\delta]/D[t;\delta]f(t)$ is called a \textbf{(nonassociative) generalised differential algebra} and also denoted $(D,\delta,d)$ \cite{pumpluen2016nonassociative}. $(D,\delta,d)$ is a nonassociative algebra over $F$ of dimension $p^{2e} n^2$, moreover $(D,\delta,d)$ is associative if and only if $d \in F$.

When $d \in F$, then $(D,\delta,d)$ is central simple over $F$ and is called the \textbf{generalised differential extension} of $D$ in \cite[p.~23]{jacobson1996finite}.

\vspace*{4mm}
\noindent For the remainder of this Section we consider examples of finite-dimensional division algebras over finite fields. These are called (finite) semifields in the literature. 

\subsection{Hughes-Kleinfeld and Knuth Semifields} \label{section:Hughes-Kleinfeld and Knuth Semifields}

Let $K$ be a finite field and $\sigma$ be a non-trivial automorphism of $K$. Choose $a_0, a_1 \in K$ such that the equation $w \sigma(w) + a_1w - a_0 = 0$ has no solution $w \in K$. In \cite[p.~215]{knuth1965finite}, Knuth defined four classes of semifields two-dimensional over $K$ with unit element $(1,0)$. Their multiplications are defined by
\begin{align*}
& (\text{K}1): (x,y)(u,v) = \big( x u + a_0 \sigma^{-2}(y) \sigma(v) , \ x v + y \sigma(u) + a_1 \sigma^{-1}(y) \sigma(v) \big) , \\
& (\text{K}2): (x,y)(u,v) = \big( x u + a_0 \sigma^{-2}(y) \sigma^{-1}(v) , \ x v + y \sigma(u) + a_1 \sigma^{-1}(y) v \big) , \\
& (\text{K}3): (x,y)(u,v) = \big( x u + a_0 y \sigma^{-1}(v) , \ x v + y \sigma(u) + a_1 y v \big) , \\
& (\text{K}4): (x,y)(u,v) = \big( x u + a_0 y \sigma(v) , \ x v + y \sigma(u) + a_1 y \sigma(v) \big) .
\end{align*}
The class of semifields defined by the multiplication $(\text{K}4)$ were first discovered by Hughes and Kleinfeld \cite{hughes1960seminuclear} and are called \textbf{Hughes-Kleinfeld semifields}.

The classes of semifields defined by the multiplications $(\text{K}2)$ and $(\text{K}4)$ can be obtained using Petit's construction:

\begin{theorem} \label{thm:S_f as Hughes Kleinfeld and Knuth Semifield}
(\cite[Theorem 5.15]{sheekey2011rank}).
Let $f(t) = t^2 - a_1 t - a_0 \in K[t;\sigma]$ where $w \sigma(w) + a_1w - a_0 \neq 0$ for all $w \in K$.
\begin{itemize}
\item[(i)] $\prescript{}{f}S$ is isomorphic to the semifield $(\text{K}2)$.
\item[(ii)] $S_f$ is isomorphic to the semifield $(\text{K}4)$.
\end{itemize}
\end{theorem}
There is a small mistake in \cite[p.~63 (5.1) and (5.2)]{sheekey2011rank} where the multiplication of two elements of $S_f$ is stated incorrectly. We give the full proof to avoid confusion.
\begin{proof}

\begin{itemize}
\item[(i)] The multiplication in $\prescript{}{f}S$ is given by
\begin{align*}
(x + y t) \prescript{}{f}\circ &(u + v t) = x u + x v t + y \sigma(u) t + t^2 \sigma^{-2}(y)\sigma^{-1}(v) \\
&= x u + x v t + y \sigma(u) t + (a_1 t + a_0) \sigma^{-2}(y)\sigma^{-1}(v) \\
&= x u + x v t + y \sigma(u) t + a_1 \sigma^{-1}(y) v t + a_0 \sigma^{-2}(y)\sigma^{-1}(v) \\
&= \big( x u + a_0 \sigma^{-2}(y)\sigma^{-1}(v) \big) + \big( x v + y \sigma(u) + a_1 \sigma^{-1}(y) v \big) t,
\end{align*}
for all $x, y, u, v \in K$. Therefore the map $\psi: \prescript{}{f}S \rightarrow (\text{K}2), \ x + y t \mapsto (x,y)$ can easily be seen to be an isomorphism.
\item[(ii)] The multiplication in $S_f$ is given by
\begin{align*}
(x + y t) \circ (u + vt) &= x u + x v t + y \sigma(u) t + y \sigma(v) t^2 \\
&= x u + x v t + y \sigma(u) t + y \sigma(v) (a_1 t + a_0) \\
&= \big( x u + y \sigma(v) a_0 \big) + \big( x v + y \sigma(u) + y \sigma(v) a_1 \big) t,
\end{align*}
for all $x, y, u, v \in K$. Therefore the map $\phi: S_f \rightarrow (\text{K}4), \ x + y t \mapsto (x,y)$ can readily be seen to be an isomorphism.
\end{itemize}
\end{proof}

\subsection{Jha-Johnson Semifields}

Jha-Johnson semifields, also called cyclic semifields, are built using irreducible semilinear transformations and generalise the Sandler and Hughes-Kleinfield semifields.

\begin{definition}
Let $K$ be a field. An additive map $T: V \rightarrow V$ on a vector space $V = K^m$ is called a \textbf{semilinear transformation} if there exists $\sigma \in \mathrm{Aut}(K)$ such that $T(\lambda v) = \sigma(\lambda)T(v),$ for all $\lambda \in K, v \in V$. The set of invertible semilinear transformations on $V$ forms a group called the \textbf{general semilinear group} and is denoted by $\Gamma \text{L}(V)$. An element $T \in \Gamma \text{L}(V)$ is said to be \textbf{irreducible}, if the only $T$-invariant subspaces of $V$ are $V$ and $\{ 0 \}$.
\end{definition}

Suppose now $K$ is a finite field.

\begin{definition}[\cite{jha1989analog}]
\newcommand*{\LargerCdot}{\raisebox{-0.25ex}{\scalebox{1.2}{$\cdot$}}}
Let $T \in \Gamma \text{L}(V)$ be irreducible and fix a $K$-basis $\{ e_0, \ldots , e_{m-1} \}$ of $V$. Define a multiplication $\LargerCdot$ on $V$ by
$$a \LargerCdot b = a(T)b = \sum_{i=0}^{m-1} a_i T^i(b),$$
where $a = \sum_{i=0}^{m-1} a_i e_i$. Then $\mathbb{S}_T = (V, \LargerCdot)$ defines a \textbf{Jha-Johnson semifield}.
\end{definition}

Let $L_{t,f}$ denote the semilinear transformation $v \mapsto tv \ \mathrm{mod}_r f$. 

\begin{theorem}  \label{thm:Jha-Johnson_is_S_f}
(\cite[ Theorems 15 and 16]{lavrauw2013semifields}).
If $\sigma$ is an automorphism of $K$ and $f(t) \in K[t;\sigma]$ is irreducible then $S_f \cong \mathbb{S}_{L_{t,f}}$. Conversely, if $T$ is any irreducible element of $\Gamma \text{L}(V)$ with automorphism $\sigma$, $\mathbb{S}_T$ is isotopic to $S_f$ for some irreducible $f(t) \in K[t;\sigma]$.
\end{theorem}

This means that every Petit algebra $S_f$ with $f(t) \in K[t;\sigma]$ irreducible is a Jha-Johnson semifield, and every Jha-Johnson semifield is isotopic to some $S_f$.

\chapter{The Structure of Petit Algebras} \label{chapter:The Structure of Petit Algebras}

In the following, let $D$ be an associative division ring with center $C$, $\sigma$ be an endomorphism of $D$, $\delta$ be a left $\sigma$-derivation of $D$ and $f(t) \in R = D[t;\sigma,\delta]$. Recall $S_f$ is a nonassociative algebra over $F = C \cap \mathrm{Fix}(\sigma) \cap \mathrm{Const}(\delta)$.

\section{Some Structure Theory} \label{section:Some Structure Theory}

In this Section we investigate the structure theory of Petit algebras. We begin by summarising some of the structure results stated by Petit in \cite{Petit1966-1967}:

\begin{theorem} \label{thm:Properties of S_f petit}
(\cite[(2), (5), (1), (14), (15)]{Petit1966-1967}). 
Let $f(t) \in R$ be of degree $m$.
\begin{itemize}
\item[(i)] If $S_f$ is not associative then
$$\mathrm{Nuc}_l(S_f) = \mathrm{Nuc}_m(S_f) = D,$$
and
$$\mathrm{Nuc}_r(S_f) = \{ g \in R \ \vert \ \mathrm{deg}(g) < m \text{ and } fg \in Rf \}.$$
\item[(ii)] The powers of $t$ are associative if and only if $t^m \circ t = t \circ t^m$ if and only if $t \in \mathrm{Nuc}_r(S_f)$.
\item[(iii)] $S_f$ is associative if and only if $f(t)$ is right invariant.
\item[(iv)] Suppose $\delta = 0$, then $\mathrm{Comm}(S_f)$ contains the set
\begin{equation} \label{eqn:Comm(S_f)}
\Big\{ \sum_{i=0}^{m-1} c_i t^i \ \vert \ c_i \in \mathrm{Fix}(\sigma) \text{ and } d c_i = c_i \sigma^i(d) \text{ for all } d \in D, \ i = 0, \ldots, m-1 \Big\}.
\end{equation}
If $t$ is left invertible the two sets are equal.
\item[(v)] Suppose $\delta = 0$ and $f(t) = t^m - \sum_{i=0}^{m-1} a_i t^i \in D[t;\sigma]$. Then $f(t)$ is right invariant if and only if $\sigma^m (z) a_i = a_i \sigma^i(z)$ and $\sigma(a_i) = a_i$ for all $z \in D$, $i \in \{ 0, \ldots, m-1 \}$.
\end{itemize}
\end{theorem}
The nuclei of $S_f$ were also calculated for special cases by Dempwolff in \cite[Proposition 3.3]{dempwolff2011autotopism}. 

\begin{remark}
(\cite[Remark 9]{pumplun2015finite}). If $f(t) = t^m - \sum_{i=0}^{m-1} a_i t^i \in D[t;\sigma]$, then $t$ is left invertible is equivalent to $a_0 \neq 0$. Indeed, if $a_0 = 0$ and there exists $g(t) \in R_m$ and $q(t) \in R$ such that $g(t)t = q(t)f(t)+1$, then the left side of the equation has constant term $0$, while the right hand side has constant term $1$, a contradiction. 

Conversely, if $a_0 \neq 0$ then defining $g(t) = a_0^{-1} t^{m-1} - \sum_{i=0}^{m-2} a_0^{-1} a_{i+1} t^i,$ we conclude $g(t)t = a_0^{-1}f(t) + 1$, therefore $g(t) \circ t = 1$ and $t$ is left invertible in $S_f$.
\end{remark}

\begin{corollary} \label{cor:Comm(S_f) = F}
Suppose $\sigma \in \mathrm{Aut}(D)$ is such that $\sigma \vert_C$ has order at least $m$ or infinite order, $f(t) \in D[t;\sigma]$ has degree $m$ and $t \in S_f$ is left invertible. Then
$\mathrm{Comm}(S_f) = F = C \cap \mathrm{Fix}(\sigma).$
\end{corollary}

\begin{proof}
$\mathrm{Comm}(S_f)$ is equal to the set \eqref{eqn:Comm(S_f)} by Theorem \ref{thm:Properties of S_f petit}(iv), in particular $F = C \cap \mathrm{Fix}(\sigma) \subseteq \mathrm{Comm}(S_f)$. Let now $\sum_{i=0}^{m-1} c_i t^i \in \mathrm{Comm}(S_f)$ and suppose, for contradiction, $c_j \neq 0$ for some $j \in \{ 1, \ldots, m-1 \}$. Then $b c_j = c_j \sigma^j(b)$ for all $b \in D$, thus $(b-\sigma^j(b))c_j = 0$ for all $b \in C$ and so $b - \sigma^j(b) = 0$ for all $b \in C$, a contradiction since $\sigma \vert_C$ has order $\geq m$. Therefore $\sum_{i=0}^{m-1} c_i t^i = c_0$ and $c_0 \in F$ by \eqref{eqn:Comm(S_f)}.
\end{proof}

\begin{proposition}
Let $L$ be a division subring of $D$ such that $\sigma \vert_L$ is an endomorphism of $L$ and $\delta \vert_L$ is a $\sigma \vert_L$-derivation of $L$. If $f(t) \in L[t;\sigma \vert_L,\delta \vert_L]$ then $L[t;\sigma \vert_L, \delta \vert_L] / L[t;\sigma \vert_L, \delta \vert_L]f(t)$ is a subring of $S_f = D[t;\sigma,\delta]/D[t;\sigma,\delta]f(t)$.
\end{proposition}

\begin{proof}
Clearly $L[t;\sigma \vert_L, \delta \vert_L] / L[t;\sigma \vert_L, \delta \vert_L]f(t)$ is a subset of $S_f$ and is a ring in its own right. Additionally $L[t;\sigma \vert_L, \delta \vert_L] / L[t;\sigma \vert_L, \delta \vert_L]f(t)$ inherits the multiplication in $S_f$ by the uniqueness of right division in $L[t;\sigma \vert_L, \delta \vert_L]$ and in $D[t;\sigma,\delta]$.
\end{proof}

Given $f(t) \in R = D[t;\sigma,\delta]$ of degree $m$, the \textbf{idealizer} $I(f) = \{ g \in R \ \vert \ fg \in Rf \}$ is the largest subalgebra of $R$ in which $Rf$ is a two-sided ideal. We then define the \textbf{eigenring} of $f(t)$ as the quotient $E(f) = I(f)/Rf$. Therefore the eigenring
$$E(f) = \{ g \in R \ \vert \ \mathrm{deg}(g) < m \text{ and } fg \in Rf \}$$
is equal to the right nucleus $\mathrm{Nuc}_r(S_f)$ by Theorem \ref{thm:Properties of S_f petit}(i), which as the right nucleus, is an associative subalgebra of $S_f$. By Theorem \ref{thm:Properties of S_f petit}(i) we obtain:

\begin{corollary} \label{cor:E(f)=S_f iff associative, S_f central}
Let $f(t) \in R$.
\begin{itemize}
\item[(i)] $E(f) = S_f$ if and only if $f(t)$ is right invariant if and only if $S_f$ is associative.
\item[(ii)] If $f(t)$ is not right invariant then $\mathrm{Cent}(S_f) = F$.
\end{itemize}
\end{corollary}

\begin{proof}
\begin{itemize}
\item[(i)] If $f(t)$ is not right invariant then $E(f) = \mathrm{Nuc}_r(S_f) \neq S_f$ by Theorem \ref{thm:Properties of S_f petit}(i). On the other hand if $f(t)$ is right invariant, then $Rf$ is a two-sided ideal, hence $f \vert_r fg$ for all $g \in R_m$ and so $E(f) = S_f$.
\item[(ii)] We have
\begin{align*}
\mathrm{Cent}(S_f) &= \mathrm{Comm}(S_f) \cap \mathrm{Nuc}(S_f) = \mathrm{Comm}(S_f) \cap D \cap E(f) \\ &= F \cap E(f) = F,
\end{align*}
by Theorem \ref{thm:Properties of S_f petit}.
\end{itemize}
\end{proof}

\begin{example}
If $K$ is a finite field, $\sigma$ is an automorphism of $K$ and $f(t) \in K[t;\sigma]$ is irreducible of degree $m$, then
$$E(f) = \{ u \ \mathrm{mod}_r f \ \vert \ u \in \mathrm{Cent}(K[t;\sigma]) \} \cong \mathbb{F}_{q^m}$$
by \cite[p.~9]{lavrauw2013semifields}.
\end{example}

\begin{theorem} \label{thm:Fix(sigma) right nucleus}
Let $f(t) = t^m - \sum_{i=0}^{m-1} a_i t^i \in D[t;\sigma,\delta]$ be such that $f(t)$ is not right invariant and $a_j \in \mathrm{Fix}(\sigma) \cap \mathrm{Const}(\delta)$ for all $j \in \{ 0, \ldots, m-1 \}$. Then the set
\begin{equation} \label{eqn:Fix(sigma) right nucleus}
\Big\{ \sum_{i=0}^{m-1} k_i t^i \ \vert \ k_i \in F = C \cap \mathrm{Fix}(\sigma) \cap \mathrm{Const}(\delta) \Big\},
\end{equation}
is contained in $\mathrm{Nuc}_r(S_f)$.
\end{theorem}

\begin{proof}
Clearly $F \subseteq \mathrm{Nuc}_r(S_f)$, therefore if we can show $t \in \mathrm{Nuc}_r(S_f)$, then \eqref{eqn:Fix(sigma) right nucleus} is contained in $\mathrm{Nuc}_r(S_f)$.
To this end, we calculate
$$t^m \circ t = \Big( \sum_{i=0}^{m-1} a_i t^i \Big) \circ t = \sum_{i=0}^{m-2} a_i t^{i+1} + a_{m-1} \sum_{i=0}^{m-1} a_i t^i,$$
and
\begin{align*}
t \circ t^m &= t \circ \Big( \sum_{i=0}^{m-1} a_i t^i \Big) = t \circ a_{m-1} t^{m-1} + t \circ \sum_{i=0}^{m-2} a_i t^i \\
&= \big( \sigma(a_{m-1})t + \delta(a_{m-1}) \big) \circ t^{m-1} + \sum_{i=0}^{m-2} \big( \sigma(a_i)t + \delta(a_i) \big) t^i \\
&= a_{m-1} \sum_{i=0}^{m-1} a_i t^i \ + \ \sum_{i=0}^{m-2} a_i t^{i+1},
\end{align*}
since $a_0, a_1, \ldots, a_{m-1} \in \mathrm{Fix}(\sigma) \cap \mathrm{Const}(\delta)$. Therefore $t^m \circ t = t \circ t^m$, which yields $t \in \mathrm{Nuc}_r(S_f)$ by Theorem \ref{thm:Properties of S_f petit}(ii).
\end{proof}

If we assume additionally $a_j \in C$ in Theorem \ref{thm:Fix(sigma) right nucleus}, i.e. we assume $f(t) \in F[t]$, then we obtain:

\begin{corollary} \label{cor:F right nucleus}
Let $f(t) \in F[t] = F[t;\sigma,\delta] \subset D[t;\sigma,\delta]$ be of degree $m$ and not right invariant. Then the set \eqref{eqn:Fix(sigma) right nucleus} is a commutative subalgebra of $\mathrm{Nuc}_r(S_f)$. Here \eqref{eqn:Fix(sigma) right nucleus} equals $F[t]/F[t]f(t)$. Furthermore, if $f(t)$ is irreducible in $F[t]$, the set \eqref{eqn:Fix(sigma) right nucleus} is a field.
\end{corollary}

\begin{proof}
$S_f$ contains the commutative subalgebra $F[t]/F[t]f(t)$ which is isomorphic to \eqref{eqn:Fix(sigma) right nucleus} because $f(t) \in F[t]$. Now, \eqref{eqn:Fix(sigma) right nucleus} is contained in $\mathrm{Nuc}_r(S_f)$ by Theorem \ref{thm:Fix(sigma) right nucleus} and thus if $f(t)$ is irreducible in $F[t]$, then $F[t]/F[t]f(t)$ is a field.
\end{proof}
When $\sigma = \mathrm{id}$, Corollary \ref{cor:F right nucleus} is precisely \cite[Proposition 2]{pumpluen2016nonassociative}.

\begin{remark}
Suppose $K$ is a finite field, $\sigma$ is an automorphism of $K$ and $F = \mathrm{Fix}(\sigma)$. If $f(t) \in F[t] \subset K[t;\sigma]$ is irreducible and not right invariant, then \eqref{eqn:Fix(sigma) right nucleus} is equal to $\mathrm{Nuc}_r(S_f)$ \cite[Theorem 3.2]{wene2000finite}.
\end{remark}

\section{So-called Right Semi-Invariant Polynomials} \label{section:Right Semi-Invariant Polynomials}

As in the previous Section, suppose $D$ is a division ring with center $C$, $\sigma$ is an endomorphism of $D$, $\delta$ is a left $\sigma$-derivation of $D$ and $f(t) \in R = D[t;\sigma,\delta]$. We now investigate conditions for $D$ to be contained in the right nucleus of $S_f$, therefore either $S_f$ is associative or $\mathrm{Nuc}(S_f) = D$ by Theorem \ref{thm:Properties of S_f petit}(i). We do this by looking at so-called right semi-invariant polynomials:

\begin{definition}
(\cite{lam1988algebraic}, \cite{lam1989invariant}).
A polynomial $f(t) \in R$ is called \textbf{right semi-invariant} if $f(t)D \subseteq Df(t)$. Similarly, $f(t)$ is \textbf{left semi-invariant} if $Df(t) \subseteq f(t)D$.
\end{definition}

We have $f(t)$ is right semi-invariant if and only if $df(t)$ is right semi-invariant for all $d \in D^{\times}$ \cite[p.~8]{lam1988algebraic}. For this reason it suffices to only consider monic $f(t)$. Furthermore, if $\sigma$ is an automorphism, then $f(t)$ is right semi-invariant if and only if it is left semi-invariant if and only if $f(t)D = Df(t)$ \cite[Proposition 2.7]{lam1988algebraic}.

For a thorough background on right semi-invariant polynomials we refer the reader to \cite{lam1988algebraic} and \cite{lam1989invariant}. Our interest in right semi-invariant polynomials stems from the following result:

\begin{theorem} \label{thm:semi-invariant iff D contained in E(f)}
$f(t) \in R$ is right semi-invariant if and only if $D \subseteq \mathrm{Nuc}_r(S_f)$. In particular, if $f(t)$ is right semi-invariant, then either $\mathrm{Nuc}(S_f) = D$ or $S_f$ is associative.
\end{theorem}

\begin{proof}
If $f(t) \in R$ is right semi-invariant, $f(t)D \subseteq Df(t) \subseteq Rf(t)$ and hence $D \subseteq E(f) = \mathrm{Nuc}_r(S_f)$. Conversely, if $D \subseteq \mathrm{Nuc}_r(S_f) = E(f)$ then for all $d \in D$, there exists $q(t) \in R$ such that $f(t)d = q(t)f(t)$. Comparing degrees, we see $q(t) \in D$ and thus $f(t)D \subseteq Df(t)$.

The second assertion follows by Theorem \ref{thm:Properties of S_f petit}(i).   
\end{proof}

\noindent The following result on the existence of a non-constant right semi-invariant polynomial is due to Lemonnier \cite{lemonnier1978dimension}:

\begin{proposition} \label{prop:lemonnier semi-invariant}
(\cite[(9.21)]{lemonnier1978dimension}).
Suppose $\sigma$ is an automorphism of $D$, then the following are equivalent:
\begin{itemize}
\item[(i)] There exists a non-constant right semi-invariant polynomial in $R$.
\item[(ii)] $R$ is not simple.
\item[(iii)] There exist $b_0, \ldots, b_n \in D$ with $b_n \neq 0$ such that $b_0 \delta_{c, \theta} + \sum_{i=1}^{n} b_i \delta^i = 0$, where $\theta$ is an endomorphism of $D$ and $\delta_{c,\theta}$ denotes the $\theta$-derivation of $D$ sending $x \in D$ to $c x - \theta(x) c$.
\end{itemize}
\end{proposition}

Combining Theorem \ref{thm:semi-invariant iff D contained in E(f)} and Proposition \ref{prop:lemonnier semi-invariant} we conclude:

\begin{corollary} \label{cor:simple and semi-invariant implies invariant}
Suppose $\sigma$ is an automorphism of $D$ and $R$ is simple. Then there are no nonassociative algebras $S_f$ with $D \subseteq \mathrm{Nuc}_r(S_f)$. In particular there are no nonassociative algebras $S_f$ with $D \subseteq \mathrm{Nuc}(S_f)$.
\end{corollary}

\begin{proof}
$R$ is not simple if and only if there exists a non-constant right semi-invariant polynomial in $R$ by Proposition \ref{prop:lemonnier semi-invariant}, and hence the assertion follows by Theorem \ref{thm:semi-invariant iff D contained in E(f)}.
\end{proof}

Theorem \ref{thm:semi-invariant iff D contained in E(f)} allows us to rephrase some of the results on semi-invariant polynomials in \cite{lam1988algebraic} and \cite{lam1989invariant}, in terms of the right nucleus of $S_f$:

\begin{theorem} \label{thm:right semi invariant conditions}
(\cite[Lemma 2.2, Corollary 2.12, Propositions 2.3 and 2.4]{lam1988algebraic}, \cite[Corollary 2.6]{lam1989invariant}).
Let $f(t) = \sum_{i=0}^{m} a_i t^i \in R$ be monic of degree $m$.
\begin{itemize}
\item[(i)] $D \subseteq \mathrm{Nuc}_r(S_f)$ if and only if $f(t)c = \sigma^m(c) f(t)$ for all $c \in D$, if and only if 
\begin{equation} \label{eqn:right semi-invariant 1}
\sigma^m(c)a_j = \sum_{i=j}^{m} a_i S_{i,j}(c)
\end{equation}
for all $c \in D$ and $j \in \{ 0, \ldots, m-1 \}$, where $S_{i,j}$ is defined as in \eqref{eqn:mult in S_f 3}.
\item[(ii)] Suppose $\sigma$ is an automorphism of $D$ of infinite inner order. Then $D \subseteq \mathrm{Nuc}_r(S_f)$ implies $S_f$ is associative.
\item[(iii)] Suppose $\delta = 0$. Then $D \subseteq \mathrm{Nuc}_r(S_f)$ if and only if
\begin{equation} \label{eqn:right semi-invariant 2}
\sigma^m(c) = a_j \sigma^j(c) a_j^{-1}
\end{equation}
for all $c \in D$ and all $j \in \{ 0, \ldots, m-1 \}$ with $a_j \neq 0$. Furthermore, $S_f$ is associative if and only if $f(t)$ satisfies \eqref{eqn:right semi-invariant 2} and $a_j \in \mathrm{Fix}(\sigma)$ for all $j \in \{ 0, \ldots, m-1 \}$.
\item[(iv)] Suppose $\delta = 0$ and $\sigma$ is an automorphism of $D$ of finite inner order $k$, i.e. $\sigma^k = I_u$ for some $u \in D^{\times}$. The polynomials $g(t) \in D[t;\sigma]$ such that $D \subseteq \mathrm{Nuc}_r(S_g)$ are precisely those of the form
\begin{equation} \label{eqn:right semi-invariant 4}
b \sum_{j=0}^{n} c_j u^{n-j}t^{jk},
\end{equation}
where $n \in \mathbb{N}$, $c_n = 1$, $c_j \in C$ and $b \in D^{\times}$. Furthermore, $S_g$ is associative if and only if $g(t)$ has the form \eqref{eqn:right semi-invariant 4} and $c_j u^{n-j} \in \mathrm{Fix}(\sigma)$ for all $j \in \{ 0, \ldots, n \}$. 
\item[(v)] Suppose $\sigma = \mathrm{id}$. Then $D \subseteq \mathrm{Nuc}_r(S_f)$ is equivalent to
\begin{equation} \label{eqn:right semi-invariant 3}
c a_j = \sum_{i=j}^{m} \binom{i}{j} a_i \delta^{i-j}(c),
\end{equation}
for all $c \in D$, $j \in  \{ 0, \ldots, m-1 \}$. Furthermore, $S_f$ is associative if and only if $f(t)$ satisfies \eqref{eqn:right semi-invariant 3} and $a_j \in \mathrm{Const}(\delta)$ for all $j \in \{ 0, \ldots, m-1 \}$.
\end{itemize}
\end{theorem}

Theorem \ref{thm:right semi invariant conditions}(iii) provides us with an alternate proof of \cite[Corollary 3.2.6]{AndrewPhD} about the nucleus of nonassociative cyclic algebras:

\begin{corollary} \label{cor:Nucleus of Nonassociative cyclic algebra}
(\cite[Corollary 3.2.6]{AndrewPhD}).
Let $A = (K/F,\sigma,a)$ be a nonassociative cyclic algebra of degree $m$ for some $a \in K \setminus F$. Then $\mathrm{Nuc}(A) = K$.
\end{corollary}

\begin{proof}
Notice $A = K[t;\sigma]/K[t;\sigma](t^m -a)$ and $t^m -a$ is right semi-invariant by Theorem \ref{thm:right semi invariant conditions}(iii). Hence $K \subseteq \mathrm{Nuc}_r(A)$ by Theorem \ref{thm:semi-invariant iff D contained in E(f)} since $A$ is not associative.
\end{proof}

%

Let $L$ be a division subring of $D$. Then we can look for conditions for $L \subseteq \mathrm{Nuc}_r(S_f)$ by generalising the definition of right semi-invariant polynomials as follows: We say $f(t) \in D[t;\sigma,\delta]$ \textbf{$L$-weak semi-invariant} if $f(t)L \subseteq D f(t)$. Clearly any right semi-invariant polynomial is also $L$-weak semi-invariant for every division subring $L$ of $D$. Moreover we obtain:

\begin{proposition} \label{prop:L-weak semi-invariant iff L subset Nuc_r(S_f)}
$f(t)$ is $L$-weak semi-invariant if and only if $L \subseteq E(f) = \mathrm{Nuc}_r(S_f)$. If $f(t)$ is $L$-weak semi-invariant but not right invariant, then $L \subseteq \mathrm{Nuc}(S_f) \subseteq D$.
\end{proposition}

\begin{proof}
If $f(t) \in R$ is $L$-weak semi-invariant, $f(t)L \subseteq Df(t) \subseteq Rf(t)$ and hence $L \subseteq E(f)$. Conversely, if $L \subseteq E(f)$ then for all $l \in L$, there exists $q(t) \in R$ such that $f(t)l = q(t)f(t)$. Comparing degrees, we see $q(t) \in D$ and thus $f(t)L \subseteq Df(t)$.

Hence if $f(t)$ is $L$-weak semi-invariant but not right invariant, then 
$$L \subseteq \mathrm{Nuc}(S_f) = E(f) \cap D \subseteq D$$
by Theorem \ref{thm:Properties of S_f petit}, which yields the second assertion.
\end{proof}

\begin{example}
Let $K$ be a field, $\sigma$ be a non-trivial automorphism of $K$, $L = \mathrm{Fix}(\sigma^j)$ be the fixed field of $\sigma^j$ for some $j >1$ and $f(t) = \sum_{i=0}^{n} a_i t^{ij} \in K[t;\sigma]$. Then
\begin{align*}
f(t)l &= \sum_{i=0}^{n} a_i t^{ij} l = \sum_{i=0}^{n} a_i \sigma^{ij}(l) t^{ij} = \sum_{i=0}^{n} a_i l t^{ij} = l f(t),
\end{align*}
for all $l \in L$ and hence $f(t) L \subseteq L f(t)$. In particular, $f(t)$ is $L$-weak semi-invariant.
\end{example}

It turns out that results similar to Theorem \ref{thm:right semi invariant conditions}(i), (iii) and (v) also hold for $L$-weak semi-invariant polynomials:

\begin{proposition} \label{prop:L-weak semi invariant conditions}
Let $f(t) = \sum_{i=0}^{m} a_i t^i \in D[t;\sigma,\delta]$ be monic of degree $m$ and $L$ be a division subring of $D$.
\begin{itemize}
\item[(i)] $f(t)$ is $L$-weak semi-invariant if and only if $f(t)c = \sigma^m(c)f(t)$ for all $c \in L$, if and only if
\begin{equation} \label{eqn:L-weak semi invariant conditions 1}
\sigma^m(c) a_j = \sum_{i=j}^{m} a_i S_{i,j}(c)
\end{equation}
for all $c \in L$, $j \in \{ 0, \ldots, m-1 \}$.
\item[(ii)] Suppose $\delta = 0$. Then $f(t)$ is $L$-weak semi-invariant if and only if  $\sigma^m(c) a_j = a_j \sigma^j(c)$ for all $c \in L$, $j \in \{ 0, \ldots, m-1 \}$.
\item[(iii)] Suppose $\sigma = \mathrm{id}$. Then $f(t)$ is $L$-weak semi-invariant if and only if
\begin{equation} \label{eqn:L-weak semi invariant conditions 2}
c a_j = \sum_{i=j}^{m} \binom{i}{j} a_i \delta^{i-j}(c)
\end{equation}
for all $c \in L$, $j \in \{ 0, \ldots, m-1 \}$.
\end{itemize} 
\end{proposition}

\begin{proof}
\begin{itemize}
\item[(i)] We have
\begin{equation} \label{eqn:L-weak semi invariant conditions 3}
f(t)c = \sum_{i=0}^{m} a_i t^i c = \sum_{i=0}^{m} a_i \sum_{j=0}^{i} S_{i,j}(c) t^j = \sum_{j=0}^{m} \sum_{i=j}^{m} a_i S_{i,j}(c) t^j
\end{equation}
for all $c \in L$, hence the $t^m$ coefficient of $f(t)c$ is $S_{m,m}(c) = \sigma^m(c)$, and so $f(t)$ is $L$-weak semi-invariant if and only if $f(t)c = \sigma^m(c)f(t)$ for all $c \in L$. Comparing the $t^j$ coefficient of \eqref{eqn:L-weak semi invariant conditions 3} and $\sigma^m(c)f(t)$ for all $j \in \{ 0, \ldots, m-1 \}$ yields \eqref{eqn:L-weak semi invariant conditions 1}.
\item[(ii)] When $\delta = 0$, $S_{i,j} = 0$ unless $i = j$ in which case $S_{j,j} = \sigma^j$. Therefore \eqref{eqn:L-weak semi invariant conditions 1} simplifies to $\sigma^m(c) a_j = a_j \sigma^j(c)$ for all $c \in L$, $j \in \{ 0, \ldots, m-1 \}$.
\item[(iii)] When $\sigma = \mathrm{id}$ we have
$$t^i c = \sum_{j=0}^{i} \binom{i}{j} \delta^{i-j}(c)$$
for all $c \in D$ by \cite[(1.1.26)]{jacobson1996finite} and thus
\begin{equation} \label{eqn:L-weak semi invariant conditions 4}
f(t) c = \sum_{i=0}^{m} a_i t^i c = \sum_{i=0}^{m} a_i \sum_{j=0}^{i} \binom{i}{j} \delta^{i-j}(c) t^j = \sum_{j=0}^{m} \sum_{i=j}^{m} \binom{i}{j} a_i \delta^{i-j}(c) t^j
\end{equation}
for all $c \in L$. Furthermore $f(t)$ is $L$-weak semi-invariant is equivalent to $f(t) c = c f(t)$ for all $c \in L$ by (i). Comparing the $t^j$ coefficient of \eqref{eqn:L-weak semi invariant conditions 4} and $c f(t) = \sum_{i=0}^{m} c a_i t^i$ for all $c \in L$, $j \in \{ 0, \ldots, m-1 \}$ yields \eqref{eqn:L-weak semi invariant conditions 2}.
\end{itemize}
\end{proof}

\section{When are Petit Algebras Division Algebras?} \label{section:When is S_f a Division Algebra?}

In this Section we look at conditions for Petit algebras to be right or left division algebras. This is closely linked to whether the polynomial $f(t)$ used in their construction is irreducible.

Given $f(t) \in R = D[t;\sigma,\delta]$, recall $S_f$ is a \textbf{right} (resp. \textbf{left}) \textbf{division algebra}, if the right multiplication $R_a : S_f \rightarrow S_f, \ x \mapsto x \circ a$, (resp. the left multiplication $L_a : S_f \rightarrow S_f, \ x \mapsto a \circ x$), is bijective for all $0 \neq a \in S_f$. Furthermore $S_f$ is a \textbf{division algebra} if it is both a right and a left division algebra. If $S_f$ is finite-dimensional over $F$, then $S_f$ is a division algebra if and only if it has no zero divisors \cite[p.~12]{schafer1966introduction}.

We say $f(t) \in R$ is \textbf{bounded} if there exists $0 \neq f^* \in R$ such that $Rf^* = f^* R$ is the largest two-sided ideal of $R$ contained in $Rf$. The element $f^*$ is determined by $f$ up to multiplication on the left by elements of $D^{\times}$.


%

The link between factors of $f(t)$ and zero divisors in the eigenring $E(f)$ is well-known:

\begin{proposition}
Let $f(t) \in R$.
\begin{itemize}
\item[(i)] (\cite[Proposition 4]{gomez2014basic}). If $f(t)$ is irreducible then $E(f)$ has no non-trivial zero divisors.
\item[(ii)] (\cite[Proposition 4]{gomez2014basic}). Suppose $\sigma$ is an automorphism and $f(t)$ is bounded. Then $f(t)$ is irreducible if and only if $E(f)$ has no non-trivial zero divisors.
\item[(iii)] (\cite[Theorem 3.3]{giesbrecht1998factoring}). If $D = \mathbb{F}$ is a finite field and $\delta = 0$, all polynomials are bounded and hence $f(t)$ is irreducible if and only if $E(f)$ is a finite field.
\end{itemize}
\end{proposition}

In general, the statement $f(t)$ is irreducible if and only if $E(f)$ has no non-trivial zero divisors is not true. Examples of reducible skew polynomials whose eigenrings are division algebras are given in \cite[Example 3]{gomez2014basic} and \cite{singer1996testing}. We prove the following result, stated but not proved by Petit in \cite[p.~13-07]{Petit1966-1967}:

\begin{proposition} \label{prop:f irreducible implies E(f) division}
If $f(t) \in R$ is irreducible then $E(f)$ is a division ring.
\end{proposition}

\begin{proof}
Let $\mathrm{End}_R(R/Rf)$ denote the endomorphism ring of the left $R$-module $R/Rf$, that is $\mathrm{End}_R(R/Rf)$ consists of all maps $\phi: R/Rf \rightarrow R/Rf$ such that $\phi(rh + r'h') = r \phi(h) + r' \phi(h')$ for all $r,r' \in R$, $h, h' \in R/Rf$.

Now $f(t)$ irreducible implies $R/Rf$ is a simple left $R$-module \cite[p.~15]{gomez2014basic}, therefore $\mathrm{End}_R(R/Rf)$ is an associative division ring by Schur's Lemma \cite[p.~33]{lam2013first}. Finally $E(f)$ is isomorphic to the ring $\mathrm{End}_R(R/Rf)$ \cite[p.~18-19]{gomez2014basic} and thus $E(f)$ is also an associative division ring.
\end{proof}

We now look at conditions for $S_f$ to be a right division algebra.

\begin{lemma} \label{lem:f(t) reducible implies S_f not division}
If $f(t) \in R$ is reducible, then $S_f$ contains zero divisors. In particular, $S_f$ is neither a left nor right division algebra.
\end{lemma}

\begin{proof}
Suppose $f(t) = g(t) h(t)$ for some $g(t), h(t) \in R$ with $\mathrm{deg}(g(t))$, $\mathrm{deg}(h(t)) < \mathrm{deg}(f(t))$, then $g(t) \circ h(t) = g(t)h(t) \ \mathrm{mod}_r f = 0$.
\end{proof}

Notice $S_f$ is a free left $D$-module of finite rank $m = \mathrm{deg}(f(t))$ and let $0 \neq a \in S_f$. Then $R_a(y+z) = (y+z) \circ a = (y \circ a) + (z \circ a) = R_a(y) + R_a(z)$ and
\begin{equation*} \label{eqn:R_a left D-linear}
R_a(k \circ z) = (k \circ z) \circ a = k \circ (z \circ a) = k \circ R_a(z),
\end{equation*}
for all $k \in D$, $y, z \in S_f$, since either $S_f$ is associative or has left nucleus equal to $D$ by Theorem \ref{thm:Properties of S_f petit}. Thus $R_a$ is left $D$-linear. We will require the following well-known Rank-Nullity Theorem:

\begin{theorem} \label{thm:Rank-Nullity}
(See for example \cite[Chapter IV, Corollary 2.14]{hungerford1980algebra}).
Let $S$ be a free left (resp. right) $D$-module of finite rank $m$ and $\phi:S \rightarrow S$ be a left (resp. right) $D$-linear map.
Then
\begin{equation*}
\mathrm{dim}(\mathrm{Ker}(\phi)) + \mathrm{dim}(\mathrm{Im}(\phi)) = m,
\end{equation*}
in particular, $\phi$ is injective if and only if it is surjective.
\end{theorem}

\begin{theorem} \label{thm:f(t) irreducible iff S_f right division}
(\cite[(6)]{Petit1966-1967}).
Let $f(t) \in R$ have degree $m$ and $0 \neq a \in S_f$. Then $R_a$ is bijective is equivalent to $1$ being a right greatest common divisor of $f(t)$ and $a$. In particular, $f(t)$ is irreducible if and only if $S_f$ is a right division algebra.
\end{theorem}

\begin{proof}
Let $0 \neq a \in S_f$. Since $S_f$ is a free left $D$-module of finite rank $m$ and $R_a$ is left $D$-linear, the Rank-Nullity Theorem \ref{thm:Rank-Nullity} implies $R_a$ is bijective if and only if it is injective which is equivalent to $\mathrm{Ker}(R_a) = \{ 0 \}$. Now $R_a(z) = z \circ a = 0$ is equivalent to $za \in Rf$, which means we can write
$$\mathrm{Ker}(R_a) = \{ z \in R_m \ \vert \ za \in Rf \}.$$
Furthermore, $R$ is a left principal ideal domain, which implies $za \in Rf$ if and only if $za \in Ra \cap Rf = Rg = Rha,$ where $g = ha$ is the least common left multiple of $a$ and $f$. Therefore $za \in Rf$ is equivalent to $z \in Rh$, and hence $\mathrm{Ker}(R_a) \neq \{ 0 \}$, if and only if there exists a polynomial of degree strictly less than $m$ in $Rh$, which is equivalent to $\mathrm{deg}(h) \leq m-1$.

Let $b \in R$ be a right greatest common divisor of $a$ and $f$. Then
$$\mathrm{deg}(f) + \mathrm{deg}(a) = \mathrm{deg}(g) + \mathrm{deg}(b) = \mathrm{deg}(ha) + \mathrm{deg}(b),$$
by \cite[Proposition 1.3.1]{jacobson1996finite}, and so $\mathrm{deg}(b) = \mathrm{deg}(f) - \mathrm{deg}(h).$ Thus $\mathrm{deg}(h) \leq m-1$ if and only if $\mathrm{deg}(b) \geq 1$, so we conclude $\mathrm{Ker}(R_a) = \{ 0 \}$ if and only if $\mathrm{deg}(b) = 0$, if and only if $1$ is a right greatest common divisor of $f(t)$ and $a$. In particular, this implies $S_f$ is a right division algebra if and only if $R_a$ is bijective for all $0 \neq a \in S_f$, if and only if $1$ is a right greatest common divisor of $f(t)$ and $a$ for all $0 \neq a \in S_f$, if and only if $f(t)$ is irreducible.
\end{proof}

We wish to determine when $S_f$ is also a left division algebra, hence when it is a division algebra.

\begin{proposition} \label{prop:S_f associative division iff irreducible}
If $f(t) \in R$ is right invariant, then $f(t)$ is irreducible if and only if $S_f$ is a division algebra.
\end{proposition}

\begin{proof}
Suppose $f(t)$ is right invariant so that $S_f$ is associative by Theorem \ref{thm:Properties of S_f petit}. If $f(t)$ is reducible then $S_f$ is not a division algebra by Lemma \ref{lem:f(t) reducible implies S_f not division}. Conversely, if $f(t)$ is irreducible the maps $R_b$ are bijective for all $0 \neq b \in S_f$ by Theorem \ref{thm:f(t) irreducible iff S_f right division}. This implies the maps $L_b$ are also bijective for all $0 \neq b \in S_f$ by \cite[Lemma 1B]{bruck1946contributions}, and so $S_f$ is a division algebra.
\end{proof}

\begin{lemma} \label{lem:L_a injective but not nec surjective}
If $f(t)$ is irreducible then $L_a$ is injective for all $0 \neq a \in S_f$.
\end{lemma}

\begin{proof}
If $f(t)$ is irreducible then $L_a(z) = a \circ z = R_z(a) = 0$ is impossible for $0 \neq z \in S_f$, as $R_z$ is injective by Theorem \ref{thm:f(t) irreducible iff S_f right division}. Thus $L_a$ is also injective.
\end{proof}

In general $L_a$ is neither left nor right $D$-linear. Therefore, when $f(t)$ is irreducible we cannot apply the Rank-Nullity Theorem to conclude $L_a$ is surjective, as we did for $R_a$ in the proof of Theorem \ref{thm:f(t) irreducible iff S_f right division}. In fact, the following Theorem shows that $L_a$ may not be surjective even if $f(t)$ is irreducible:

\begin{theorem} \label{thm:L_t surjective iff sigma surjective}
Let $f(t) = t^m - \sum_{i=0}^{m-1} a_i t^i \in D[t;\sigma]$ where $a_0 \neq 0$. Then for every $j \in \{ 1, \ldots, m-1 \}$, $L_{t^j}$ is surjective if and only if $\sigma$ is surjective. In particular, if $\sigma$ is not surjective then $S_f$ is not a left division algebra.
\end{theorem}

\begin{proof}
We first prove the result for $j = 1$: Given $z = \sum_{i=0}^{m-1} z_i t^i \in S_f$, we have
\begin{equation} \label{eqn:L_t surjective iff sigma surjective 1}
\begin{split}
L_t(z) &= t \circ z = \sum_{i=0}^{m-2} \sigma(z_{i})t^{i+1} + \sigma(z_{m-1})t \circ t^{m-1} \\ 
&= \sum_{i=1}^{m-1} \sigma(z_{i-1})t^i + \sigma(z_{m-1}) \sum_{i=0}^{m-1} a_i t^i.
\end{split}
\end{equation}
\begin{itemize}
\item[($\Rightarrow$)] Suppose $L_t$ is surjective, then given any $b \in D$ there exists $z \in S_f$ such that $t \circ z = b$. The $t^0$-coefficient of $L_t(z)$ is $\sigma(z_{m-1}) a_0$ by \eqref{eqn:L_t surjective iff sigma surjective 1}, and thus for all $b \in D$ there exists $z_{m-1} \in D$ such that $\sigma(z_{m-1}) a_0 = b$. Therefore $\sigma$ is surjective.
\item[($\Leftarrow$)] Suppose $\sigma$ is surjective and let $g = \sum_{i=0}^{m-1} g_i t^i \in S_f$. Define $$z_{m-1} = \sigma^{-1}(g_0 a_0^{-1} ), \ z_{i-1} = \sigma^{-1}(g_i) - z_{m-1} \sigma^{-1}(a_i)$$ for all $i \in \{ 1, \ldots , m-1 \}$. Then
\begin{equation*}
\begin{split}
L_t(z) &= \sigma(z_{m-1}) a_0 + \sum_{i=1}^{m-1} \big( \sigma(z_{i-1}) + \sigma(z_{m-1} a_i \big) t^i = \sum_{i=0}^{m-1} g_i t^i = g,
\end{split}
\end{equation*}
by \eqref{eqn:L_t surjective iff sigma surjective 1}, which implies $L_t$ is surjective.
\end{itemize}
Hence $L_t$ surjective is equivalent to $\sigma$ surjective. To prove the result for all $j \in \{ 1, \ldots, m-1 \}$ we show that
\begin{equation} \label{eqn:L_t surjective iff sigma surjective 2}
L_{t^j} = L_t^j,
\end{equation}
for all $j \in \{ 1, \ldots, m-1 \}$, then it follows $\sigma$ is surjective if and only if $L_t$ is surjective if and only if $L_t^j = L_{t^j}$ is surjective. In the special case when $D = \mathbb{F}_q$ is a finite field, $\sigma$ is an automorphism and $f(t)$ is monic and irreducible, the equality \eqref{eqn:L_t surjective iff sigma surjective 2} is proven in \cite[p.~12]{lavrauw2013semifields}. A similar proof also works more generally in our context: suppose inductively that $L_{t^j} = L_t^j$ for some $j \in \{ 1, \ldots, m-2 \}$. Then $L_t^j(b) = t^j b \ \mathrm{mod}_r f$ for all $b \in R_m$. Let $L_t^j(b) = b'$ so that $t^j b = qf + b'$ for some $q \in R$. We have
\begin{align*}
L_t^{j+1}(b) &= L_t(L_t^j(b)) = L_t(b') = L_t(t^j b - q f) = t \circ (t^j b - q f) \\ &= (t^{j+1} b - tqf) \ \mathrm{mod}_r f = t^{j+1}b \ \mathrm{mod}_r f = L_{t^{j+1}}(b),
\end{align*}
hence \eqref{eqn:L_t surjective iff sigma surjective 2} follows by induction.
\end{proof}

We can use Theorems \ref{thm:f(t) irreducible iff S_f right division} and \ref{thm:L_t surjective iff sigma surjective} to find examples of Petit algebras which are right but not left division algebras:

\begin{corollary} \label{cor:S_f right but not left division algebra}
Suppose $\sigma$ is not surjective and $f(t) \in D[t;\sigma]$ is irreducible. Then $S_f$ is a right division algebra but not a left division algebra.
\end{corollary}

\begin{example} \label{example:S_f right but not left division algebra}
Let $K$ be a field, $y$ be an indeterminate and define $\sigma: K(y) \rightarrow K(y)$ by $\sigma \vert_K = \mathrm{id}$ and $\sigma(y) = y^2$. Then $\sigma$ is an injective but not surjective endomorphism of $K(y)$ \cite[p.~123]{berrick2000introduction}. For $a(y) \in K[y]$ denote by $\mathrm{deg}_y(a(y))$ the degree of $a(y)$ as a polynomial in $y$.

Let $f(t) = t^2 - a(y) \in K(y)[t;\sigma]$ where $0 \neq a(y) \in K[y]$ is such that $3 \nmid \mathrm{deg}_y(a(y))$. We will show later in Corollary \ref{cor:t^2-a(y) in K(y)[t;sigma] irreducibility} that $f(t)$ is irreducible in $K(y)[t;\sigma]$, hence $S_f$ is a right, but not a left division algebra by Corollary \ref{cor:S_f right but not left division algebra}.
\end{example}

The following result was stated but not proved by Petit \cite[(7)]{Petit1966-1967}:
\begin{theorem} \label{thm:S_f_division_iff_irreducible}
(\cite[(7)]{Petit1966-1967}).
Let $f(t) \in D[t; \sigma, \delta]$ be such that $S_f$ is a finite-dimensional $F$-vector space or a right $\mathrm{Nuc}_r(S_f)$-module, which is free of finite rank. Then $S_f$ is a division algebra if and only if $f(t)$ is irreducible.
\end{theorem}

\begin{proof}
When $S_f$ is associative the assertion follows by Proposition \ref{prop:S_f associative division iff irreducible} so suppose $S_f$ is not associative. If $f(t)$ is reducible, $S_f$ is not a division algebra by Lemma \ref{lem:f(t) reducible implies S_f not division}. Conversely, suppose $f(t)$ is irreducible so that $S_f$ is a right division algebra by Theorem \ref{thm:f(t) irreducible iff S_f right division}. Let $0 \neq a \in S_f$ be arbitrary, then $L_a$ is injective for all $0 \neq a \in S_f$ by Lemma \ref{lem:L_a injective but not nec surjective}. We prove $L_a$ is surjective, hence $S_f$ is also a left division algebra:
\begin{itemize}
\item[(i)] Suppose $S_f$ is a finite-dimensional $F$-vector space. Then since $F \subseteq \mathrm{Nuc}(S_f)$, we have
$$L_a(k \circ z) = a \circ (k \circ z) = (a \circ k) \circ z = (k \circ a) \circ z = k \circ (a \circ z) = k \circ L_a(z)$$
and
$$L_a(z \circ k) = a \circ (z \circ k) = (a \circ z) \circ k = L_a(z) \circ k,$$
for all $k \in F$, $z \in S_f$. Therefore $L_a$ is $F$-linear, and thus $L_a$ is surjective by the Rank-Nullity Theorem \ref{thm:Rank-Nullity}.
\item[(ii)] Suppose $S_f$ is a free right $\mathrm{Nuc}_r(S_f)$-module of finite rank, then $E(f)$ is a division ring by Proposition \ref{prop:f irreducible implies E(f) division}. Furthermore, we have
$$L_a(z \circ k) = a \circ (z \circ k) = (a \circ z) \circ k = L_a(z) \circ k$$
for all $k \in \mathrm{Nuc}_r(S_f)$, $z \in S_f$ and so $L_a$ is right $\mathrm{Nuc}_r(S_f)$-linear. Therefore $L_a$ is surjective by the Rank-Nullity Theorem \ref{thm:Rank-Nullity}.
\end{itemize}
\end{proof}

%

\begin{theorem} \label{thm:L weak irreducible iff division}
Let $\sigma$ be an automorphism of $D$, $L$ be a division subring of $D$ such that $D$ is a free right $L$-module of finite rank, and $f(t) \in D[t;\sigma,\delta]$ be $L$-weak semi-invariant. Then $S_f$ is a division algebra if and only if $f(t)$ is irreducible.
In particular if $\sigma$ is an automorphism of $D$ and $f(t)$ is right semi-invariant then $S_f$ is a division algebra if and only if $f(t)$ is irreducible.
\end{theorem}

\begin{proof}
If $f(t)$ is reducible then $S_f$ is not a division algebra by Lemma \ref{lem:f(t) reducible implies S_f not division}. Conversely, suppose $f(t)$ is irreducible. Then $S_f$ is a right division algebra by Theorem \ref{thm:f(t) irreducible iff S_f right division} so we are left to show $S_f$ is also a left division algebra. Let $0 \neq a \in S_f$ be arbitrary and recall $L_a$ is injective by Lemma \ref{lem:L_a injective but not nec surjective}. Since $f(t)$ is $L$-weak semi-invariant, $L \subseteq \mathrm{Nuc}_r(S_f)$ which implies
$$L_a(z \circ \lambda) = a \circ (z \circ \lambda) = (a \circ z) \circ \lambda = L_a(z) \circ \lambda,$$ 
for all $z \in S_f$, $\lambda \in L$. Hence $L_a$ is right $L$-linear.

$S_f$ is a free right $D$-module of rank $m = \mathrm{deg}(f)$ because $\sigma$ is an automorphism. Since $D$ is a free right $L$-module of finite rank then also $S_f$ is a free right $L$-module of finite rank. Thus the Rank-Nullity Theorem \ref{thm:Rank-Nullity} implies $L_a$ is bijective as required.
\end{proof}


\section{Semi-Multiplicative Maps} \label{section:Semi-Multiplicative Maps}

\begin{definition}
A \textbf{map of degree $m$} over a field $F$, is a map $M: V \rightarrow W$ between two finite-dimensional vector spaces $V$ and $W$ over $F$, such that $M(\alpha v) = \alpha^m M(v)$ for all $\alpha \in F$, $v \in V$, and such that the map $M: V \times \cdots \times V \rightarrow W$ defined by
$$M(v_1, \ldots , v_m) = \sum_{1 \leq i_1 < \cdots < i_l \leq m} (-1)^{m-l} M(v_{i_1} + \ldots + v_{i_l}),$$
$(1 \leq l \leq m)$ is $m$-linear over $F$. A map $M: V \rightarrow F$ of degree $m$ is called a \textbf{form of degree $m$} over $F$.
\end{definition}

\begin{definition}
Consider a finite-dimensional nonassociative algebra $A$ over a field $F$ containing a subalgebra $D$. A map $M : A \rightarrow D$ of degree $m$ is called \textbf{left semi-multiplicative} if $M(d g) = M(d) M(g)$, for all $d \in D$, $g \in A$. \textbf{Right semi-multiplicative} maps are defined similarly.
\end{definition}

As before let $D$ be a division ring with center $C$, $\sigma$ be an endomorphism of $D$, $\delta$ be a left $\sigma$-derivation of $D$, and $f(t) \in D[t;\sigma,\delta]$ be of degree $m$. In his Ph.D. thesis \cite[\S 4.2]{AndrewPhD}, Steele defined and studied a left semi-multiplicative map on nonassociative cyclic algebras. In this Section, we show that when $D$ is commutative and $S_f$ is finite-dimensional over $F = C \cap \mathrm{Fix}(\sigma) \cap \mathrm{Const}(\delta)$, then we can similarly define a left semi-multiplicative map $M_f$ for $S_f$.

In the classical theory of associative central simple algebras of degree $n$, the reduced norm is a multiplicative form of degree $n$. The maps $M_f$ can be seen as a generalisation of the reduced norm.

\vspace*{4mm}
Consider $S_f$ as a free left $D$-module of rank $m = \mathrm{deg}(f(t))$ with basis $\{ 1,t,\ldots, t^{m-1} \}$, and recall the right multiplication $R_g: S_f \rightarrow S_f, \ h \mapsto h \circ g$ is left $D$-linear for all $0 \neq g \in S_f$ by the argument on page \pageref{eqn:R_a left D-linear}. Define
\begin{equation*}
\lambda: S_f \rightarrow \mathrm{End}_D(S_f), \ g \mapsto R_g,
\end{equation*}
which induces a map 
\begin{equation*}
\lambda: S_f \rightarrow \mathrm{Mat}_m(D), \ g \mapsto W_g,
\end{equation*}
where $W_g \in \mathrm{Mat}_m(D)$ is the matrix representing $R_g$ with respect to the basis $\{ 1,t,\ldots, t^{m-1} \}$. If we represent $h = h_0 + h_1t + \ldots + h_{m-1} t^{m-1} \in S_f$ as the row vector $(h_0, h_1, \ldots, h_{m-1})$ with entries in $D$, then we can write the product of two elements in $S_f$ as $h \circ g = h W_g$.

When $D$ is commutative, define $M_f: S_f \rightarrow D$ by $M_f(g) = \mathrm{det}(W_g)$. Notice this definition does not make sense unless $D$ is commutative, otherwise $W_g$ is a matrix with entries in the noncommutative ring $D$, and as such we cannot take its determinant.

\begin{proposition} \label{prop:f(t) two-sided then M_f is multiplicative}
Suppose $f(t)$ is right invariant, i.e. $S_f$ is associative, then $W_g W_h = W_{g \circ h}$ for all $g,h \in S_f$. In particular, if $D$ is commutative then $M_f$ is multiplicative.
\end{proposition}

\begin{proof}
We have 
$$y W_{g \circ h} = y \circ (g \circ h) = (y \circ g) \circ h = (y W_g) W_h = y (W_g W_h)$$
for all $y, g, h \in S_f$, where we have used the associativity in $S_f$ and the associativity of matrix multiplication. This means $W_g W_h = W_{g \circ h}$ for all $g, h \in S_f$. If $D$ is commutative, then
\begin{align*}
M_f(g \circ h) &= \mathrm{det}(W_{g \circ h}) = \mathrm{det}(W_g W_h) = \mathrm{det}(W_g) \mathrm{det}(W_h) = M_f(g) M_f(h)
\end{align*}
for all $g, h \in S_f$, therefore $M_f$ is multiplicative.
\end{proof}

In general, $W_{g \circ h} \neq W_g W_h$ for $g, h \in S_f$ unless $S_f$ is associative since the map $g \mapsto W_g$ is not an $F$-algebra homomorphism. Nevertheless we obtain:

\begin{proposition} \label{prop:W_d W_g = W_dg}
$W_d W_g = W_{d \circ g}$ for all $d \in D$, $g \in S_f$. In particular, if $D$ is commutative and $S_f$ is finite-dimensional over $F$, then $M_f$ is left semi-multiplicative.
\end{proposition}

\begin{proof}
Consider $d \in D$ as an element of $S_f$ so that $d = d + 0 t + \ldots + 0 t^{m-1}$. When $S_f$ is associative the assertion follows by Proposition \ref{prop:f(t) two-sided then M_f is multiplicative}, otherwise $D = \mathrm{Nuc}_m(S_f)$ by Theorem \ref{thm:Properties of S_f petit} and so
\begin{align*}
y W_{d \circ g} = y \circ (d \circ g) = (y \circ d) \circ g = (y W_d) W_g = y (W_d W_g)
\end{align*}
for all $d \in D$, $y, g \in S_f$. Thus $W_d W_g = W_{d \circ g}$.

If $D$ is commutative and $S_f$ is finite-dimensional over $F$, then
\begin{align*}
M_f(d \circ g) &= \mathrm{det}(W_{d \circ g}) = \mathrm{det}(W_d W_g) = \mathrm{det}(W_d) \mathrm{det}(W_g) = M_f(d) M_f(g)
\end{align*}
for all $d \in D$, $g \in S_f$ and so $M_f$ is left semi-multiplicative.
\end{proof}

\begin{examples}
\begin{itemize}
\item[(i)] Let $f(t) = t^2 - a_1 t - a_0 \in D[t;\sigma,\delta]$. Given $g = g_0 + g_1t \in S_f$ with $g_0, g_1 \in D$, the matrix $W_g$ has the form
$$\begin{pmatrix}
g_0 & g_1 \\
\sigma(g_1)a_0 + \delta(g_0) & \sigma(g_0) + \sigma(g_1)a_1 + \delta(g_1)
\end{pmatrix}.$$
\item[(ii)] Let $f(t) = t^m - a \in D[t; \sigma]$, then given $g = g_0 + g_1 t + \ldots + g_{m-1}t^{m-1} \in S_f$, $g_i \in D$, the matrix $W_g$ has the form
\begin{equation*}
W_g =
\begin{pmatrix}
g_0 & g_1 & g_2 & \cdots & g_{m-1} \\
\sigma(g_{m-1})a & \sigma(g_0) & \sigma(g_1) & \cdots & \sigma(g_{m-2}) \\
\sigma^2(g_{m-2})a & \sigma^2(g_{m-1}) \sigma(a) & \sigma^2(g_0) & \cdots & \sigma^2(g_{m-3}) \\
\sigma^3(g_{m-3})a & \sigma^3(g_{m-2}) \sigma(a) & \sigma^3(g_{m-1}) \sigma^2(a) & \cdots & \sigma^3(g_{m-4}) \\
\vdots & \vdots & \vdots &  & \vdots \\
\sigma^{m-1}(g_1) a & \sigma^{m-1}(g_2) \sigma(a) & \sigma^{m-1}(g_3) \sigma^2(a) & \cdots & \sigma^{m-1}(g_0)
\end{pmatrix}.
\end{equation*}
In other words, $W_g = (W_{ij})_{i,j = 0, \ldots, m-1}$, where
\begin{equation*}
W_{ij} = \begin{cases}
\sigma^i(g_{j-i}) & \text{ if } i \leq j, \\
\sigma^i(g_{m-i+j}) \sigma^j(a) & \text{ if } i > j. 
\end{cases}
\end{equation*}
If $D$ is a finite field, the matrix $W_g$ is the $(\sigma,a)$-circulant matrix $M_a^{\sigma}$ in \cite{fogarty2015circulant}. In this case, Proposition \ref{prop:W_d W_g = W_dg} is \cite[Remark 3.2(b)]{fogarty2015circulant}.
\end{itemize}
\end{examples}

We now look at the connection between $M_f$ and zero divisors in $S_f$:

\begin{theorem} \label{thm:semi-mult division}
Suppose $D$ is commutative.
\begin{itemize}
\item[(i)] Let $0 \neq g \in S_f$. If $g$ is not a right zero divisor in $S_f$ then $M_f(g) \neq 0$.
\item[(ii)] $S_f$ has no non-trivial zero divisors if and only if $M_f(g) \neq 0$ for all $0 \neq g \in S_f$, if and only if $S_f$ is a right division algebra.
\item[(iii)] If $S_f$ is a finite-dimensional left $F$-vector space or a free of finite rank right $\mathrm{Nuc}_r(S_f)$-module, then $S_f$ is a division algebra is equivalent to $M_f(g) \neq 0$ for all $0 \neq g \in S_f$.
\end{itemize}
\end{theorem}

\begin{proof}
\begin{itemize}
\item[(i)] If $W_g$ is a singular matrix, then the equation
$$h \circ g = (h_0, \ldots, h_{m-1}) W_g = 0,$$
has a non-trivial solution $(h_0, \ldots, h_{m-1}) \in D^{m}$, contradicting the assumption that $g$ is not a right zero divisor in $S_f$.
\item[(ii)] Suppose $M_f(g) \neq 0$ for all $0 \neq g \in S_f$. If $g, h \in S_f$ are non-zero and $h \circ g = h W_g = 0$ then $h = 0 W_g^{-1} = 0$, a contradiction. Hence $S_f$ has no non-trivial zero divisors. Conversely, if $S_f$ has no non-trivial zero divisors then $M_f(g) \neq 0$ for all $0 \neq g \in S_f$ by (i).

Additionally, $S_f$ contains no non-trivial zero divisors if and only if the right multiplication map $R_g : S_f \rightarrow S_f, \ x \mapsto x \circ g$ is injective for all non-zero $g \in S_f$, if and only if $S_f$ is a right division algebra by the proof of Theorem \ref{thm:f(t) irreducible iff S_f right division}.
\item[(iii)] Follows from (ii) and Theorem \ref{thm:S_f_division_iff_irreducible}.
\end{itemize}
\end{proof}


\chapter{Irreducibility Criteria in Skew Polynomial Rings} \label{chapter:Irreducibility Criteria for Polynomials in a Skew Polynomial Ring}

Let $D$ be a division ring with center $C$, $\sigma$ be an endomorphism of $D$ and $\delta$ be a left $\sigma$-derivation. Throughout this Chapter we assume without loss of generality $f(t) \in R = D[t;\sigma,\delta]$ is monic.

In Section \ref{section:When is S_f a Division Algebra?}, we saw that whether $S_f$ is a division algebra or not is closely linked to whether the polynomial $f(t)$ used in its construction is irreducible. For instance, $S_f$ is a right division algebra if and only if $f(t)$ is irreducible by Theorem \ref{thm:f(t) irreducible iff S_f right division}. This motivates the study of factorisation and irreducibility of skew polynomials which we do in the present Chapter. The results we obtain, in conjunction with Theorems \ref{thm:f(t) irreducible iff S_f right division} and \ref{thm:S_f_division_iff_irreducible}, yield criteria for some Petit algebras to be (right) division algebras.

It is well-known that a skew polynomial can always be factored as a product of irreducible skew polynomials. This factorisation is in general not unique, however, the degrees of the factors are unique up to permutation:

\begin{theorem} \label{thm:Ore factorisation is similar in pairs}
(\cite[Theorem 1]{ore1933theory}).
Every non-zero polynomial $f(t) \in R$ factorises as $f(t)  = f_1(t) \cdots f_n(t)$ where $f_i(t) \in R$ is irreducible for all $i \in \{ 1, \ldots, n \}$. Furthermore, if $f(t) = g_1(t) \cdots g_s(t)$ is any other factorisation of $f(t)$ as a product of irreducible $g_i \in R$, then $s = n$ and there exists a permutation $\pi: \{ 1, \ldots, n \} \rightarrow \{ 1, \ldots, n \}$ such that $f_i$ is similar to $g_{\pi(i)}$. In particular, $f_i$ and $g_{\pi(i)}$ have the same degree for all $i \in \{ 1, \ldots, n \}$.
\end{theorem}

We first restrict our attention to the case where $\delta = 0$.

\section{Irreducibility Criteria in \texorpdfstring{$R = D[t;\sigma]$}{R = D[t;sigma]}} \label{section:Irreducibility Criteria in D[t;sigma]}

Let $f(t) = t^m - \sum_{i=0}^{m-1} a_i t^i \in R = D[t;\sigma]$. In order to study when $f(t)$ is irreducible, we first determine the remainder after dividing $f(t)$ on the right by $(t-b), \ b \in D$. By \cite[p.~15]{jacobson1996finite} we have the identity 
\begin{equation} \label{eqn:Right division identity}
\begin{split}
t^i &- \sigma^{i-1}(b) \sigma^{i-2}(b) \cdots b \\ &= \Big(t^{i-1} + \sigma^{i-1}(b)t^{i-2} + \ldots + \sigma^{i-1}(b) \cdots \sigma(b) \Big)(t-b),
\end{split}
\end{equation}
for all $i \in \mathbb{N}$. Multiplying \eqref{eqn:Right division identity} on the left by $a_i$ and summing over $i$ yields 
$$f(t) = q(t)(t-b) + N_m (b) - \sum_{i=0}^{m-1} a_i N_i(b),$$
for some $q(t) \in R$, where $N_i (b) = \sigma^{i-1} (b) \cdots \sigma(b) b$ for $i > 0$ and $N_0 (b) = 1$. Therefore the remainder after dividing $f(t)$ on the right by $(t-b)$ is $N_m (b) - \sum_{i=0}^{m-1} a_i N_i(b)$,
and we conclude:

\begin{proposition} \label{prop:rightdivdegree1}
(\cite[p.~16]{jacobson1996finite}). $(t-b) \vert_r f(t)$ is equivalent to $$a_m N_m(b) - \sum_{i=0}^{m-1} a_i N_i (b) = 0.$$
\end{proposition}

When $\sigma$ is an automorphism of $D$, we can also determine the remainder after dividing $f(t)$ on the left by $(t-b), \ b \in D$: Similarly to \eqref{eqn:Right division identity} we have the identity
\begin{equation} \label{eqn:Left division identity}
\begin{split}
t^i - b \sigma^{-1}(b)& \cdots \sigma^{1-i}(b) = (t-b) \Big(t^{i-1 } + \sigma^{-1}(b)t^{i-2} \\ &+ \sigma^{-1}(b) \sigma^{-2}(b)t^{i-3} + \ldots + \sigma^{-1}(b) \sigma^{-2}(b) \cdots \sigma^{1-i}(b)\Big)
\end{split}
\end{equation}
for all $i \in \mathbb{N}$. Multiplying \eqref{eqn:Left division identity} on the right by $\sigma^{-i}(a_i)$, and using $a_i t^i = t^i \sigma^{-i} (a_i)$ gives
\begin{equation*}
\begin{split}
&a_i t^i - b \sigma^{-1}(b) \cdots \sigma^{1-i}(b) \sigma^{-i}(a_i) \\ 
&= (t-b) \Big(t^{i-1} + \sigma^{-1}(b)t^{i-2} + \ldots +  \sigma^{-1}(b) \sigma^{-2}(b) \cdots \sigma^{1-i}(b) \Big) \sigma^{-i}(a_i).
\end{split}
\end{equation*}
Summing over $i$, we obtain
$$f(t) = (t-b)q(t) + M_m (b) - \sum_{i=0}^{m-1} M_i(b) \sigma^{-i}(a_i),$$
for some $q(t) \in R$ where $M_i(b)$ are defined by $M_0 (b) = 1$, $M_1(b) = b$ and $M_i (b) = b \sigma^{-1}(b) \cdots \sigma^{1-i}(b)$ for $i \geq 2$. We immediately conclude:

\begin{proposition} \label{prop:leftdivdegree1}
Suppose $\sigma$ is an automorphism of $D$. Then $(t-b) \vert_l f(t) \ $ if and only if $ \ M_m (b) - \sum_{i=0}^{m-1} M_i(b) \sigma^{-i}(a_i) = 0.$
\end{proposition}
A careful reading of Propositions \ref{prop:rightdivdegree1} and \ref{prop:leftdivdegree1} yields the following:

\begin{corollary} \label{cor:left iff right divisor}
Suppose $\sigma$ is an automorphism and $f(t) = t^m - a \in D[t;\sigma]$. Then $f(t)$ has a left linear divisor if and only if it has a right linear divisor.
\end{corollary}

\begin{proof}
Let $b \in D$, then $(t-b) \vert_r f(t)$ is equivalent to $\sigma^{m-1}(b) \cdots \sigma(b) b = a$ by Proposition \ref{prop:rightdivdegree1}, if and only if $c \sigma^{-1}(c) \cdots \sigma^{1-m}(c) = a$ where $c = \sigma^{m-1}(b)$, if and only if $(t-c) \vert_l f(t)$ by Proposition \ref{prop:leftdivdegree1}.
\end{proof}

Using Propositions \ref{prop:rightdivdegree1} and \ref{prop:leftdivdegree1} we obtain criteria for some skew polynomials of degree two or three to be irreducible. The following was stated but not proven by Petit in \cite[(17), (18)]{Petit1966-1967}:

\begin{theorem} \label{thm:Petit_factor}
\begin{itemize}
\item[(i)] Suppose $\sigma$ is an endomorphism of $D$. Then $f(t) = t^2 - a_1 t - a_0 \in D[t;\sigma]$ is irreducible if and only if 
\begin{equation*} \label{eqn:Petit (17)}
\sigma(b)b - a_1 b - a_0 \neq 0,
\end{equation*}
for all $b \in D.$ 
\item[(ii)] Suppose $\sigma$ is an automorphism. Then $f(t) = t^3 - a_2 t^2 - a_1 t - a_0 \in D[t;\sigma]$ is irreducible if and only if 
\begin{equation*} \label{eqn:Petit (18) 1}
\sigma^2 (b) \sigma(b) b - \sigma^2 (b)\sigma(b) a_2 - \sigma^2 (b) \sigma(a_1) - \sigma^2 (a_0) \neq 0,
\end{equation*}
and 
\begin{equation*} \label{eqn:Petit (18) 2}
\sigma^2 (b) \sigma(b) b - a_2 \sigma(b) b - a_1 b - a_0 \neq 0,
\end{equation*}
for all $b \in D.$
\end{itemize}
\end{theorem}

\begin{proof}
\begin{itemize}
\item[(i)] Since $\mathrm{deg}(f(t)) = 2$, we have $f(t)$ is irreducible if and only if $(t-b) \nmid_r f(t)$ for all $b \in D$, if and only if
$$N_2 (b) - a_1 N_1 (b) - a_0 N_0(b) = \sigma(b) b - a_1 b - a_0 \neq 0,$$ 
for all $b \in D$ by Proposition \ref{prop:rightdivdegree1}.
\item[(ii)] Here $\mathrm{deg}(f(t)) = 3$ and so $f(t)$ is irreducible if and only if $(t-b) \nmid_r f(t)$ and $(t-b) \nmid_l f(t)$ for all $b \in D$, if and only if
\begin{equation*} \label{eqn:Petit_factor 1}
\sigma^2 (b) \sigma(b) b - a_2 \sigma(b)b - a_1 b - a_0 \neq 0,
\end{equation*}
and
\begin{equation} \label{eqn:Petit_factor 2}
b \sigma^{-1}(b) \sigma^{-2}(b) - b \sigma^{-1}(b) \sigma^{-2}(a_2) - b \sigma^{-1}(a_1) - a_0 \neq 0,
\end{equation}
for all $b \in D$ by Propositions \ref{prop:rightdivdegree1} and \ref{prop:leftdivdegree1}. Applying $\sigma^2$ to \eqref{eqn:Petit_factor 2} we obtain the assertion.
\end{itemize}
\end{proof}

When $f(t)$ has the form $f(t) = t^3 - a \in D[t;\sigma]$, we obtain the following simplification of Theorem \ref{thm:Petit_factor}(ii):

\begin{corollary} \label{cor:Irreducibility t^3-a}
Suppose $\sigma$ is an automorphism of $D$, then $f(t) = t^3 - a \in D[t;\sigma]$ is irreducible is equivalent to $\sigma^2 (b) \sigma(b) b \neq a$ for all $b \in D$.
\end{corollary}

\begin{proof}
Recall $f(t)$ has a right linear divisor if and only if it has a left linear divisor by Corollary \ref{cor:left iff right divisor}. Therefore $f(t)$ is irreducible if and only if $(t-b) \nmid_r f(t)$ for all $b \in D$, if and only if
$\sigma^2 (b) \sigma(b) b \neq a$ for all $b \in D$  by Proposition \ref{prop:rightdivdegree1}.
\end{proof}

\begin{corollary} \label{cor:t^2-a(y) in K(y)[t;sigma] irreducibility}
Let $K$ be a field, $y$ be an indeterminate and define $\sigma: K(y) \rightarrow K(y)$ by $\sigma \vert_K = \mathrm{id}$ and $\sigma(y) = y^2$. For $a(y) \in K[y]$ denote by $\mathrm{deg}_y(a(y))$ the degree of $a(y)$ as a polynomial in $y$. Let $f(t) = t^2 - a(y) \in K(y)[t;\sigma]$ where $0 \neq a(y) \in K[y]$ is such that $3 \nmid \mathrm{deg}_y(a(y))$. Then $f(t)$ is irreducible in $K(y)[t;\sigma]$.
\end{corollary}

\begin{proof}
Note that $\sigma$ is an injective but not surjective endomorphism of $K(y)$ by \cite[p.~123]{berrick2000introduction}. We have $f(t)$ is irreducible is equivalent to $\sigma(b(y))b(y) \neq a(y)$ for all $b(y) \in K(y)$ by Theorem \ref{thm:Petit_factor}. Given $0 \neq b(y) \in K(y)$, write $b(y) = c(y) / d(y)$ for some non zero $c(y), d(y) \in K[y]$.

If $\sigma(b(y))b(y) \notin K[y]$ then $\sigma(b(y))b(y) \neq a(y)$ because $a(y) \in K[y]$. Conversely suppose $\sigma(b(y))b(y) \in K[y]$ and let $c(y) = \sum_{i=0}^{l} \lambda_i y^i$, \ $d(y) = \sum_{j=0}^{n} \mu_j y^j$ for some $\lambda_i, \mu_j \in K$ with $\lambda_l, \mu_n \neq 0$. Then $\sigma(c(y)) = \sum_{i=0}^{l} \lambda_i y^{2i}$, \quad $\sigma(d(y)) = \sum_{j=0}^{n} \mu_j y^{2j}$, and so
\begin{align*}
\sigma(b(y))b(y) &= \frac{\sigma(c(y))c(y)}{\sigma(d(y))d(y)} = \frac{\sum_{i=0}^{l} \lambda_i y^{2i} \sum_{k=0}^{l} \lambda_k y^{k}}{\sum_{j=0}^{n} \mu_j y^{2j} \sum_{s=0}^{n} \mu_s y^{s}} = \frac{\sum_{i=0}^{l} \sum_{k=0}^{l} \lambda_i \lambda_k y^{2i+k}}{\sum_{j=0}^{n} \sum_{s=0}^{n} \mu_j \mu_s y^{2j+s}}.
\end{align*}
This means $\mathrm{deg}_y \big( \sigma(b(y))b(y) \big) = 3l - 3n$ is a multiple of $3$, thus if $3 \nmid \mathrm{deg}_y(a(y))$ then $\sigma(b(y))b(y) \neq a(y)$ and $f(t)$ is irreducible.
\end{proof}

Consider the field extension $\mathbb{C} / \mathbb{R}$ where $\mathrm{Gal}(\mathbb{C} / \mathbb{R}) = \{ \mathrm{id} , \sigma \}$ and $\sigma$ denotes complex conjugation. By \cite[Corollary 6]{pumplun2015factoring}, any non-constant $g(t) \in \mathbb{C}[t;\sigma]$ decomposes into a product of linear and irreducible quadratic skew polynomials, in particular, every polynomial of degree $\geq 3$ is reducible. As a Corollary of Theorem \ref{thm:Petit_factor}(i) we now give some irreducibility criteria for $f(t) \in \mathbb{C}[t;\sigma]$ of degree $2$:

\begin{corollary} \label{cor:Irreducibility_degree_four_complex}
\begin{itemize}
\item[(i)] Let $f(t) = t^2 - a \in \mathbb{C}[t;\sigma]$, then $f(t)$ is irreducible if and only if $a \in \mathbb{C} \setminus \mathbb{R}$ or $a \in \mathbb{R}^{-} = \{ r \in \mathbb{R} \ \vert \ r < 0 \}$.
\item[(ii)] \cite[Corollary 2.6]{bergen2015factorizations} Let $f(t) = t^2 - a_1 t - a_0 \in \mathbb{C}[t;\sigma]$ with $a_0, a_1 \in \mathbb{R}$, then $f(t)$ is irreducible in $\mathbb{C}[t;\sigma]$ if and only if $a_1^2 + 4a_0 < 0$ if and only if $f(t)$ is irreducible in $\mathbb{R}[t]$. Moreover, if $f(t)$ is reducible, then the factorisation of $f(t)$ into monic linear polynomials is unique when $a_1 \neq 0$, whereas $f(t)$ factors an infinite number of ways into monic linear factors when $a_1 = 0$.
\item[(iii)] Let $f(t) = t^2 - a_1 t - a_0 \in \mathbb{C}[t;\sigma]$ where $a_0 \in \mathbb{R}$ and $a_1 = \lambda + \mu i$ for some $\lambda , \mu \in \mathbb{R}^{\times}$. Then $f(t)$ is irreducible if and only if $a_0 < -(\lambda^2 + \mu^2)/4$. In particular, if $a_0 \geq 0$ then $f(t)$ is reducible.
\end{itemize}
\end{corollary} 

\begin{proof}
\begin{itemize}
\item[(i)] We have $f(t)$ is reducible is equivalent to $\sigma(b)b = a$ for some $b = c + d i \in \mathbb{C}$ by Theorem \ref{thm:Petit_factor}, which is equivalent to $c^2 + d^2 = a$ for some $c, d \in \mathbb{R}$. Therefore if $0 > a \in \mathbb{R}$ or $a \in \mathbb{C} \setminus \mathbb{R}$, then $f(t)$ must be irreducible. On the other hand, if $0 < a \in \mathbb{R}$ then setting $b = \sqrt{a}$ gives $\sigma(b)b = b^2=a$ as required.
\item[(iii)] We have $f(t)$ is reducible if and only if $\sigma(b)b - a_1 b - a_0 = 0$ for some $b = c + d i \in \mathbb{C}$ by Theorem \ref{thm:Petit_factor}, if and only if
\begin{align*}
c^2 + d^2 - \lambda c + \mu d - a_0 - (\lambda d + \mu c)i = 0,
\end{align*}
if and only if $c^2 + d^2 - \lambda c + \mu d - a_0 = 0$ and $\lambda d + \mu c = 0$ for some $c, d \in \mathbb{R}$. Therefore $f(t)$ is reducible if and only if
$$\Big( 1 + \frac{\mu^2}{\lambda^2} \Big) c^2 + \Big( - \lambda - \frac{\mu^2}{\lambda} \Big) c - a_0 = 0,$$
for some $c \in \mathbb{R}$, if and only if
$$\Big( - \lambda - \frac{\mu^2}{\lambda} \Big) ^2 + 4 a_0 \Big( 1 + \frac{\mu^2}{\lambda^2} \Big) \geq 0,$$
if and only if $a_0 \geq - (\lambda^2 + \mu^2)/4$.
\end{itemize}
\end{proof}

\begin{lemma} \label{lem:f(bt)=q(bt)g(bt)}
Let $f(t) \in R = D[t;\sigma]$ and suppose $f(t) = q(t) g(t)$ for some $q(t), g(t) \in R$. Then $f(bt) = q(bt) g(bt)$ for all $b \in F = C \cap \mathrm{Fix}(\sigma)$.
\end{lemma}

\begin{proof}
Write $q(t) = \sum_{i=0}^{l} q_i t^i$, \ $g(t) = \sum_{j=0}^{n} g_j t^j$, then 
$$f(t) = q(t) g(t) = \sum_{i=0}^{l} \sum_{j=0}^{n} q_i t^i g_j t^j = \sum_{i=0}^{l} \sum_{j=0}^{n} q_i \sigma^i(g_j) t^{i+j},$$
and so
\begin{align*}
q(bt) g(bt) &= \sum_{i=0}^{l} q_i (bt)^i \sum_{j=0}^{n} g_j (bt)^j = \sum_{i=0}^{l} \sum_{j=0}^{n} q_i \sigma^i(g_j) b^{i+j} t^{i+j} \\
&= \sum_{i=0}^{l} \sum_{j=0}^{n} q_i \sigma^i(g_j) (bt)^{i+j} = f(bt), \\
\end{align*}
for all $b \in F$.
\end{proof}

The following result was stated as Exercise by Bourbaki in \cite[p.~344]{bourbaki1973elements} and proven in the special case where $\sigma$ is an automorphism of order $m$ in \cite[Proposition 3.7.5]{cohn1995skew}:

\begin{theorem} \label{thm:bourbaki}
Let $\sigma$ be an endomorphism of $D$, $f(t) = t^m - a \in R = D[t;\sigma]$ and suppose $F = C \cap \mathrm{Fix}(\sigma)$ contains a primitive $m^{\text{th}}$ root of unity. If $g(t) \in R$ is a monic irreducible polynomial dividing $f(t)$ on the right, then the degree $d$ of $g(t)$ divides $m$ and $f(t)$ is the product of $m/d$ polynomials of degree $d$.
\end{theorem}

\begin{proof}
Let $g(t) \in R$ be a monic irreducible polynomial of degree $d$ dividing $f(t)$ on the right, and $\omega \in F$ be a primitive $m^{\text{th}}$ root of unity.

Define $g_{i} (t) = g(\omega^i t)$ for all $i \in \{ 0, \ldots, m-1 \}$. Then $\bigcap_{i=0}^{m-1} R g_{i}(t)$ is an ideal of $R$, and since $R$ is a left principle ideal domain, we have
\begin{equation} \label{eqn:bourbaki 1}
Rh(t) = \bigcap_{i=0}^{m-1} R g_{i}(t),
\end{equation}
for a suitably chosen $h(t) \in R$. Furthermore, we may assume $h(t)$ is monic, otherwise if $h(t)$ has leading coefficient $d \in D^{\times}$, then $Rh(t) = R(d^{-1}h(t))$. 

We show $f(t) \in Rh(t)$: As $g(t)$ right divides $f(t)$, we can write $f(t) = q(t)g(t)$ for some $q(t) \in R$. In addition, we have $(\omega t)^i = \omega^i t^i$ for all $i \in \{ 0, \ldots, m-1 \}$ because $\omega \in F$, therefore
$$f(\omega^i t) = \omega^{mi} t^m -a = t^m - a = f(t) = q(\omega^i t) g(\omega^i t),$$ 
by Lemma \ref{lem:f(bt)=q(bt)g(bt)} and so $g_{i}(t)$ right divides $f(t)$ for all $i \in \{ 0, \ldots, m-1 \}$. This means
$$f(t) \in \bigcap_{i=0}^{m-1} R g_{i}(t) = Rh(t),$$ 
in particular, $Rh(t)$ is not the zero ideal. 

We next show $h(\omega^i t) = h(t)$ for all $i \in \{ 0, \ldots , m-1 \}$: For simplicity we only do this for $i = 1$, the other cases are similar. Notice $h(t) \in \bigcap_{j=0}^{m-1} R g_{j}(t)$ by \eqref{eqn:bourbaki 1} and thus there exists $q_0(t), \ldots, q_{m-1}(t) \in R$ such that $h(t) = q_j(t) g_{j}(t)$, for all  $j \in \{ 0, \ldots , m-1 \}$. Therefore
$$h(\omega t) = q_{m-1}(\omega t) g_{m-1}(\omega t) = q_{m-1}(\omega t) g_0(t),$$
and
$$h(\omega t) = q_j(\omega t) g_{j}(\omega t) = q_j(\omega t) g_{j+1}(t) \in Rg_{j+1}(t),$$
for all $j \in \{ 0, \ldots , m-2 \}$ by Lemma \ref{lem:f(bt)=q(bt)g(bt)}, which implies
$$h(\omega t) \in \bigcap_{j=0}^{m-1} R g_{j}(t) = Rh(t).$$
As a result $h(\omega t) = k (t) h(t)$ for some $k(t) \in R$, and by comparing degrees, we conclude $0 \neq k(t) = k \in D$.
Suppose $h(t)$ has degree $l$ and write
$$h(t) = a_0 + \ldots + a_{l-1}t^{l-1} + t^l, \ a_j \in D,$$
here $Rg(t) \supseteq Rh(t)$ and $f(t) \in Rh(t)$ which yields $\mathrm{deg}(g(t)) = d \leq l \leq m$. Since $h(\omega t) = k h(t)$, we have $k t^l = (\omega t)^l = \big( \prod_{j=0}^{l-1} \sigma^j(\omega) \big) t^l = \omega^l t^l$ which implies $k = \omega^l$. Clearly, the coefficients $a_j$ must be zero for all $j \in \{ 1, \ldots, l-1 \}$, otherwise $a_j (\omega t)^j = k a_j t^j = \omega^l a_j t^j$ giving $\omega^j = \omega^l$, a contradiction as $\omega$ is a primitive $m^{\text{th}}$ root of unity. This means $h(t) = t^l + a_0$, and with
$$\omega^l t^l + a_0 = h(\omega t) = k h(t) = \omega^l (t^l + a_0) = \omega^l t^l + \omega^l a_0,$$
we obtain $\omega^l = 1$. This implies $l = m$ and $k = \omega^m = 1$, hence $h(\omega t) = h(t)$.

We next prove $h(t) = f(t)$: Now $f(t) \in Rh(t)$ implies $f(t) = t^m - a = p(t)(t^m + a_0)$ for some $p \in R$. Comparing degrees we see $p \in D^{\times}$, thus $t^m - a = p(t^m + a_0) = pt^m + p a_0$ which yields $p=1$ , $a_0 = -a$ and $f(t) = h(t)$.

Finally, $\bigcap_{i=0}^{m-1} R g_{i}(t) = Rf(t)$ is equivalent to $f(t)$ being the least common left multiple  of the $g_{i}(t)$, $i \in \{ 0, \ldots, m-1 \}$ \cite[p.~10]{jacobson1996finite}. As a result, we can write
$$f(t) = q_{i_r}(t) q_{i_{r-1}}(t) \cdots q_{i_1}(t),$$
by \cite[p.~496]{ore1933theory}, where $i_1 = 0 < i_2 < \ldots < i_r \leq m-1$ and each $q_{i_s}(t) \in R$ is similar to $g_{i_s}(t)$. Similar polynomials have the same degree \cite[p.~14]{jacobson1996finite} so $r = m/d$, and $f(t)$ factorises into $m/d$ irreducible polynomials of degree $d$.
\end{proof}

Theorem \ref{thm:bourbaki} implies the following result, which improves \cite[(19)]{Petit1966-1967} by making $\sigma^{m-1}(a) \neq \sigma^{m-1} (b) \cdots \sigma(b) b$ for all $b \in D$ a superfluous condition:

\begin{theorem} \label{thm:Petit(19)}
Suppose $m$ is prime, $\sigma$ is an endomorphism of $D$ and $F$ contains a primitive $m^{\text{th}}$ root of unity. Then $f(t) = t^m - a \in D[t;\sigma]$ is irreducible if and only if it has no right linear divisors, if and only if
$$a \neq \sigma^{m-1} (b) \cdots \sigma(b) b$$
 for all $b \in D$.
\end{theorem}

\begin{proof}
Let $g(t) \in D[t;\sigma]$ be an irreducible polynomial of degree $d$ dividing $f(t)$ on the right. Without loss of generality $g(t)$ is monic, otherwise if $g(t)$ has leading coefficient $c \in D^{\times}$, then $c^{-1}g(t)$ is monic and also right divides $f(t)$. Thus $d$ divides $m$ by Theorem \ref{thm:bourbaki} and since $m$ is prime, either $d = m$, in which case $g(t) = f(t)$, or $d = 1$, which means $f(t)$ can be written as a product of $m$ linear factors. Therefore $f(t)$ is irreducible if and only if $(t-b) \nmid_r f(t)$ for all $b \in D$, if and only if $a \neq \sigma^{m-1} (b) \cdots \sigma(b) b$, for all $b \in D$ by Proposition \ref{prop:rightdivdegree1}.
\end{proof}

%
%

\begin{lemma} \label{quantum plane sigma(b(y))=b(qy)}
Let $K$ be a field, $y$ be an indeterminate, and $\sigma$ be the automorphism of $K(y)$ such that $\sigma \vert_K = \mathrm{id}$ and $\sigma(y) = qy$ for some $1 \neq q \in K^{\times}$. Let $b(y) \in K(y)$ and write $b(y) = c(y)/d(y)$ for some $c(y), d(y) \in K[y]$ with $d(y) \neq 0$. Then $\sigma^j(b(y)) = c(q^j y)/ d(q^jy)$ for all $j \in \mathbb{N}$.
\end{lemma}

\begin{proof}
Write $c(y) = \lambda_0 + \lambda_1 y + \ldots + \lambda_l y^l$ and $d(y) = \mu_0 + \mu_1 y + \ldots + \mu_n y^n$ for some $\lambda_i, \mu_j \in K$, then
\begin{align*}
\sigma&^j(b(y)) = \frac{\sigma^j(c(y))}{\sigma^j(d(y))} 
= \frac{\sigma^j(\lambda_0) + \sigma^j(\lambda_1 y) + \ldots + \sigma^j(\lambda_l y^l)}{\sigma^j(\mu_0) + \sigma^j(\mu_1 y) + \ldots + \sigma^j(\mu_n y^n)} \\
&= \frac{\lambda_0 + \lambda_1 \sigma^j(y) + \ldots + \lambda_l \sigma^j(y^l)}{\mu_0 + \mu_1 \sigma^j(y) + \ldots + \mu_n \sigma^j(y^n)} = \frac{\lambda_0 + \lambda_1 q^j y + \ldots + \lambda_l (q^j y)^l}{\mu_0 + \mu_1 q^j y + \ldots + \mu_n (q^j y)^n} = b(q^j y).
\end{align*}
\end{proof}

\noindent For $a(y) \in K[y]$ denote $\mathrm{deg}_y (a(y))$ the degree of $a(y)$ as a polynomial in $y$.

\begin{corollary} \label{cor:t^m-a(y) in K(y)[t;sigma] irreducible}
Let $K(y)$ and $\sigma$ be as in Lemma \ref{quantum plane sigma(b(y))=b(qy)}. Suppose $m$ is prime, $K$ contains a primitive $m^{\text{th}}$ root of unity, and $f(t) = t^m - a(y) \in K(y)[t;\sigma]$ where $0 \neq a(y) \in K[y]$ is such that $m \nmid \mathrm{deg}_y (a(y))$. Then $f(t)$ is irreducible in $K(y)[t;\sigma]$.
\end{corollary}

\begin{proof}
We have $f(t)$ is irreducible if and only if
$$N_m(b(y)) = \sigma^{m-1}(b(y)) \cdots \sigma(b(y)) b(y) \neq a(y)$$
for all $b(y) \in K(y)$ by Theorem \ref{thm:Petit(19)}. Given $b(y) \in K(y)$, write $b(y) = c(y)/d(y)$ for some $c(y), d(y) \in K[y]$ with $d(y) \neq 0$, then
\begin{align*}
N_m(b(y)) &= \frac{c(q^{m-1}y) \cdots c(qy) c(y)}{d(q^{m-1}y) \cdots d(qy) d(y)}
\end{align*}
by Lemma \ref{quantum plane sigma(b(y))=b(qy)}. If $N_m(y) \notin K[y]$, we immediately conclude $N_m(b(y)) \neq a(y)$ because $a(y) \in K[y]$. Conversely if $N_m(b(y)) \in K[y]$, then
$$\mathrm{deg}_y(N_m(b(y))) = m \mathrm{deg}_y(c(y)) - m \mathrm{deg}_y(d(y))$$
for all $0 \neq b(y) \in K(y)$. Therefore $\mathrm{deg}_y(N_m(b(y)))$ is a multiple of $m$, thus $N_m(b(y)) \neq a(y)$ for all $b(y) \in K(y)$, and so $f(t)$ is irreducible.
\end{proof}

Recall $N_{K/F}(K^{\times}) \subseteq F^{\times}$ for any finite field extension $K/F$, therefore we obtain the following Corollary of Theorem's \ref{thm:Petit_factor} and \ref{thm:Petit(19)}:

\begin{corollary} \label{cor:nonassociative cyclic algebra is division}
Let $K/F$ be a cyclic Galois field extension of degree $m$ with $\mathrm{Gal}(K/F) = \langle \sigma \rangle$.
\begin{itemize}
\item[(i)] If $m=2$ then $f(t) = t^2 - a \in K[t;\sigma]$ is irreducible for all $a \in K \setminus F$.
\item[(ii)] If $m=3$ then $f(t) = t^3 - a \in K[t;\sigma]$ is irreducible for all $a \in K \setminus F$.
\item[(iii)] If $m$ is prime and $F$ contains a primitive $m^{\text{th}}$ root of unity then $f(t) = t^m - a \in K[t;\sigma]$ is irreducible for all $a \in K \setminus F$.
\end{itemize}
\end{corollary}

\begin{proof}
We have $\sigma^{m-1}(b) \cdots \sigma(b)b = N_{K/F}(b) \in F$, for all $b \in K$, hence the result follows by Corollary \ref{cor:Irreducibility t^3-a} and Theorems \ref{thm:Petit_factor} and \ref{thm:Petit(19)}.
\end{proof}

Recently in \cite[Theorem 3.1]{bergen2015factorizations}, it was shown that in the special case where $K$ is an algebraically closed field and $\sigma$ is an automorphism of $K$ of order $n \geq 2$, that every non-constant reducible skew polynomial in $K[t;\sigma]$ can be written as a product of irreducible skew polynomials of degree less than or equal to $n$. 
Notice that in this case, the Artin-Schreier Theorem implies $\mathrm{char}(K) = 0$ and $n=2$ \cite[p.~242]{lam2013first}. Therefore we can immediately improve \cite[Theorem 3.1]{bergen2015factorizations} to the following:

\begin{theorem} \label{thm:algebraically closed factorisation}
Let $K$ be an algebraically closed field and $\sigma$ be an automorphism of $K$ of order $n \geq 2$.  Then $\mathrm{char}(K) = 0$, $n = 2$ and every non-constant reducible skew polynomial in $K[t;\sigma]$ can be written as a product of linear and irreducible quadratic skew polynomials.
\end{theorem}

\vspace*{4mm}
We now extend some of our previous arguments to find criteria for skew polynomials of degree $4$ to be irreducible. We will see that the conditions for $f(t)$ to be irreducible become complicated when $f(t)$ has degree $4$.

Suppose $\sigma$ is an automorphism of $D$ and $f(t) = t^4 - a_3 t^3 - a_2 t^2 - a_1 t - a_0 \in R = D[t;\sigma]$. Then either $f(t)$ is irreducible, $f(t)$ is divisible by a linear factor from the right, from the left, or $f(t) = g (t) h(t)$ for some $g(t), h(t) \in R$ of degree $2$. In Propositions \ref{prop:rightdivdegree1} and \ref{prop:leftdivdegree1} we computed the remainders after dividing $f(t)$ by a linear polynomial on the right and the left. In order to obtain irreducibility criteria for $f(t)$, we wish to find the remainder after dividing $f(t)$ by $t^2 - c t - d \ , \ (c, d \in D)$ on the right. To do this we use the identities
\begin{equation} \label{eqn:degree 4 t^2 identity 1}
t^2 = (t^2 - c t - d) + (c t + d),
\end{equation}
\begin{equation} \label{eqn:degree 4 t^2 identity 2}
t^3 = (t + \sigma(c)) \big( t^2 - c t - d \big) + \big( \sigma(d) + \sigma(c)c \big) t + \sigma(c) d,
\end{equation}
and
\begin{equation} \label{eqn:degree 4 t^2 identity 3}
\begin{split}
t^4 &= \big( t^2 + \sigma^2 (c) t + \sigma^2 (d) + \sigma^2 (c) \sigma(c) \big) \big( t^2 - c t - d \big) \\ &+ \big( \sigma^2 (c) \sigma(c) c + \sigma^2 (d)c + \sigma^2 (c) \sigma(d) \big) t + \sigma^2 (d)d + \sigma^2 (c) \sigma(c) d.
\end{split}
\end{equation}
If we define
$$M_0 (c , d)(t) = 1, \ M_1 (c , d)(t) = t, \ M_2 (c , d)(t) = c t + d$$
$$M_3 (c , d)(t) = \big( \sigma(d) + \sigma(c)c \big)t + \sigma(c) d,$$
$$M_4 (c ,d)(t) = \big( \sigma^2 (c) \sigma(c) c + \sigma^2 (d)c + \sigma^2 (c) \sigma(d) \big) t + \sigma^2 (d)d + \sigma^2 (c) \sigma(c) d,$$
then multiplying \eqref{eqn:degree 4 t^2 identity 1}, \eqref{eqn:degree 4 t^2 identity 2} and \eqref{eqn:degree 4 t^2 identity 3} on the left by $a_i$ and summing over $i$ yields
$$f(t) = q(t) \big( t^2 - c t - d \big) + M_4 (c , d)(t) - \sum_{i=0}^{3} a_i M_i (c , d )(t)$$
for some $q(t) \in R$. This means the remainder after dividing $f(t)$ on the right by $(t^2 - c t - d)$ is
$$M_4 (c ,d )(t) - \sum_{i=0}^3 a_i M_i (c , d )(t),$$
which evidently implies:

\begin{proposition} \label{prop:degree 4 right divide by quadratic}
$(t^2 - c t - d) \vert_r f(t)$ is equivalent to
$$\sigma^2 (c) \sigma(c)c + \sigma^2 (d)c + \sigma^2 (c) \sigma(d) - a_3 \big( \sigma(d) + \sigma(c)c \big) - a_2 c - a_1 = 0,$$
and
$$\sigma^2 (d)d + \sigma^2 (c) \sigma(c) d - a_3 \sigma(c)d - a_2 d - a_0 = 0.$$
\end{proposition}

Together Propositions \ref{prop:rightdivdegree1}, \ref{prop:leftdivdegree1} and \ref{prop:degree 4 right divide by quadratic} yield:

\begin{theorem} \label{thm:degree 4 irreducibility criteria}
$f(t)$ is irreducible if and only if 
\begin{equation} \label{eqn:degree 4 irreducible 1}
\sigma^3 (b) \sigma^2 (b) \sigma(b) b + a_3 \sigma^2 (b) \sigma(b) b + a_2 \sigma(b) b + a_1 b + a_0 \neq 0,
\end{equation}
and
\begin{equation} \label{eqn:degree 4 irreducible 2}
\begin{split}
&\sigma^3 (b) \sigma^2 (b) \sigma (b) b + \sigma^3 (b) \sigma^2 (b) \sigma (b) a_3 \\ &+ \sigma^3 (b) \sigma^2 (b) \sigma (a_2) + \sigma^3 (b) \sigma^2 (a_1) + \sigma^3 (a_0) \neq 0,
\end{split}
\end{equation}
for all $b \in D$, and for every $ c , d \in D$, we have
\begin{equation} \label{eqn:degree 4 irreducible 3}
\sigma^2 (c) \sigma(c)c + \sigma^2 (d)c + \sigma^2 (c) \sigma(d) + a_3 (\sigma(d) + \sigma(c)c) + a_2 c + a_1 \neq 0,
\end{equation}
 or
\begin{equation} \label{eqn:degree 4 irreducible 4}
\sigma^2 (d)d + \sigma^2 (c) \sigma(c) d + a_3 \sigma(c)d + a_2 d + a_0 \neq 0.
\end{equation}
i.e., $f(t)$ is irreducible if and only if \eqref{eqn:degree 4 irreducible 1} and \eqref{eqn:degree 4 irreducible 2} and (\eqref{eqn:degree 4 irreducible 3} or \eqref{eqn:degree 4 irreducible 4}) holds.
\end{theorem}

\begin{proof}
$f(t)$ is irreducible if and only if $(t-b) \nmid_r f(t)$ for all $b \in D$, $(t-b) \nmid_l f(t)$ for all $b \in D$ and $(t^2 - c t - d) \nmid_r f(t)$ for all $c, d \in D$. Therefore the result follows from Propositions \ref{prop:rightdivdegree1}, \ref{prop:leftdivdegree1} and \ref{prop:degree 4 right divide by quadratic}.
\end{proof}

We briefly consider the special case where $f(t)$ has the form $f(t) = t^4 - a \in R$:

\begin{lemma} \label{lem:t^4-afactorisation1}
Let $f(t) = t^4 - a \in R$. Suppose $(t-b) \vert_r f(t)$, then
$$f(t) = (t + \sigma^3(b))(t^2 + \sigma^2(b) \sigma(b))(t - b),$$
and
$$f(t) = (t^2 + \sigma^3(b) \sigma^2(b))(t + \sigma(b))(t-b),$$
are factorisations of $f(t)$. In particular $(t + \sigma(b))(t-b) = t^2 -\sigma(b)b$ also right divides $f(t)$.
\end{lemma}

\begin{proof}
Multiplying out these factorisations gives $t^4 - \sigma^3(b) \sigma^2(b) \sigma(b) b$ which is equal to $f(t)$ by Proposition \ref{prop:rightdivdegree1}.
\end{proof}

Lemma \ref{lem:t^4-afactorisation1} implies that if $f(t) = t^4 - a$ has a right linear divisor then it also has a right quadratic divisor. Therefore in this case Theorem \ref{thm:degree 4 irreducibility criteria} simplifies to:

\begin{theorem} \label{thm:t^4-a irreducibility criteria delta=0}
$f(t) = t^4 - a \in R$ is reducible if and only if 
$$\sigma^2 (c) \sigma(c)c + \sigma^2 (d)c + \sigma^2 (c) \sigma(d) = 0 \quad \text{and} \quad \sigma^2 (d)d + \sigma^2 (c) \sigma(c) d =  a,$$
for some $c, d \in D.$
\end{theorem}

\begin{proof}
Recall $f(t)$ has a right linear divisor if and only if it has a left linear divisor by  Corollary \ref{cor:left iff right divisor}. Moreover if $f(t)$ has a right linear divisor then it also has a quadratic right divisor by Lemma \ref{lem:t^4-afactorisation1}, therefore $f(t)$ is reducible if and only if $(t^2 - c t - d) \vert_r f(t)$ for some $c, d \in D$. The result now follows from Proposition \ref{prop:degree 4 right divide by quadratic}.
\end{proof}

%



\section{Irreducibility Criteria in Skew Polynomial Rings over Finite Fields}

Let $K = \mathbb{F}_{p^h}$ be a finite field of order $p^h$ for some prime $p$ and $\sigma$ be a non-trivial $\mathbb{F}_p$-automorphism of $K$. This means 
$\sigma: K \rightarrow K, \ k \mapsto k^{p^r},$
for some $r \in \{ 1, \ldots, h-1 \}$ is a power of the Frobenius automorphism. Here $\sigma$ has order $n = h/ \mathrm{gcd}(r,h)$. Algorithms for efficiently factorising polynomials in $K[t;\sigma]$ exist, see \cite{giesbrecht1998factoring} or more recently \cite{caruso2017new}, however our methods are purely algebraic and employ the previously developed theory. All of our previous results from Section \ref{section:Irreducibility Criteria in D[t;sigma]} hold in $K[t;\sigma]$. In this Section we focus on polynomials of the form $f(t) = t^m - a \in K[t;\sigma]$.

We will require the following well-known result:
\begin{lemma} \label{lem:gcd number theory result}
$\mathrm{gcd}(p^h-1,p^r-1) = p^{\mathrm{gcd}(h,r)}-1.$
\end{lemma}

\begin{proof}
Let $d = \mathrm{gcd}(r,h)$ so that $h = dn$. We have
$$p^{h}-1 = (p^d-1)(p^{d(n-1)} + \ldots + p^d+1),$$
therefore $p^h-1$ is divisible by $p^d-1$. A similar argument shows $(p^d-1) \vert (p^r-1)$. Suppose that $c$ is a common divisor of $p^h-1$ and $p^r-1$, this means $p^h \equiv p^r \equiv 1 \ \mathrm{mod} \ (c)$. Write $d = hx + ry$ for some integers $x,y$, then we have
$$p^d = p^{hx + ry} = (p^h)^x (p^r)^y \equiv 1 \ \mathrm{mod} \ (c)$$
which implies $c \vert (p^d-1)$ and hence $p^d-1 = \mathrm{gcd}(p^h-1,p^r-1)$.
\end{proof}

Given $k \in K^{\times}$, we have $k \in \mathrm{Fix}(\sigma)$ if and only if $k^{p^r-1} = 1$, if and only if $k$ is a $(p^r-1)^{\mathrm{th}}$ root of unity. It is well-known there are $\mathrm{gcd}(p^r-1,p^h-1)$ such roots of unity in $K$, see for example \cite[Proposition II.2.1]{koblitz1994course}, thus $$\vert \mathrm{Fix}(\sigma) \vert = \mathrm{gcd}(p^r-1,p^h-1) + 1 = p^{\mathrm{gcd}(r,h)}$$ by Lemma \ref{lem:gcd number theory result} and so $\mathrm{Fix}(\sigma) \cong \mathbb{F}_{q}$ where $q = p^{\mathrm{gcd}(r,h)}$. \label{page:Fix(sigma) isomorphic to}

\begin{proposition}
\begin{itemize}
\item[(i)] Suppose $n \in \{ 2,3 \}$, then $f(t) = t^n - a \in K[t;\sigma]$ is irreducible if and only if $a \in K \setminus \mathrm{Fix}(\sigma)$. In particular there are precisely $p^{h} - q$ irreducible polynomials in $K[t;\sigma]$ of the form $t^n-a$ for some $a \in K$.
\item[(ii)] Suppose $n$ is a prime and $n \vert (q-1)$. Then $f(t) = t^n - a \in K[t;\sigma]$ is irreducible if and only if $a \in K \setminus \mathrm{Fix}(\sigma)$. In particular there are precisely $p^{h} - q$ irreducible polynomials in $K[t;\sigma]$ of the form $t^n-a$ for some $a \in K$.
\end{itemize}
\end{proposition}

\begin{proof}
Here $\sigma$ has order $n$ and $K/\mathrm{Fix}(\sigma)$ is a cyclic Galois field extension of degree $n$ with Galois group generated by $\sigma$.
\begin{itemize}
\item[(i)] $f(t)$ is irreducible if and only if
$\prod_{l=0}^{n-1} \sigma^l(b) = N_{K/\mathrm{Fix}(\sigma)}(b) \neq a$ for all $b \in K$ by Theorem \ref{thm:Petit_factor} or Corollary \ref{cor:Irreducibility t^3-a}, where $N_{K/\mathrm{Fix}(\sigma)}$ is the field norm. It is well-known that as $K$ is a finite field, $N_{K/\mathrm{Fix}(\sigma)}: K^{\times} \rightarrow \mathrm{Fix}(\sigma)^{\times}$ is surjective and so $f(t)$ is irreducible if and only if $a \notin \mathrm{Fix}(\sigma)$. There are $p^{h} - q$ elements in $K \setminus \mathrm{Fix}(\sigma)$, hence there are precisely $p^{h} - q$ irreducible polynomials of the form $t^n-a$ for some $a \in K$.
\item[(ii)] Notice $\mathrm{Fix}(\sigma) \cong \mathbb{F}_{q}$ contains a primitive $n^{\mathrm{th}}$ root of unity because $n \vert (q-1)$ \cite[Proposition II.2.1]{koblitz1994course}. The rest of the proof is similar to (i) but using Theorem \ref{thm:Petit(19)}.
\end{itemize}
\end{proof}

Let $a, b \in K$ and recall $(t-b) \vert_r (t^m-a)$ is equivalent to
$$a = \sigma^{m-1}(b) \cdots \sigma(b)b = b^s$$
by Proposition \ref{prop:rightdivdegree1} where $s = \sum_{j=0}^{m-1}p^{rj} = (p^{mr}-1)/(p^r-1)$. Suppose $z$ is a \textbf{primitive element} of $K$, that is $z$ generates the multiplicative group $K^{\times}$. Writing $b = z^l$ for some $l \in \mathbb{Z}$ yields $(t-b) \vert_r (t^m-a)$ if and only if $a = z^{ls}$. This implies the following:

\begin{proposition} \label{prop:finite fields irreducibility primitive element}
Let $f(t) = t^m-a \in K[t;\sigma]$ and write $a \in K$ as $a = z^u$ for some $u \in \{ 0, \ldots, p^h-2 \}$. 
\begin{itemize}
\item[(i)] $(t-b) \nmid_r f(t)$ for all $b \in K$ if and only if 
 $u \notin \mathbb{Z} s \ \mathrm{mod} \ (p^h-1).$
\item[(ii)] If $m \in \{2, 3\}$ then $f(t)$ is irreducible if and only if
 $u \notin \mathbb{Z} s \ \mathrm{mod} \ (p^h-1).$
\item[(iii)] Suppose $m$ is a prime divisor of $(q-1)$, then $f(t)$ is irreducible if and only if  $u \notin \mathbb{Z} s \ \mathrm{mod} \ (p^h-1).$
\end{itemize}
\end{proposition}

\begin{proof}
\begin{itemize}
\item[(i)] $(t-b) \nmid_r f(t)$ for all $b \in K$ if and only if $a = z^u \neq z^{l s}$ for all $l \in \mathbb{Z}$, if and only if $u \notin \mathbb{Z} s \ \mathrm{mod} \ (p^h-1)$.
\item[(ii)] $f(t)$ has a left linear divisor if and only if it has a right linear divisor by Corollary \ref{cor:left iff right divisor}. Therefore if $m \in \{2, 3 \}$ then $f(t)$ is irreducible if and only if $(t-b) \nmid_r f(t)$ for all $b \in K$ and so the assertion follows by (i).
\item[(iii)] If $m$ is a prime divisor of $(q-1)$ then $\mathrm{Fix}(\sigma) \cong \mathbb{F}_{q}$ contains a primitive $m^{\mathrm{th}}$ root of unity. Therefore the result follows by (i) and Theorem \ref{thm:Petit(19)}.
\end{itemize}
\end{proof}

\begin{corollary} \label{cor:Finite field t^m-a irreducibility criteria}
\begin{itemize}
\item[(i)] There exists $a \in K$ such that $(t-b) \nmid_r (t^m-a)$ for all $b \in K$ if and only if $\mathrm{gcd}(s,p^h-1) > 1.$
\item[(ii)] \cite[(22)]{Petit1966-1967} Suppose $m \in \{ 2,3 \}$ or $m$ is a prime divisor of $(q-1)$. Then there exists $a \in K^{\times}$ such that $t^m-a \in K[t;\sigma]$ is irreducible if and only if $\mathrm{gcd}(s,p^h-1) > 1.$
\end{itemize}
\end{corollary}

\begin{proof}
There exists $u \in \{ 0, \ldots, p^h-2 \}$ such that $u \notin \mathbb{Z}s \ \mathrm{mod} \ (p^h-1)$, if and only if $s$ does not generate $\mathbb{Z}_{p^h-1}$, if and only if $\mathrm{gcd}(s,p^h-1) > 1.$
Hence the result follows by Proposition \ref{prop:finite fields irreducibility primitive element}.
\end{proof}

When $p \equiv 1 \ \mathrm{mod} \ m$, it becomes simpler to apply Corollary \ref{cor:Finite field t^m-a irreducibility criteria}:

\begin{corollary} \label{cor:p=1mod m finite field irreducibility criteria}
Suppose $p \equiv 1 \ \mathrm{mod} \ m$.
\begin{itemize}
\item[(i)] There exists $a \in K$ such that $(t-b) \nmid_r (t^m-a)$ for all $b \in K$.
\item[(ii)] If $p$ is an odd prime, then there exists $a \in K^{\times}$ such that $t^2-a \in K[t;\sigma]$ is irreducible.
\item[(iii)] If $m = 3$, then there exists $a \in K^{\times}$ such that $t^3-a \in K[t;\sigma]$ is irreducible.
\item[(iv)] Suppose $m$ is a prime divisor of $(q-1)$, then there exists $a \in K^{\times}$ such that $t^m-a \in K[t;\sigma]$ is irreducible.
\end{itemize}
\end{corollary}

\begin{proof}
We have
$$s \ \mathrm{mod} \ m = \sum_{i=0}^{m-1} (p^{ri} \ \mathrm{mod} \ m) \ \mathrm{mod} \ m = (\sum_{i=0}^{m-1} 1) \ \mathrm{mod} \ m = 0,$$
and $p^h \equiv 1 \ \mathrm{mod} \ m$. This means $m \vert s$ and $m \vert (p^h-1)$, therefore $\mathrm{gcd}( s,p^h-1) \geq m$ and so the assertion follows by Corollary \ref{cor:Finite field t^m-a irreducibility criteria}.
\end{proof}

\section{Irreducibility Criteria in \texorpdfstring{$R = D[t;\sigma,\delta]$}{R = D[t;sigma,delta]}}

Let $D$ be a division ring with center $C$, $\sigma$ be an endomorphism of $D$ and $\delta$ be a left $\sigma$-derivation of $D$. In this Section we investigate irreducibility criteria in $R = D[t;\sigma,\delta]$ generalising some of our results from Section \ref{section:Irreducibility Criteria in D[t;sigma]}.

Let $f(t) = t^m - \sum_{i=0}^{m-1} a_i t^i \in R$ and define a sequence of maps $N_i : D \rightarrow D, \ i \geq 0$, recursively by
$$N_{i+1}(b) = \sigma(N_i(b))b + \delta(N_i(b)), \ N_0 (b) = 1,$$
e.g. $N_0 (b) = 1, \ N_1(b) = b, \ N_2(b) = \sigma(b)b + \delta(b), \ldots$

Let $r \in D$ be the unique remainder after right division of $f(t)$ by $(t-b)$, then
$$r = N_m(b) - \sum_{i=0}^{m-1} a_i N_i (b),$$
by \cite[Lemma 2.4]{lam1988vandermonde}. This evidently implies:

\begin{proposition} \label{prop:rightlinearfactordelta}
$(t-b) \vert_r f(t)$ is equivalent to $N_m(b) - \sum_{i=0}^{m-1} a_i N_i (b) = 0$.
\end{proposition}

Now suppose $\sigma$ is an automorphism of $D$. We wish to find the remainder after left division of $f(t)$ by $(t-b)$. Define a sequence of maps $M_i:D \rightarrow D$, $i \geq 0$, recursively by
$$M_{i+1}(b) = b \sigma^{-1}(M_i(b)) - \delta(\sigma^{-1}(M_i(b))), \ M_0(b) = 1,$$
for example $M_0(b) = 1$, \ $M_1(b) = b$, \ $M_2(b) = b \sigma^{-1}(b) - \delta(\sigma^{-1}(b))$, \ldots

Recall from page \pageref{sigma automorphism right polynomial ring} that since $\sigma$ is an automorphism, we can also view $R$ as a right polynomial ring. In particular this means we can write $f(t) = t^m - \sum_{i=0}^{m-1} a_i t^i \in R$ in the form $f(t) = t^m - \sum_{i=0}^{m-1} t^i a_i'$ for some uniquely determined $a_i' \in D$.

\begin{proposition} \label{prop:leftlinearfactordelta}
$(t-b) \vert_l f(t)$ is equivalent to $M_m(b) - \sum_{i=0}^{m-1} M_i(b) a_i' = 0$. In particular, $(t-b) \vert_l (t^m - a)$ if and only if $M_m(b) \neq a$.
\end{proposition}

\begin{proof}
We first show $t^n - M_n(b) \in (t-b)R$ for all $b \in D$ and $n \geq 0$: If $n = 0$ then $t^0 - M_0(b) = 1 - 1 = 0 \in (t-b)R$ as required. Suppose inductively $t^n - M_n(b) \in (t-b)R$ for some $n \geq 0$, then
\begin{align*}
&t^{n+1} - M_{n+1}(b) = t^{n+1} - b \sigma^{-1}(M_n(b)) + \delta(\sigma^{-1}(M_n(b))) \\
&= t^{n+1} + (t-b) \sigma^{-1}(M_n(b)) - t \sigma^{-1}(M_n(b)) + \delta(\sigma^{-1}(M_n(b))) \\
&= t^{n+1} + (t-b) \sigma^{-1}(M_n(b)) - M_n(b) t - \delta(\sigma^{-1}(M_n(b))) + \delta(\sigma^{-1}(M_n(b))) \\
&= (t-b) \sigma^{-1}(M_n(b)) + (t^n - M_n(b))t \in (t-b)R,
\end{align*}
as $t^n-M_n(b) \in (t-b)R$. Therefore $t^n - M_n(b) \in (t-b)R$ for all $b \in D$, $n \geq 0$ by induction.

As a result, there exists $q_i(t) \in R$ such that $t^i = (t-b) q_i(t) + M_i(b)$, for all $i \in \{ 0, \ldots, m \}$. Multiplying on the right by $a_i'$ and summing over $i$ yields
$$f(t) = (t-b)q(t) + M_m(b) - \sum_{i=0}^{m-1} M_i(b) a_i',$$
for some $q(t) \in R$.
\end{proof}

Using Propositions \ref{prop:rightlinearfactordelta} and \ref{prop:leftlinearfactordelta} we obtain criteria for skew polynomials of degree $2$ and $3$ to be irreducible:

\begin{theorem} \label{thm:irredcriteriadelta}
\begin{itemize}
\item[(i)] Suppose $\sigma$ is an endomorphism of $D$, then $f(t) = t^2 - a_1 t - a_0 \in R$ is irreducible if and only if  $\sigma(b)b + \delta(b) - a_1 b - a_0 \neq 0$  for all $b \in D$.
\item[(ii)] Suppose $\sigma$ is an automorphism of $D$ and $f(t) = t^3 - a_2 t^2 - a_1 t - a_0 \in R$. Write $f(t) = t^3 - t^2 a_2' - t a_1' - a_0'$ for some unique $a_0', a_1', a_2' \in D$, then $f(t)$ is irreducible if and only if
\begin{equation} \label{eqn:irredcriteriadelta 1}
N_m(b) - \sum_{i=0}^{2}a_i N_i(b) \neq 0,
\end{equation}
and
\begin{equation} \label{eqn:irredcriteriadelta 2}
M_m(b) - \sum_{i=0}^{2} M_i(b) a_i' \neq 0,
\end{equation}
for all $b \in D$.
\end{itemize}
\end{theorem}

\begin{proof}
\begin{itemize}
\item[(i)] We have $f(t)$ is irreducible if and only if it has no right linear factors, if and only if
$$N_2 (b) - a_1 N_1(b) - a_0 N_0 (b) = \sigma(b)b + \delta(b) - a_1 b - a_0 \neq 0,$$
for all $b \in D$ by Proposition \ref{prop:rightlinearfactordelta}.
\item[(ii)] We have $f(t)$ is irreducible if and only if it has no left or right linear factors, if and only if \eqref{eqn:irredcriteriadelta 1} and \eqref{eqn:irredcriteriadelta 2} hold for all $b \in D$ by Propositions \ref{prop:rightlinearfactordelta} and \ref{prop:leftlinearfactordelta}.
\end{itemize}
\end{proof}

We now prove an analogous result to Theorem \ref{thm:Petit(19)}: 

\begin{theorem} \label{thm:Petit(19), delta not 0}
Suppose $\sigma$ is an endomorphism of $D$, $m$ is prime, $\mathrm{Char}(D) \neq m$ and $C \cap \mathrm{Fix}(\sigma)$ contains a primitive $m^{\text{th}}$ root of unity $\omega$. Then $f(t) = t^m - a \in R$ is irreducible if and only if $N_m(b) \neq a$ for all $b \in D$.
\end{theorem}

\begin{proof}
Recall
$$\delta(b^n) = \sum_{i=0}^{n-1} \sigma(b)^i \delta(b)b^{n-1-i},$$
for all $b \in D$, $n \geq 1$ by \eqref{eqn:goodearl} and so
\begin{align*}
0 &= \delta(1) = \delta(\omega^m) = \sum_{i=0}^{m-1} \sigma(\omega)^i \delta(\omega) \omega^{m-1-i} = \sum_{i=0}^{m-1} \omega^i \delta(\omega) \omega^{m-1-i} \\ &= \sum_{i=0}^{m-1} \delta(\omega) \omega^{m-1} = \delta(\omega) \omega^{m-1} m,
\end{align*}
where we have used $\omega \in C \cap \mathrm{Fix}(\sigma)$. Therefore $\omega \in \mathrm{Const}(\delta)$ because $\mathrm{Char}(D) \neq m$, hence also $\omega^i \in \mathrm{Const}(\delta)$ and so $(\omega t)^i = \omega^i t^i$ for all $i \in \{ 1, \ldots , m \}$. Furthermore if $b \in D$, then $(t-b) \nmid_r f(t)$ is equivalent to $N_m(b) \neq a$ by Proposition \ref{prop:rightlinearfactordelta}. The proof now follows exactly as in Theorem's \ref{thm:bourbaki} and \ref{thm:Petit_factor}.
\end{proof}

Setting $m = 3$ and $\sigma = \mathrm{id}$ in Theorem \ref{thm:Petit(19), delta not 0} yields:

\begin{corollary} \label{cor:Petit(19), delta not 0, m=3}
Suppose $\mathrm{Char}(D) \neq 3$, $\sigma = \mathrm{id}$ and $C$ contains a primitive $3^{\text{rd}}$ root of unity. Then $f(t) = t^3 - a \in D[t; \delta]$ is irreducible if and only if
$$N_3(b) = b^3 + 2 \delta(b)b + b\delta(b) + \delta^2(b) \neq a,$$
for all $b \in D$. 
\end{corollary}

\section{Irreducibility Criteria in \texorpdfstring{$D[t;\delta]$}{D[t;delta]} where \texorpdfstring{$\mathrm{Char}(D) = p$}{Char(D) = p}} \label{section:Irreducibility Criteria in D[t;delta] where Char(D) = p}

Suppose $\mathrm{Char}(D) = p \neq 0$, $\sigma = \mathrm{id}$ and
$$f(t) = t^{p^e} - a_{1} t^{p^{e-1}} - \ldots - a_e t - d \in D[t;\delta].$$
In $D[t;\delta]$ we have the equalities
\begin{equation} \label{p-power formula char p}
(t-b)^p = t^p - V_p(b), \qquad V_p(b) = b^p + \delta^{p-1}(b) + * \in D,
\end{equation}
for all  $b \in D$, with $*$ a sum of commutators of $b, \delta(b), \ldots, \delta^{p-2}(b)$ \cite[p.~17-18]{jacobson1996finite}. E.g. $V_2(b) = b^2 + \delta(b)$ and $V_3(b) = b^3 + \delta^2(b) + \delta(b) b - b \delta(b)$. In particular, if $D$ is commutative or $b$ commutes with all of its derivatives, then $* = 0$ and the formula simplifies to
\begin{equation} \label{eqn:V_p(b) form when D commutative}
V_p(b) = b^p + \delta^{p-1}(b).
\end{equation}
We can iterate \eqref{p-power formula char p} to obtain
\begin{equation} \label{p-power formula char p iterated}
(t-b)^{p^i} = t^{p^i} - V_{p^i}(b),
\end{equation}
for all $i \in \mathbb{N}$ where $V_{p^i}(b) = V_p^i(b) = V_p(V_p( \cdots (V_p(b)) \cdots )$. We thus have $a_i t^{p^i} = a_i (t-b)^{p^i} + a_i V_{p^i}(b)$ and by summing over $i$ we conclude:

\begin{proposition} \label{prop:(t-b) right divides f(t) in Char p}
(\cite[Proposition 1.3.25]{jacobson1996finite}).
$(t-b) \vert_r f(t)$ is equivalent to 
\begin{equation} \label{eqn:(t-b) right divides f(t) in Char p}
V_{p^e}(b) - a_{1} V_{p^{e-1}}(b) - \ldots - a_e b - d = 0.
\end{equation}
\end{proposition}

Looking instead at left division of $f(t)$ by a linear polynomial gives:

\begin{proposition} \label{prop:(t-b) left divides f(t) in Char p}
$(t-b) \vert_l f(t)$ is equivalent to
\begin{equation*}
\begin{split}
V_{p^e}(b) - \big( V_{p^{e-1}}(b)& a_{1} - \delta^{p^{e-1}}(a_{1}) \big) - \ldots - \big( V_p(b)a_{e-1} - \delta^p(a_{e-1}) \big) \\& - (b a_e - \delta(a_e)) - d = 0.
\end{split}
\end{equation*}
\end{proposition}

\begin{proof}
We have $t^{p^k} a = (t-b)^{p^k}a + V_{p^k}(b)a,$ for all $a, b \in D$ and $k \geq 1$ by \eqref{p-power formula char p iterated}. Moreover, iterating the relation $ta = at + \delta(a)$ yields
\begin{equation*}
t^{p^k} a = \sum_{i=0}^{p^k} \binom{p^k}{i} \delta^{p^k - i}(a) t^i,
\end{equation*}
for all $a \in D$, $k \geq 1$ \cite[(1.1.26)]{jacobson1996finite}. This implies $t^{p^k} a = \delta^{p^k}(a) + at^{p^k}$ for all $a \in D$, $k \geq 1$ because $\binom{p^k}{i} = 0$ for all $i \in \{ 1, \ldots, p^k -1 \}$. Therefore 
$$at^{p^k} = (t-b)^{p^k}a + V_{p^k}(b)a - \delta^{p^k}(a),$$
and hence
\begin{equation*}
\begin{split}
f(t) &= t^{p^e} - a_{1} t^{p^{e-1}} - \ldots - a_e t - d \\
&= (t-b)q(t) + V_{p^e}(b) - \big( V_{p^{e-1}}(b) a_{1} - \delta^{p^{e-1}}(a_{1}) \big) \\
& \qquad \qquad - \ldots - \big( V_p(b)a_{e-1} - \delta^p(a_{e-1}) \big) - ba_e + \delta(a_e) - d.
\end{split}
\end{equation*}
\end{proof}

\begin{corollary} \label{cor:right iff left division in Char p}
Suppose $\mathrm{Char}(D) = p \neq 0$ and $f(t) = t^p - a_1 t - a_0 \in D[t;\delta]$, where $a_1 \in C \cap \mathrm{Const}(\delta)$. Then $(t-b) \vert_r f(t)$ if and only if  $(t-b) \vert_l f(t)$.
\end{corollary}

\begin{proof}
Recall $(t-b) \vert_r f(t)$ if and only if $V_p(b) - a_1 b - a_0 = 0$ by Proposition \ref{prop:(t-b) right divides f(t) in Char p}, and $(t-b) \vert_l f(t)$ if and only if $V_p(b) - b a_1 + \delta(a_1) - a_0 = 0$ by Proposition \ref{prop:(t-b) left divides f(t) in Char p}. When $a_1 \in C \cap \mathrm{Const}(\delta)$, these two conditions are equivalent.
\end{proof}

When $p=3$, Propositions \ref{prop:(t-b) right divides f(t) in Char p} and \ref{prop:(t-b) left divides f(t) in Char p} yield the following:

\begin{corollary} \label{cor:t^3-a_1t-a_0 Char 3 irreducibility}
Let $\mathrm{Char}(D) = 3$ and $f(t) = t^3 - a_1 t - a_0 \in D[t;\delta]$. 
\begin{itemize}
\item[(i)] $f(t)$ is irreducible if and only if
$$V_3(b) - a_1 b - a_0 \neq 0 \ \text{ and } \ V_3(b) - b a_1 + \delta(a_1) - a_0 \neq 0,$$
for all $b \in D$. 
\item[(ii)] Let $a_1 \in C \cap \mathrm{Const}(\delta)$, then $f(t)$ is irreducible if and only if $V_3(b) - a_1 b - a_0 \neq 0$  for all $b \in D$.
\item[(iii)] Suppose $D$ is commutative, then $f(t)$ is irreducible if and only if
$$b^3 + \delta^2(b) - a_1 b - a_0 \neq 0 \ \text{ and } \ b^3 + \delta^2(b) - b a_1 + \delta(a_1) - a_0 \neq 0,$$
for all $b \in D$.
\end{itemize}
\end{corollary}

\begin{proof}
Notice $f(t)$ is irreducible if and only if $(t-b) \nmid_r f(t)$ and $(t-b) \nmid_l f(t)$ for all $b \in D$ and thus (i) follows by Propositions \ref{prop:(t-b) right divides f(t) in Char p} and \ref{prop:(t-b) left divides f(t) in Char p}. (ii) follows from (i) and Corollary \ref{cor:right iff left division in Char p} and (iii) follows from (i) and \eqref{eqn:V_p(b) form when D commutative}.
\end{proof}


\begin{proposition}
Suppose $\mathrm{Char}(D) = p \neq 0$, $D$ is commutative and $\delta$ is a non-trivial derivation of $D$ such that $\delta^p = 0$. Let $f(t) = t^p - a \in D[t;\delta]$ where $a \notin \mathrm{Const}(\delta)$, then
\begin{itemize}
\item[(i)] $(t-b) \nmid_r f(t)$ and $(t-b) \nmid_l f(t)$ for all $b \in D$.
\item[(ii)] If $p = 3$ then $f(t)$ is irreducible.
\end{itemize}
\end{proposition}

\begin{proof}
\begin{itemize}
\item[(i)] Recall $(t-b) \vert_r f(t)$ is equivalent to $(t-b) \vert_l f(t)$ by Corollary \ref{cor:right iff left division in Char p}, and that $(t-b) \vert_r f(t)$ if and only if $V_p(b) = a$ if and only if $b^p + \delta^{p-1}(b) = a$ by Proposition \ref{prop:(t-b) right divides f(t) in Char p}. Now $b^p \in \mathrm{Const}(\delta)$ for all $b \in D$ as $D$ is commutative \cite[p.~60]{kolchin1973differential}, also $\delta^{p-1}(b) \in \mathrm{Const}(\delta)$ as $\delta^p = 0$. Therefore if $a \notin \mathrm{Const}(\delta)$ then $b^p + \delta^{p-1}(b) \neq a$ for all $b \in D$.
\item[(ii)] If $p = 3$, then $f(t)$ is irreducible if and only if $(t-b) \nmid_r f(t)$ and $(t-b) \nmid_l f(t)$ for all $b \in D$, hence the assertion follows by (i).
\end{itemize}
\end{proof}

When $f(t)$ has the form $f(t) = t^p - t - a \in D[t;\delta]$, we have:
\begin{theorem}
(\cite[Lemmas 4 and 6]{amitsur1954non}).
For any $a \in D$, the polynomial $f(t) = t^p - t - a \in D[t;\delta]$ is either a product of commuting linear factors or irreducible. Furthermore, $f(t)$ is irreducible if and only if $V_p(b) - b - a \neq 0,$
for all $b \in D$. In particular, if $D$ is commutative then $f(t)$ is irreducible if and only if
$$b^p + \delta^{p-1}(b) - b - a \neq 0,$$
for all $b \in D$.
\end{theorem}

\chapter{Isomorphisms Between Petit Algebras} \label{chapter:Isomorphisms Between some Petit Algebras}

We now investigate isomorphisms between some Petit Algebras. The results we obtain in this Chapter will then be applied in Chapter \ref{chapter:Automorphisms of S_f} to study the automorphism groups of Petit algebras. 

Let $D$ be an associative division ring with center $C$, $\sigma$ be an automorphism of $D$ and $\delta$ be a left $\sigma$-derivation of $D$. Suppose $D'$, $C'$, $\sigma'$ and $\delta'$ are defined similarly. Let $f(t) \in R = D[t;\sigma,\delta]$, \ $g(t) \in R' = D'[t;\sigma',\delta']$, and recall $S_f = R/Rf$ is an algebra over $F = C \cap \mathrm{Fix}(\sigma) \cap \mathrm{Const}(\delta)$ and $S_g = R'/R'g$ is an algebra over $F' = C' \cap \mathrm{Fix}(\sigma') \cap \mathrm{Const}(\delta')$. Denote by $\circ_f$ the multiplication in $S_f$ and by $\circ_g$ the multiplication in $S_g$. If $f(t)$ and $g(t)$ are not right invariant and $F = F'$, we have the following necessary conditions for $S_f$ to be $F$-isomorphic to $S_g$:


\begin{proposition} \label{prop:isomorphism_necessity}
Suppose $f(t)$, $g(t)$ are not right invariant, $F = F'$ and $S_f$ is $F$-isomorphic to $S_g$. Then 
\begin{itemize}
\item[(i)] $D \cong D'$.
\item[(ii)] $\mathrm{Nuc}_r(S_f) \cong \mathrm{Nuc}_r(S_g)$.
\item[(iii)] $S_f$ is a division algebra if and only if $S_g$ is a division algebra.
\item[(iv)] If $S_f$ is a finite-dimensional left $F$-vector space, then $\mathrm{deg}(f(t)) = \mathrm{deg}(g(t))$.
\end{itemize}
\end{proposition}

\begin{proof}
\begin{itemize}
\item[(i)] We have $\mathrm{Nuc}_l(S_f) = D$ and $\mathrm{Nuc}_l(S_g) = D'$ by Theorem \ref{thm:Properties of S_f petit}(i) because $S_f$ and $S_g$ are not associative. Any isomorphism preserves the left nucleus and so $D \cong D'$.
\item[(ii)] Any isomorphism also preserves the right nucleus, thus $\mathrm{Nuc}_r(S_f) \cong \mathrm{Nuc}_r(S_g)$. 
\item[(iii)] Let $L_{a,f}: S_f \rightarrow S_f, \ x \mapsto a \circ_f x$ denote the left multiplication in $S_f$ and $L_{b,g}: S_g \rightarrow S_g, \ x \mapsto b \circ_g x$ the left multiplication in $S_g$. Similarly denote the right multiplication maps $R_{a,f}$ and $R_{b,g}$. Suppose $\phi:S_f \rightarrow S_g$ is an $F$-isomorphism and $S_f$ is a division algebra, so that $L_{a,f}$ and $R_{a,f}$ are bijective for all non-zero $a \in S_f$. Furthermore we have
$$L_{b,g}(x) = \phi(L_{\phi^{-1}(b),f}(\phi^{-1}(x))) \text{ and } R_{b,g}(x) = \phi(R_{\phi^{-1}(b),f}(\phi^{-1}(x))),$$
for all $x \in S_g$, $0 \neq b \in S_g$. These imply $L_{b,g}$ and $R_{b,g}$ are bijective for all non-zero $b \in S_g$ because $\phi, \phi^{-1}, L_{\phi^{-1}(b),f}$ and $R_{\phi^{-1}(b),f}$ are all bijective, hence $S_g$ is a division algebra. The reverse implication is proven analogously.
\item[(iv)] $S_f$ and $S_g$ must have the same dimension as left $F$-vector spaces, and since $D \cong D'$, this implies $\mathrm{deg}(f(t)) = \mathrm{deg}(g(t))$.
\end{itemize}
\end{proof}

\section{Automorphisms of \texorpdfstring{$R = D[t;\sigma,\delta]$}{R = D[t;sigma,delta]}}

Henceforth we assume $D = D'$, $\sigma = \sigma'$, $\delta = \delta'$ and $f(t), g(t) \in R = D[t;\sigma,\delta]$ have degree $m$. Automorphisms of $R$ can be used to define isomorphisms between Petit algebras:

\begin{theorem} \label{thm:lavrauw isomorphic petit algebras generalisation} 
Let $\Theta$ be an $F$-automorphism of $R$. Given $f(t) \in R$, define $g(t) = l \Theta(f(t)) \in R$ for some $l \in D^{\times}$. Then $\Theta$ induces an $F$-isomorphism $S_f \cong S_g$.
\end{theorem}

\begin{proof}
$\Theta(t)$ has degree $1$ by the argument in \cite[p.~4]{lam1992homomorphisms}, therefore $\Theta$ preserves the degree of elements of $R$ and thus $\Theta \vert_{R_m}: R_m \rightarrow R_m$ is well defined, hence bijective and $F$-linear. Let $a, b \in R_m$, then there exist unique $q(t), r(t) \in R$ with $\mathrm{deg}(r(t)) < m$ such that $ab = q(t)f(t) + r(t)$. We have
\begin{align*}
\Theta (a \circ_f b) &= \Theta(ab - qf) = \Theta(a)\Theta(b) - \Theta(q)\Theta(f) \\ &= \Theta(a)\Theta(b) - \Theta(q) l^{-1} l \Theta(f) = \Theta(a) \circ_{g} \Theta(b),
\end{align*}
and so $\Theta \vert_{R_m}: S_f \rightarrow S_{g}$ is an $F$-isomorphism between algebras.
\end{proof}

When $D = K$ is a finite field and $f(t) \in K[t;\sigma]$ is irreducible, then \cite[Theorem 7]{lavrauw2013semifields} states that an automorphism $\Theta$ of $K[t;\sigma]$ restricts to an isomorphism between $S_f$ and $S_{\Theta(f)}$. Therefore Theorem \ref{thm:lavrauw isomorphic petit algebras generalisation} is a generalisation of \cite[Theorem 7]{lavrauw2013semifields}.


In order to employ Theorem \ref{thm:lavrauw isomorphic petit algebras generalisation}, we first investigate what the automorphisms of $R$ look like: Given $\tau \in \mathrm{Aut}_{F}(D)$ and $p(t) \in R$, the map
\begin{equation*}
\Theta: R \rightarrow R, \ \sum_{i=0}^{n} b_i t^i \mapsto \sum_{i=0}^{n} \tau(b_i) p(t)^i,
\end{equation*}
is an $F$-homomorphism if and only if
\begin{equation} \label{eqn:isomorphism between skew polynomial rings}
p(t) \tau(b) = \tau(\sigma(b))p(t) + \tau(\delta(b)),
\end{equation}
for all $b \in D$ by \cite[p.~4]{lam1992homomorphisms}. Furthermore, by a simple degree argument we see $\Theta$ is injective if and only if $p(t)$ has degree $\geq 1$, and $\Theta$ is bijective if and only if $p(t)$ has degree $=1$ \cite[p.~4]{lam1992homomorphisms}. Therefore, by \eqref{eqn:isomorphism between skew polynomial rings} we conclude:

\begin{proposition} \label{prop:isomorphisms between skew polynomial rings}
Let $c \in D$, $d \in D^{\times}$, then
\begin{equation} \label{eqn:isomorphisms between skew polynomial rings 0}
\Theta_{\tau,c,d}: R \rightarrow R, \ \sum_{i=0}^n b_i t^i \mapsto \sum_{i=0}^n \tau(b_i) (c + d t)^i,
\end{equation}
is an $F$-automorphism of $R$ if and only if 
\begin{equation} \label{eqn:isomorphisms between skew polynomial rings 1}
c \tau(b) + d \delta(\tau(b)) = \tau(\sigma(b))c + \tau(\delta(b)),
\end{equation}
and
\begin{equation} \label{eqn:isomorphisms between skew polynomial rings 2}
d \sigma(\tau(b)) = \tau(\sigma(b))d,
\end{equation}
for all $b \in D$.
\end{proposition}

\begin{proof}
Let $p(t) = c + dt$, then $\Theta_{\tau,c,d}: R \rightarrow R$ is bijective because $p(t)$ has degree $1$, furthermore $\Theta_{\tau,c,d}$ is an $F$-automorphism if and only if
$$(c + dt) \tau(b) = c \tau(b) + d \sigma(\tau(b))t + d \delta(\tau(b)) = \tau(\sigma(b))(c + dt) + \tau(\delta(b)),$$
by \eqref{eqn:isomorphism between skew polynomial rings}. Comparing the coefficients of $t^0$ and $t$ yields \eqref{eqn:isomorphisms between skew polynomial rings 1} and \eqref{eqn:isomorphisms between skew polynomial rings 2} as required.
\end{proof}

Looking closely at the conditions \eqref{eqn:isomorphisms between skew polynomial rings 1} and \eqref{eqn:isomorphisms between skew polynomial rings 2} yields the following:

\begin{corollary} \label{cor:Automorphisms of D[t;sigma,delta]}
\begin{itemize}
\item[(i)] Suppose $c \in D$, $d \in D^{\times}$. Then $\Theta_{\mathrm{id},c,d}$ is an $F$-automorphism of $R$ if and only if $d \in C^{\times}$ and
$$cb + d \delta(b) = \sigma(b) c + \delta(b)$$
for all $b \in D$.
\item[(ii)] $\Theta_{\tau,0,1}$ is an $F$-automorphism of $R$ if and only if $\sigma \circ \tau = \tau \circ \sigma$ and $\delta \circ \tau = \tau \circ \delta$.
\item[(iii)] Let $c \in D$, then $\Theta_{\tau,c,1}$ is an $F$-automorphism of $R$ if and only if $\sigma \circ \tau = \tau \circ \sigma$ and
\begin{equation} \label{eqn:Automorphisms of D[t;sigma,delta] 1}
c \tau(b) + \delta(\tau(b)) = \tau(\sigma(b)) c + \tau(\delta(b)),
\end{equation}
for all $b \in D$.
\item[(iv)] Suppose $c \in D^{\times}, 1 \neq d \in C^{\times}$ and $\delta$ is the inner $\sigma$-derivation
\begin{equation} \label{eqn:Automorphisms of D[t;sigma,delta] 2}
\delta: D \rightarrow D, \ b \mapsto c(1-d)^{-1} b - \sigma(b) c(1-d)^{-1}.
\end{equation}
Then $\Theta_{\mathrm{id},c,d}$ is an $F$-automorphism of $R$.
\end{itemize}
\end{corollary}

\begin{proof}
(ii) and (iii) follow immediately from \eqref{eqn:isomorphisms between skew polynomial rings 1} and \eqref{eqn:isomorphisms between skew polynomial rings 2}. 
\begin{itemize}
\item[(i)] Setting $\tau = \mathrm{id}$ in \eqref{eqn:isomorphisms between skew polynomial rings 2} yields $d \sigma(b) = \sigma(b) d$ for all $b \in D$, which is equivalent to $d \in C^{\times}$. The result follows by setting $\tau = \mathrm{id}$ in \eqref{eqn:isomorphisms between skew polynomial rings 1}.
\item[(iv)] If $\delta$ has the form \eqref{eqn:Automorphisms of D[t;sigma,delta] 2} and $\tau = \mathrm{id}$, then
\begin{align*}
c b &+ d \delta(b) = c b + d \big( c (1-d)^{-1} b - \sigma(b) c (1-d)^{-1} \big) \\
&= (1-d)^{-1} \big( c b(1-d) + d c b - d \sigma(b) c \big) = (1-d)^{-1}(c b - d \sigma(b) c) \\
&= (1-d)^{-1} \big( \sigma(b) c (1-d) + c b - \sigma(b) c \big) = \sigma(b) c + \delta(b),
\end{align*}
and $d \sigma(b) = \sigma(b) d$, for all $b \in D$ because $d \in C^{\times}$. Therefore $\Theta_{\mathrm{id},c,d}$ is an $F$-automorphism of $R$ by Proposition \ref{prop:isomorphisms between skew polynomial rings}.
\end{itemize}
\end{proof}

We can use Theorem \ref{thm:lavrauw isomorphic petit algebras generalisation} together with Proposition \ref{prop:isomorphisms between skew polynomial rings} and Corollary \ref{cor:Automorphisms of D[t;sigma,delta]}, to find isomorphisms between some Petit algebras:

\begin{corollary}
Let $f(t) = t^m - \sum_{i=0}^{m-1} a_i t^i \in R$.
\begin{itemize}
\item[(i)] Suppose $\tau \in \mathrm{Aut}_{F}(D)$ and $c \in D$, $d \in D^{\times}$ are such that \eqref{eqn:isomorphisms between skew polynomial rings 1} and \eqref{eqn:isomorphisms between skew polynomial rings 2} hold for all $b \in D$. Let $g(t) = (c+dt)^m - \sum_{i=0}^{m-1} \tau(a_i) (c+dt)^i \in R,$ then $S_f \cong S_g$.
\item[(ii)] Suppose $\tau \in \mathrm{Aut}_{F}(D)$ is such that $\sigma \circ \tau = \tau \circ \sigma$ and $\delta \circ \tau = \tau \circ \delta$ and define $g(t) = t^m - \sum_{i=0}^{m-1} \tau(a_i)t^i \in R$, then $S_f \cong S_g$.
\item[(iii)] Suppose $\tau \in \mathrm{Aut}_{F}(D)$ and $c \in D$ are such that $\sigma \circ \tau = \tau \circ \sigma$ and \eqref{eqn:Automorphisms of D[t;sigma,delta] 1} holds for all $b \in D$. If $g(t) = (c+t)^m - \sum_{i=0}^{m-1} \tau(a_i)(c+t)^i \in R$, then $S_f \cong S_g$.
\item[(iv)] Suppose $\delta$ is the inner $\sigma$-derivation given by \eqref{eqn:Automorphisms of D[t;sigma,delta] 2} for some $c \in D^{\times}, 1 \neq d \in C^{\times}$. If $g(t) = (c + dt)^m - \sum_{i=0}^{m-1} a_i (c + dt)^i \in R$, then $S_f \cong S_g$.
\end{itemize}
\end{corollary}

\begin{proof}
In (i), (ii), (iii) and (iv) the maps $\Theta_{\tau,c,d}$, $\Theta_{\tau,0,1}$, $\Theta_{\tau,c,1}$ and $\Theta_{\mathrm{id},c,d}$ resp. are $F$-automorphisms of $R$ by Proposition \ref{prop:isomorphisms between skew polynomial rings} and Corollary \ref{cor:Automorphisms of D[t;sigma,delta]}. Applying these automorphisms to $f(t)$ gives $g(t)$ in each case, thus the assertion follows by Theorem \ref{thm:lavrauw isomorphic petit algebras generalisation}.
\end{proof}

\section{Isomorphisms Between \texorpdfstring{$S_f$}{S\_f} and \texorpdfstring{$S_g$}{S\_g} when \texorpdfstring{$f(t), g(t) \in D[t;\delta]$}{f(t), g(t) in D[t;delta]}}

Let $D$ be an associative division ring with center $C$, $\delta$ be a derivation of $D$ and $F = C \cap \mathrm{Const}(\delta)$. We briefly look into the isomorphisms between some Petit algebras $S_f$ and $S_g$ in the special case where $f(t), g(t) \in R = D[t;\delta]$. Here Proposition \ref{prop:isomorphisms between skew polynomial rings} becomes:

\begin{corollary} \label{cor:Automorphisms of D[t;delta]}
Let $c, d \in D$ and $\tau \in \mathrm{Aut}_{F}(D)$, then the map $\Theta_{\tau,c,d}$ defined by \eqref{eqn:isomorphisms between skew polynomial rings 0} is an $F$-automorphism of $R$, if and only if $d \in C^{\times}$ and
\begin{equation} \label{eqn:isomorphisms between skew polynomial rings sigma=id}
c \tau(b) + d \delta(\tau(b)) = \tau(b)c + \tau(\delta(b)).
\end{equation}
In particular, if $c \in C$ then $\Theta_{\tau,c,1}$ is an $F$-automorphism if and only if $\delta \circ \tau = \tau \circ \delta$. 
\end{corollary}

\begin{proof}
With $\sigma = \mathrm{id}$, \eqref{eqn:isomorphisms between skew polynomial rings 2} becomes $d \tau(b) = \tau(b) d$ for all $b \in D$, i.e. $d \in C$. Finally, setting $\sigma = \mathrm{id}$ in \eqref{eqn:isomorphisms between skew polynomial rings 1} yields \eqref{eqn:isomorphisms between skew polynomial rings sigma=id}, hence the result follows by Proposition \ref{prop:isomorphisms between skew polynomial rings}.
\end{proof}

Recall from \eqref{p-power formula char p} that when $D$ has characteristic $p \neq 0$, we have
$$(t-b)^p = t^p - V_p(b), \ V_p(b) = b^p + \delta^{p-1}(b) + *$$
for all $b \in D$, where $*$ is a sum of commutators of $b, \delta(b), \ldots, \delta^{p-2}(b)$. An iteration yields $(t-b)^{p^e} = t^{p^e} - V_{p^e}(b)$, for all $b \in D$ with $V_{p^e}(b) = V_p^e(b) = V_p(\ldots(V_p(b)\ldots)$. We use Corollary \ref{cor:Automorphisms of D[t;delta]}, together with Theorem \ref{thm:lavrauw isomorphic petit algebras generalisation} to obtain sufficient conditions for some Petit algebras to be isomorphic:

\begin{corollary} \label{cor:Isomorphisms of Petit algebras D[t;delta]}
\begin{itemize}
\item[(a)] Suppose $D$ has arbitrary characteristic and $f(t) = t^m - \sum_{i=0}^{m-1} a_i t^i \in R = D[t;\delta]$.
\begin{itemize}
\item[(i)] Let $\tau \in \mathrm{Aut}_{F}(D)$ and $g(t) = (t-c)^m - \sum_{i=0}^{m-1} \tau(a_i) (t-c)^i \in R$ for some $c \in C$. If $\tau \circ \delta = \delta \circ \tau$, then $S_f \cong S_{g}$.
\item[(ii)] Let
$g(t) = (t-c)^m - \sum_{i=0}^{m-1} a_i (t-c)^i \in R$
for some $c \in C$, then $S_f \cong S_{g}$.
\end{itemize}
\item[(b)] Suppose $\mathrm{Char}(D) = p \neq 0$ and $f(t) = t^{p^e} + a_1 t^{p^{e-1}} + \ldots + a_e t + d \in R$.
\begin{itemize}
\item[(i)] Let $\tau \in \mathrm{Aut}_{F}(D)$. If $\tau \circ \delta = \delta \circ \tau$ and
\begin{align*}
g(t) &= t^{p^e} + \tau(a_1) t^{p^{e-1}} + \ldots + \tau(a_e) t + \tau(d) - V_{p^e}(c) \\ & \qquad - \tau(a_1) V_{p^{e-1}}(c) - \ldots - \tau(a_e)c \in R
\end{align*}
for some $c \in C$, then $S_f \cong S_{g}$.
\item[(ii)] Let
$$g(t) = f(t) - V_{p^e}(c) - a_1 V_{p^{e-1}}(c) - \ldots - a_e c \in R$$
for some $c \in C$, then $S_f \cong S_{g}$.
\end{itemize}
\end{itemize}
\end{corollary}

\begin{proof}
\begin{itemize}
\item[(i)] The maps $\Theta_{\tau,-c,1}$ are automorphisms of $R$ for all $c \in C$ by Corollary \ref{cor:Automorphisms of D[t;delta]}. In (a) we have $g(t) = \Theta_{\tau,-c,1}(f(t))$ and so $S_f \cong S_{g}$ by Theorem \ref{thm:lavrauw isomorphic petit algebras generalisation}.

In (b), since $\mathrm{Char}(D) = p \neq 0$, we have
\begin{align*}
\Theta_{\tau,-c,1}&(f(t)) = (t-c)^{p^e} + \tau(a_1)(t-c)^{p^{e-1}} + \ldots + \tau(a_e)(t-c) + \tau(d) \\
&= t^{p^e} - V_{p^e}(c) + \tau(a_1) \big( t^{p^{e-1}} - V_{p^{e-1}}(c) \big) + \ldots + \tau(a_e)(t-c) + \tau(d)\\
&= g(t),
\end{align*}
and hence $S_f \cong S_{g}$ by Theorem \ref{thm:lavrauw isomorphic petit algebras generalisation}.
\item[(ii)] follows from (i) by setting $\tau = \mathrm{id}$.
\end{itemize}
\end{proof}

\section{Isomorphisms Between \texorpdfstring{$S_f$}{S\_f} and \texorpdfstring{$S_g$}{S\_g} when \texorpdfstring{$f(t), g(t) \in D[t;\sigma]$}{f(t), g(t) in D[t;sigma]}}

Let $D$ be an associative division ring with center $C$ and $\sigma$ be an automorphism of $D$. Suppose $\delta = 0$ so that
$$f(t) = t^m - \sum_{i=0}^{m-1} a_i t^i, \ g(t) = t^m - \sum_{i=0}^{m-1} b_i t^i \in R = D[t;\sigma],$$
then $S_f$ and $S_g$ are nonassociative algebras over $F = C \cap \mathrm{Fix}(\sigma)$. Throughout this Section, if we assume $\sigma$ has order $\geq m-1$, we include infinite order. We begin by looking at the automorphisms of $R$. Here Proposition \ref{prop:isomorphisms between skew polynomial rings} becomes:

\begin{corollary} \label{cor:Automorphisms of D[t;sigma]}
Let $c \in D$, $d \in D^{\times}$ and $\tau \in \mathrm{Aut}_{F}(D)$. Then $\Theta_{\tau,c,d}$ is an $F$-automorphism of $R$ if and only if $c \tau(b) = \tau(\sigma(b))c$ and $d \sigma(\tau(b)) = \tau(\sigma(b)) d$ for all $b \in D$. In particular, if $\sigma \circ \tau = \tau \circ \sigma$ then $\Theta_{\tau,0,d}$ is an $F$-automorphism if and only if $d \in C^{\times}$.
\end{corollary}

We employ Corollary \ref{cor:Automorphisms of D[t;sigma]}, together with Theorem \ref{thm:lavrauw isomorphic petit algebras generalisation}, to find sufficient conditions for $S_f$ and $S_g$ to be $F$-isomorphic. When $f(t)$ and $g(t)$ are not right invariant, $\sigma$ commutes with all $F$-automorphisms of $D$, and $\sigma \vert_C$ has order at least $m-1$, these conditions are also necessary:

\begin{theorem} \label{thm:general_isomorphism}
\begin{itemize}
\item[(i)] Suppose there exists $k \in C^{\times}$ such that
\begin{equation} \label{eqn:isomorphism necessity division case Q_id,k}
a_i = \Big( \prod_{l=i}^{m-1} \sigma^l(k) \Big) b_i,
\end{equation}
for all $i \in \{ 0, \ldots, m-1 \}$. Then $S_f \cong S_g$. Furthermore, for every such $k \in C^{\times}$ the maps $Q_{\mathrm{id},k}: S_f \rightarrow S_g$,
\begin{equation} \label{eqn:form of Q_id,k isomorphism}
Q_{\mathrm{id},k}: \sum_{i=0}^{m-1} x_i t^i \mapsto x_0 + \sum_{i=1}^{m-1} x_i \big( \prod_{l=0}^{i-1} \sigma^l(k) \big) t^i,
\end{equation}
are $F$-isomorphisms between $S_f$ and $S_g$.
\item[(ii)] Suppose there exists $\tau \in \mathrm{Aut}_{F}(D)$ and $k \in C^{\times}$ such that $\sigma$ commutes with $\tau$ and
\begin{equation} \label{eqn:isomorphism necessity division case}
\tau(a_i) = \Big( \prod_{l=i}^{m-1} \sigma^l(k) \Big) b_i,
\end{equation}
for all $i \in \{ 0, \ldots, m-1 \}$. Then $S_f \cong S_g$. Furthermore, for every such $\tau$ and $k$ the maps $Q_{\tau,k}: S_f \rightarrow S_g$,
\begin{equation} \label{eqn:isomorphism form of Q_tau,k}
Q_{\tau,k}: \sum_{i=0}^{m-1} x_i t^i \mapsto \tau(x_0) + \sum_{i=1}^{m-1} \tau(x_i) \big( \prod_{l=0}^{i-1} \sigma^l(k) \big) t^i,
\end{equation}
are $F$-isomorphisms between $S_f$ and $S_g$.
\item[(iii)] Suppose $f(t)$, $g(t)$ are not right invariant, $\sigma \vert_C$ has order at least $m-1$, and $\sigma$ commutes with all $F$-automorphisms of $D$. Then $S_f \cong S_g$ if and only if there exists $\tau \in \mathrm{Aut}_{F}(D)$ and $k \in C^{\times}$ such that \eqref{eqn:isomorphism necessity division case} holds for all $i \in \{ 0, \ldots, m-1 \}$. Every such $\tau$ and $k$ gives rise to an $F$-isomorphism $Q_{\tau,k}: S_f \rightarrow S_g$ and these are the only $F$-isomorphisms between $S_f$ and $S_g$.
\end{itemize}
\end{theorem}

\begin{proof}
\begin{itemize}
\item[(i)] Recall $\Theta_{\mathrm{id},0,k}$ is an $F$-automorphism of $R$ by Corollary \ref{cor:Automorphisms of D[t;sigma]}, moreover
\begin{align*}
\Theta_{\mathrm{id},0,k}(f(t)) &= (kt)^m - \sum_{i=0}^{m-1} a_i (kt)^i \\
&= \big( \prod_{l=0}^{m-1} \sigma^l(k) \big) t^m - a_0 - \sum_{i=1}^{m-1} a_i \big( \prod_{l=0}^{i-1} \sigma^l(k) \big) t^i \\
&=  \big( \prod_{l=0}^{m-1} \sigma^l(k) \big) \Big( t^m - \sum_{i=0}^{m-1} b_i t^i \big) = \big( \prod_{l=0}^{m-1} \sigma^l(k) \big) g(t),
\end{align*}
by \eqref{eqn:isomorphism necessity division case Q_id,k} and thus $\Theta_{\mathrm{id},0,k}$ restricts to an isomorphism between $S_f$ and $S_g$ by Theorem \ref{thm:lavrauw isomorphic petit algebras generalisation}. This restriction is $Q_{\mathrm{id},k}$ by a straightforward calculation.
\item[(ii)] The proof is similar to (i) using that $\Theta_{\tau,0,k} \in \mathrm{Aut}_{F}(R)$ by Corollary \ref{cor:Automorphisms of D[t;sigma]}.
\item[(iii)] We are left to prove that $S_f \cong S_g$ implies there exists $\tau \in \mathrm{Aut}_{F}(D)$ and $k \in C^{\times}$ such that \eqref{eqn:isomorphism necessity division case} holds for all $i \in \{ 0, \ldots, m-1 \}$, and that all $F$-isomorphisms between $S_f$ and $S_g$ have the form $Q_{\tau,k}$ where $\tau$ and $k$ satisfy \eqref{eqn:isomorphism necessity division case}.

Suppose $S_f \cong S_g$ and let $Q: S_f \rightarrow S_g$ be an $F$-isomorphism. $f(t)$ and $g(t)$ are not right invariant which is equivalent to $S_f$ and $S_g$ not being associative, therefore $\mathrm{Nuc}_l(S_f) = \mathrm{Nuc}_l(S_g) = D$ by Theorem \ref{thm:Properties of S_f petit}(i). Since any isomorphism preserves the left nucleus, $Q(D) = D$ and so $Q \vert_D = \tau$ for some $\tau \in \mathrm{Aut}_{F}(D)$. Suppose $Q(t) = \sum_{i=0}^{m-1} k_i t^i$ for some $k_i \in D$, then we have
\begin{equation} \label{eqn:general_isomorphism_theoremI}
Q(t \circ_f z) = Q(t) \circ_g Q(z) = \Big( \sum_{i=0}^{m-1} k_i t^i \Big)  \circ_g \tau(z) = \sum_{i=0}^{m-1} k_i \sigma^{i}(\tau(z)) t^i,
\end{equation}
and
\begin{equation} \label{eqn:general_isomorphism_theoremII}
Q(t \circ_f z) = Q(\sigma(z)t) = \sum_{i=0}^{m-1} \tau(\sigma(z)) k_i t^i
\end{equation}
for all $z \in D$. Comparing the coefficients of $t^i$ in \eqref{eqn:general_isomorphism_theoremI} and \eqref{eqn:general_isomorphism_theoremII} we obtain
\begin{equation} \label{eqn:general_isomorphism_theoremIII}
k_i \sigma^i(\tau(z)) = k_i \tau(\sigma^i(z)) = \tau(\sigma(z)) k_i,
\end{equation}
for all $i \in \{ 0, \ldots m-1 \}$ and all $z \in D$ as $\sigma$ and $\tau$ commute. In particular
\begin{equation*}
k_i \tau(\sigma^i(z)) = k_i \tau(\sigma(z)),
\end{equation*}
for all $i \in \{ 0, \ldots m-1 \}$ and all $z \in C$. This means
$$k_i \tau \big( \sigma^i(z) - \sigma(z) \big) = 0,$$
for all $i \in \{ 0, \ldots m-1 \}$ and all $z \in C$, i.e. $k_i = 0$ or $\sigma \vert_C = \sigma^i \vert_C$ for all $i \in \{0, \ldots, m-1 \}$. Now, $\sigma \vert_C$ has order at least $m-1$ or infinite order which means $\sigma^i \vert_C \neq \sigma \vert_C$ for all $1 \neq i \in \{ 0, \ldots, m-1 \}$ and thus $k_i = 0$ for all $1 \neq i \in \{ 0, \ldots, m-1 \}$. Therefore $Q(t) = kt$ for some $k \in D^{\times}$ such that $k \tau(\sigma(z)) = \tau(\sigma(z))k$ for all $z \in D$ by \eqref{eqn:general_isomorphism_theoremIII}. Hence $k \in C^{\times}$. 
 
Furthermore, we have
$$Q(z t^i) = Q(z) \circ_g Q(t)^i = \tau(z) (kt)^i = \tau(z) \Big( \prod_{l=0}^{i-1} \sigma^l(k) \Big) t^i,$$
for all $i \in \{ 1, \ldots, m-1 \}$ and all $z \in D$. Thus $Q$ has the form
$$Q_{\tau,k}: \sum_{i=0}^{m-1} x_i t^i \mapsto \tau(x_0) + \sum_{i=1}^{m-1} \tau(x_i) \big( \prod_{l=0}^{i-1} \sigma^l(k) \big) t^i,$$
for some $k \in C^{\times}$. Moreover, with $t^m = t \circ_f t^{m-1}$, also
\begin{equation} \label{eqn:general_isomorphism_theoremV}
\begin{split}
Q(t^m) &= Q \big( \sum_{i=0}^{m-1} a_i t^i \big) = \sum_{i=0}^{m-1} Q(a_i) \circ_g Q(t)^i \\
&= \tau(a_0) + \sum_{i=1}^{m-1} \tau(a_i) \big( \prod_{l=0}^{i-1} \sigma^l(k) \big) t^i,
\end{split}
\end{equation}
and $Q(t \circ_f t^{m-1}) = Q(t) \circ_g Q(t)^{m-1}$, i.e.
\begin{equation} \label{eqn:general_isomorphism_theoremVI}
Q(t^m) = Q(t) \circ_g Q(t)^{m-1} = \Big( \prod_{l=0}^{m-1} \sigma^l(k) \Big) t^m = \Big( \prod_{l=0}^{m-1} \sigma^l(k) \Big) \sum_{i=0}^{m-1} b_i t^i.
\end{equation}
Comparing \eqref{eqn:general_isomorphism_theoremV} and \eqref{eqn:general_isomorphism_theoremVI} gives $\tau(a_i) = \Big( \prod_{l=i}^{m-1} \sigma^l(k) \Big) b_i$, for all $i \in \{ 0, \ldots, m-1 \}$, and hence $Q$ has the form $Q_{\tau,k}$ where $\tau \in \mathrm{Aut}_{F}(D)$ and $k \in C^{\times}$ satisfy \eqref{eqn:isomorphism necessity division case}.
\end{itemize}
\end{proof}

\begin{remark}
Suppose $D = \mathbb{F}_{p^n}$ is a finite field of order $p^n$, $\sigma: \mathbb{F}_{p^n} \rightarrow \mathbb{F}_{p^n}, \ k \mapsto k^p$ is the Frobenius automorphism and $f(t), g(t) \in \mathbb{F}_{p^n}[t;\sigma]$ are irreducible of degree $m \in \{ 2, 3 \}$. Then $S_f$ and $S_g$ are isomorphic semifields if and only if there exists $\tau \in \mathrm{Aut}(\mathbb{F}_{p^n})$ and $k \in \mathbb{F}_{p^n}^{\times}$ are such that \eqref{eqn:isomorphism necessity division case} holds for all $i \in \{ 0, \ldots, m-1 \}$ by \cite[Theorems 4.2 and 5.4]{wene2000finite}. Therefore Theorem \ref{thm:general_isomorphism}(iii) can be seen as a generalisation of \cite[Theorems 4.2 and 5.4]{wene2000finite}.
\end{remark}

We obtain the following Corollaries of Theorem \ref{thm:general_isomorphism}:

\begin{corollary}
Let $f(t) = t^m -a \in R$ and define $g(t) = t^m - k^m a \in R$ for some $k \in F^{\times}$. Then $S_f \cong S_g$.
\end{corollary}

\begin{proof}
We have $\Big( \prod_{l=0}^{m-1} \sigma^l(k) \Big) a = k^m a$ and so $S_f \cong S_g$ by Theorem \ref{thm:general_isomorphism}(i).
\end{proof}

We next take a closer look at the equation \eqref{eqn:isomorphism necessity division case} to obtain necessary conditions for some $S_f$ and $S_g$ to be isomorphic.

\begin{corollary}
Suppose $f(t)$ and $g(t)$ are not right invariant, $\sigma$ commutes with all $F$-automorphisms of $D$ and $\sigma \vert_C$ has order at least $m-1$. If $S_f \cong S_g$ then $a_i = 0$ is equivalent to $b_i = 0$ for all $i \in \{ 0 , \ldots, m-1 \}$.
\end{corollary}

\begin{proof}
If $a_i \neq b_i$ then $\tau(a_i) \neq \Big( \prod_{l=i}^{m-1} \sigma^l(k) \Big) b_i$ for all $\tau \in \mathrm{Aut}_{F}(D)$ and all $k \in C^{\times}$ and so $S_f \ncong S_g$ by Theorem \ref{thm:general_isomorphism}(iii), a contradiction.
\end{proof}

Now suppose $C/F$ is a proper field extension of finite degree, $D$ is finite-dimensional as an algebra over $C$, and $D$ is also finite-dimensional when considered as an algebra over $F$. Let $N_{D/C}$ denote the norm of $D$ considered as an algebra over $C$, $N_{D/F}$ the norm of $D$ considered as an algebra over $F$ and $N_{C/F}$ be the norm of the field extension $C/F$. We have $N_{D/F}(z) = N_{C/F}(N_{D/C}(z))$ and $N_{D/F}(\tau(z)) = N_{D/F}(z)$, for all $\tau \in \mathrm{Aut}_{F}(D)$ and all $z \in D$ \cite[p.~547, 548]{magurn2002algebraic}.

\begin{corollary} \label{cor:norm isomorphism}
Suppose $f(t)$ and $g(t)$ are not right invariant, $\sigma$ commutes with all $F$-automorphisms of $D$ and $\sigma \vert_C$ has order at least $m-1$.
\begin{itemize}
\item[(i)] If $b_0 \neq 0$ and $N_{D/F}(a_0 b_0^{-1}) \notin F^{\times m}$ then $S_f \ncong S_g$.
\item[(ii)] If there exists $i \in \{ 0, \ldots, m-1 \}$ such that $b_i \neq 0$ and $N_{D/F}(a_i b_i^{-1}) \notin F^{\times (m-i)}$ then $S_f \ncong S_g$.
\end{itemize}
\end{corollary}

\begin{proof}
We prove (ii) since setting $i = 0$ in (ii) yields (i). Suppose, for a contradiction, that $S_f \cong S_g$. Then there exists $\tau \in \mathrm{Aut}_{F}(D)$ and $k \in C^{\times}$ such that $\tau(a_i) = \Big( \prod_{l=i}^{m-1} \sigma^l(k) \Big) b_i$
by Theorem \ref{thm:general_isomorphism}(iii). Applying $N_{D/F}$ we obtain
\begin{equation*}
\begin{split}
N_{D/F}(\tau(a_i)) &= N_{D/F}(a_i) = N_{D/F} \big( \big( \prod_{l=i}^{m-1} \sigma^l(k) \big) b_i \big) = N_{D/F}(k)^{m-i} N_{D/F}(b_i).
\end{split}
\end{equation*}
This implies $N_{D/F}(a_i b_i^{-1}) = N_{D/F}(k)^{m-i}$ by the multiplicity of the norm, but here $N_{D/F}(k)^{m-i} \in F^{\times (m-i)}$, a contradiction.
\end{proof}

In \cite{AndrewPhD}, isomorphisms between two nonassociative cyclic algebras were briefly investigated. We show Theorem \ref{thm:general_isomorphism}(iii) specialises to \cite[Proposition 3.2.8]{AndrewPhD} and \cite[Corollary 6.2.5]{AndrewPhD}.

Let $K/F$ be a cyclic Galois field extension of degree $m$ with $\mathrm{Gal}(K/F) = \langle \sigma \rangle$ and $a, b \in K \setminus F$. Then Theorem \ref{thm:general_isomorphism}(iii) shows when the nonassociative cyclic algebras $(K/F,\sigma,a)$ and $(K/F,\sigma,b)$ are isomorphic:

\begin{corollary} \label{cor:isomorphisms of nonassociative cyclic algebras}
(\cite[Proposition 3.2.8]{AndrewPhD}).
$(K/F,\sigma,a) \cong (K/F,\sigma,b)$ if and only if there exists $k \in K^{\times}$ and $j \in \{ 0, \ldots, m-1 \}$ such that $\sigma^j(a) = \big( \prod_{l=0}^{m-1} \sigma^l(k) \big) b = N_{K/F}(k) b$.
\end{corollary}

Now suppose additionally $K$ and $F$ are finite fields. It is well-known that the norm $N_{K/F}:K^{\times} \rightarrow F^{\times}$ is surjective for finite extensions of finite fields and so by Corollary \ref{cor:isomorphisms of nonassociative cyclic algebras}, we conclude:

\begin{corollary} \label{cor:isomorphisms between nonassociative cyclic algebras over finite fields}
(\cite[Corollary 6.2.5]{AndrewPhD}).
$(K/F,\sigma,a) \cong (K/F,\sigma,b)$ if and only if $\sigma^j(a) = k b$ for some $j \in \{ 0, \ldots, m-1 \}$ and some $k \in F^{\times}$.
\end{corollary}


\chapter{Automorphisms of Petit Algebras} \label{chapter:Automorphisms of S_f}

Let $D$ be an associative division ring with center $C$, $\sigma$ be a ring automorphism of $D$, $\delta$ be a left $\sigma$-derivation of $D$ and $f(t) \in R = D[t;\sigma,\delta]$. Here $S_f$ is a nonassociative algebra over $F = C \cap \mathrm{Fix}(\sigma) \cap \mathrm{Const}(\delta)$. Since $S_f = S_{df}$ for all $d \in D^{\times}$ we assume w.l.o.g. that $f(t)$ is monic, otherwise, if $f(t)$ has leading coefficient $d \in D^{\times}$, then consider $d^{-1}f(t)$. In this Chapter we study the automorphism groups of Petit algebras building upon our results in Chapter \ref{chapter:Isomorphisms Between some Petit Algebras}. We obtain partial results in the most general case where $\sigma$ is not necessarily the identity and $\delta$ is not necessarily $0$, however, most of our attention is given to the cases where $f(t)$ is either a differential or twisted polynomial (see Sections
\ref{section:Automorphisms of S_f, f(t) in D[t;delta]} and \ref{section:Automorphisms of S_f when f in D[t;sigma]} respectively). Later in Chapter \ref{chapter:Automorphisms of Nonassociative Cyclic Algebras} we go on to study the automorphism groups of nonassociative cyclic algebras.

\vspace*{4mm}
If $f(t) \in R$ is not right invariant, then $S_f$ is not associative and any automorphism of $S_f$ extends automorphisms of $\mathrm{Nuc}_l(S_f)$, $\mathrm{Nuc}_m(S_f)$ and $\mathrm{Nuc}_r(S_f)$ because the nuclei are invariant under automorphisms. Since $\mathrm{Nuc}_l(S_f) = \mathrm{Nuc}_m(S_f) = D$ and $\mathrm{Nuc}_r(S_f) = E(f)$ by Theorem \ref{thm:Properties of S_f petit}, we conclude:

\begin{lemma} \label{lem:automorphism restricts to right nucleus}
If $f(t) \in R$ is not right invariant, any $F$-automorphism of $S_f$ extends an $F$-automorphism of $D$ and an $F$-automorphism of $E(f)$.
\end{lemma}

By Theorem \ref{thm:lavrauw isomorphic petit algebras generalisation}, automorphisms $\Theta: R \rightarrow R$ induce isomorphisms between $S_f$ and $S_{\Theta(f(t))}$. This means if $\Theta(f(t)) = l f(t)$ for some $l \in D^{\times}$, then $\Theta$ induces an isomorphism $S_f \cong S_{l \Theta(f(t))} = S_f$, i.e. $\Theta$ induces an automorphism of $S_f$:

\begin{theorem} \label{thm:automorphism of skew polynomial induces automorphism of S_f}
If $\Theta$ is an $F$-automorphism of $R$ and $\Theta(f(t)) = l f(t)$ for some $l \in D^{\times}$, then $\Theta(b \circ c) = \Theta(b) \circ \Theta(c)$ for all $b, c \in S_f$, i.e. $\Theta$ induces an $F$-automorphism of $S_f$.
\end{theorem}

Theorem \ref{thm:automorphism of skew polynomial induces automorphism of S_f} allows us to find automorphisms of Petit algebras which are induced by automorphisms of $R$, later focusing on the special cases where $R$ is a differential or twisted polynomial ring. More generally, when $\sigma$ is not necessarily the identity, and $\delta$ is not necessarily $0$, we can still use Theorem \ref{thm:automorphism of skew polynomial induces automorphism of S_f} to obtain non-trivial automorphisms of $S_f$ for some quadratic $f(t)$:

\begin{proposition} \label{prop:aut of S_f from automorphism of D[t;sigma,delta]}
Suppose $1 \neq d \in C^{\times}$ is such that $d \sigma(d) = 1$ and $\delta$ is the inner $\sigma$-derivation given by
$$\delta: D \rightarrow D, \ b \mapsto c(1-d)^{-1} b - \sigma(b) c(1-d)^{-1},$$
for some $c \in D^{\times}$. If
$f(t) = t^2 - (c - d \sigma(c))(1 - d)^{-1} t - a \in R,$
then the map
$$H_{\mathrm{id},c,d}: S_f \rightarrow S_f, \ x_0 + x_1t \mapsto x_0 + x_1(c + d t),$$
is a non-trivial $F$-automorphism of $S_f$. Moreover, if additionally $d$ is a primitive $n^{\text{th}}$ root of unity for some $n > 1$, then $\langle H_{\mathrm{id},c,d} \rangle$ is a cyclic subgroup of $\mathrm{Aut}_{F}(S_f)$ of order $n$.
\end{proposition}

\begin{proof}
Recall
$$\Theta :R \rightarrow R, \ \sum_{i=0}^{n} b_i t^i \mapsto \sum_{i=0}^n b_i (c + d t)^i,$$
is an $F$-automorphism by Corollary \ref{cor:Automorphisms of D[t;sigma,delta]}. We have
\begin{align*}
\Theta(f(t)) &= (c + dt)^2 - (c - d \sigma(c))(1-d)^{-1} (c + d t) - a \\
&= d \sigma(d) t^2 + \Big( c d + d \sigma(c) + d \delta(d) - (c - d \sigma(c)) (1-d)^{-1} d \Big) t \\ & \qquad + \Big( c^2 + d \delta(c) - (c - d \sigma(c)) (1-d)^{-1} c - a \Big) \\
&= t^2 - (c - d \sigma(c))(1-d)^{-1} t - a = f(t),
\end{align*}
where we have used that $d \sigma(d) = 1$, hence the first assertion follows by Theorem \ref{thm:automorphism of skew polynomial induces automorphism of S_f}.

To prove the final assertion we first show $H_{\mathrm{id},c,d}^n$ has the form
\begin{equation} \label{eqn:form of H^n contain primitive nth root}
H_{\mathrm{id},c,d}^n(x_0 + x_1t) = x_0 + x_1 c (1-d^n)(1-d)^{-1} + x_1 d^n t,
\end{equation}
for all $n \in \mathbb{N}$ by induction: For $n = 1$ we have
$$x_0 + x_1 c (1-d^n)(1-d)^{-1} + x_1 d^n t = x_0 + x_1 c + x_1 d t = H_{\mathrm{id},c,d}(x_0 + x_1t)$$
as required. Assume as induction hypothesis that \eqref{eqn:form of H^n contain primitive nth root} holds for some $n \geq 1$, then
\begin{align*}
H_{\mathrm{id},c,d}&^{n+1}(x_0 + x_1t) = H_{\mathrm{id},c,d} \big( H_{\mathrm{id},c,d}^n(x_0 + x_1 t) \big) \\
&= H_{\mathrm{id},c,d} \big( x_0 + x_1 c (1-d^n)(1-d)^{-1} + x_1 d^n t \big) \\
&= x_0 + x_1 c (1-d^n)(1-d)^{-1} + x_1 d^n (c + d t) \\
&= x_0 + x_1 c \sum_{j=0}^{n-1} d^j + x_1 c d^n + x_1 d^{n+1}t = x_0 + x_1 c \sum_{j=0}^{n} d^j + x_1 d^{n+1}t \\
&= x_0 + x_1 c (1-d^{n+1})(1-d)^{-1} + x_1 d^{n+1} t,
\end{align*}
and thus \eqref{eqn:form of H^n contain primitive nth root} holds by induction. In particular \eqref{eqn:form of H^n contain primitive nth root} implies $H_{\mathrm{id},c,d}^n = \mathrm{id}$ if and only if $d^n = 1$, therefore if $d$ is a primitive $n^{\text{th}}$ root of unity, $\langle H_{\mathrm{id},c,d} \rangle$ is a cyclic subgroup of $\mathrm{Aut}_{F}(S_f)$ of order $n$.
\end{proof}

Setting $d = -1$ in Proposition \ref{prop:aut of S_f from automorphism of D[t;sigma,delta]} gives:

\begin{corollary}
Suppose $\mathrm{Char}(D) \neq 2$ and $\delta$ is the inner $\sigma$-derivation given by
$$\delta: D \rightarrow D, \ b \mapsto \frac{c}{2}b - \sigma(b) \frac{c}{2},$$
for some $c \in D^{\times}$. Then for 
$$f(t) = t^2 - \frac{1}{2}(c + \sigma(c)) t - a \in R,$$
the map $H_{\mathrm{id},c,-1}$ is an $F$-automorphism of $S_f$. Moreover $\{ \mathrm{id}, H_{\mathrm{id},c,-1} \}$ is a subgroup of $\mathrm{Aut}_{F}(S_f)$.
\end{corollary}

\begin{example}
Suppose $K = F(i)$ is a quadratic separable extension of $F$ with non-trivial automorphism $\sigma$ and let $\delta$ be the inner $\sigma$-derivation
$$\delta:K \rightarrow K, \ b \mapsto \frac{c}{1-i}(b-\sigma(b)),$$
for some $c \in K$. Then $i \sigma(i) = -i^2 = 1$ and $i$ is a primitive $4^{\text{th}}$ root of unity, so for
$$f(t) = t^2 - \frac{c - i \sigma(c)}{1-i}t - a \in K[t;\sigma,\delta],$$
the map $H_{\mathrm{id},c,i}$ is an automorphism of $S_f$ of order $4$ by Proposition \ref{prop:aut of S_f from automorphism of D[t;sigma,delta]}.
\end{example}

\begin{proposition}
Let $f(t) = t^m - \sum_{i=0}^{m-1} a_i t^i \in F[t] = F[t;\sigma,\delta] \subset R$, then for all $\tau \in \mathrm{Aut}_{F}(D)$ such that $\sigma \circ \tau = \tau \circ \sigma$ and $\delta \circ \tau = \tau \circ \delta$, the maps
$$H_{\tau,0,1}:S_f \rightarrow S_f, \ \sum_{i=0}^{m-1} b_i t^i \mapsto \sum_{i=0}^{m-1} \tau(b_i) t^i,$$
are $F$-automorphisms of $S_f$.
\end{proposition}

\begin{proof}
Let $\tau \in \mathrm{Aut}_{F}(D)$, then
$$\Theta: R \rightarrow R, \ \sum_{i=0}^n b_i t^i \mapsto \sum_{i=0}^n \tau(b_i) t^i,$$
is an $F$-automorphism if and only if $\sigma \circ \tau = \tau \circ \sigma$ and $\delta \circ \tau = \tau \circ \delta$ by Corollary \ref{cor:Automorphisms of D[t;sigma,delta]}. If all the coefficients of $f(t)$ are contained in $F$, a straightforward computation yields $\Theta(f(t)) = f(t)$, hence $H_{\tau,0,1} \in \mathrm{Aut}_{F}(S_f)$ by Theorem \ref{thm:automorphism of skew polynomial induces automorphism of S_f}.
\end{proof}

Given a nonassociative $F$-algebra $A$, an element $0 \neq c \in A$ has a \textbf{left inverse} $c_l \in A$ if $R_c(c_l) = c_l c = 1$, and a \textbf{right inverse} $c_r \in A$ if $L_c(c_r) = c c_r = 1$. We say $c$ is \textbf{invertible} if it has both a left and a right inverse. In \cite[p.~233]{wene2010inner} the definition of inner automorphisms of associative algebras is generalised to finite nonassociative division algebras, also called finite semifields. We can generalise this definition further to arbitrary nonassociative algebras:

\begin{definition}
An automorphism $G \in \mathrm{Aut}_F(A)$ is an \textbf{inner automorphism} if there is an element $0 \neq c \in A$ with left inverse $c_l$, such that $G(x) = (c_l x)c$ for all $x \in A$.
\end{definition}

Equivalently, an automorphism $G$ is inner if there is an element $0 \neq c \in A$ with left inverse $c_l$, such that $G(x) = R_c(L_{c_l}(x))$ for all $x \in A$.

When $A$ is a finite semifield, every non-zero element of $A$ is left invertible and our definition reduces to the definition of inner automorphisms given in \cite{wene2010inner}. Given an inner automorphism $G \in \mathrm{Aut}_F(A)$ and some $H \in \mathrm{Aut}_F(A)$ then a straightforward calculation shows $H^{-1} \circ G \circ H$ is the inner automorphism $A \rightarrow A, \ x \mapsto \big( H^{-1}(c_l)x \big) H^{-1}(c)$.

If $c \in \mathrm{Nuc}(A)$ is invertible, then $c_l = c_r$ as $\mathrm{Nuc}(A)$ is an associative algebra. We denote this element $c^{-1}$. Let $G_c$ denote the map $G_c: A \rightarrow A, \ x \mapsto (c^{-1} x)c$ for all invertible $c \in \mathrm{Nuc}(A)$.

We remark that if $c \in D$ and $f(t) \in D[t;\sigma,\delta]$ is not right invariant, then multiplying $1 = c_l c$ on the right by $c_r$ and using $c \in \mathrm{Nuc}_m(S_f)$ yields $c_r = (c_l c)c_r = c_l (c c_r) = c_l = c^{-1}$.

\begin{proposition} \label{prop:Inner automorphisms nucleus}
$Q = \{ G_c \ \vert \ c \in \mathrm{Nuc}(A) \text{ is invertible} \}$ is a subgroup of $\mathrm{Aut}_F(A)$ consisting of inner automorphisms.
\end{proposition}

\begin{proof}
We have
\begin{align*}
G_c(x) G_c(y) &= \big( (c^{-1} x)c \big) \big( (c^{-1} y)c \big) = (c^{-1} x) [ c ( (c^{-1} y)c ) ] = (c^{-1} x) [ [ c c^{-1} y ] c ] \\ &= (c^{-1} x) (yc) = \big( (c^{-1} x)y \big) c = (c^{-1} (xy))c = G_c(xy),
\end{align*}
for all $x, y \in A$, where we have used $c^{-1} \in \mathrm{Nuc}(A)$. Hence $G_c$ is multiplicative. Now $G_c$ is clearly bijective and $F$-linear, and so the maps $G_c$ are $F$-automorphisms of $A$ for all invertible $c \in \mathrm{Nuc}(A)$. 

We now prove $Q$ is a subgroup of $\mathrm{Aut}_F(A)$: Clearly $\mathrm{id}_A = G_1 \in Q$. Let $G_c, G_d \in Q$ for some invertible $c, d \in \mathrm{Nuc}(A)$, then
$$G_d(G_c(x)) = \big( d^{-1}((c^{-1}x)c) \big) d = (d^{-1} c^{-1} x) cd = (cd)^{-1}xcd = G_{cd}(x)$$
for all $x \in A$ and so $Q$ is closed under composition. Finally $G_c \circ G_{c^{-1}} = G_1 = \mathrm{id}$, therefore $Q$ is also closed under inverses since $c^{-1} \in \mathrm{Nuc}(A)$. Thus $Q$ is a subgroup of $\mathrm{Aut}_F(A)$.
\end{proof}

In the case where $A$ is a finite semifield, Proposition \ref{prop:Inner automorphisms nucleus} was proven in \cite[2 Lemma]{wene2010inner}. By Proposition \ref{prop:Inner automorphisms nucleus} we conclude:

\begin{corollary} \label{cor:inner automorphisms of S_f, D[t;sigma,delta]}
Given $f(t) \in R = D[t;\sigma,\delta]$, the maps $G_c(x) = (c^{-1} \circ x) \circ c$ are inner automorphisms of $S_f$ for all invertible $c \in \mathrm{Nuc}(S_f)$. In particular, if $f(t)$ is right semi-invariant the maps $G_c$ are inner automorphisms of $S_f$ for all $c \in D^{\times}$.
\end{corollary}

\begin{proof}
The first assertion follows immediately from Proposition \ref{prop:Inner automorphisms nucleus}. We have $f(t) \in R$ is right semi-invariant is equivalent to $D \subseteq \mathrm{Nuc}_r(S_f)$ by Theorem \ref{thm:semi-invariant iff D contained in E(f)}, thus either $S_f$ is associative or $\mathrm{Nuc}(S_f) = D \cap \mathrm{Nuc}_r(S_f) = D$. Therefore the maps $G_c$ are inner automorphisms of $S_f$ for all $c \in D^{\times}$.
\end{proof}

\section{Automorphisms of \texorpdfstring{$S_f$}{S\_f}, \texorpdfstring{$f(t) \in R = D[t;\delta]$}{f(t) in R = D[t;delta]}} \label{section:Automorphisms of S_f, f(t) in D[t;delta]}

Now suppose $D$ is an associative division ring of characteristic $p \neq 0$ and center $C$, $\sigma = \mathrm{id}$ and $0 \neq \delta$ is a derivation of $D$. In this Section we investigate the automorphisms of $S_f$ in the special case where
$$f(t) = t^{p^e} + a_1 t^{p^{e-1}} + \ldots + a_e t + d \in R = D[t;\delta].$$
Here $S_f$ is a nonassociative algebra over $F = C \cap \mathrm{Const}(\delta)$.

Recall from \eqref{p-power formula char p} that as $D$ has characteristic $p$, we can write $(t-b)^p = t^p - V_p(b)$, where $V_p(b) = b^p + \delta^{p-1}(b) + *$ for all $b \in D$, with $*$ a sum of commutators of $b, \delta(b), \ldots, \delta^{p-2}(b)$. In particular, if $b \in C$ then $* = 0$ and $V_p(b) = b^p + \delta^{p-1}(b)$. An iteration yields $(t-b)^{p^e} = t^{p^e} - V_{p^e}(b)$, for all $b \in D$ with $V_{p^e}(b) = V_p^e(b) = V_p(\ldots(V_p(b)\ldots)$.

The automorphisms of $R$ are described in Corollary \ref{cor:Automorphisms of D[t;delta]}. Together with Theorem \ref{thm:automorphism of skew polynomial induces automorphism of S_f}, this leads us to the following result, which generalises \cite[Proposition 7]{pumpluen2016nonassociative} in which $\tau = \mathrm{id}$:

\begin{corollary} \label{cor:auto of S_f D[t;delta] from auto of D[t;delta]}
Suppose $\tau \in \mathrm{Aut}_{F}(D)$ commutes with $\delta$ and
$$f(t) = t^{p^e} + a_1 t^{p^{e-1}} + \ldots + a_e t + d \in R,$$
where $a_1, \ldots, a_e \in \mathrm{Fix}(\tau)$. Then for any $b \in C$ such that 
\begin{equation} \label{eqn:auto of S_f D[t;delta] from auto of D[t;delta]}
V_{p^e}(b) + a_1 V_{p^{e-1}}(b) + \ldots + a_e b + d = \tau(d),
\end{equation}
the map
\begin{equation} \label{eqn:form of H_tau,-b,1 sigma=id}
H_{\tau,-b,1}: S_f \rightarrow S_f, \ \sum_{i=0}^{p^e-1} x_i t^i \mapsto \sum_{i=0}^{p^e-1} \tau(x_i)(t-b)^i,
\end{equation}
is an $F$-automorphism of $S_f$.
\end{corollary}

\begin{proof}
The map
$$\Theta: R \rightarrow R, \ \sum_{i=0}^n b_i t^i \mapsto \sum_{i=0}^n \tau(b_i) (t-b)^i,$$
is an $F$-automorphism for all $b \in C$ by Corollary \ref{cor:Automorphisms of D[t;delta]}. Furthermore, a close inspection of the proof of Corollary \ref{cor:Isomorphisms of Petit algebras D[t;delta]} shows $\Theta(f(t)) = f(t)$ if and only if \eqref{eqn:auto of S_f D[t;delta] from auto of D[t;delta]} holds, thus the assertion follows by Theorem \ref{thm:automorphism of skew polynomial induces automorphism of S_f}.
\end{proof}


Let $f(t) = t^p - t - d \in R$, then $H_{\tau,-b,1} \in \mathrm{Aut}_{F}(S_f)$ for all $b \in C$, $\tau \in \mathrm{Aut}_{F}(D)$ such that $\tau \circ \delta = \delta \circ \tau$ and $$\tau(d) = b - V_p(b) + d = b + d - b^p - \delta^{p-1}(b)$$ by Corollary \ref{cor:auto of S_f D[t;delta] from auto of D[t;delta]}.
In addition, if $f(t)$ is not right invariant, $\delta$ commutes with all $F$-automorphisms of $D$ and $F \subsetneq C$, these are the only $F$-automorphisms of $S_f$:

\begin{theorem} \label{thm:automorphisms of S_f t^p-t-a derivation}
Let $f(t) = t^p - t - d \in R$. Then for all $\tau \in \mathrm{Aut}_{F}(D)$ and $b \in C$ such that $\tau$ commutes with $\delta$ and 
\begin{equation} \label{eqn:automorphisms of S_f t^p-t-a derivation}
\tau(d) = b + d - b^p - \delta^{p-1}(b),
\end{equation}
the maps $H_{\tau,-b,1}$ given by \eqref{eqn:form of H_tau,-b,1 sigma=id} are $F$-automorphisms of $S_f$. Moreover, if $f(t)$ is not right invariant, $\delta$ commutes with all $F$-automorphisms of $D$ and $F \subsetneq C$, these are all the automorphisms of $S_f$.
\end{theorem}

\begin{proof}
Suppose $f(t)$ is not right invariant, $F \subsetneq C$ and $H \in \mathrm{Aut}_F(S_f)$. By Corollary \ref{cor:auto of S_f D[t;delta] from auto of D[t;delta]} we are left to show $H$ has the form $H_{\tau,-b,1}$ for some $\tau \in \mathrm{Aut}_{F}(D)$ and $b \in C$ satisfying \eqref{eqn:automorphisms of S_f t^p-t-a derivation}.

Since $f(t)$ is not right invariant, $S_f$ is not associative and $\mathrm{Nuc}_l(S_f) = D$ by Theorem \ref{thm:Properties of S_f petit}. Any automorphism of $S_f$ must preserve the left nucleus, thus $H(D) = D$ and so $H \vert_D = \tau$ for some $\tau \in \mathrm{Aut}_{F}(D)$. Write $H(t) = \sum_{i=0}^{p-1} b_i t^i$ for some $b_i \in D$, then we have
\begin{equation} \label{eqn:automorphisms of S_f t^p-t-a derivation 1}
H(t \circ z) = H(t) \circ H(z) = \sum_{i=0}^{p-1} b_i t^i \circ \tau(z) = \sum_{i=0}^{p-1} \sum_{j=0}^i \binom{i}{j} b_i \delta^{i-j}(\tau(z))t^j,
\end{equation}  
and
\begin{equation} \label{eqn:automorphisms of S_f t^p-t-a derivation 2}
H(t \circ z) = H(zt + \delta(z)) = \tau(z) \sum_{i=0}^{p-1} b_i t^i + \tau(\delta(z)),
\end{equation} 
for all $z \in D$. Comparing the coefficients of $t^{p-2}$ in \eqref{eqn:automorphisms of S_f t^p-t-a derivation 1} and \eqref{eqn:automorphisms of S_f t^p-t-a derivation 2} we obtain
$$\tau(z) b_{p-2} = b_{p-2} \tau(z) + \binom{p-1}{p-2} b_{p-1} \delta(\tau(z)),$$
for all $z \in D$. In particular,
$$\binom{p-1}{p-2} b_{p-1} \delta(\tau(z)) = 0,$$
for all $z \in C$, therefore $b_{p-1} = 0$ since $\binom{p-1}{p-2} \neq 0$ and $\delta(\tau(z)) = \tau(\delta(z)) \neq 0$ for all $z \in C$ with $z \notin \mathrm{Const}(\delta)$. Such a $z$ exists because $F \subsetneq C$.

Similarly, looking in turn at the coefficients of $t^{p-3}, \ldots, t^2, t$ in \eqref{eqn:automorphisms of S_f t^p-t-a derivation 1} and \eqref{eqn:automorphisms of S_f t^p-t-a derivation 2} yields $b_{p-1} = b_{p-2} = \ldots = b_2 = 0$, and comparing the coefficients of $t^0$ in \eqref{eqn:automorphisms of S_f t^p-t-a derivation 1} and \eqref{eqn:automorphisms of S_f t^p-t-a derivation 2} we get
\begin{equation} \label{eqn:automorphisms of S_f t^p-t-a derivation 3}
\tau(z) b_0 + \tau(\delta(z)) = b_0 \tau(z) + b_1 \delta(\tau(z)),
\end{equation}
for all $z \in D$. In particular, this means $\tau(\delta(z)) = b_1 \tau(\delta(z))$ for all $z \in C$ since $\tau$ and $\delta$ commute. Hence $b_1 = 1$ as $F \subsetneq C$. We conclude $\tau(z) b_0 = b_0 \tau(z)$ for all $z \in D$ by \eqref{eqn:automorphisms of S_f t^p-t-a derivation 3}, thus $b_0 \in C$ and so $H(t) = t - b$ for some $b \in C$. 

We now show $H$ has the form \eqref{eqn:form of H_tau,-b,1 sigma=id}: We have
$$H(z \circ t^i) = H(z) \circ H(t)^i = \tau(z)(t-b)^i,$$
for all $i \in \{ 1, \ldots, m-1 \}$, $z \in D$, and so $H$ has the form
$H_{\tau,-b,1}$ for some $b \in C$. Moreover, with $t^p = t \circ t^{p-1}$, also
\begin{equation} \label{eqn:automorphisms of S_f t^p-t-a derivation 4}
H(t^p) = H(t + d) = t - b + \tau(d),
\end{equation}
and $H(t \circ t^{p-1}) = H(t) \circ H(t)^{p-1}$, i.e.
\begin{align} \label{eqn:automorphisms of S_f t^p-t-a derivation 5}
\begin{split}
H(t^p) &= H(t) \circ H(t)^{p-1} = (t-b)^p \ \mathrm{mod}_r \ f \\ &= t^p - V_p(b) \ \mathrm{mod}_r \ f = t + d - V_p(b).
\end{split}
\end{align}
Comparing \eqref{eqn:automorphisms of S_f t^p-t-a derivation 4} and \eqref{eqn:automorphisms of S_f t^p-t-a derivation 5} gives
\begin{equation} \label{eqn:tau(d) = d - V_p(b) + b}
\tau(d) = d - V_p(b) + b = b + d - b^p - \delta^{p-1}(b),
\end{equation}
and thus $H$ has the form $H_{\tau,-b,1}$, where $\tau \in \mathrm{Aut}_{F}(D)$ and $b \in C$ satisfy \eqref{eqn:tau(d) = d - V_p(b) + b}.
\end{proof}

When $f(t) = t^p - t - d \in R$, we see $H_{\mathrm{id}, -1,1} \in \mathrm{Aut}_{F}(S_f)$ by Theorem \ref{thm:automorphisms of S_f t^p-t-a derivation}. Additionally, a straightforward calculation shows $H_{\mathrm{id}, -1,1}^i(t) = t-i$ for all $i \in \mathbb{N}$, therefore since $\mathrm{Char}(D) = p$, we have $H_{\mathrm{id}, -1,1}^i \neq \mathrm{id}$ for all $i \in \{ 1, \ldots, p-1 \}$. Furthermore
$$\Theta : R \rightarrow R: \sum_{i=0}^{n} b_i t^i \mapsto \sum_{i=0}^{n} b_i (t-1)^i$$
is an automorphism of order $p$ \cite[p.~90]{amitsur1954non}, thus $\langle H_{\mathrm{id}, -1,1} \rangle$ is a cyclic subgroup of $\mathrm{Aut}_{F}(S_f)$ of order $p$ \cite[Lemma 9]{pumpluen2016nonassociative}.

\vspace*{4mm}
Notice the equation \eqref{eqn:automorphisms of S_f t^p-t-a derivation} in Theorem \ref{thm:automorphisms of S_f t^p-t-a derivation} is remarkably similar to \eqref{eqn:(t-b) right divides f(t) in Char p} in Proposition \ref{prop:(t-b) right divides f(t) in Char p}. Comparing them yields a connection between the automorphisms of $S_f$ and factors of certain differential polynomials:

\begin{proposition} \label{prop:automorphisms right linear division f(t)=t^p-t-a}
Let $\tau \in \mathrm{Aut}_{F}(D)$, $f(t) = t^p - t - d \in R$ and $g(t) = t^p - t - (d - \tau(d)) \in R$.
\begin{itemize}
\item[(i)] Suppose $\tau$ commutes with $\delta$. If $b \in C^{\times}$ is such that $(t-b) \vert_r g(t)$, then $H_{\tau,-b,1} \in \mathrm{Aut}_{F}(S_f)$.
\item[(ii)] Given $b \in C^{\times}$, if $(t-b) \vert_r (t^p-t)$ then $H_{\mathrm{id},-b,1} \in \mathrm{Aut}_{F}(S_f)$.
\item[(iii)] Suppose $f(t)$ is not right invariant, $\delta$ commutes with all $F$-automorphisms of $D$ and $F \subsetneq C$. Then $H_{\tau,-b,1} \in \mathrm{Aut}_{F}(S_f)$ if and only if $b \in C^{\times}$ and $(t-b) \vert_r g(t).$
In particular, if $b \in C^{\times}$ then $H_{\mathrm{id},-b,1} \in \mathrm{Aut}_{F}(S_f)$ if and only if $(t-b) \vert_r (t^p-t)$.
\item[(iv)] Suppose $f(t)$ is not right invariant, $\delta$ commutes with all $F$-automorphisms of $D$ and $F \subsetneq C$. If $g(t)$ is irreducible then $H_{\tau,-b,1} \notin \mathrm{Aut}_F(S_f)$ for all $b \in C^{\times}$.
\end{itemize}
\end{proposition}

\begin{proof}
\begin{itemize}
\item[(i)] We have $(t-b) \vert_r g(t)$ is equivalent to $V_p(b) - b - d + \tau(d) = 0$ by Proposition \ref{prop:(t-b) right divides f(t) in Char p}. As $\tau$ commutes with $\delta$ and $b \in C^{\times}$, this yields $H_{\tau,-b,1} \in \mathrm{Aut}_{F}(S_f)$ by Theorem \ref{thm:automorphisms of S_f t^p-t-a derivation}.
\item[(ii)] follows by setting $\tau = \mathrm{id}$ in (i).
\item[(iii)] If $f(t)$ is not right invariant and $F \subsetneq C$, then $H_{\tau,-b,1} \in \mathrm{Aut}_{F}(S_f)$ if and only if $\tau$ commutes with $\delta$, $b \in C^{\times}$ and $\tau(d) = b - V_p(b) + d$ by Theorem \ref{thm:automorphisms of S_f t^p-t-a derivation}. Now $(t-b) \vert_r g(t)$ is equivalent to $V_p(b) - b - d + \tau(d) = 0$ by Proposition \ref{prop:(t-b) right divides f(t) in Char p} and the assertion follows.
\item[(iv)] If $g(t)$ is irreducible then in particular, $(t-b) \nmid_r g(t)$ for all $b \in C^{\times}$ and the assertion follows by (iii).
\end{itemize}
\end{proof}

\section{Automorphisms of \texorpdfstring{$S_f$}{S\_f}, \texorpdfstring{$f(t) \in R = D[t;\sigma]$}{f(t) in R = D[t;sigma]}} \label{section:Automorphisms of S_f when f in D[t;sigma]}

Now suppose $D$ is an associative division ring with center $C$, $\mathrm{Char}(D)$ is arbitrary, $\sigma$ is a non-trivial automorphism of $D$ and $f(t) \in R = D[t;\sigma]$. Thus $S_f$ is a nonassociative algebra over $F = C \cap \mathrm{Fix}(\sigma)$. Throughout this Section, if we assume $\sigma$ has order $\geq m-1$ we include infinite order. From Theorem \ref{thm:general_isomorphism} we obtain:

\begin{theorem} \label{thm:automorphism_of_S_f_division_case}
Let $f(t) = t^m - \sum_{i=0}^{m-1} a_i t^i \in R$.
\begin{itemize}
\item[(i)] For every $k \in C^{\times}$ such that
\begin{equation} \label{eqn:automorphism necessary id}
a_i = \Big( \prod_{l=i}^{m-1} \sigma^l(k) \Big) a_i,
\end{equation}
for all $i \in \{ 0, \ldots, m-1 \}$, the maps
\begin{equation}
H_{\mathrm{id},k}: \sum_{i=0}^{m-1} x_i t^i \mapsto x_0 + \sum_{i=1}^{m-1} x_i \big( \prod_{l=0}^{i-1} \sigma^l(k) \big) t^i,
\end{equation}
are $F$-automorphisms of $S_f$. Furthermore
$$\{ H_{\mathrm{id},k} \ \vert \ k \in K^{\times} \text{ satisfies } \eqref{eqn:automorphism necessary id} \text{ for all } i \in \{ 0, \ldots, m-1 \} \}$$
is a subgroup of $\mathrm{Aut}_{F}(S_f)$.
\item[(ii)] For every $\tau \in \mathrm{Aut}_{F}(D)$ and $k \in C^{\times}$ such that $\tau$ commutes with $\sigma$ and 
\begin{equation} \label{eqn:automorphism necessary}
\tau(a_i) = \Big( \prod_{l=i}^{m-1}\sigma^l(k) \Big) a_i,
\end{equation}
for all $i \in \{ 0, \ldots, m-1 \}$, the maps 
\begin{equation} \label{automorphism_of_Sf form of H}
H_{\tau , k} : \sum_{i=0}^{m-1} x_i t^i \mapsto \tau(x_0) + \sum_{i=1}^{m-1} \tau(x_i) \big( \prod_{l=0}^{i-1} \sigma^l(k) \big) t^i,
\end{equation}
are $F$-automorphisms of $S_f$. Moreover
\begin{align*}
\{ H_{\tau,k} \ \vert \ \tau \in \mathrm{Aut}_{F}(D), \ \tau \circ \sigma = \sigma \circ \tau, \ k \in K^{\times} \text{ satisfies } \eqref{eqn:automorphism necessary} \text{ for all } i \}
\end{align*}
is a subgroup of $\mathrm{Aut}_{F}(S_f)$.
\item[(iii)] Suppose $f(t)$ is not right invariant, $\sigma$ commutes with all $F$-automorphisms of $D$ and $\sigma \vert_C$ has order at least $m-1$. A map $H: S_f \rightarrow S_f$ is an $F$-automorphism of $S_f$ if and only if $H$ has the form $H_{\tau,k}$, where $\tau \in \mathrm{Aut}_{F}(D)$ and $k \in C^{\times}$ are such that \eqref{eqn:automorphism necessary} holds for all $i \in \{ 0, \ldots, m-1 \}$.
\end{itemize}
\end{theorem}

\begin{proof}
Note that the inverse of $H_{\tau,k}$ is $H_{\tau^{-1},\tau^{-1}(k^{-1})}$ and $H_{\tau,k} \circ H_{\rho,b} = H_{\tau \rho, \tau(b)k}$. The rest of the proof is trivial using Theorem \ref{thm:general_isomorphism}.

\end{proof}

The automorphisms $H_{\tau,k}$ in Theorem \ref{thm:automorphism_of_S_f_division_case} are restrictions of automorphisms 
$$\Theta: R \rightarrow R, \ \sum_{i=0}^{n} b_i t^i \mapsto \sum_{i=0}^{n} \tau(b_i) (kt)^i,$$
by the proof of Theorem \ref{thm:general_isomorphism}. If $f(t)$ is not right invariant, $\sigma$ commutes with all $F$-automorphisms of $D$ and $\sigma \vert_C$ has order at least $m-1$, then Theorem \ref{thm:automorphism_of_S_f_division_case}(iii) shows that all $F$-automorphisms of $S_f$ are restrictions of automorphisms of $R$. Conversely, when $\sigma \vert_C$ has order $< m-1$ and $\sigma$ commutes with all $F$-automorphisms of $D$, the automorphisms $H_{\tau,k}$ are restrictions of automorphisms of $R$ and form a subgroup of $\mathrm{Aut}_F(S_f)$. Moreover we have:

\begin{proposition} \label{prop:automorphism_of_Sf_division_caseII}
Suppose $f(t) = t^m - \sum_{i=0}^{m-1} a_i t^i \in R$ is not right invariant, $\sigma$ commutes with all $F$-automorphisms of $D$ and $\sigma \vert_C$ has order $n < m-1$. Let $H \in \mathrm{Aut}_{F}(S_f)$ and $N = \mathrm{Nuc}_r(S_f)$. Then $H \vert_D = \tau$ for some $\tau \in \mathrm{Aut}_{F}(D)$, $H_N \in \mathrm{Aut}_F(N)$ and $H(t) = g(t)$ with
\begin{equation} \label{eqn:H(t) form when ord(sigma) < m-1}
g(t) = k_1t + k_{1+n}t^{1+n} +  k_{1+2n}t^{1+2n} + \ldots + k_{1+sn} t^{1+sn},
\end{equation}
for some $k_{1+ln}\in D$.
\end{proposition}

\begin{proof}
Let  $H: S_f \rightarrow S_f$ be an automorphism. Then $H \vert_N \in \mathrm{Aut}_F(N)$ and $H \vert_D = \tau$ for some $\tau \in \mathrm{Aut}_{F}(D)$ by Lemma \ref{lem:automorphism restricts to right nucleus}.

Suppose $H(t) = \sum_{i=0}^{m-1} k_i t^i$ for some $k_i \in D$. Comparing the coefficients of $t$ in $H(t \circ z) = H(t) \circ H(z)= H(\sigma(z)t)$ we obtain $k_i = 0$ or $\sigma^{i}(z) = \sigma(z)$, for all $i \in\{ 0, \ldots, m-1\}$ and all $z \in C$. 
Now, since $\sigma \vert_C$ has order $n < m-1$, $\sigma^i(z) = \sigma(z)$ for all $z \in C$ if and only if $i = 1 + nl$ for some $l \in \mathbb{Z}$. Therefore $k_i = 0$ for every $i \neq 1 + nl$, $l \in \mathbb{N} \cup \{ 0 \}$, $i \in  \{ 0, \ldots, m-1 \}$ and hence $H(t)$ has the form \eqref{eqn:H(t) form when ord(sigma) < m-1} for some $s$ with $sn < m-1$.
\end{proof}

Suppose $\sigma$ commutes with all $F$-automorphisms of $D$ and $\sigma \vert_C$ has order at least $m-1$, then the automorphism groups of $S_f$ for $f(t) = t^m-a \in R$ not right invariant, are essential to understanding the automorphism groups of all the algebras $S_g$, as for all nonassociative $S_g$ with $g(t) = t^m - \sum_{i=0}^{m-1} b_i t^i \in R$ and $b_0 = a$, $\mathrm{Aut}_F(S_g)$ is a subgroup of $\mathrm{Aut}_{F}(S_f)$:

\begin{theorem} \label{thm:Aut(S_f) subgroup}
Suppose $\sigma$ commutes with all $F$-automorphisms of $D$ and $\sigma \vert_C$ has order at least $m-1$. Let $g(t) = t^m - \sum_{i=0}^{m-1} b_i t^i \in R$ not be right invariant.
\begin{itemize}
\item[(i)] If $f(t) = t^m - b_0 \in R$ is not right invariant, then $\mathrm{Aut}_{F}(S_g)$ is a subgroup of $\mathrm{Aut}_{F}(S_f)$.
\item[(ii)] If $f(t) = t^m - \sum_{i=0}^{m-1} a_i t^i \in R$ is not right invariant and $a_j \in \{ 0 , b_j \}$ for all $j \in \{ 0, \ldots , m-1 \}$, then $\mathrm{Aut}_{F}(S_g)$ is a subgroup of $\mathrm{Aut}_{F}(S_f)$.   
\end{itemize}
If additionally all automorphisms of $S_g$ are extensions of the identity, i.e. have the form $H_{\mathrm{id},k}$ for some $k \in C^{\times}$, then $\mathrm{Aut}_{F}(S_g)$ is a normal subgroup of $\mathrm{Aut}_{F}(S_f)$.
\end{theorem}

\begin{proof}
\begin{itemize}
\item[(i)] Let $H \in \mathrm{Aut}_{F}(S_g)$, then $H$ has the form $H_{\tau,k}$ where $\tau \in \mathrm{Aut}_{F}(D)$ and $k \in C^{\times}$ satisfy $\tau(b_i) = \big( \prod_{l=i}^{m-1} \sigma^l(k) \big) b_i$ for all $i \in \{ 0, \ldots, m-1 \}$ by Theorem \ref{thm:automorphism_of_S_f_division_case}(iii). In particular, $\tau(b_0) = \Big( \prod_{l=0}^{m-1} \sigma^l(k) \Big) b_0,$ thus $H_{\tau,k}$ is also an automorphism of $S_f$, again by Theorem \ref{thm:automorphism_of_S_f_division_case}(iii). Therefore $\mathrm{Aut}_{F}(S_g) \leq \mathrm{Aut}_{F}(S_f)$.

Suppose additionally that all automorphisms of $S_g$ are extensions of the identity. Let $H_{\rho,k} \in \mathrm{Aut}_{F}(S_f)$, $H_{\mathrm{id},z} \in \mathrm{Aut}_{F}(S_g)$ for some $\rho \in \mathrm{Aut}_{F}(D)$ and $k, z \in C^{\times}$. The inverse of $H_{\rho,k}$ is $H_{\rho^{-1},\rho^{-1}(k^{-1})}$, furthermore
\begin{align*}
H_{\rho,k} \Big( &H_{\mathrm{id},z} \Big( H_{\rho^{-1},\rho^{-1}(k^{-1})} \Big( \sum_{i=0}^{m-1} x_i t^i \Big) \Big) \Big) \\
&= H_{\rho,k} \Big( H_{\mathrm{id},z} \Big( \rho^{-1}(x_0) + \sum_{i=1}^{m-1} \rho^{-1}(x_i) \big( \prod_{l=0}^{i-1} \sigma^l(\rho^{-1}(k^{-1})) \big) t^i \Big) \Big) \\
&= H_{\rho,k} \Big( \rho^{-1}(x_0) + \sum_{i=1}^{m-1} \rho^{-1}(x_i) \big( \prod_{l=0}^{i-1} \sigma^l(\rho^{-1}(k^{-1})) \big) \big( \prod_{l=0}^{i-1} \sigma^l(z) \big) t^i \Big) \\
&= x_0 + \sum_{i=1}^{m-1} x_i \big( \prod_{l=0}^{i-1} \rho(\sigma^l(z)) \big) t^i = x_0 + \sum_{i=1}^{m-1} x_i \big( \prod_{l=0}^{i-1} \sigma^l(\rho(z)) \big) t^i \\
&= H_{\mathrm{id}, \rho(z)} \Big( \sum_{i=0}^{m-1} x_i t^i \Big),
\end{align*}
because $\sigma$ and $\rho$ commute. Therefore if we show $H_{\mathrm{id}, \rho(z)} \in \mathrm{Aut}_{F}(S_g)$, then indeed $\mathrm{Aut}_{F}(S_g)$ is a normal subgroup of $\mathrm{Aut}_{F}(S_f)$.

As $H_{\mathrm{id},z} \in \mathrm{Aut}_{F}(S_g)$, Theorem \ref{thm:automorphism_of_S_f_division_case}(iii) implies $\prod_{l=i}^{m-1} \sigma^l(z) = 1$, for all $i \in \{0, \ldots, m-1 \}$ such that $b_i \neq 0$. Applying $\rho$ and using that $\sigma$ and $\rho$ commute, we obtain
$$\rho(1) = 1 = \rho \Big( \prod_{l=i}^{m-1} \sigma^l(z) \Big) = \prod_{l=i}^{m-1} \sigma^l(\rho(z)),$$
for all $i \in \{0, \ldots, m-1 \}$ such that $b_i \neq 0$. Thus
$$b_i = \Big( \prod_{l=i}^{m-1} \sigma^l(\rho(z)) \Big) b_i,$$
for all $i \in \{ 0,\ldots, m-1 \}$ and hence $H_{\mathrm{id}, \rho(z)} \in \mathrm{Aut}_{F}(S_g)$ by Theorem \ref{thm:automorphism_of_S_f_division_case}.
\item[(ii)] The proof is analogous to (i).
\end{itemize}
\end{proof}

\subsection{Cyclic Subgroups of \texorpdfstring{$\mathrm{Aut}_F(S_f)$}{Aut(S\_f)}}

We now give some conditions for $\mathrm{Aut}_F(S_f)$ to have cyclic subgroups of certain order. In the special case when the coefficients of $f(t)$ are contained in $F$ we obtain the following:

\begin{proposition} \label{prop:Aut(S_f) Coefficients in F 2}
Let $f(t) = t^m - \sum_{i=0}^{m-1} a_i t^i \in F[t;\sigma] \subseteq R$.
\begin{itemize}
\item[(i)] If $\sigma$ has finite order $n$, then $\langle H_{\sigma,1} \rangle \cong \mathbb{Z}/n \mathbb{Z}$ is a subgroup of $\mathrm{Aut}_{F}(S_f)$.
\item[(ii)] Suppose $D = K$ is a cyclic Galois field extension of $F$ of prime degree $m$, $\mathrm{Gal}(K/F) = \langle \sigma \rangle$, $a_0 \neq 0$ and not all $a_1, \ldots, a_{m-1}$ are zero. Then $\mathrm{Aut}_{F}(S_f) = \langle H_{\sigma,1} \rangle \cong \mathbb{Z}/m \mathbb{Z}$.
\end{itemize}
\end{proposition}

\begin{proof}
\begin{itemize}
\item[(i)] Since $a_i \in F$, we have $\sigma(a_i) = a_i = \Big( \prod_{l=i}^{m-1} \sigma^l(1) \Big) a_i,$ for all $i \in \{ 0, \ldots, m-1 \}$ and so $H_{\sigma,1} \in \mathrm{Aut}_{F}(S_f)$ by Theorem \ref{thm:automorphism_of_S_f_division_case}(ii). Furthermore, $H_{\sigma^j,1} \circ H_{\sigma^l,1} = H_{\sigma^{l+j},1}$ and $H_{\sigma^n,1} = H_{\mathrm{id},1}$, hence $\langle H_{\sigma,1} \rangle = \{ H_{\mathrm{id},1}, H_{\sigma,1}, \ldots, H_{\sigma^{n-1}, 1} \} \cong \mathbb{Z}/n \mathbb{Z}$ is a cyclic subgroup of order $n$.
\item[(ii)] The automorphisms of $S_f$ are exactly the maps $H_{\sigma^j,k}$ for some $j \in \{ 0, \ldots, m-1 \}$ and $k \in K^{\times}$ such that 
\begin{equation} \label{eqn:Aut(S_f) Coefficients in F, field caseIII}
\sigma^j(a_i) = a_i = \Big( \prod_{l=i}^{m-1} \sigma^l(k) \Big) a_i,
\end{equation}
for all $i \in \{ 0, \ldots, m-1 \}$ by Theorem \ref{thm:automorphism_of_S_f_division_case}(iii). The maps $H_{\sigma^j,1}$ are therefore automorphisms of $S_f$ for all $j \in \{ 0, \ldots, m-1 \}$. We show that these are all the automorphisms of $S_f$: We have $N_{K/F}(k) = 1$ because $a_0 \neq 0$, hence by Hilbert's Theorem 90, there exists $\alpha \in K$ such that $k = \sigma(\alpha)/\alpha$. Let $q \in \{ 1, \ldots, m-1 \}$ be such that $a_q \neq 0$, then
$$1 = \prod_{l=q}^{m-1} \sigma^l(k) = \prod_{l=q}^{m-1} \sigma^l \big( \frac{\sigma(\alpha)}{\alpha} \big) = \frac{\prod_{l=q+1}^{m} \sigma^l(\alpha)}{\prod_{l=q}^{m-1} \sigma^l(\alpha)} = \frac{\alpha}{\sigma^q(\alpha)},$$
by \eqref{eqn:Aut(S_f) Coefficients in F, field caseIII}.
This means $\alpha \in \mathrm{Fix}(\sigma^q) = F$ as $m$ is prime. Therefore
$k = \sigma(\alpha)/\alpha = \alpha/\alpha = 1$
as required.
\end{itemize}
\end{proof}

If $F$ contains a primitive $n^{\text{th}}$ root of unity for some $n \geq 2$, then there exist algebras $S_f$ whose automorphism groups contain a cyclic subgroup of order $n$:

\begin{theorem} \label{thm:Primitive root then subgroup of order m}
Let $f(t) = t^m - a \in R$. If $n \vert m$ and $F$ contains a primitive $n^{\text{th}}$ root of unity $\omega$, then $\mathrm{Aut}_{F}(S_f)$ contains a cyclic subgroup of order $n$ generated by $H_{\mathrm{id},\omega}$.
\end{theorem}

\begin{proof}
We have
$\omega \sigma(\omega) \cdots \sigma^{m-1}(\omega) = \omega^{m} = 1$
and thus $H_{\mathrm{id},\omega} \in \mathrm{Aut}_{F}(S_f)$ by Theorem \ref{thm:automorphism_of_S_f_division_case}(i). Notice $H_{\mathrm{id},\omega^j}$ is not the identity automorphism for all $j \in \{ 1, \ldots, n-1 \}$ and that $H_{\mathrm{id}, \omega^n} = H_{\mathrm{id},1}$ is the identity. Furthermore
\begin{align*}
H_{\mathrm{id},\omega^j} \Big( H_{\mathrm{id},\omega^l} \Big( \sum_{i=0}^{m-1} x_i t^i \Big) \Big) &= H_{\mathrm{id},\omega^j} \Big( \sum_{i=0}^{m-1} x_i \omega^{li} t^i \Big) = \sum_{i=0}^{m-1} x_i \omega^{li} \omega^{ji} t^i \\ &= H_{\mathrm{id},\omega^{j+l}} \Big( \sum_{i=0}^{m-1} x_i t^i \Big),
\end{align*}
thus $H_{\mathrm{id},\omega^j} \circ H_{\mathrm{id},\omega^l} = H_{\mathrm{id},\omega^{j+l}}$ for all $j, l \in \{ 0, \ldots, n-1 \}$. This implies $\langle H_{\mathrm{id},\omega} \rangle = \{ H_{\mathrm{id},1}, H_{\mathrm{id},\omega}, \ldots, H_{\mathrm{id},\omega^{n-1}} \}$ is a cyclic subgroup.
\end{proof}

\begin{proposition} \label{prop:primitive root of unity t^ml - sum a_im t^mi}
Suppose $F$ contains a primitive $n^{\text{th}}$ root of unity $\omega$ and let $f(t) = t^{nm} - \sum_{i=0}^{m-1} a_{in} t^{in} \in R$. Then $\mathrm{Aut}_{F}(S_f)$ contains a cyclic subgroup of order $n$ generated by $H_{\mathrm{id},\omega}$. 
\end{proposition}

\begin{proof}
We have
$$\Big( \prod_{l=in}^{nm-1} \sigma^l(\omega) \Big) a_{in} = \omega^{mn-in} a_{in} = a_{in},$$
for all $i \in \{ 0, \ldots, m-1 \}$ which implies $H_{\mathrm{id},\omega} \in \mathrm{Aut}_{F}(S_f)$ by Theorem \ref{thm:automorphism_of_S_f_division_case}(i). The rest of the proof is similar to Theorem \ref{thm:Primitive root then subgroup of order m}.
\end{proof}

$F$ contains a primitive $2^{\text{nd}}$ root of unity whenever $\mathrm{Char}(F) \neq 2$, namely $-1$. Therefore setting $n = 2$ in Proposition \ref{prop:primitive root of unity t^ml - sum a_im t^mi} yields:

\begin{corollary}
If $\mathrm{Char}(F) \neq 2$ and $f(t) = t^{2m} - \sum_{i=0}^{m-1} a_{2i} t^{2i} \in R$, then $\{ H_{\mathrm{id},1}, H_{\mathrm{id},-1} \}$ is a subgroup of $\mathrm{Aut}_{F}(S_f)$ of order $2$.
\end{corollary}

\subsection{Inner Automorphisms}

In this subsection we consider the case where $\sigma \vert_C$ has finite order $m$, and look at the inner automorphisms of $S_f$.

\begin{proposition} \label{prop:automorphisms sigma vertC order m}
Suppose $\sigma \vert_C$ has finite order $m$ and $f(t) = t^m - a \in R$. Then the maps
$$G_c: S_f \rightarrow S_f, \ \sum_{i=0}^{m-1} x_i t^i \mapsto \sum_{i=0}^{m-1} x_i c^{-1} \sigma^i(c) t^i,$$
are inner automorphisms for all $c \in C^{\times}$. Furthermore, $\{ G_c \ \vert \ c \in C^{\times} \}$ is a non-trivial subgroup of $\mathrm{Aut}_F(S_f)$.
\end{proposition}

\begin{proof}
Let $c \in C^{\times}$, then
$\prod_{l=0}^{i-1} \sigma^l \big( c^{-1}\sigma(c) \big) = c^{-1} \sigma^i(c)$
for all $i \geq 1$, thus $G_c = H_{\mathrm{id},k}$ where $k = c^{-1} \sigma(c)$. Moreover, we have
$$\prod_{l=0}^{m-1} \sigma^l(c^{-1}\sigma(c)) = c^{-1} \sigma^m(c) = c^{-1}c = 1,$$
and hence $G_c = H_{\mathrm{id}, k} \in \mathrm{Aut}_F(S_f)$ by Theorem \ref{thm:automorphism_of_S_f_division_case}(i). A simple calculation shows
$G_c \Big( \sum_{i=0}^{m-1} x_i t^i \Big) = \Big( c^{-1} \sum_{i=0}^{m-1} x_i t^i \Big) c,$
and so $G_c$ are inner automorphisms for all $c \in C^{\times}$.

A straightforward calculation shows $G_c \circ G_d = G_{cd}$ for all $c,d \in C^{\times}$, therefore $\{ G_c \ \vert \ c \in C^{\times} \}$ is closed under composition. Additionally, we have $G_1 = \mathrm{id}_{S_f}$ and $G_c \circ G_{c^{-1}} = G_1$, hence $\{ G_c \ \vert \ c \in C^{\times} \}$ forms a subgroup of $\mathrm{Aut}_F(S_f)$.
 Finally, $G_c$ is not the identity for all $c \in C \setminus F$ which yields the assertion.
\end{proof}

\begin{example}
We use the same set-up as in Hanke \cite[p.~200]{hanke2005twisted}: Let $K = \mathbb{Q}(\alpha)$ where $\alpha$ is a root of $x^3 + x^2 - 2x - 1 \in \mathbb{Q}[x]$. Then $K / \mathbb{Q}$ is a cyclic Galois field extension of degree $3$, its Galois group is generated by $\tilde{\sigma}: \alpha \mapsto \alpha^2 - \alpha + 1$. Let also $L = K(\beta)$ where $\beta$ is a root of $x^3 + (\alpha-2)x^2 - (\alpha+1)x + 1 \in K[x]$. Then $L/K$ is a cyclic Galois field extension of degree $3$ and there is $\tau \in \mathrm{Gal}(L/K)$ with $\tau(\beta) = \beta^2 + (\alpha-2)\beta - \alpha$. Define $\pi = \alpha^2 + 2\alpha - 1 \in K$, then $D = (L/K,\tau,2\pi)$ is an associative cyclic division algebra of degree $3$ and $\tilde{\sigma}$ extends to an automorphism $\sigma$ of $D$.

Suppose $f(t) = t^3 - a \in D[t;\sigma]$ and note that $\mathrm{Cent}(D) = K$ and $\sigma \vert_K = \tilde{\sigma}$ has order $3$. Therefore $\{ G_c \ \vert \ c \in K^{\times} \}$ is a non-trivial subgroup of $\mathrm{Aut}_F(S_f)$ consisting of inner automorphisms by Proposition \ref{prop:automorphisms sigma vertC order m}.
\end{example}

\begin{corollary} \label{cor:G_c subgroup order j cyclic}
Suppose $\sigma \vert_C$ has finite order $m$ and $f(t) = t^m - a \in R$. Let $c \in C^{\times}$ and suppose there exists $j \in \mathbb{N}$ such that $c^j \in F$. Let $j$ be minimal. Then $\langle G_c \rangle \cong \mathbb{Z}/j \mathbb{Z}$ is a cyclic subgroup of $\mathrm{Aut}_{F}(S_f)$ consisting of inner automorphisms. 
\end{corollary}

\begin{proof}
The maps $G_c, G_{c^2}, \ldots$ are all automorphisms of $S_f$ by Proposition \ref{prop:automorphisms sigma vertC order m}, furthermore a straightforward calculation shows $G_{c^i} \circ G_{c^l} = G_{c^{i+l}}$ for all $i, l \in \mathbb{N}$. Notice $G_c^j = G_{c^j}$ is the identity automorphism if and only if $c^j \in F$, then by the minimality of $j$ we conclude
$\langle G_c \rangle = \{ G_c, G_{c^2}, \ldots, G_{c^{j-1}}, \mathrm{id} \}$ is a cyclic subgroup.
\end{proof}

If $D$ is finite-dimensional over $C$, then since $\sigma \vert_C$ has finite order $m$, $\sigma$ has inner order $m$ by the Skolem-Noether Theorem. That is $\sigma^m$ is an inner automorphism $I_u: x \mapsto u^{-1}xu$ for some $u \in D^{\times}$, where we can choose $u \in D^{\times}$ such that $\sigma(u) = u$ \cite[Theorem 1.1.22]{jacobson1996finite}. Given $f(t) = \sum_{j=0}^n a_j u^{n-j} t^{jm} \in D[t;\sigma]$ such that $a_n =1$ and $a_j \in C$, then $f(t)$ is right semi-invariant by Theorem \ref{thm:right semi invariant conditions}. Therefore as a direct consequence of Corollary \ref{cor:inner automorphisms of S_f, D[t;sigma,delta]}, we obtain:

\begin{corollary}
Suppose $\sigma$ and $m$ are as above, and let $f(t) = \sum_{j=0}^n a_j u^{n-j} t^{jm} \in R$ where $a_n =1$ and $a_j \in C$. Then the maps 
$$G_c : S_f \rightarrow S_f, \ \sum_{i=0}^{mn-1} x_i t^i \mapsto \sum_{i=0}^{mn-1} c^{-1} x_i \sigma^i(c) t^i,$$
are inner automorphisms for all $c \in D^{\times}$.
\end{corollary}

\subsection{Necessary Conditions for \texorpdfstring{$H_{\tau,k} \in \mathrm{Aut}_F(S_f)$}{H\_\{tau,k\} to be an Automorphism of S\_f}}

We next take a closer look at the equality \eqref{eqn:automorphism necessary}, to obtain necessary conditions for $\tau \in \mathrm{Aut}_F(D)$ to extend to $H_{\tau,k} \in \mathrm{Aut}_F(S_f)$. The more non-zero coefficients $f(t)$ has the more restrictive these conditions become.

\begin{proposition} \label{prop:Conditions for automorphism group trivial}
Let $\sigma, \tau \in \mathrm{Aut}_F(D)$ and $k \in C^{\times}$ be such that \eqref{eqn:automorphism necessary} holds for all $i \in \{ 0, \ldots, m-1 \}$, i.e. $\tau(a_i) = \big( \prod_{l=i}^{m-1} \sigma^l(k) \big) a_i$ for all $i \in \{ 0, \ldots, m-1 \}$.
\begin{itemize}
\item[(i)] If $a_{m-1} \neq 0$ then
$$\tau(a_i) = \big( \prod_{l=i}^{m-1} \sigma^{l-m+1} \big( \tau(a_{m-1})a_{m-1}^{-1} \big) \big) a_i$$
for all $i \in \{ 0, \ldots, m-1 \}$.
\item[(ii)] If two consecutive $a_s, a_{s+1} \in \mathrm{Fix}(\tau)^{\times}$ then $k = 1$.
\item[(iii)] If $a_{m-1} \in \mathrm{Fix}(\tau)^{\times}$ then $k=1$.
\item[(iv)] If there is $i \in \{ 0, \ldots, m-1 \}$ such that $a_i \in \mathrm{Fix}(\tau)^{\times}$ then $1 = \prod_{l=i}^{m-1} \sigma^l(k).$
\end{itemize}
\end{proposition}

\begin{proof}
\begin{itemize}
\item[(i)] Since $a_{m-1} \neq 0$, \eqref{eqn:automorphism necessary} implies $\tau(a_{m-1}) = \sigma^{m-1}(k) a_{m-1}$, hence $k = \sigma^{-m+1} \big( \tau(a_{m-1}) a_{m-1}^{-1} \big)$. Subbing this back into \eqref{eqn:automorphism necessary} yields the assertion.
\item[(ii)] If there are two consecutive $a_s, a_{s+1} \in \mathrm{Fix}(\tau)^{\times}$, then by \eqref{eqn:automorphism necessary} we conclude $\prod_{l=s}^{m-1} \sigma^l(k) = 1 = \prod_{l=s+1}^{m-1} \sigma^l(k)$, thus cancelling gives $\sigma^s(k) = 1$, i.e. $k = 1$. 

The proof of (iii) and (iv) is similar to \cite[Proposition 9]{brownautomorphism2017}, but we need not assume $D$ is a field: 
\item[(iii)] Since $a_{m-1} \in \mathrm{Fix}(\tau)^{\times}$, \eqref{eqn:automorphism necessary} yields $\tau(a_{m-1}) = a_{m-1} = \sigma^{m-1}(k) a_{m-1}$, thus $\sigma^{m-1}(k) = 1$ and so $k = 1$.
\item[(iv)] We have $\tau(a_i) = a_i = \prod_{l=i}^{m-1} \sigma^l(k) a_i$ by \eqref{eqn:automorphism necessary}, hence $1 = \prod_{l=i}^{m-1} \sigma^l(k)$.
\end{itemize}
\end{proof}

The condition \eqref{eqn:automorphism necessary} heavily restricts the choice of available $k$ to $k=1$ in many cases. Therefore in many instances we conclude $\mathrm{Aut}_F(S_f)$ is isomorphic to a subgroup of $\mathrm{Aut}_F(D)$ or is trivial:

\begin{corollary}
Suppose $f(t) = t^m - \sum_{i=0}^{m-1} a_i t^i \in R$ is not right invariant, $\sigma$ commutes with all $F$-automorphisms of $D$ and $\sigma$ has order at least $m-1$. Suppose also that one of the following holds:
\begin{itemize}
\item[(i)] $a_{m-1} \neq 0$ and there exists $j \in \{ 0, \ldots, m-2 \}$ such that
$$\tau(a_j) \neq \big( \prod_{l=j}^{m-1} \sigma^{l-m+1} \big( \tau(a_{m-1})a_{m-1}^{-1} \big) \big) a_j,$$
for all $\mathrm{id} \neq \tau \in \mathrm{Aut}_F(D)$.
\item[(ii)] $a_{m-1} \in F^{\times}$ and for all $\mathrm{id} \neq \tau \in \mathrm{Aut}_F(D)$ there exists $j \in \{ 0, \ldots, m-2 \}$ such that $a_j \notin \mathrm{Fix}(\tau)$.
\item[(iii)] There are two consecutive $a_s, a_{s+1} \in F^{\times}$ and for all $\mathrm{id} \neq \tau \in \mathrm{Aut}_F(D)$ there exists $j \in \{ 0, \ldots, m-1 \}$ such that $a_j \notin \mathrm{Fix}(\tau)$.
\end{itemize}
Then $\mathrm{Aut}_F(S_f)$ is trivial.
\end{corollary}

\begin{proof}
Suppose $H \in \mathrm{Aut}_F(S_f)$, then $H = H_{\tau,k}$ for some $\tau \in \mathrm{Aut}_F(D)$ and $k \in C^{\times}$ satisfying \eqref{eqn:automorphism necessary} by Theorem \ref{thm:automorphism_of_S_f_division_case}.
\begin{itemize}
\item[(i)] We have $\tau = \mathrm{id}$ by Proposition \ref{prop:Conditions for automorphism group trivial}(i), therefore \eqref{eqn:automorphism necessary} implies $a_{m-1} = \sigma^{m-1}(k) a_{m-1}$. This means $k = 1$ and $H = H_{\mathrm{id},1}$ is trivial.
\item[(ii)] We have $k = 1$ by Proposition \ref{prop:Conditions for automorphism group trivial}(iii), therefore \eqref{eqn:automorphism necessary} implies $\tau(a_i) = a_i$ for all $i \in \{ 0,\ldots, m-1 \}$ and thus $\tau = \mathrm{id}$. Therefore $H = H_{\mathrm{id},1}$ and $\mathrm{Aut}_F(S_f)$ is trivial.
\item[(iii)] Proposition \ref{prop:Conditions for automorphism group trivial}(ii) yields $k=1$, therefore by \eqref{eqn:automorphism necessary} we have $\tau(a_i) = a_i$ for all $i \in \{ 0,\ldots, m-1 \}$, hence $\tau = \mathrm{id}$ and $H = H_{\mathrm{id},1}$.
\end{itemize}
\end{proof}

Denote by $\mathrm{Cent}_{\mathrm{Aut}(D)}(\sigma)$ the centralizer of $\sigma$ in $\mathrm{Aut}_{F}(D)$. If the coefficients of $f(t)$ are all in $F$, we have:

\begin{proposition} \label{prop:Aut(S_f) Coefficients in F}
Let $f(t) = t^m - \sum_{i=0}^{m-1} a_i t^i \in F[t] = F[t;\sigma] \subset R$.
\begin{itemize}
\item[(i)] $\{ H_{\tau,1} \ \vert \ \tau \in \mathrm{Cent}_{\mathrm{Aut}(D)}(\sigma) \} \cong \mathrm{Cent}_{\mathrm{Aut}(D)}(\sigma)$ is a subgroup of $\mathrm{Aut}_{F}(S_f)$.
\item[(ii)] Suppose $f(t)$ is not right invariant, $\sigma$ commutes with all $F$-automorphisms of $D$, $a_{m-1} \neq 0$ and $\sigma \vert_C$ has order at least $m-1$. Then
$$\mathrm{Aut}_{F}(S_f) = \{ H_{\tau,1} \ \vert \ \tau \in \mathrm{Aut}_{F}(D) \} \cong \mathrm{Aut}_{F}(D).$$
\end{itemize}
\end{proposition}

\begin{proof}
\begin{itemize}
\item[(i)] We have $\tau(a_i) = a_i = \Big( \prod_{l=i}^{m-1} \sigma^l(1) \Big) a_i$, for all $i \in \{ 0, \ldots, m-1 \}$, $\tau \in \mathrm{Cent}_{\mathrm{Aut}(D)}(\sigma)$, therefore $\{ H_{\tau,1} \ \vert \ \tau \in \mathrm{Cent}_{\mathrm{Aut}(D)}(\sigma) \}$ is a subset of $\mathrm{Aut}_{F}(S_f)$ by Theorem \ref{thm:automorphism_of_S_f_division_case}(ii). Furthermore, $H_{\tau,1} \circ H_{\rho,1} = H_{\tau \rho,1}$ for all $\tau, \rho \in \mathrm{Cent}_{\mathrm{Aut}(D)}(\sigma)$, hence $\{ H_{\tau,1} \ \vert \ \tau \in \mathrm{Cent}_{\mathrm{Aut}(D)}(\sigma) \}$ is a subgroup of $\mathrm{Aut}_{F}(S_f)$ because $\mathrm{Cent}_{\mathrm{Aut}(D)}(\sigma)$ is a group.
\item[(ii)] In this case we prove the subgroup in (i) is all of $\mathrm{Aut}_F(S_f)$: Let $H \in \mathrm{Aut}_{F}(S_f)$, then $H$ has the form $H_{\tau,k}$ for some $\tau \in \mathrm{Aut}_{F}(D)$, $k \in C^{\times}$ such that $\tau(a_i) = a_i = \Big( \prod_{l=i}^{m-1} \sigma^l(k) \Big) a_i$, for all $i \in \{ 0, \ldots, m-1 \}$ by Theorem \ref{thm:automorphism_of_S_f_division_case}(iii). In particular, $a_{m-1} = \sigma^{m-1}(k) a_{m-1}$ which implies $k = 1$ since $a_{m-1}\neq 0$. Thus $H = H_{\tau,1}$ as required.
\end{itemize}
\end{proof}

Now suppose $C/F$ is a proper field extension of finite degree, and $D$ is finite-dimensional as an algebra over $C$ and over $F$. Let $N_{D/C}$ denote the norm of $D$ considered as an algebra over $C$, $N_{D/F}$ denote the norm of $D$ considered as an algebra over $F$, and $N_{C/F}$ denote the norm of the field extension $C/F$. Recall $N_{D/F}(z) = N_{C/F}(N_{D/C}(z))$ for all $z \in D$ \cite[\S 7.4]{jacobson1985basic} and $N_{D/F}(\tau(z)) = N_{D/F}(z)$ for all $\tau \in \mathrm{Aut}_{F}(D)$ and all $z \in D$ \cite[p.~547]{magurn2002algebraic}. Applying $N_{D/C}$ to \eqref{eqn:automorphism necessary} yields a necessary condition for $H_{\tau,k}$ to be an automorphism of $S_f$:

\begin{proposition} \label{prop:automorphism division norm argument}
Suppose $\sigma$ commutes with all $F$-automorphisms of $D$, $\sigma \vert_C$ has order at least $m - 1$ and $f(t) = t^m - \sum_{i=0}^{m-1} a_i t^i \in D[t;\sigma]$ is not right invariant. If $H_{\tau,k} \in \mathrm{Aut}_{F}(S_f)$, then $N_{C/F}(k)$ is a $[D:C](m-i)$th root of unity for all $i \in \{ 0, \ldots, m-1 \}$ such that $a_i \neq 0$. In particular, if $D$ is commutative and $a_0 \neq 0$ then $N_{C/F}(k)$ is an $m$th root of unity.
\end{proposition}

\begin{proof}
Applying $N_{D/F}$ to \eqref{eqn:automorphism necessary} we obtain
\begin{equation*}
\begin{split}
N_{D/F}&(\tau(a_i)) = N_{D/F}(a_i) 
= N_{D/F} \Big( \Big( \prod_{l=i}^{m-1} \sigma^l(k) \Big) a_i \Big) \\ 
&= N_{D/F}(k)^{m-i} N_{D/F}(a_i) = N_{C/F}(k)^{[D:C](m-i)} N_{D/F}(a_i),
\end{split}
\end{equation*}
for all $i \in \{ 0, \ldots, m-1 \}$ since $N_{D/F}(k) = N_{C/F}(N_{D/C}(k)) = N_{C/F}(k)^{[D:C]}$ for all $k \in C^{\times}$. This yields $N_{C/F}(k)^{[D:C](m-i)} = 1$ or $a_i = 0$  for every $i \in \{ 0, \ldots, m-1 \}$. 
\end{proof}

\subsection{Connections Between Automorphisms of \texorpdfstring{$S_f$}{S\_f}, \texorpdfstring{$f(t) = t^m-a \in D[t;\sigma]$}{f(t) = t\^{}m - a in D[t;sigma]}, and Factors of Skew Polynomials}

Suppose $D$ is an associative division ring with center $C$, $\sigma$ is a non-trivial ring automorphism of $D$ and $f(t) = t^m-a \in R = D[t;\sigma]$. We compare Theorem \ref{thm:automorphism_of_S_f_division_case} and Proposition \ref{prop:rightdivdegree1} to obtain connections between the automorphisms of $S_f$ and factors of certain skew polynomials:

\begin{proposition} \label{prop:automorphisms right divisors of t^m-tau(a)a^-1}
\begin{itemize}
\item[(i)] Suppose $\tau \in \mathrm{Aut}_{F}(D)$ commutes with $\sigma$ and $k \in C^{\times}$, then $H_{\tau,k} \in \mathrm{Aut}_{F}(S_f)$ if $(t-k) \vert_r (t^m - \tau(a) a^{-1})$.
\item[(ii)] Suppose $(t-k) \vert_r (t^m-1)$ for some $k \in C^{\times}$, then $H_{\mathrm{id},k} \in \mathrm{Aut}_{F}(S_f)$.
\item[(iii)] Suppose $f(t)$ is not right invariant, $\sigma$ commutes with all $F$-automorphisms of $D$ and $\sigma \vert_C$ has order at least $m-1$. If $k \in C^{\times}$ and $\tau \in \mathrm{Aut}_{F}(D)$, then $H_{\tau,k} \in \mathrm{Aut}_{F}(S_f)$ if and only if $(t-k) \vert_r (t^m - \tau(a)a^{-1})$. In particular, $H_{\mathrm{id},k} \in \mathrm{Aut}_{F}(S_f)$ if and only if $(t-k) \vert_r (t^m - 1)$.
\end{itemize}
\end{proposition}

\begin{proof}
\begin{itemize}
\item[(i)] If $(t-k) \vert_r (t^m - \tau(a) a^{-1})$ then $\prod_{l=0}^{m-1} \sigma^l(k) = \tau(a)a^{-1}$ by Proposition \ref{prop:rightdivdegree1}, thus $H_{\tau,k} \in \mathrm{Aut}_{F}(S_f)$ by Theorem \ref{thm:automorphism_of_S_f_division_case}(i).
\item[(ii)] follows by setting $\tau = \mathrm{id}$ in (i).
\item[(iii)] If $(t-k) \vert_r (t^m - \tau(a) a^{-1})$ then $H_{\tau,k} \in \mathrm{Aut}_{F}(S_f)$ by (i). Conversely if $H_{\tau,k} \in \mathrm{Aut}_{F}(S_f)$, then $\tau(a) = \big( \prod_{l=0}^{m-1} \sigma^l(k) \big) a$ by Theorem \ref{thm:automorphism_of_S_f_division_case}(iii), and thus $(t-k) \vert_r (t^m - \tau(a) a^{-1})$ by Proposition \ref{prop:rightdivdegree1}.
\end{itemize}
\end{proof}

When $D$ is commutative, Proposition \ref{prop:automorphisms right divisors of t^m-tau(a)a^-1} shows there is an injective map between the monic linear right divisors of $t^m-1$ in $R$ and the automorphisms of $S_f$ of the form $H_{\mathrm{id},k}$. If additionally $f(t)$ is not right invariant, $\sigma$ commutes with all $F$-automorphisms of $D$ and $\sigma$ has order at least $m-1$, then there is a bijection between the monic linear right divisors of $t^m-1$ in $R$ and the automorphisms of $S_f$ of the form $H_{\mathrm{id},k}$.

\begin{corollary} \label{cor:H_id,k automorphism link to right divisors}
Suppose $\sigma$ commutes with all $F$-automorphisms of $D$, $f(t)$ is not right invariant, and $\sigma \vert_C$ has order at least $m-1$. 
\begin{itemize}
\item[(i)] If all $F$-automorphisms of $S_f$ have the form $H_{\mathrm{id},k}$ for some $k \in C^{\times}$, then $(t-b) \nmid_r (t^m-\tau(a)a^{-1}),$ for all $\mathrm{id} \neq \tau \in \mathrm{Aut}_{F}(D)$, $b \in C^{\times}$. In addition, if $D$ is commutative, $m$ is prime and $\mathrm{Fix}(\sigma)$ contains a primitive $m^{\text{th}}$ root of unity, then $t^m - \tau(a)a^{-1} \in R$ is irreducible for all $\mathrm{id} \neq \tau \in \mathrm{Aut}_{F}(D)$.
\item[(ii)] If $t^m - \tau(a)a^{-1} \in R$ is irreducible for all $\mathrm{id} \neq \tau \in \mathrm{Aut}_{F}(D)$, then
$$\mathrm{Aut}_{F}(S_f) = \Big\{ H_{\mathrm{id},k} \ \vert \ k \in C^{\times} \text{ such that } \prod_{l=0}^{m-1} \sigma^l(k) = 1 \Big\} .$$
\end{itemize}
\end{corollary}

\begin{proof}
\begin{itemize}
\item[(i)] The first assertion follows immediately from Proposition \ref{prop:automorphisms right divisors of t^m-tau(a)a^-1}. If $D$ is commutative this means $(t-b) \nmid_r (t^m-\tau(a)a^{-1})$ for all $b \in D^{\times}$ and all $\mathrm{id} \neq \tau \in \mathrm{Aut}_{F}(D)$. If also $\mathrm{Fix}(\sigma)$ contains a primitive $m^{\text{th}}$ root of unity, then $t^m-\tau(a)a^{-1} \in R$ is irreducible for all $\mathrm{id} \neq \tau \in \mathrm{Aut}_{F}(D)$ by Theorem \ref{thm:Petit(19)}.
\item[(ii)] Suppose $t^m - \tau(a)a^{-1} \in R$ is irreducible for all $\mathrm{id} \neq \tau \in \mathrm{Aut}_{F}(D)$, then in particular, $(t-b) \nmid_r (t^m-\tau(a)a^{-1})$ for all $b \in C^{\times}$, $\mathrm{id} \neq \tau \in \mathrm{Aut}_{F}(D)$. Therefore $H_{\tau,b} \notin \mathrm{Aut}_{F}(S_f)$ for all $b \in C^{\times}$, and all $\mathrm{id} \neq \tau \in \mathrm{Aut}_{F}(D)$ by Proposition \ref{prop:automorphisms right divisors of t^m-tau(a)a^-1} and so
$$\mathrm{Aut}_{F}(S_f) = \Big\{ H_{\mathrm{id},k} \ \vert \ k \in C^{\times} \text{ such that } \prod_{l=0}^{m-1} \sigma^l(k) = 1 \Big\}$$
by Theorem \ref{thm:automorphism_of_S_f_division_case}.
\end{itemize}
\end{proof}

\section[Automorphisms of some Jha-Johnson Semifields]{Automorphisms of Jha-Johnson Semifields Obtained from Skew Polynomial Rings} \label{section:Automorphisms of Jha-Johnson Semifields Obtained from Skew Polynomial Rings}

In this Section we study the automorphism groups of the Jha-Johnson semifields which arise from Petit's algebra construction. Let $K = \mathbb{F}_{p^h}$ be a finite field of order $p^h$ for some prime $p$ and $\sigma$ be a non-trivial $\mathbb{F}_p$-automorphism of $K$, i.e. $\sigma: K \rightarrow K, \ k \mapsto k^{p^r}$ for some $r \in \{ 1, \ldots, h-1 \}$ is a power of the Frobenius automorphism. Notice that $F = \mathrm{Fix}(\sigma) \cong \mathbb{F}_{q}$ where $q = p^{\mathrm{gcd}(r,h)}$, $\sigma$ has order $n = h/ \mathrm{gcd}(r,h)$, and $\sigma$ commutes with all $\mathbb{F}_q$-automorphisms of $K$.

Suppose $f(t) \in R = K[t;\sigma]$ is monic and irreducible, so that $S_f$ is a Jha-Johnson semifield \cite[Theorem 15]{lavrauw2013semifields} (see Theorem \ref{thm:Jha-Johnson_is_S_f}). Recall that when $f(t) = t^m - a \in K[t;\sigma]$ is irreducible, $a \in K \setminus F$ and $n \geq m$, then $S_f$ is called a Sandler semifield \cite{sandler1962autotopism}. The automorphism groups of Sandler semifields are particularly relevant as for all Jha-Johnson semifields $S_g$ with $g(t) = t^m - \sum_{i=0}^{m-1} b_i t^i \in K[t;\sigma]$ irreducible and $b_0 = a$, $\mathrm{Aut}_{F}(S_g)$ is a subgroup of $\mathrm{Aut}_{F}(S_f)$ by Theorem \ref{thm:Aut(S_f) subgroup}. We study the automorphisms of Sandler semifields in Section \ref{section:Automorphisms of Sandler Semifields}. When $m = n$, Sandler semifields are precisely the nonassociative cyclic algebras over $F$, and their automorphism groups are studied in detail in Chapter \ref{chapter:Automorphisms of Nonassociative Cyclic Algebras}.

We remark that the results in this Section hold more generally when $f(t) \in R$ is not necessarily irreducible, here $S_f$ is a nonassociative algebra over $F$, but $S_f$ is not a semifield unless $f(t)$ is irreducible (Theorem \ref{thm:S_f_division_iff_irreducible}).

Theorem \ref{thm:automorphism_of_S_f_division_case} becomes:

\begin{theorem} \label{thm:automorphism_of_S_f_finite field case}
Let $f(t) = t^m - \sum_{i=0}^{m-1} a_i t^i \in K[t;\sigma]$ and define $s_i = (p^{rm}-p^{ri})/(p^r-1).$
\begin{itemize}
\item[(i)] Suppose $k \in K^{\times}$ is an $s_i$th root of unity for all $i \in \{ 0, \ldots, m - 1 \}$ with $a_i \neq 0$. Then $H_{\mathrm{id},k} \in \mathrm{Aut}_{F}(S_f)$.
\item[(ii)] Suppose $k \in K^{\times}$ and $j \in \{ 0, \ldots, n-1 \}$ are such that $\sigma^j(a_i) = k^{s_i} a_i$ for all $i \in \{ 0 , \ldots, m-1 \}$, then $H_{\sigma^j,k} \in \mathrm{Aut}_{F}(S_f)$.
\item[(iii)] Suppose $f(t)$ is not right invariant and $n \geq m-1$. Then $H \in \mathrm{Aut}_{F}(S_f)$ if and only if $H = H_{\sigma^j,k}$ for some $k \in K^{\times}$, $j \in \{ 0, \ldots, n-1 \}$ such that $\sigma^j(a_i) = k^{s_i} a_i$ for all $i \in \{ 0 , \ldots, m-1 \}$.
\end{itemize}
\end{theorem}

\begin{proof}
\begin{itemize}
\item[(i)] We have
$$\Big( \prod_{l=i}^{m-1} \sigma^l(k) \Big) a_i = \Big( \prod_{l=i}^{m-1} k^{p^{rl}} \Big) a_i = k^{s_i} a_i = a_i,$$
for all $i \in \{ 0, \ldots, m-1 \}$ and thus $H_{\mathrm{id},k} \in \mathrm{Aut}_{F}(S_f)$ by Theorem \ref{thm:automorphism_of_S_f_division_case}(i).
\item[(ii)] Similarly to (i) we have
$$\Big( \prod_{l=i}^{m-1} \sigma^l(k) \Big) a_i = \Big( \prod_{l=i}^{m-1} k^{p^{rl}} \Big) a_i = k^{s_i} a_i = \sigma^j(a_i)$$
for all $i \in \{ 0, \ldots, m-1 \}$, and thus $H_{\sigma^j,k} \in \mathrm{Aut}_{F}(S_f)$ by Theorem \ref{thm:automorphism_of_S_f_division_case}(ii).
\item[(iii)] follows by (ii) and Theorem \ref{thm:automorphism_of_S_f_division_case}(iii).
\end{itemize}
\end{proof}

As a direct consequence of Theorem \ref{thm:Aut(S_f) subgroup}, we have:
\begin{corollary} 
Let $n \geq m-1$ and $g(t) = t^m - \sum_{i=0}^{m-1} b_i t^i \in R$ not be right invariant.
\begin{itemize}
\item[(i)] If $f(t) = t^m - b_0 \in R$ is not right invariant then $\mathrm{Aut}_{F}(S_g)$ is a subgroup of $\mathrm{Aut}_{F}(S_f)$.
\item[(ii)] If $f(t) = t^m - \sum_{i=0}^{m-1} a_i t^i \in R$ is not right invariant and $a_j \in \{ 0 , b_j \}$ for all $j \in \{ 0, \ldots , m-1 \}$, then $\mathrm{Aut}_{F}(S_g)$ is a subgroup of $\mathrm{Aut}_{F}(S_f)$.   
\end{itemize}
\end{corollary}

Proposition \ref{prop:Aut(S_f) Coefficients in F 2} becomes:

\begin{corollary}
Let $f(t) = t^m - \sum_{i=0}^{m-1} a_i t^i \in F[t;\sigma] \subseteq K[t;\sigma]$. Then $\langle H_{\sigma,1} \rangle \cong \mathbb{Z}/n \mathbb{Z}$ is a cyclic subgroup of $\mathrm{Aut}_{F}(S_f)$.
\end{corollary}

The Hughes-Kleinfeld semifields can be written in the form $S_f$ for irreducible $f(t) \in K[t;\sigma]$ of degree $2$ by Theorem \ref{thm:S_f as Hughes Kleinfeld and Knuth Semifield}. Setting $m=2$ in Theorem \ref{thm:automorphism_of_S_f_finite field case} gives a description of the automorphisms of Hughes Kleinfeld semifields. In particular we obtain \cite[Proposition 8.3.1]{AndrewPhD} and \cite[Corollary 8.3.6]{AndrewPhD} as straightforward corollaries of Theorem \ref{thm:automorphism_of_S_f_division_case}(iii).

\subsection{Automorphisms of Sandler Semifields} \label{section:Automorphisms of Sandler Semifields}

In this Subsection we study the automorphism group of $S_f$ when $f(t) = t^m-a \in K[t;\sigma]$. In particular, when $a \in K \setminus F$ and $\sigma$ has order $n \geq m$ this yields results on the automorphisms of Sandler semifields. Let $$s = \frac{p^{rm}-1}{p^r-1}.$$
Theorem \ref{thm:automorphism_of_S_f_finite field case} implies:

\begin{corollary} \label{cor:automorphisms of Sandler semifields}
Let $f(t) = t^m - a \in K[t;\sigma]$.
\begin{itemize}
\item[(i)] The set $\{ H_{\mathrm{id},k} \ \vert \ k \in K, \ k^s=1 \}$ is a cyclic subgroup of  $\mathrm{Aut}_F(S_f)$ of order $\mathrm{gcd}(s, p^h-1)$.
\item[(ii)] If $p \equiv 1 \ \mathrm{mod} \ m$, then there are at least $m$ automorphisms of $S_f$ of the form $H_{\mathrm{id},k}$. In particular, $\mathrm{Aut}_{F}(S_f)$ is not trivial.
\item[(iii)] Suppose $h$ is even and at least one of $r, m$ are even. If $p \equiv -1 \ \mathrm{mod} \ m$, then there are at least $m$ automorphisms of $S_f$ of the form $H_{\mathrm{id},k}$. In particular, $\mathrm{Aut}_{F}(S_f)$ is not trivial.
\item[(iv)] Suppose $p$ is an odd prime and $m = 2$. Then $\mathrm{Aut}_{F}(S_f)$ is not trivial.
\end{itemize}
\end{corollary}

\begin{proof}
\begin{itemize}
\item[(i)] If $k \in K^{\times}$ is an $s^{\text{th}}$ root of unity then $H_{\mathrm{id},k} \in \mathrm{Aut}_{F}(S_f)$ by Theorem \ref{thm:automorphism_of_S_f_finite field case}(i). Moreover a straightforward calculation shows $H_{\mathrm{id},k} \circ H_{\mathrm{id},l} = H_{\mathrm{id},kl} \in \{ H_{\mathrm{id},k} \ \vert \ k^s=1 \}$ for all $s^{\text{th}}$ roots of unity $k, l \in K^{\times}$. This means $\{ H_{\mathrm{id},k} \ \vert \ k^s=1 \}$ is isomorphic to the cyclic subgroup of $K^{\times}$ consisting of all $s^{\text{th}}$ roots of unity. There are precisely $\mathrm{gcd}(s, p^h-1)$ $s^{\text{th}}$ roots of unity in $K$ by \cite[Proposition II.2.1]{koblitz1994course} which yields the assertion.
\item[(ii)] We have $\mathrm{gcd}(s, p^h-1) \geq m$ by the proof of Corollary \ref{cor:p=1mod m finite field irreducibility criteria}, therefore there are at least $m$ automorphisms of $S_f$ of the form $H_{\mathrm{id},k}$ by (i).
\item[(iii)] We have $p^h \equiv (-1)^h \ \mathrm{mod} \ m \equiv 1 \ \mathrm{mod} \ m$ because $h$ is even. If $r$ is even then $p^r \equiv 1 \ \mathrm{mod} \ m$ and
\begin{align*}
s \ \mathrm{mod} \ m & \equiv \Big( \sum_{i=0}^{m-1} (p^{ri} \ \mathrm{mod} \ m) \Big) \ \mathrm{mod} \ m  \equiv \big( \sum_{i=0}^{m-1} 1 \big) \ \mathrm{mod} \ m \equiv 0 \ \mathrm{mod} \ m.
\end{align*}
On the other hand, if $r$ is odd then $m$ must be even, therefore $p^r \equiv -1 \ \mathrm{mod} \ m$ and
\begin{align*}
s \ \mathrm{mod} \ m & \equiv \Big( \sum_{i=0}^{m-1} (p^{ri} \ \mathrm{mod} \ m) \Big) \ \mathrm{mod} \ m \\ & \equiv \big( \sum_{i=0}^{m-1} (-1)^i \big) \ \mathrm{mod} \ m \equiv 0 \ \mathrm{mod} \ m.
\end{align*}
In either case, $m \vert (p^h-1)$ and $m \vert s$. Hence $\mathrm{gcd}(s, p^h-1) \geq m$, so there are at least $m$ automorphisms of $S_f$ of the form $H_{\mathrm{id},k}$ by (i).
\item[(iv)] $p$ is odd so $p \equiv 1 \ \mathrm{mod} \ 2$ and the result follows by (ii).
\end{itemize}
\end{proof}

Corollary \ref{cor:automorphisms of Sandler semifields}(i) implies that if $\mathrm{Aut}_F(S_f)$ is trivial for a single $f(t) = t^m-a \in K[t;\sigma]$, then all $t^m-c \in K[t;\sigma]$, $c \in K$, are reducible: Indeed if $\mathrm{Aut}_{F}(S_f)$ is trivial then $\mathrm{gcd}(s, p^h-1) = 1$ by Corollary \ref{cor:automorphisms of Sandler semifields}(i), and so $t^m - c \in K[t;\sigma]$ is reducible for all $c \in K$ by Corollary \ref{cor:Finite field t^m-a irreducibility criteria}.

\begin{proposition} \label{prop:inner automorphisms of S_f over finite fields}
Let $f(t) = t^m-a \in K[t;\sigma]$.
\begin{itemize}
\item[(i)] If $c \in K^{\times}$ is a $(p^{rm}-1)$ root of unity, then the map
$$G_c : S_f \rightarrow S_f, \ \sum_{i=0}^{m-1} x_i t^i \mapsto \Big( c^{-1} \sum_{i=0}^{m-1} x_i t^i \Big) c,$$
is an inner automorphism of $S_f$.
\item[(ii)] If $p^{\mathrm{gcd}(rm,h)} - p^{\mathrm{gcd}(r,h)} > 0$, then there exists a non-trivial inner automorphism of $S_f$ of the form $G_c$ for some $c \in K^{\times}$ which is a $(p^{rm}-1)$ root of unity, but not a $(p^r-1)$ root of unity.
\end{itemize}
\end{proposition}

\begin{proof}
Let $k = c^{p^r-1}$.
\begin{itemize}
\item[(i)] We have $k^{(p^{rm}-1)/(p^r-1)} = 1$ and so $H_{\mathrm{id},k} \in \mathrm{Aut}_{F}(S_f)$ by Theorem \ref{thm:automorphism_of_S_f_finite field case}(i). Furthermore
\begin{align*}
H_{\mathrm{id},k} &\big( \sum_{i=0}^{m-1} x_i t^i \big) = x_0 + \sum_{i=1}^{m-1} x_i \Big( \prod_{l=0}^{i-1} \sigma^l(k) \Big) t^i = x_0 + \sum_{i=1}^{m-1} x_i \Big( \prod_{l=0}^{i-1} k^{p^{rl}} \Big) t^i \\
&= x_0 + x_1 k t + \sum_{i=2}^{m-1} x_i k^{(p^{ri}-1)/(p^r-1)} t^i = \sum_{i=0}^{m-1} x_i c^{-1} \sigma^i(c) t^i \\ &= \Big( c^{-1} \sum_{i=0}^{m-1} x_i t^i \Big) c,
\end{align*}
so we conclude $H_{\mathrm{id},k} = G_c$ is an inner automorphism.
\item[(ii)] Notice $H_{\mathrm{id},k}$ is the identity if and only if $k = 1$, i.e. if and only if $c$ is a $(p^r-1)$th root of unity. Moreover every $(p^r-1)$th root of unity $c$ is also a $(p^{rm}-1)$th root of unity because $(p^r-1) \vert (p^{rm}-1)$. Therefore if the number of $(p^{rm}-1)$ roots of unity in $K$ is strictly greater than the number of $(p^{r}-1)$ roots of unity, then there exists a non-trivial automorphism of $S_f$ of the form $G_c$ for some $c \in K^{\times}$ which is a $(p^{rm}-1)$ root of unity, but not a $(p^r-1)$ root of unity.

Finally, the number of $(p^{rm}-1)$ roots of unity in $K$ is strictly greater than the number of $(p^{r}-1)$ roots of unity if and only if
\begin{align*}
\mathrm{gcd}&(p^{rm}-1,p^h-1) - \mathrm{gcd}(p^r-1,p^h-1) > 0,
\end{align*}
if and only if
$$p^{\mathrm{gcd}(rm,h)} - 1 - (p^{\mathrm{gcd}(r,h)}-1) = p^{\mathrm{gcd}(rm,h)} - p^{\mathrm{gcd}(r,h)} > 0$$
by Lemma \ref{lem:gcd number theory result}.
\end{itemize}
\end{proof}

When $K$ contains a primitive $(p^{rm}-1)^{\text{th}}$ primitive root of unity, Proposition \ref{prop:inner automorphisms of S_f over finite fields} leads to:

\begin{corollary} \label{cor:inner automorphism cyclic subgroup over finite fields}
Suppose $f(t) = t^m-a \in K[t;\sigma]$ and $(rm) \vert h$, then $K$ contains a primitive $(p^{rm}-1)$ root of unity $c$ and $\mathrm{Aut}_{F}(S_f)$ contains a cyclic subgroup of inner automorphisms of order $(p^{rm}-1)/(p^r-1)$ generated by $G_c$.
\end{corollary}

\begin{proof}
$K$ contains a primitive $(p^{rm}-1)$ root of unity if and only if $(p^{rm}-1) \vert (p^h-1)$ by \cite[Proposition II.2.1]{koblitz1994course}, if and only if
$$p^{rm}-1 = \mathrm{gcd}(p^{rm}-1, p^h-1) = p^{\mathrm{gcd}(rm,h)}-1$$
by Lemma \ref{lem:gcd number theory result}, if and only if $rm = \mathrm{gcd}(rm,h)$ if and only if $(rm) \vert h$. Let $c \in K^{\times}$ be a primitive $(p^{rm}-1)$ root of unity, then $G_c$ is a non-trivial inner automorphism of $S_f$ by Proposition \ref{prop:inner automorphisms of S_f over finite fields}. Additionally, notice $k = c^{p^r-1}$ is a primitive $(p^{rm}-1)/(p^r-1)$ root of unity and $G_c = H_{\mathrm{id},k}$ by the proof of Proposition \ref{prop:inner automorphisms of S_f over finite fields}. Thus $\langle G_c \rangle = \{ G_c, G_{c^2}, \ldots, \mathrm{id} \}$ is a cyclic subgroup of $\mathrm{Aut}_{F}(S_f)$ of order $(p^{rm}-1)/(p^r-1)$.
\end{proof}

\begin{proposition}
Let $f(t) = t^m-a \in K[t;\sigma]$ and $l \in \mathbb{N}$ be such that $l \vert m$ and $l \vert (p^{\mathrm{gcd}(r,h)}-1)$. Then $F$ contains a primitive $l^{\text{th}}$ root of unity $\omega$ and $\mathrm{Aut}_{F}(S_f)$ contains a cyclic subgroup of order $l$ generated by $H_{\mathrm{id},\omega}$.
\end{proposition}

\begin{proof}
$F$ contains a primitive $l^{\text{th}}$ root of unity is equivalent to $l \vert (p^{\mathrm{gcd}(r,h)}-1)$, so the result follows by Theorem \ref{thm:Primitive root then subgroup of order m}.
\end{proof}

\chapter{Automorphisms of Nonassociative Cyclic Algebras} \label{chapter:Automorphisms of Nonassociative Cyclic Algebras}

Throughout this Chapter, let $K/F$ be a cyclic Galois field extension of degree $m$ with $\mathrm{Gal}(K/F) = \langle \sigma \rangle$, and  $$A = (K/F,\sigma,a) = K[t;\sigma]/K[t;\sigma](t^m-a)$$ be a nonassociative cyclic algebra of degree $m$ for some $a \in K \setminus F$. Here $A$ is not associative by Theorem \ref{thm:Properties of S_f petit}(v).

In this Chapter we investigate the automorphisms of nonassociative cyclic algebras using results in Chapter \ref{chapter:Automorphisms of S_f} and those by Steele in \cite[\S 6.3]{AndrewPhD}. 

In particular, we will prove there always exist non-trivial inner automorphisms of a nonassociative cyclic algebra and find conditions for all automorphisms to be inner. These conditions are closely related to whether or not the field $F$ contains a primitive $m^{\text{th}}$ root of unity. We pay special attention to the case where $K/F$ is an extension of finite fields in Section \ref{section:Automorphisms of Nonassociative Cyclic Algebras over Finite Fields}, and we completely determine the automorphism group of nonassociative cyclic algebras of prime degree different from $\mathrm{Char}(F)$.

The field norm $N_{K/F}:K \rightarrow F$ is given by $N_{K/F}(k) = \prod_{l=0}^{m-1} \sigma^l(k)$. Recall Hilbert's Theorem 90, which states $N_{K/F}(b) = 1$ if and only if $b = \sigma(c)c^{-1}$ for some $c \in K^{\times}$, thus $\mathrm{Ker}(N_{K/F}) = \Delta^\sigma (1)$, where $\Delta^\sigma (l) = \{ \sigma(c) l c^{-1} \ \vert \ c \in K^{\times} \}$ is the \textbf{$\sigma$-conjugacy class} of $l \in K^{\times}$ \cite{lam1988vandermonde}.

In this case Theorem \ref{thm:automorphism_of_S_f_division_case}(iii) immediately becomes:

\begin{corollary} \label{cor:t^m-a automorphism field result}
(\cite[Corollary 3.2.9]{AndrewPhD}).
A map $H: A \rightarrow A$ is an $F$-automorphism of $A$ if and only if $H = H_{\sigma^j, k}$ for some $j \in \{ 0, \ldots, m-1 \}$ and $k \in K^{\times}$ such that $\sigma^j(a) = N_{K/F}(k)a$, where $H_{\sigma^j, k}$ is defined as in \eqref{automorphism_of_Sf form of H}.
\end{corollary}

Corollary \ref{cor:t^m-a automorphism field result} leads us to the following Theorem on the inner automorphisms of nonassociative cyclic algebras:

\begin{theorem} \label{thm:t^m-a_automorphism_field}
\begin{itemize}
\item[(i)] The maps
\begin{equation} \label{eqn:form of G_c t^m-a automorphism}
G_c: A \rightarrow A, \ \sum_{i=0}^{m-1} x_i t^i \mapsto \sum_{i=0}^{m-1} x_i c^{-1} \sigma^i(c) t^i,
\end{equation}
are inner automorphisms for all $c \in K^{\times}$.
\item[(ii)] Let $c \in K^{\times}$, then $G_c = H_{\mathrm{id},k}$ where $k = c^{-1} \sigma(c)$. Furthermore every $H_{\mathrm{id}, l} \in \mathrm{Aut}_F(A)$ is inner and can be written in the form $G_c$ for some $c \in K^{\times}$. 
\item[(iii)] Let $c, d \in K^{\times}$, then $G_c = G_d$ if and only if $c^{-1} \sigma(c) = d^{-1}\sigma(d)$.
\item[(iv)] There exists a non-trivial inner automorphism of $A$. Therefore the automorphism group of $A$ is not trivial.
\item[(v)] $N = \{ G_c \ \vert \ c \in K^{\times} \} \cong \mathrm{Ker}(N_{K/F})$ is an abelian normal subgroup of $\mathrm{Aut}_F(A)$. In particular, $H \circ G_c \circ H^{-1}$ is inner for all $c \in K^{\times}$ and all $H \in \mathrm{Aut}_F(A)$.
\item[(vi)] If $\sigma^j(a) a^{-1} \notin N_{K/F}(K^{\times})$ for all $j \in \{ 1, \ldots, m-1 \}$ then
$$\mathrm{Aut}_F(A) = \{ G_c \ \vert \ c \in K^{\times} \} \cong \mathrm{Ker}(N_{K/F}),$$
and all automorphisms of $A$ are inner. In particular, $\mathrm{Aut}_F(A)$ is abelian.
\item[(vii)] Let $c \in K \setminus F$ and suppose there exists $j \in \mathbb{N}$ such that $c^j \in F^{\times}$. Let $j$ be minimal. Then $\langle G_c \rangle \cong \mathbb{Z}/j \mathbb{Z}$ is a cyclic subgroup of $\mathrm{Aut}_F(A)$.
\end{itemize}
\end{theorem}

\begin{proof}
\begin{itemize}
\item[(i)] 
A straightforward calculation shows
$$G_c \Big( \sum_{i=0}^{m-1} x_i t^i \Big) = \Big( c^{-1} \sum_{i=0}^{m-1} x_i t^i  \Big) c,$$
for all $c \in K^{\times}$. Furthermore, $D = \mathrm{Nuc}(A)$ by Corollary \ref{cor:Nucleus of Nonassociative cyclic algebra} and hence $G_c$ are inner automorphisms by Corollary \ref{cor:inner automorphisms of S_f, D[t;sigma,delta]}.
\item[(ii)] If $k = c^{-1} \sigma(c)$, then
\begin{align*}
H_{\mathrm{id}, k} \big( \sum_{i=0}^{m-1} x_i t^i \big) &= x_0 + \sum_{i=1}^{m-1} x_i \Big( \prod_{l=0}^{i-1} \sigma^l(k) \Big) t^i \\ &= x_0 + \sum_{i=1}^{m-1} x_i c^{-1} \sigma^i(c) t^i = G_c \big( \sum_{i=0}^{m-1} x_i t^i \big),
\end{align*}
and thus $G_c = H_{\mathrm{id},k}$. Suppose $H_{\mathrm{id},l} \in \mathrm{Aut}_F(A)$ for some $l \in K^{\times}$, then $N_{K/F}(l) = 1$ by Corollary \ref{cor:t^m-a automorphism field result} and so there exists $c \in K^{\times}$ such that $l = c^{-1}\sigma(c)$ by Hilbert 90. This means $H_{\mathrm{id},l} = G_c$ by the above calculation.
\item[(iii)] Let $k = c^{-1} \sigma(c)$ and $l = d^{-1} \sigma(d)$, so that $G_c = H_{\mathrm{id},k}$ and $G_d = H_{\mathrm{id},l}$ by (ii). Therefore $G_c = G_d$ if and only if $H_{\mathrm{id},k} = H_{\mathrm{id},l}$ if and only if $k = l$.
\item[(iv)] $G_c$ is a non-trivial inner automorphism of $A$ for all $c \in K \setminus F$.
\item[(v)] Note $N = \{ H_{\mathrm{id},k} \ \vert \ k \in \mathrm{Ker}(N_{K/F}) \}$ by (ii) and Corollary \ref{cor:t^m-a automorphism field result} and $N$ is a subgroup of $\mathrm{Aut}_F(A)$ by Theorem \ref{thm:automorphism_of_S_f_division_case}(i). Furthermore, a straightforward calculation shows $G_c \circ G_d = G_{cd} = G_{dc} = G_d \circ G_c$, for all $c, d \in K^{\times}$, i.e. $N$ is abelian. We are left to prove $N$ is a normal subgroup of $\mathrm{Aut}_F(A)$: Let $c \in K^{\times}$ and $H \in \mathrm{Aut}_F(A)$. Then $H = H_{\sigma^j,k}$ for some $j \in \{ 0, \ldots, m-1 \}$, $k \in K^{\times}$ such that $\sigma^j(a) = N_{K/F}(k)a$ by Corollary \ref{cor:t^m-a automorphism field result}. Additionally the inverse of $H_{\sigma^j,k}$ is $H_{\sigma^{-j},\sigma^{-j}(k^{-1})}$, and we have
\begin{align*}
H_{\sigma^j,k} \Big( &G_c \Big( H_{\sigma^j,k}^{-1} \Big( \sum_{i=0}^{m-1} x_i t^i \Big) \Big) \Big) \\
&= H_{\sigma^j,k} \Big( G_c \Big( \sigma^{-j}(x_0) + \sum_{i=1}^{m-1} \sigma^{-j}(x_i) \big( \prod_{l=0}^{i-1} \sigma^{l-j}(k^{-1}) \big) t^i \Big) \Big) \\
&= H_{\sigma^j,k} \Big( \sigma^{-j}(x_0) + \sum_{i=1}^{m-1} \sigma^{-j}(x_i) \big( \prod_{l=0}^{i-1} \sigma^{l-j}(k^{-1}) \big) c^{-1} \sigma^i(c) t^i \Big) \\
&= x_0 + \sum_{i=1}^{m-1} x_i \big( \prod_{l=0}^{i-1} \sigma^{l}(k^{-1}) \big) \sigma^j(c^{-1}) \sigma^{i+j}(c) \big( \prod_{l=0}^{i-1} \sigma^{l}(k) \big) t^i \\
&= x_0 + \sum_{i=1}^{m-1} x_i \sigma^j(c^{-1}) \sigma^{j+i}(c) t^i = G_{\sigma^j(c)} \Big( \sum_{i=0}^{m-1} x_i t^i \Big),
\end{align*}
hence $H_{\sigma^j,k} \circ G_c \circ H_{\sigma^j,k}^{-1} = G_{\sigma^j(c)} \in N$ which yields the assertion.
\item[(vi)] Here $\mathrm{Aut}_{F}(A) = \{ H_{\mathrm{id},k} \ \vert \ k \in \mathrm{Ker}(N_{K/F}) \}$ by Corollary \ref{cor:t^m-a automorphism field result} and hence the result follows by (ii) and (iii).
\item[(vii)] This is just Corollary \ref{cor:G_c subgroup order j cyclic}.
\end{itemize}
\end{proof}

In some instances, $\mathrm{Aut}_F(A)$ is a non-abelian group:

\begin{proposition}
Suppose $H_{\sigma^j,k} \in \mathrm{Aut}_F(A)$ for some $j \in \{ 1, \ldots, m-1 \}$ and $k \in K^{\times}$. If there exists $c \in K^{\times}$ and $i \in \{ 1, \ldots, m-1 \}$ such that $c^{-1} \sigma^i(c) \notin \mathrm{Fix}(\sigma^j)$, then $\mathrm{Aut}_F(A)$ is a non-abelian group.
\end{proposition}

\begin{proof}
We have
$$H_{\sigma^j,k}(G_c(t^i)) = H_{\sigma^j,k} \big( c^{-1} \sigma^i(c)t^i \big) = \sigma^j \big( c^{-1} \sigma^i(c) \big) \Big( \prod_{l=0}^{i-1} \sigma^l(k) \Big) t^i,$$
and
$$G_c(H_{\sigma^j,k}(t^i)) = G_c \Big( \Big( \prod_{l=0}^{i-1} \sigma^l(k) \Big) t^i \Big) = \Big( \prod_{l=0}^{i-1} \sigma^l(k) \Big) c^{-1} \sigma^i(c) t^i,$$
for all $i \in \{ 1, \ldots, m-1 \}$. Thus if there exists $i \in \{ 1, \ldots, m-1 \}$ such that $c^{-1} \sigma^i(c) \notin \mathrm{Fix}(\sigma^j)$, then $G_c \circ H_{\sigma^j,k} \neq H_{\sigma^j,k} \circ G_c$ and $\mathrm{Aut}_F(A)$ is not abelian.
\end{proof}

We now take a closer look at the inner automorphisms in the case where $F$ does not contain certain primitive roots of unity:

\begin{theorem} \label{thm:Aut(S_f) F does not contain primitive root of unity}
Suppose $a \in K^{\times}$ does not lie in any proper subfield of $K$ and $F$ does not contain a non-trivial $m^{\text{th}}$ root of unity. Then 
$$\mathrm{Aut}_F(A) = \{ G_c \ \vert \ c \in K^{\times} \} \cong \mathrm{Ker}(N_{K/F}),$$
and all automorphisms of $A$ are inner.
\end{theorem}

\begin{proof}
We first prove every automorphism of $A$ has the form $H_{\mathrm{id},k}$, then
$$\mathrm{Aut}_F(A) = \{ G_c \ \vert \ c \in K^{\times} \} \cong \mathrm{Ker}(N_{K/F}),$$
by Theorem \ref{thm:t^m-a_automorphism_field} and all automorphisms of $A$ are inner. Suppose, for a contradiction, that there exists $j \in \{ 1, \ldots, m-1 \}$ and $k \in K^{\times}$ such that $H_{\sigma^j,k} \in \mathrm{Aut}_{F}(A)$. This implies $H_{\sigma^j,k}^2 = H_{\sigma^j,k} \circ H_{\sigma^j,k} \in \mathrm{Aut}_{F}(A)$, and
\begin{equation} \label{eqn:Aut(S_f) F does not contain primitive root of unity 1}
\begin{split}
H_{\sigma^j,k}^2 & \Big( \sum_{i=0}^{m-1} x_i t^i \Big) = \sigma^{2j}(x_0)  + \sum_{i=1}^{m-1} \sigma^{2j}(x_i) \Big( \prod_{q=0}^{i-1} \sigma^{j+q}(k) \sigma^q(k) \Big) t^i.
\end{split}
\end{equation}
Now $H_{\sigma^j,k}^2$ must have the form $H_{\sigma^{2j},l}$ for some $l \in K^{\times}$ by Corollary \ref{cor:t^m-a automorphism field result}, and comparing \eqref{automorphism_of_Sf form of H} and \eqref{eqn:Aut(S_f) F does not contain primitive root of unity 1} yields $l = k \sigma^j(k)$. Similarly, $H_{\sigma^j,k}^3 = H_{\sigma^{3j},s} \in \mathrm{Aut}_{F}(A)$ where $s = k \sigma^j(k) \sigma^{2j}(k)$.
Continuing in this manner we conclude the maps $H_{\sigma^j,k}, H_{\sigma^{2j},l}, H_{\sigma^{3j},s}, \ldots$ are all $F$-automorphisms of $A$, therefore
\begin{align} \label{eqn:Aut(S_f) F does not contain primitive root of unity 2}
\begin{split}
\sigma^j(a) &= N_{K/F}(k) a, \\
\sigma^{2j}(a) &= N_{K/F}(k \sigma^j(k)) a = N_{K/F}(k)^2 a, \\
\vdots & \qquad \qquad \vdots \\
a = \sigma^{n j}(a) &= N_{K/F}(k)^{n} a,
\end{split}
\end{align}
by Corollary \ref{cor:t^m-a automorphism field result}, where $n = m/\mathrm{gcd}(j,m)$ is the order of $\sigma^j$.

Note that $\sigma^{ij}(a) \neq a$ for all $i \in \{ 1, \ldots, n-1 \}$ since $a$ is not contained in any proper subfield of $K$. Therefore $N_{K/F}(k)^{n} = 1$ and $N_{K/F}(k)^i \neq 1$ for all $i \in \{1, \ldots, n - 1 \}$ by \eqref{eqn:Aut(S_f) F does not contain primitive root of unity 2}, i.e. $N_{K/F}(k)$ is a primitive $n^{\text{th}}$ root of unity, thus also an $m^{\text{th}}$ root of unity, a contradiction.
\end{proof}

\begin{example} \label{ex:cubic cyclic extension inner automorphisms}
Let $F = \mathbb{Q}$ and $K = \mathbb{Q}(\theta)$ where $\theta$ is a root of $T(y) = y^3 + y^2 - 2y - 1 \in F[y]$. Then $K/F$ is a cubic cyclic Galois field extension, its Galois group is generated by $\sigma$ where $\sigma(\theta) = - \theta^2 - \theta +1$ \cite[p.~199]{hanke2005twisted}. Suppose $A = (K/F,\sigma,a)$ for some $a \in K \setminus F$. As $F$ does not contain a non-trivial $3^{\text{rd}}$ root of unity, Theorem \ref{thm:Aut(S_f) F does not contain primitive root of unity} implies $\mathrm{Aut}_F(A) = \{ G_c \ \vert \ c \in K^{\times} \} \cong \mathrm{Ker}(N_{K/F}),$ and all automorphisms of $A$ are inner.
\end{example}

We now investigate the automorphisms of $A$ in the case when $F$ contains a primitive $m^{\text{th}}$ root of unity. It is well-known that if $F$ contains a primitive $m^{\text{th}}$ root of unity and $K/F$ is cyclic of degree $m$ (where $m$ and $\mathrm{Char}(F)$ are coprime), then $K = F(d)$ where $d$ is a root of the irreducible polynomial $x^m - e \in F[x]$ for some $e \in F^{\times}$ \cite[\S VI.6]{lang2002algebra}. 

\begin{lemma} \label{lem:eigenvalues and eigenvectors of cyclic extension automorphisms}
Suppose $F$ contains a primitive $m^{\text{th}}$ root of unity and either $F$ has characteristic $0$ or $\mathrm{gcd}(\mathrm{char}(F),m)=1$. Write $K = F(d)$ where $d$ is a root of the irreducible polynomial $x^m - e$ for some $e \in F^{\times}$. Then $\sigma^j(\lambda d^i) = \omega^{ij} \lambda d^i$, for all $\lambda \in F$ and $i, j \in \{ 0, \ldots, m-1 \}$ with $\omega \in F^{\times}$ a primitive $m^{\text{th}}$ root of unity. Furthermore, if $m$ is prime then $\lambda d^i$ are the only possible eigenvectors of $\sigma^j$.
\end{lemma}

\begin{proof}
When $m$ is prime this is \cite[Lemma 6.2.7]{AndrewPhD}. We are left to prove the first assertion when $m$ is not necessarily prime, this is similar to the first part of the proof of \cite[Lemma 6.2.7]{AndrewPhD}: We have $\sigma(d^m) = \sigma(e) = e = d^m$, therefore the action of $\sigma$ on $d$ is given by $\sigma(d) = \omega d$ where $\omega$ is a primitive $m^{\text{th}}$ root of unity. Thus $\sigma^j(d) = \omega^jd$ and $\sigma^j(d^i) = \omega^{ij}d^i$ for all $i,j \in \{ 0, \ldots, m-1 \}$.
\end{proof}

When $m$ is prime, Corollary \ref{cor:t^m-a automorphism field result} and Lemma \ref{lem:eigenvalues and eigenvectors of cyclic extension automorphisms} yield:

\begin{proposition} \label{prop:contains rot of unity Ker}
Suppose $m$ is prime, $F$ has characteristic not $m$ and contains a primitive $m^{\text{th}}$ root of unity. Write $K = F(d)$ where $d$ is a root of the irreducible polynomial $x^m-e \in F[x]$ for some $e \in F^{\times}$.
\begin{itemize}
\item[(i)] If $a \neq \lambda d^i$ for all $\lambda \in F^{\times}$, $i \in \{ 1, \ldots, m-1 \}$ then $\mathrm{Aut}_F(A) \cong \mathrm{Ker}(N_{K/F})$ and all automorphisms of $A$ are inner.
\item[(ii)] Suppose $a = \lambda d^i$ for some $\lambda \in F^{\times}$, $i \in \{ 1, \ldots, m-1 \}$. If there exists $j \in \{ 1, \ldots, m-1 \}$ and $k \in K^{\times}$ such that $N_{K/F}(k) = \omega^{ij}$ where $\omega \in F^{\times}$ is the primitive $m^{\text{th}}$ root of unity satisfying $\sigma(d) = \omega d$, then $\langle H_{\sigma^j,k} \rangle$ is a cyclic subgroup of $\mathrm{Aut}_F(A)$ of order $m^2$. Otherwise $\mathrm{Aut}_F(A) \cong \mathrm{Ker}(N_{K/F})$ and all automorphisms of $A$ are inner.
\end{itemize}
\end{proposition}

\begin{proof} 
\begin{itemize}
\item[(i)] If $a \neq \lambda d^i$ for all $\lambda \in F^{\times}$, $i \in \{ 1, \ldots, m-1 \}$ then $\sigma^j(a) \neq la$ for all $l \in F^{\times}$, $j \in \{ 1, \ldots, m-1 \}$ by Lemma \ref{lem:eigenvalues and eigenvectors of cyclic extension automorphisms}. In particular, this means $\sigma^j(a) \neq N_{K/F}(k)a$ for all $k \in K^{\times}$ and so $H_{\sigma^j,k}$ is not an automorphism of $A$ for all $j \in \{ 1, \ldots, m-1 \}$, $k \in K^{\times}$ by Corollary \ref{cor:t^m-a automorphism field result}. Therefore $\mathrm{Aut}_F(A) = \{ H_{\mathrm{id},k} \ \vert \ N_{K/F}(k)=1 \}$ again by Corollary \ref{cor:t^m-a automorphism field result} and $\mathrm{Aut}_F(A) = \{ G_c \ \vert \ c \in K^{\times} \} \cong \mathrm{ker}(N_{K/F})$ by Theorem \ref{thm:t^m-a_automorphism_field}, hence all automorphisms of $A$ are inner.
\item[(ii)] Suppose there exists $j \in \{ 1, \ldots, m-1 \}$ and $k \in K$ such that $N_{K/F}(k) = \omega^{ij}$ where $\omega \in F^{\times}$ is a primitive $m^{\text{th}}$ root of unity. Then
$$\sigma^j(a) = \omega^{ij}a = N_{K/F}(k)a,$$
by Lemma \ref{lem:eigenvalues and eigenvectors of cyclic extension automorphisms} which implies $H_{\sigma^j,k} \in \mathrm{Aut}_F(A)$ by Corollary \ref{cor:t^m-a automorphism field result}. Since $m$ is prime, $\sigma^j$ has order $m$, and $H_{\sigma^j,k} \circ \ldots \circ H_{\sigma^j,k}$ ($m$-times) becomes $H_{\mathrm{id},b}$ where $b = \omega^{ij} = N_{K/F}(k)$. As $b$ is a primitive $m^{\text{th}}$ root of unity, $H_{\mathrm{id},b}$ has order $m$ so the subgroup generated by $H_{\sigma^j,k}$ has order $m^2$.

On the other hand, if $N_{K/F}(k) \neq \omega^{ij}$ for all $k \in K^{\times}$, $j \in \{ 1, \ldots, m-1 \}$, then $\sigma^j(a) \neq N_{K/F}(k) a$ for all $k \in K^{\times}$, $j \in \{ 1, \ldots, m-1 \}$ by Lemma \ref{lem:eigenvalues and eigenvectors of cyclic extension automorphisms}, and hence $H_{\sigma^j,k} \notin \mathrm{Aut}_{F}(A)$ for all $j \in \{ 1,\ldots, m-1 \}$, $k \in K^{\times}$ by Corollary \ref{cor:t^m-a automorphism field result}. Therefore $\mathrm{Aut}_F(A) = \{ G_c \ \vert \ c \in K^{\times} \} \cong \mathrm{ker}(N_{K/F})$ by Theorem \ref{thm:t^m-a_automorphism_field}.
\end{itemize}
\end{proof}

\section{Automorphisms of Nonassociative Quaternion Algebras} \label{section:Automorphisms of Nonassociative Quaternion Algebras}

We now study the automorphisms of nonassociative cyclic algebras of degree $2$, i.e. nonassociative quaternion algebras. Suppose $\mathrm{Char}(F) \neq 2$, $K/F$ is a quadratic separable field extension with non-trivial automorphism $\sigma$, and write $K = F(\sqrt{b})$ for some $b \in K^{\times}$. Let $A = (K/F,\sigma,a)$, $a \in K \setminus F$, be a nonassociative quaternion algebra.

\begin{theorem} \label{thm:Automorphisms of nonassociative quaternion algebras}
\begin{itemize}
\item[(i)] A map $H$ is an automorphism of $A$ if and only if $H = H_{\sigma^j,k}$ for some $j \in \{ 0,1 \}$ and $k \in K^{\times}$ such that $\sigma^j(a) = N_{K/F}(k)a$.
\item[(ii)] The map $H_{\mathrm{id},c^{-1}\sigma(c)} = G_c$ defined as in \eqref{eqn:form of G_c t^m-a automorphism} is an inner automorphism of $A$ for all $c \in K^{\times}$. Moreover every automorphism of $A$ of the form $H_{\mathrm{id},k}$ can also be written in the form $G_c$ for some $c \in K^{\times}$.
\item[(iii)] $G_c$ is not trivial if and only if $c \in K \setminus F$. In particular there exists a non-trivial inner automorphism of $A$.
\item[(iv)] $\{ G_c \ \vert \ c \in K^{\times} \}$ are the only inner automorphisms of $A$.
\end{itemize}
\end{theorem}

\begin{proof}
(i), (ii) and (iii) follow immediately from Corollary \ref{cor:t^m-a automorphism field result} and Theorem \ref{thm:t^m-a_automorphism_field}. The proof of (iv) is similar to \cite[Lemmas 2 and 3]{wene2006auto} with $\alpha = \mathrm{id}$:

Suppose, for a contradiction, that $\{ G_c \ \vert \ c \in K^{\times} \}$ are not the only inner automorphisms of $A$. This means there exists an element $0 \neq r + st \in A$ with left inverse $u+vt$, such that $s \neq 0$ and
$$H: A \rightarrow A, \ x_0 + x_1t \mapsto [(u + vt) \circ (x_0 + x_1t)] \circ (r + s t),$$
is an automorphism. We have
\begin{align} \label{eqn:t^m-a_automorphism_field 0}
\begin{split}
1 &= (u + vt) \circ (r + st) = u r + v \sigma(s)a + \big( u s + v \sigma(r) \big) t,
\end{split}
\end{align}
and comparing the coefficients of $t$ in \eqref{eqn:t^m-a_automorphism_field 0} yields
\begin{equation} \label{eqn:t^m-a_automorphism_field 1}
u s + v \sigma(r) = 0.
\end{equation}
Any automorphism must preserve the left nucleus, so $H(K) = K$ and $H \vert_K = \sigma^j$ for some $j \in \{ 0, 1 \}$. This implies
\begin{align*}
H(k) &= [(u + vt) \circ k] \circ (r + st) = (uk + v \sigma(k)t) \circ (r + s t) \\
&= u k r + v \sigma(k) \sigma(s) a + (u k s + v \sigma(k)\sigma(r))t = \sigma^j(k) 
\end{align*}
for all $k \in K$, in particular
\begin{equation} \label{eqn:t^m-a_automorphism_field 2}
k u s + \sigma(k) v \sigma(r) = 0.
\end{equation}
Therefore $k u s = \sigma(k) u s$ for all $k \in K$ by \eqref{eqn:t^m-a_automorphism_field 1}, \eqref{eqn:t^m-a_automorphism_field 2}, hence $u = 0$ since $s \neq 0$ and $\sigma$ is not trivial. Furthermore $v \sigma(r) = 0$ by \eqref{eqn:t^m-a_automorphism_field 1} which means $r = 0$ because $\sigma$ is injective and $v \neq 0$. Now $A$ is a division algebra by Corollary \ref{cor:nonassociative cyclic algebra is division} and Theorem \ref{thm:S_f_division_iff_irreducible} (see also \cite[p.~369]{waterhouse}), moreover
$$\frac{1}{\sigma(s)a}t \circ st = 1,$$
and so the left inverse of $st$ is unique and equal to $(1/(\sigma(s)a)) t$. We conclude $H$ has the form
\begin{align*}
H(x_0 + x_1t) &= \big[ \frac{1}{\sigma(s)a}t \circ (x_0 + x_1t) \big] \circ s t = \Big( \frac{\sigma(x_0)}{\sigma(s)a}t + \frac{\sigma(x_1)}{\sigma(s)} \Big) \circ st \\
&= \sigma(x_0) + \frac{\sigma(x_1)s}{\sigma(s)}t,
\end{align*}
that is $H = H_{\sigma,b}$ where $b = \sigma(s)^{-1} s$, and
$$\sigma(a) = \sigma(s)^{-1} s \sigma^2(s)^{-1} \sigma(s) a = s \sigma^2(s)^{-1} a = s s^{-1} a = a$$ by Corollary \ref{cor:t^m-a automorphism field result}, a contradiction because $a \in K \setminus F$.
\end{proof}

Theorem \ref{thm:Automorphisms of nonassociative quaternion algebras}(i) is also proven by Waterhouse in \cite[p.~370-371]{waterhouse}. Setting $m = 2$ in Theorem \ref{thm:t^m-a_automorphism_field}(vi) and Proposition \ref{prop:contains rot of unity Ker}(i) yields:

\begin{corollary} \label{cor:inner automorphisms of nonassociative quaternion algebras}
\begin{itemize}
\item[(i)] If $a = \lambda_0 + \lambda_1 \sqrt{b}$ for some $\lambda_0, \lambda_1 \in F^{\times}$, then $\mathrm{Aut}_F(A) = \{ G_c \ \vert \ c \in K^{\times} \} \cong \mathrm{Ker}(N_{K/F})$ and all automorphisms of $A$ are inner.
\item[(ii)] If $a = \lambda \sqrt{b}$ for some $\lambda \in F^{\times}$ and $N_{K/F}(k) \neq -1$ for all $k \in K^{\times}$, then $\mathrm{Aut}_F(A) = \{ G_c \ \vert \ c \in K^{\times} \} \cong \mathrm{Ker}(N_{K/F})$ and all automorphisms of $A$ are inner.
\end{itemize}
\end{corollary}

\begin{proof}
\begin{itemize}
\item[(i)] $-1 \in F^{\times}$ is a primitive $2^{\text{nd}}$ root of unity so the result follows by Proposition \ref{prop:contains rot of unity Ker}(i).
\item[(ii)] We have $\sigma(a)a^{-1} = -a a^{-1} = -1$ so the assertion follows by Theorem \ref{thm:t^m-a_automorphism_field}(vi).
\end{itemize}
\end{proof}

\begin{example} \label{ex:automorphisms_of_nonassociative_quaternion_algebras}
Consider the nonassociative quaternion algebra $A = (\mathbb{C}/\mathbb{R},\sigma,a)$ where $a \in \mathbb{C} \setminus \mathbb{R}$ and $\sigma$ denotes complex conjugation. Given $k = k_0 + k_1 i \in \mathbb{C}$, $k_0, k_1 \in \mathbb{R}$, we have
$$N_{\mathbb{C}/\mathbb{R}}(k) = k \sigma(k) = (k_0 + k_1 i)(k_0 - k_1 i) = k_0^2 + k_1^2 \neq -1,$$
and so $\mathrm{Aut}_{\mathbb{R}}(A) = \{ G_c \ \vert \ c \in \mathbb{C}^{\times} \} \cong \mathrm{Ker}(N_{\mathbb{C}/\mathbb{R}})$ by Corollary \ref{cor:inner automorphisms of nonassociative quaternion algebras}. Now since
$\mathrm{Ker}(N_{\mathbb{C}/\mathbb{R}}) = \{ k_0 + k_1 i \in \mathbb{C} \ \vert \ k_0^2 + k_1^2 = 1 \},$
we conclude $\mathrm{Aut}_{\mathbb{R}}(A)$ is isomorphic to the unit circle in $\mathbb{C}$.
\end{example}

Recall the \textbf{dicyclic group} of order $4l$ has the presentation
\begin{equation} \label{eqn:dicyclic group presentation}
\mathrm{Dic}_l = \langle x, y \ \vert \ y^{2l}=1, \ x^2 = y^l, \ x^{-1}yx = y^{-1} \rangle.
\end{equation}
The \textbf{semidirect product} $\mathbb{Z} / s \mathbb{Z} \rtimes \mathbb{Z} / n \mathbb{Z}$ between cyclic groups $\mathbb{Z} / s \mathbb{Z}$ and $\mathbb{Z} / n \mathbb{Z}$ corresponds to a choice of integer $l$ such that $l^n \equiv 1 \ \mathrm{ mod} \ s$. It can be described by the presentation
\begin{equation} \label{eqn:semidirect product of cyclic algebras}
\mathbb{Z} / s \mathbb{Z} \rtimes_l \mathbb{Z} / n \mathbb{Z} = \langle x,y \ \vert \ x^s = 1, \ y^n= 1, \ yxy^{-1} = x^l \rangle.
\end{equation}

\begin{theorem} \label{thm:semidirect and dicyclic m=2}
Let $a = \lambda \sqrt{b}$ for some $\lambda \in F^{\times}$ and suppose there exists $k \in K^{\times}$ such that $N_{K/F}(k) = -1$. For every $c \in K \setminus F$ for which there exists a positive integer $j$ such that $c^j \in F^{\times}$, pick the smallest such $j$.
\begin{itemize}
\item[(i)] If $j$ is even then $\mathrm{Aut}_F(A)$ contains a subgroup isomorphic to the dicyclic group of order $2j$.
\item[(ii)] If $j$ is odd then $\mathrm{Aut}_F(A)$ contains a subgroup isomorphic to the semidirect product $\mathbb{Z} / j \mathbb{Z} \rtimes_{j-1} \mathbb{Z} / 4 \mathbb{Z}.$
\end{itemize}
\end{theorem}

\begin{proof}
Since $\sigma(a) = -a = N_{K/F}(k)a$, $H_{\sigma,k} \in \mathrm{Aut}_F(A)$ by Corollary \ref{cor:t^m-a automorphism field result}. $\langle G_c \rangle$ is a cyclic subgroup of $\mathrm{Aut}_F(A)$ of order $j$ by Theorem \ref{thm:t^m-a_automorphism_field}, furthermore, a straightforward calculation shows that
\begin{equation} \label{eqn:semidirect and dicyclic m=2 I}
\langle H_{\sigma,k} \rangle = \{ H_{\sigma,k}, H_{\mathrm{id},-1}, H_{\sigma,-k}, H_{\mathrm{id}, 1} \}.
\end{equation}
\begin{itemize}
\item[(i)] Suppose $j$ is even and write $j = 2l$. We prove first that $G_{c^l} = H_{\mathrm{id}, -1}$. Write $c^l = \lambda + \mu \sqrt{b}$ for some $\lambda, \mu \in F$. Then
$$c^j = c^{2l} = \lambda^2 + \mu^2 b + 2 \lambda \mu \sqrt{b} \in F,$$
which implies $2 \lambda \mu = 0$, i.e. $\lambda = 0$ or $\mu = 0$. By the minimality of $j$, $c^l \notin F$ thus $\lambda = 0$ and hence $c^l = \mu \sqrt{b}$. We have
\begin{align*}
G_{c^l}(x_0 + x_1t) &= x_0 + x_1 (\mu \sqrt{b})^{-1} \sigma(\mu \sqrt{b})t 
= x_0 + x_1 \mu^{-1} b^{-1} \sqrt{b}(- \mu \sqrt{b})t \\
&= x_0 - x_1t = H_{\mathrm{id}, -1}(x_0 + x_1t),
\end{align*}
that is $G_{c^l} = H_{\mathrm{id}, -1}$.

Next we prove $(H_{\sigma,k})^{-1} \circ G_c \circ H_{\sigma,k} = G_c^{-1}$: Simple calculations show $(H_{\sigma,k})^{-1} = H_{\sigma,-k}$ and $G_c^{-1} = G_{\sigma(c)}$. Furthermore
\begin{align*}
H_{\sigma,-k} \big( G_c \big( H_{\sigma,k} \big( x_0 + x_1t \big) \big) \big) &= H_{\sigma,-k} \big( G_c \big( \sigma(x_0) + \sigma(x_1)kt \big) \big) \\
&= H_{\sigma,-k} \big( \sigma(x_0) + \sigma(x_1)k c^{-1}\sigma(c) t \big) \\
&= x_0 - x_1 \sigma(k) \sigma(c^{-1})ckt \\
&= x_0 + x_1 \sigma(c^{-1})c t = G_{\sigma(c)}\big( x_0 + x_1t \big),
\end{align*}
that is $(H_{\sigma,k})^{-1} \circ G_c \circ H_{\sigma,k} = G_c^{-1}$.

We conclude $H_{\sigma,k}^2 = H_{\mathrm{id}, -1} = G_{c^l} = G_c^l$, $G_c^{2l} = \mathrm{id}$ and $(H_{\sigma,k})^{-1} \circ G_c \circ H_{\sigma,k} = G_c^{-1}$, hence $\langle H_{\sigma,k}, G_c \rangle$ has the presentation \eqref{eqn:dicyclic group presentation} as required.
\item[(ii)] Suppose $j$ is odd, then $\langle G_c \rangle$ is cyclic of order $j$ so does not contain $H_{\mathrm{id}, -1}$, because $H_{\mathrm{id}, -1}$ has order $2$. This implies $\langle H_{\sigma,k} \rangle \cap \langle G_c \rangle = \{ \mathrm{id} \}$ by \eqref{eqn:semidirect and dicyclic m=2 I}. Furthermore,
$(H_{\sigma,k})^{-1} \circ G_c \circ H_{\sigma,k} = G_c^{-1} = G_c^{j-1} = G_{c^{j-1}},$
similarly to the argument in (i). Notice that
$$(j-1)^4 = j^4 - 4j^3 + 6j^2 - 4j + 1 \equiv 1 \text{ mod }(j),$$
thus $\mathrm{Aut}_F(A)$ contains the subgroup 
$$\langle G_c \rangle \rtimes_{j-1} \langle H_{\sigma,k} \rangle \cong \mathbb{Z} / j \mathbb{Z} \rtimes_{j-1} \mathbb{Z} / 4 \mathbb{Z}$$
as required.
\end{itemize}
\end{proof}

In particular if we choose $c = \sqrt{b}$ in Theorem \ref{thm:semidirect and dicyclic m=2} then clearly $j = 2$ which means $\mathrm{Aut}_F(A)$ contains the dicyclic group of order $4$, which is the cyclic group of order $4$:

\begin{corollary}
If $a = \lambda \sqrt{b}$ for some $\lambda \in F^{\times}$ and there exists $k \in K^{\times}$ such that $k \sigma(k) = -1$ then $\mathrm{Aut}_F(A)$ contains a subgroup isomorphic to $\mathbb{Z}/4\mathbb{Z}$.
\end{corollary}


\begin{examples}
\begin{itemize}
\item[(i)] Suppose $F = \mathbb{Q}(i)$, $K = F(\sqrt{-3})$ and $\sigma: K \rightarrow K$ is the $F$-automorphism sending $\sqrt{-3}$ to $-\sqrt{-3}$. Let $A = (K/F,\sigma, \lambda \sqrt{-3})$ for some $\lambda \in F^{\times}$, $c = 1 + \sqrt{-3} \in K$, and notice $i \sigma(i) = i^2 = -1$. We have $c^2 = -2+2\sqrt{-3}$ and $c^3 = -8$ so that $c, c^2 \in K \setminus F$ and $c^3 \in F$. This means $\mathrm{Aut}_F(A)$ contains a subgroup isomorphic to the semidirect product $\mathbb{Z} / 3\mathbb{Z} \rtimes_2 \mathbb{Z} / 4 \mathbb{Z}$ by Theorem \ref{thm:semidirect and dicyclic m=2}.
\item[(ii)] Suppose $F = \mathbb{Q}(i)$, $K = F(\sqrt{-1/12})$ and $\sigma: K \rightarrow K$ the $F$-automorphism sending $\sqrt{-1/12}$ to $-\sqrt{-1/12}$. Let $$A = (K/F,\sigma,\lambda \sqrt{-1/12})$$
for some $\lambda \in F^{\times}$, $c = 1 + 2 \sqrt{-1/12} \in K$ and notice $i \sigma(i) = i^2 = -1$. Then
$$c^2 = \frac{2}{3} + \frac{2i}{\sqrt{3}}, \ \ c^3 = \frac{8i}{3\sqrt{3}}, \ \ c^4 = \frac{-8}{9} + \frac{8i}{3 \sqrt{3}},$$
$$c^5 = \frac{-16}{9} + \frac{16i}{9\sqrt{3}} \ \text{ and } \ c^6 = \frac{-64}{27}.$$
Hence $c, c^2, c^3, c^4, c^5 \in K \setminus F$ and $c^6 \in F$. Therefore $\mathrm{Aut}_F(A)$ contains the dicyclic group of order $12$ by Theorem \ref{thm:semidirect and dicyclic m=2}.
\item[(iii)] Suppose $F = \mathbb{Q}(i)(\sqrt{5})$, $K = F \big( \sqrt{2\sqrt{5}-5} \big)$ and $\sigma: K \rightarrow K$ is the $F$-automorphism sending $\sqrt{2\sqrt{5}-5}$ to $-\sqrt{2\sqrt{5}-5}$. Let $$A = (K/F,\sigma,\lambda \sqrt{2\sqrt{5}-5})$$ for some $\lambda \in F^{\times}$, $c = 1 + \sqrt{2\sqrt{5}-5} \in K$ and notice $i \sigma(i) = i^2 = -1$. Then
\begin{align*}
c^2 &= -4 + 2 \sqrt{5} + 2\sqrt{2\sqrt{5}-5} \in K \setminus F, \\
c^3 &= -14 + 6 \sqrt{5} - 2 \sqrt{2\sqrt{5}-5} + 2 \sqrt{5} \sqrt{2\sqrt{5}-5} \in K \setminus F, \\
c^4 &= 16 - 8 \sqrt{5} - 16 \sqrt{2\sqrt{5}-5} + 8 \sqrt{5} \sqrt{2\sqrt{5}-5} \in K \setminus F, \\
c^5 &= 176 - 80 \sqrt{5} \in F,
\end{align*}
so we conclude $\mathrm{Aut}_F(A)$ contains a subgroup isomorphic to the semidirect product $\mathbb{Z} / 5\mathbb{Z} \rtimes_4 \mathbb{Z} / 4 \mathbb{Z}$ by Theorem \ref{thm:semidirect and dicyclic m=2}.
\end{itemize}
\end{examples}


\section{Automorphisms of Nonassociative Cyclic Algebras over Finite Fields} \label{section:Automorphisms of Nonassociative Cyclic Algebras over Finite Fields}

In \cite[p.~88-92]{AndrewPhD}, the automorphisms of nonassociative cyclic algebras over finite fields were briefly investigated. In this Section we 
continue this investigation. In particular, we completely determine the automorphism group of nonassociative cyclic algebras of prime degree $p$ where $\mathrm{Char}(F) \neq p$, this was done only for $p = 2$ in \cite{AndrewPhD}.

Let now $F = \mathbb{F}_q$ be a finite field of order $q$ and $K = \mathbb{F}_{q^m}$ for some $m \geq 2$. Then $K/F$ is a cyclic Galois extension of degree $m$, say $\mathrm{Gal}(K/F) = \langle \sigma \rangle$. Suppose $A = (K/F,\sigma,a)$ is a nonassociative cyclic algebra for some $a \in K \setminus F$. Let $\alpha$ be a primitive element of $K$, i.e. $K^{\times} = \langle \alpha \rangle$. We recall the well-known fact that since $K$ and $F$ are finite fields, the field norm $N_{K / F} : K^{\times} \rightarrow F^{\times}$ is surjective. Therefore by Corollary \ref{cor:t^m-a automorphism field result}, the problem of finding for which $j \in \{ 0, \ldots, m-1 \}$ there exists $H_{\sigma^j,k} \in \mathrm{Aut}_F(A)$ for some $k \in K^{\times}$, reduces to finding which $j \in \{0 , \ldots, m-1 \}$ exist such that $\sigma^j(a) a^{-1} \in F^{\times}$.

Let $s = (q^m-1)/(q-1)$ and notice $s = \sum_{i=0}^{m-1}q^i$.

\begin{lemma} \label{lem:number_of_solutions_to_norm_equation}
(\cite[Proposition 6.3.1]{AndrewPhD}).
For every $b \in F^{\times}$ there are precisely $s$ elements $k \in K^{\times}$ such that $N_{K/F}(k) = b$. Furthermore, $\mathrm{Ker}(N_{K/F}) \cong \mathbb{Z} / s \mathbb{Z}$ is a cyclic subgroup of the multiplicative group $K^{\times}$. 
\end{lemma}

If $\mathrm{gcd}(m,q-1) = 1$ then $F$ does not contain a non-trivial $m^{\text{th}}$ root of unity by \cite[p.~42]{koblitz1994course}. Therefore Theorems \ref{thm:t^m-a_automorphism_field} and \ref{thm:Aut(S_f) F does not contain primitive root of unity} become:

\begin{theorem} \label{thm:t^m-a_automorphism_field finite}
\begin{itemize}
\item[(i)] $\langle G_{\alpha} \rangle$ is a cyclic normal subgroup of $\mathrm{Aut}_F(A)$ of order $s$. Moreover every automorphism of $A$ of the form $H_{\mathrm{id}, k}$ for some $k \in K^{\times}$, is contained in $\langle G_{\alpha} \rangle$.
\item[(ii)] All automorphisms contained in $\langle G_{\alpha} \rangle$ are inner, and if $m = 2$ these are all the inner automorphisms of $A$.
\item[(iii)] If $\mathrm{gcd}(m,q-1) = 1$ and $a$ is not contained in any proper subfield of $K$, then $\mathrm{Aut}_F(A) = \langle G_{\alpha} \rangle$ is a cyclic group of order $s$ and all automorphisms of $A$ are inner.
\end{itemize}
\end{theorem}

\begin{proof}
\begin{itemize}
\item[(i)] If $K^{\times} = \langle \alpha \rangle$ then $F^{\times} = \langle\alpha^s \rangle$ by \cite[p.~303]{nicholson2012introduction}. This means $\alpha^s \in F^{\times}$ but $\alpha^j \notin F$ for all $j \in \{ 1, \ldots, s-1 \}$ which implies $\langle G_{\alpha} \rangle$ is a cyclic subgroup of $\mathrm{Aut}_F(A)$ of order $s$ by Theorem \ref{thm:t^m-a_automorphism_field}. Since $\mathrm{Ker}(N_{K/F}) \cong \mathbb{Z} / s \mathbb{Z}$ by Lemma \ref{lem:number_of_solutions_to_norm_equation}, every automorphism $H_{\mathrm{id}, k}$ is contained in $\langle G_{\alpha} \rangle$ by Theorem \ref{thm:t^m-a_automorphism_field}.
\item[(ii)] follows by (i) and Theorem \ref{thm:t^m-a_automorphism_field} and (iii) follows by (i) and Theorem \ref{thm:Aut(S_f) F does not contain primitive root of unity}.
\end{itemize}
\end{proof}

\begin{lemma} \label{lem:n divides s} 
Suppose $m \vert (q-1)$, then $m \vert s$.
\end{lemma}

\begin{proof}
We prove first
\begin{equation} \label{eqn:Lemma:n divides s I} 
(q-1) \vert \Big( \big( \sum_{i=0}^{m-1} q^i \big) - m \Big)
\end{equation}
for all $m \geq 2$ by induction: Clearly \eqref{eqn:Lemma:n divides s I} holds for $m = 2$. Suppose \eqref{eqn:Lemma:n divides s I} holds for some $m \geq 2$, then
\begin{equation} \label{eqn:Lemma:n divides s II}
\big( \sum_{i=0}^{m} q^i \big) - (m+1) =  \big( \sum_{i=0}^{m-1} q^i \big) - m + q^m - 1 =  \big( \sum_{i=0}^{m-1} q^i \big) - m +  \big( \sum_{i=0}^{m-1} q^i \big) (q-1).
\end{equation}
Now, $(q-1) \vert \Big( \big( \sum_{i=0}^{m-1} q^i \big) - m \Big)$ by induction hypothesis, thus $(q-1)$ divides \eqref{eqn:Lemma:n divides s II} which implies \eqref{eqn:Lemma:n divides s I} holds by induction. In particular since $m \vert (q-1)$, \eqref{eqn:Lemma:n divides s I} yields
$$m \vert \Big( \big( \sum_{i=0}^{m-1} q^i \big) - m \Big),$$
therefore $m$ divides $\big( \sum_{i=0}^{m-1} q^i \big) - m + m = s$ as required.
\end{proof}

\begin{lemma} \label{lem:n^2 does not divide s} 
Suppose $m \vert (q-1)$.
\begin{itemize}
\item[(i)] If $m$ is odd then $m^2 \nmid (ls)$ for all $l \in \{ 1, \ldots, m-1 \}$.
\item[(i)] If $(q-1)/m$ is even then $m^2 \nmid (ls)$ for all $l \in \{ 1, \ldots, m-1 \}$.
\end{itemize}
\end{lemma}

\begin{proof}
Write $q = 1 + rm$ for some $r \in \mathbb{N}$, then
\begin{align*}
q^j &= (1+rm)^j = \sum_{i=0}^j \binom{j}{i} (rm)^i \equiv \sum_{i=0}^1 \binom{j}{i} (rm)^i \ \mathrm{ mod} \ (m^2) \\ & \equiv (1 + jrm) \ \mathrm{ mod} \ (m^2)
\end{align*}
for all $j \geq 1$. Therefore
\begin{align*}
ls &= l \sum_{j=0}^{m-1} q^j \equiv l \big( 1 + \sum_{j=1}^{m-1}(1+jrm) \big) \ \mathrm{ mod} \ (m^2) \\ & \equiv \Big( lm + lrm \frac{(m-1)m}{2} \Big) \ \mathrm{ mod} \ (m^2),
\end{align*}
for all $l \in \{ 1, \ldots, m-1 \}$.
If $m$ is odd or $r = (q-1)/m$ is even, then
$$\frac{lr(m-1)}{2} \in \mathbb{Z}$$
and so
$$lrm \frac{(m-1)m}{2} \ \mathrm{ mod} \ (m^2) \equiv 0$$
for all $l \in \{ 1, \ldots, m-1 \}$. This means
$$ls \equiv lm \ \mathrm{ mod} \ (m^2) \not\equiv 0 \ \mathrm{ mod} \ (m^2),$$
that is, $m^2 \nmid (ls)$ for all $l \in \{ 1, \ldots, m-1 \}$.
\end{proof}

\begin{theorem} \label{thm:Automorphisms of nonassociative cyclic algebras over finite fields}
Suppose $\mathrm{gcd}(\mathrm{Char}(F),m) = 1$ and $m \vert (q-1)$. Then we can write $K = F(d)$ where $d$ is a root of the irreducible polynomial $x^m - e \in F[x]$ as in Lemma \ref{lem:eigenvalues and eigenvectors of cyclic extension automorphisms}. Let $A = (K/F,\sigma,\lambda d^i)$ for some $i \in \{ 1 , \ldots, m-1 \}$, $\lambda \in F^{\times}$. If $m$ is odd or $(q-1)/m$ is even, then $\mathrm{Aut}_F(A)$ is a group of order $ms$ and contains a subgroup isomorphic to the semidirect product of cyclic groups
\begin{equation} \label{eqn:Automorphisms of nonassociative cyclic algebras over finite fields semidirect not nec prime}
\mathbb{Z} / \Big( \frac{s}{m} \Big) \mathbb{Z} \rtimes_{q} \mathbb{Z} / (m \mu) \mathbb{Z},
\end{equation}
where $\mu = m/\mathrm{gcd}(i,m)$.
Moreover if $i$ and $m$ are coprime, then
\begin{equation} \label{eqn:Automorphisms of nonassociative cyclic algebras over finite fields semidirect not nec prime 2}
\mathrm{Aut}_F(A) \cong \mathbb{Z} / \Big( \frac{s}{m} \Big) \mathbb{Z} \rtimes_{q} \mathbb{Z} / (m^2) \mathbb{Z}.
\end{equation}
\end{theorem}

\begin{proof}
Let $\tau: K \rightarrow K, \ k \mapsto k^q$, then $\tau^j(\lambda d^i) = \omega^{ij} \lambda d^i$, for all $j \in \{ 0, \ldots, m-1 \}$ where $\omega \in F^{\times}$ is a primitive $m^{\text{th}}$ root of unity by Lemma \ref{lem:eigenvalues and eigenvectors of cyclic extension automorphisms}. As $\tau$ generates $\mathrm{Gal}(K/F)$, the automorphisms of $A$ are precisely the maps $H_{\tau^j,k}$, where $j \in \{ 0, \ldots, m-1 \}$ and $k \in K^{\times}$ are such that $\tau^j(\lambda d^i) = N_{K/F}(k) \lambda d^i$ by Corollary \ref{cor:t^m-a automorphism field result}. Moreover there are exactly $s$ elements $k \in K^{\times}$ with $N_{K/F}(k) = \omega^{ij}$ by Lemma \ref{lem:number_of_solutions_to_norm_equation}, and each of these elements corresponds to a unique automorphism of $A$. Therefore $\mathrm{Aut}_F(A)$ is a group of order $ms$.

Choose $k \in K^{\times}$ such that $N_{K/F}(k) = \omega^i$ so that $H_{\tau,k} \in \mathrm{Aut}_F(A)$. As $\tau$ has order $m$, $H_{\tau,k} \circ \ldots \circ H_{\tau,k}$ ($m$-times) becomes $H_{\mathrm{id},b}$ where $b = \omega^i = N_{K/F}(k)$. Notice $\omega^i$ is a primitive $\mu^{\text{th}}$ root of unity where $\mu = m/\mathrm{gcd}(i,m)$, therefore $H_{\mathrm{id},b}$ has order $\mu$, and thus the subgroup of $\mathrm{Aut}_F(A)$ generated by $H_{\tau,k}$ has order $m \mu$.

$\langle G_{\alpha} \rangle$ is a cyclic subgroup of $\mathrm{Aut}_F(A)$ of order $s$ by Theorem \ref{thm:t^m-a_automorphism_field finite} where $\alpha$ is a primitive element of $K$. Furthermore, $m \vert s$ by Lemma \ref{lem:n divides s} and so $\langle G_{\alpha^m} \rangle$ is a cyclic subgroup of $\mathrm{Aut}_F(A)$ of order $s/m$. We will prove $\mathrm{Aut}_F(A)$ contains the semidirect product $\langle G_{\alpha^m} \rangle \rtimes_q \langle H_{\tau,k} \rangle$, by showing it can be written in the presentation \eqref{eqn:semidirect product of cyclic algebras}:

The inverse of $H_{\tau,k}$ in $\mathrm{Aut}_F(A)$ is $H_{\tau^{-1},\tau^{-1}(k^{-1})}$ and so
\begin{align*}
H_{\tau,k} \Big( & G_{\alpha^m}  \Big( H_{\tau,k}^{-1} \Big( \sum_{j=0}^{m-1} x_j t^j \Big) \Big) \Big) \\
&= H_{\tau,k} \Big( G_{\alpha^m} \Big( \tau^{-1}(x_0) + \sum_{j=1}^{m-1} \tau^{-1}(x_j) \big( \prod_{l=0}^{j-1} \sigma^{l}(\tau^{-1}(k^{-1})) \big) t^j \Big) \Big) \\
&= H_{\tau,k} \Big( \tau^{-1}(x_0) + \sum_{j=1}^{m-1} \tau^{-1}(x_j) \big( \prod_{l=0}^{j-1} \sigma^{l}(\tau^{-1}(k^{-1})) \big) \alpha^{-m} \sigma^j(\alpha^m) t^j \Big) \\
&= x_0 + \sum_{j=1}^{m-1} x_j \tau(\alpha^{-m}) \tau(\sigma^{j}(\alpha^m)) t^j = x_0 + \sum_{j=1}^{m-1} x_j \tau(\alpha^{-m}) \sigma^{j}(\tau(\alpha^m)) t^j \\
&= G_{\tau(\alpha^m)} \Big( \sum_{j=0}^{m-1} x_j t^j \Big) = G_{\alpha^{mq}} \Big( \sum_{j=0}^{m-1} x_j t^j \Big),
\end{align*}
that is $H_{\tau,k} \circ G_{\alpha^m} \circ H_{\tau,k}^{-1} = G_{\alpha^{mq}} = (G_{\alpha^m})^q.$

Notice $q^m = qs -s + 1$, i.e. $q^m \equiv 1 \text{ mod } (s)$, and so $q^{m \mu} \equiv 1 \text{ mod } (s)$. Then $m \vert s$ by Lemma \ref{lem:n divides s}, hence $q^{m \mu} \equiv 1 \text{ mod } (s/m)$. In order to prove $\mathrm{Aut}_F(A)$ contains $\langle G_{\alpha^m} \rangle \rtimes_q \langle H_{\tau,k} \rangle$, we are left to show that $\langle H_{\tau,k} \rangle \cap \langle G_{\alpha^m} \rangle = \{ \mathrm{id} \}$.

Suppose, for a contradiction, that $\langle H_{\tau,k} \rangle \cap \langle G_{\alpha^m} \rangle \neq \{ \mathrm{id} \}$. Then $H_{\mathrm{id},\omega^l} \in \langle G_{\alpha^m} \rangle$ for some $l \in \{ 1, \ldots, m-1 \}$. Therefore $\langle G_{\alpha^m} \rangle$ contains a subgroup of order $m/\mathrm{gcd}(l,m)$ generated by $H_{\mathrm{id}, \omega^l}$ and so $(m/\mathrm{gcd}(l,m)) \vert (s / m)$. This means $m^2 \vert (s \mathrm{gcd}(l,m))$, a contradiction by Lemma \ref{lem:n^2 does not divide s}.

Therefore $\mathrm{Aut}_F(A)$ contains the subgroup
$$\langle G_{\alpha^m} \rangle \rtimes_q \langle H_{\tau,k} \rangle \cong \mathbb{Z} / \Big( \frac{s}{m} \Big) \mathbb{Z} \rtimes_{q} \mathbb{Z} / (m \mu) \mathbb{Z}.$$
If $\mathrm{gcd}(i,m) = 1$ this subgroup has order $ms$ and since $|\mathrm{Aut}_F(A)| = ms$, this is all of $\mathrm{Aut}_F(A)$.
\end{proof}

We can completely determine the automorphism group of nonassociative cyclic algebras of prime degree $m$ different from $\mathrm{Char}(F)$. If $F$ does not contain a primitive $m^{\text{th}}$ root of unity, i.e. if $m \nmid (q-1)$, then $\mathrm{Aut}_F(A) = \langle G_{\alpha} \rangle \cong \mathbb{Z}/ s \mathbb{Z}$ by Theorem \ref{thm:t^m-a_automorphism_field finite}(iii), and all automorphisms of $A$ are inner. Otherwise we have: 


\begin{theorem} \label{thm:Automorphisms of nonassociative cyclic algebras over finite fields prime}
Suppose $m$ is prime and $m \vert (q-1)$. Then we can write $K = F(d)$ where $d$ is a root of the irreducible polynomial $x^m - e \in F[x]$ as in Lemma \ref{lem:eigenvalues and eigenvectors of cyclic extension automorphisms}. Let $A = (K/F,\sigma,a)$ for some $a \in K \setminus F$.
\begin{itemize}
\item[(i)] If $a \neq \lambda d^i$ for any $i \in \{ 0, \ldots, m-1 \}$, $\lambda\in F^\times$, then $\mathrm{Aut}_F(A) = \langle G_{\alpha} \rangle \cong \mathbb{Z}/ s \mathbb{Z}$ and all automorphisms of $A$ are inner.
\item[(ii)] (\cite[Theorem 6.3.5]{AndrewPhD}). If $m = 2$ and $a = \lambda d$ for some $\lambda \in F^{\times}$, then $\mathrm{Aut}_F(A)$ is the dicyclic group of order $2q + 2$.
\item[(iii)] If $m > 2$ and $a = \lambda d^i$ for some $i \in \{ 1 , \ldots, m-1 \}$, $\lambda \in F^{\times}$, then
$$\mathrm{Aut}_F(A) \cong \mathbb{Z} / \Big( \frac{s}{m} \Big) \mathbb{Z} \rtimes_{q} \mathbb{Z} / (m^2) \mathbb{Z}.$$
\end{itemize}
\end{theorem}

\begin{proof}
\begin{itemize}
\item[(i)] This is \cite[Corollary 6.3.3]{AndrewPhD} together with Theorem \ref{thm:t^m-a_automorphism_field finite}(i).
\item[(iii)] follows immediately from Theorem \ref{thm:Automorphisms of nonassociative cyclic algebras over finite fields}.
\end{itemize}
\end{proof}



\chapter{Generalisation of the \texorpdfstring{$S_f$}{S\_f} Construction} \label{chapter:Generalisation of the S_f Construction}

Until now, we have studied the construction $S_f = D[t;\sigma,\delta]/D[t;\sigma,\delta]f$ in the case where $D$ is an associative division ring. In this Chapter we generalise this construction using the skew polynomial ring $S[t;\sigma,\delta]$, where $S$ is any associative unital ring, $\sigma$ is an injective endomorphism of $S$ and $\delta$ is a $\sigma$-derivation of $S$. While $S[t;\sigma,\delta]$ is in general neither left nor right Euclidean (unless $S$ is a division ring), we are still able to right divide by polynomials $f(t) \in S[t;\sigma,\delta]$ whose leading coefficient is a unit. Moreover, if $\sigma$ is an automorphism we are also able to left divide by such $f(t)$. Therefore, when $f(t)$ has an invertible leading coefficient, it is possible to generalise the construction of the algebras $S_f$ to this setting.
\vspace{5mm}

In the following, let $S$ be a unital associative ring, $\sigma$ be an injective endomorphism of $S$ and $\delta$ be a left $\sigma$-derivation of $S$.
An element $0 \neq b \in S$ is called \textbf{right-invertible} if there exists $b_r \in S$ such that $b b_r = 1$, and \textbf{left-invertible} if there exists $b_l \in S$ such that $b_l b = 1$. An element $b$ which is both left and right invertible is called \textbf{invertible} (or a \textbf{unit}), and $b_l = b_r$ is called the \textbf{inverse} of $b$ and denoted $b^{-1}$ \cite[p.~4]{lam2013first}. We say a non-zero ring is a \textbf{domain} if it has no non-trivial zero divisors. A commutative domain is called an \textbf{integral domain}.

Recall from \eqref{eqn:mult in S_f 3} that $S_{n,j}$ denotes the sum of all monomials in $\sigma$ and $\delta$ that are of degree $j$ in $\sigma$ and degree $n - j$ in $\delta$. The equality
\begin{equation}
(bt^n)(ct^m) = \sum_{j=0}^n b(S_{n,j}(c))t^{j+m}, \tag{\ref{eqn:mult in S_f 2}}
\end{equation}
for all $b,c \in S$, holds more generally for $S$ any unital associative ring by \cite[p.~4]{leroy2013sigma}.

The degree function satisfies
\begin{equation} \label{eqn:generalised degree function}
\mathrm{deg} \big( g(t)h(t) \big) \leq \mathrm{deg}(g(t)) + \mathrm{deg}(h(t))
\end{equation}
for all $g(t), h(t) \in S[t;\sigma,\delta]$. In general \eqref{eqn:generalised degree function} is not an equality unless $S$ is a domain, or $g(t)$ has an invertible leading coefficient, or $h(t)$ has an invertible leading coefficient: 

Indeed, if $S$ is a domain this is \cite[p.~12]{gomez2014basic}, and the equality in \eqref{eqn:generalised degree function} implies $S[t;\sigma,\delta]$ is also a domain. Suppose now $S$ is not necessarily a domain, $g(t)$ has degree $l$ and leading coefficient $g_l$ and $h(t)$ has degree $n$ and leading coefficient $h_n$. If $h_n$ is invertible then $\sigma^l(h_n)$ is also invertible and
$$g(t)h(t)  = g_{l} \sigma^{l}(h_{n}) t^{l + n} + \text{ lower degree terms}.$$
Here $g_{l} \sigma^{l}(h_{n}) \neq 0$ since invertible elements are not zero divisors. Similarly, if $g_l$ is invertible then $g_{l} \sigma^{l}(h_{n}) \neq 0$. In either case \eqref{eqn:generalised degree function} is an equality.

\vspace{4mm}
$S[t;\sigma,\delta]$ is in general neither left nor right Euclidean, nor a left or right principal ideal domain. Nevertheless, we can still perform right division by a polynomial whose leading coefficient is a unit. Additionally, when $\sigma$ is an automorphism we can also left divide by such a polynomial. When $\delta = 0$ and $\sigma \in \mathrm{Aut}(S)$ this was proven for special cases of $S$, for instance in \cite[p.~4]{ducoat2015skew}, \cite[p.~4]{jitman2010skew} and \cite[p.~391]{mcdonald1974finite}:

\begin{theorem} \label{thm:generalised S_f euclidean division}
Let $f(t) = a_m t^m - \sum_{i=0}^{m-1} a_i t^i \in R = S[t;\sigma,\delta]$ and suppose $a_m$ is invertible.
\begin{itemize}
\item[(i)] For all $g(t) \in R$, there exists uniquely determined $r(t), q(t) \in R$ with $\mathrm{deg}(r(t)) < m$, such that $g(t) = q(t) f(t) + r(t)$.
\item[(ii)] Suppose additionally $\sigma$ is an automorphism of $S$. Then for all $g(t) \in R$, there exists uniquely determined $r(t), q(t) \in R$ with $\mathrm{deg}(r(t)) < m$, such that $g(t) = f(t) q(t) + r(t)$.
\end{itemize}
\end{theorem}

\begin{proof}
Let $g(t) = \sum_{i=0}^{l} g_i t^i \in R$ have degree $l$. 
\begin{itemize}
\item[(i)] Suppose $l < m$, then $g(t) = 0 q(t) + g(t)$. Moreover, we have
$$\mathrm{deg}(q(t)f(t)) = \mathrm{deg}(q(t)) + \mathrm{deg}(f(t)) \geq m,$$
for all $0 \neq q(t) \in R$ as $f(t)$ has an invertible leading coefficient. Therefore if $g(t) = q(t)f(t) + r(t)$ then $q(t) = 0$ and so $r(t) = g(t)$ as required.

Now suppose $l \geq m$, then
\begin{align*}
g(t) - g_l \sigma^{l-m}(&a_m^{-1}) t^{l-m}f(t) = 
g(t) - g_l \sigma^{l-m}(a_m^{-1}) t^{l-m} \Big( a_m t^m - \sum_{i=0}^{m-1} a_i t^i \Big) \\
&= g(t) - g_l \sigma^{l-m}(a_m^{-1}) t^{l-m} a_m t^m + \sum_{i=0}^{m-1} g_l \sigma^{l-m}(a_m^{-1}) t^{l-m} a_i t^i \\
&= g(t) - g_l \sigma^{l-m}(a_m^{-1}) \Big( \sum_{j=0}^{l-m} S_{l-m,j} (a_m) t^j \Big) t^m \\ & \qquad \qquad + \sum_{i=0}^{m-1} g_l \sigma^{l-m}(a_m^{-1}) \Big( \sum_{j=0}^{l-m} S_{l-m,j}(a_i) t^j \Big) t^i \\
&= g(t) - g_l \sigma^{l-m}(a_m^{-1}) S_{l-m,l-m} (a_m) t^l 
\\ & \qquad \qquad - g_l \sigma^{l-m}(a_m^{-1}) \sum_{0 \leq j \leq l-m-1} S_{l-m,j} (a_m) t^{j+m} \\
& \qquad \qquad + \sum_{i=0}^{m-1} \sum_{j=0}^{l-m} g_l \sigma^{l-m}(a_m^{-1}) S_{l-m,j}(a_i)t^{i+j} \\
&= g(t) - g_l \sigma^{l-m}(a_m^{-1}) \sigma^{l-m} (a_m) t^l 
\\ & \qquad \qquad - \sum_{0 \leq j \leq l-m-1} g_l \sigma^{l-m}(a_m^{-1}) S_{l-m,j} (a_m) t^{j+m} \\
& \qquad \qquad + \sum_{i=0}^{m-1} \sum_{j=0}^{l-m} g_l \sigma^{l-m}(a_m^{-1}) S_{l-m,j}(a_i)t^{i+j} \\
&= g(t) - g_l t^l - \sum_{0 \leq j \leq l-m-1} g_l \sigma^{l-m}(a_m^{-1}) S_{l-m,j} (a_m) t^{j+m} \\
& \qquad \qquad + \sum_{i=0}^{m-1} \sum_{j=0}^{l-m} g_l \sigma^{l-m}(a_m^{-1}) S_{l-m,j}(a_i)t^{i+j},
\end{align*}
where we have used $S_{l-m,l-m}(a_m) = \sigma^{l-m}(a_m)$. Hence the polynomial $g(t) - g_l \sigma^{l-m}(a_m^{-1}) t^{l-m}f(t)$ has degree $< l$ and by iteration of this argument, we find $q(t), r(t) \in R$ with $\mathrm{deg}(r(t)) < m$, such that $g(t) = q(t)f(t) + r(t).$

We now prove uniqueness of $r(t)$ and $q(t)$. Suppose
$$g(t) = q_1(t) f(t) + r_1(t) = q_2(t) f(t) + r_2(t),$$
and so $\big( q_1(t) - q_2(t) \big) f(t) = r_2(t) - r_1(t).$
If $q_1(t) - q_2(t) \neq 0$, then the left hand side of the equation has degree $\geq m$ since $f(t)$ has an invertible leading coefficient, while the right hand side has degree $< m$. Thus $q_1(t) = q_2(t)$ and so $r_1(t) = r_2(t)$.
\item[(ii)] If we assume $\sigma$ is an automorphism, then (ii) is proven similarly to (i) by showing $g(t) - f(t) \sigma^{-m}(a_m^{-1}) \sigma^{-m}(g_l) t^{l-m}$ has degree $<l$ and iterating this argument. Uniqueness is proven analogously to (i).
\end{itemize}
\end{proof}

Let $\mathrm{mod}_r f$ denote the remainder after right division by $f(t)$. Since remainders are uniquely determined by Theorem \ref{thm:generalised S_f euclidean division}(i), the skew polynomials of degree $< m$ canonically represent the elements of the left $S[t;\sigma,\delta]$-module $S[t;\sigma,\delta]/S[t;\sigma,\delta]f$. Similarly, when $\sigma$ is an automorphism, the skew polynomials of degree $< m$ canonically represent the elements of the right $S[t;\sigma,\delta]$-module $S[t;\sigma,\delta]/fS[t;\sigma,\delta]$.

\begin{definition}
Let $f(t) \in R = S[t;\sigma,\delta]$ be of degree $m$ with invertible leading coefficient $a_m \in S$ and $R_m = \{ g(t) \in R \ \vert \ deg(g(t)) < m \}.$
\begin{itemize}
\item[(i)] Then $R_m$ together with the multiplication $g \circ h = g h \ \mathrm{mod}_r f$, becomes a unital nonassociative ring $S_f = (R_m, \circ)$, also denoted $R/Rf$.
\item[(ii)] Suppose additionally $\sigma$ is an automorphism and let $\mathrm{mod}_l f$ denote the remainder after left division by $f(t)$. Then $R_m$ together with the multiplication $g \prescript{}{f}\circ \ h = g h \ \mathrm{mod}_l f$, becomes a unital nonassociative ring $_f S = (R_m, \prescript{}{f}\circ)$, also denoted $R/fR$.
\end{itemize}
\end{definition}

$S_f$ and $_f S$ are unital nonassociative algebras over $$S_0 = \{ c \in S \ \vert \ c \circ h = h \circ c \text{ for all } h \in S_f \} = \mathrm{Comm}(S_f) \cap S,$$ which is a commutative subring of $S$. If $S$ is a division ring then this construction is precisely Petit's algebra construction given in Chapter \ref{chapter:Preliminaries}.

\begin{remarks}
\begin{itemize}
\item[(i)] Let $g(t), h(t) \in R_m$ be such that $\mathrm{deg}(g(t)h(t)) < m$. Then the multiplication $g \circ h$ is the same as that in $R$.
\item[(ii)] $S_f$ is associative if and only if $f(t)$ is right invariant \cite[Theorem 4(ii)]{pumplun2015finite}.
\item[(iii)] We have $qf + r = (qd^{-1})(df) + r$, for all $q,r \in R$, and all invertible $d \in S$. It follows that $S_f = S_{df}$ for all invertible $d \in S$. 
\item[(iv)] If $\mathrm{deg}(f(t)) = m = 1$ then $R_m = S$ and $S_f \cong S$.
\end{itemize}
\end{remarks}

Henceforth suppose $\mathrm{deg}(f(t)) = m \geq 2$. We may assume w.l.o.g. again that $f(t)$ is monic. Similarly to Proposition \ref{prop:_fS opposite algebra of some S_g}, it suffices to only consider the algebras $S_f$ since we still have following anti-isomorphism:

\begin{proposition} 
(\cite[Proposition 3]{pumplun2015finite}).
Let $f(t) \in S[t;\sigma,\delta]$ where $\sigma \in \mathrm{Aut}(S)$ and $f(t)$ has invertible leading coefficient. The canonical anti-isomorphism
$$\psi: S[t;\sigma,\delta] \rightarrow S^{\mathrm{op}}[t;\sigma^{-1},-\delta \sigma^{-1}], \ \ \sum_{i=0}^{n} b_i t^i \mapsto \sum_{i=0}^{n} \Big( \sum_{j=0}^{i} S_{n,j}(b_i) \Big) t^i,$$
between the skew polynomial rings $S[t;\sigma,\delta]$ and $S^{\mathrm{op}}[t;\sigma^{-1},-\delta \sigma^{-1}]$, induces an anti-isomorphism between the rings $S_f = S[t;\sigma,\delta] / S[t;\sigma,\delta]f$, and
$$_{\psi(f)}S = S^{\mathrm{op}}[t;\sigma^{-1},-\delta \sigma^{-1}] / \psi(f)S^{\mathrm{op}}[t;\sigma^{-1},-\delta \sigma^{-1}].$$
\end{proposition}

We now take a closer look at $S_f$ in the special case where $\delta = 0$.

\begin{proposition} \label{prop:f(t) two-sided delta=0 generalised}
Suppose $f(t) = t^m - \sum_{i=0}^{m-1} a_i t^i \in R = S[t;\sigma]$. If $\sigma^m(z) a_i = a_i \sigma^i(z)$ and $\sigma(a_i) = a_i$ for all $z \in S$, $i \in \{ 0, \ldots, m-1 \}$, then $f(t)$ is right invariant and $S_f$ is associative.
\end{proposition}

\begin{proof}
We have
\begin{align*}
f(t) z t^j &= t^m z t^j - \sum_{i=0}^{m-1} a_i t^i z t^j = \sigma^m(z) t^{m+j} - \sum_{i=0}^{m-1} a_i \sigma^i(z) t^{i+j} \\
&= \sigma^m(z) t^j t^m - \sum_{i=0}^{m-1} \sigma^m(z) a_i t^j t^i = \sigma^m(z) t^j t^m - \sum_{i=0}^{m-1} \sigma^m(z) t^j a_i t^i \\
&= \sigma^m(z) t^j f(t) \in Rf,
\end{align*}
for all $z \in S$, $j \in \{ 0, \ldots, m-1 \}$ since $t^ja_i = \sigma^j(a_i)t^j = a_i t^j$ and $\sigma^m(z) a_i = a_i \sigma^i(z)$. Therefore by distributivity $fR \subseteq Rf$. This implies $bfc \in Rf$ for all $b,c \in R$, thus $Rf$ is a two-sided ideal and so $S_f$ is associative by \cite[Theorem 4(ii)]{pumplun2015finite}. 
\end{proof}

\begin{proposition} \label{prop:Comm(S_f) delta=0 generalised}
Suppose $f(t) = t^m - \sum_{i=0}^{m-1} a_i t^i \in S[t;\sigma]$.
\begin{itemize}
\item[(i)] (\cite[Theorem 8]{pumplun2015finite}). $\mathrm{Comm}(S_f)$ contains the set
\begin{equation} \label{eqn:Comm(S_f) generalised}
\Big\{ \sum_{i=0}^{m-1} c_i t^i \ \vert \ c_i \in \mathrm{Fix}(\sigma) \text{ and } b c_i = c_i \sigma^i(b) \text{ for all } b \in S \Big\}.
\end{equation}
\item[(ii)] If $S$ is a domain and $a_0 \neq 0$ the two sets are equal.
\item[(iii)] If $a_0$ is invertible the two sets are equal.
\end{itemize}
\end{proposition}

\begin{proof}
\begin{itemize}
\item[(ii)] Let $c = \sum_{i=0}^{m-1} c_i t^i \in \mathrm{Comm}(S_f)$, then in particular
\begin{equation} \label{eqn:Comm(S_f) delta=0 generalised 1}
t \circ c = t \circ \Big( \sum_{i=0}^{m-1} c_i t^i \Big) = \sum_{i=0}^{m-2} \sigma(c_i) t^{i+1} + \sigma(c_{m-1}) \sum_{i=0}^{m-1} a_i t^i,
\end{equation}
and
\begin{equation} \label{eqn:Comm(S_f) delta=0 generalised 2}
c \circ t = \Big( \sum_{i=0}^{m-1} c_i t^i \Big) \circ t = \sum_{i=0}^{m-2} c_i t^{i+1} + c_{m-1} \sum_{i=0}^{m-1} a_i t^i,
\end{equation}
must be equal. Comparing the $t^0$ coefficient in \eqref{eqn:Comm(S_f) delta=0 generalised 1} and \eqref{eqn:Comm(S_f) delta=0 generalised 2} yields
\begin{equation} \label{eqn:Comm(S_f) delta=0 generalised 3}
(\sigma(c_{m-1})-c_{m-1}) a_0 = 0,
\end{equation}
which implies $c_{m-1} \in \mathrm{Fix}(\sigma)$ because $S$ is a domain and $a_0 \neq 0$. Comparing the $t^j$ coefficients in \eqref{eqn:Comm(S_f) delta=0 generalised 1} and \eqref{eqn:Comm(S_f) delta=0 generalised 2} gives
\begin{equation*}
\sigma(c_{j-1}) + \sigma(c_{m-1}) a_j = c_{j-1} + c_{m-1} a_j,
\end{equation*}
for all $j \in \{ 1, \ldots, m-1 \}$, hence $c_{j-1} \in \mathrm{Fix}(\sigma)$ for all $j \in \{ 1, \ldots, m-1 \}$ because $c_{m-1} \in \mathrm{Fix}(\sigma)$.

Since $c \in \mathrm{Comm}(S_f)$ we also have $b \circ c = c \circ b$ for all $b \in S$. As a result
\begin{equation} \label{eqn:Comm(S_f) delta=0 generalised 4}
b \circ c = \sum_{i=0}^{m-1} b c_i t^i,
\end{equation}
and
\begin{equation} \label{eqn:Comm(S_f) delta=0 generalised 5}
c \circ b = \sum_{i=0}^{m-1} c_i t^i b = \sum_{i=0}^{m-1} c_i \sigma^i(b) t^i
\end{equation}
must be equal. Comparing the coefficients of the powers of $t$ in \eqref{eqn:Comm(S_f) delta=0 generalised 4} and \eqref{eqn:Comm(S_f) delta=0 generalised 5} yields $b c_i = c_i \sigma^i(b)$ for all $b \in S$, $i \in \{ 0, \ldots, m-1 \}$ as required.
\item[(iii)] The proof is similar to (ii), but \eqref{eqn:Comm(S_f) delta=0 generalised 3} implies $c_{m-1} \in \mathrm{Fix}(\sigma)$ because $a_0$ is invertible, hence not a zero divisor.
\end{itemize}
\end{proof}

\begin{corollary} \label{cor:Comm(S_f) = F generalised}
Suppose $S$ is a central simple algebra over a field $C$ and $\sigma \in \mathrm{Aut}(S)$ is such that $\sigma \vert_C$ has order at least $m$. Let $f(t) = t^m - \sum_{i=0}^{m-1} a_i t^i \in S[t;\sigma]$ where $a_0 \in S$ is invertible. Then $\mathrm{Comm}(S_f) = C \cap \mathrm{Fix}(\sigma)$.
\end{corollary}

\begin{proof}
$\mathrm{Comm}(S_f)$ is equal to the set \eqref{eqn:Comm(S_f) generalised} by Proposition \ref{prop:Comm(S_f) delta=0 generalised}, in particular $C \cap \mathrm{Fix}(\sigma) \subseteq \mathrm{Comm}(S_f)$.

Let $\sum_{i=0}^{m-1} c_i t^i \in \mathrm{Comm}(S_f)$ and suppose, for contradiction, $c_j \neq 0$ for some $j \in \{ 1, \ldots, m-1 \}$. Then $bc_j = c_j \sigma^j(b)$ for all $b \in S$, in particular
$(b - \sigma^j(b))c_j = 0$
for all $b \in C$. We have $0 \neq b - \sigma^j(b)$ for some $b \in C$ since $\sigma \vert_C$ has order $\geq m$. Moreover $b - \sigma^j(b) \in C$, therefore it is invertible and hence $(b - \sigma^j(b))c_j \neq 0$ for some $b \in C$ because invertible elements are not zero divisors, a contradiction. Thus $\sum_{i=0}^{m-1} c_i t^i = c_0 \in C \cap \mathrm{Fix}(\sigma)$ by \eqref{eqn:Comm(S_f) generalised}.
\end{proof}

\section{When does \texorpdfstring{$S_f$}{S\_f} Contain Zero Divisors?}

We investigate when the algebras $S_f$ contain zero divisors. If the ring $S$ contains zero divisors then clearly so does $S_f$, therefore we only consider the case where $S$ is a domain. Furthermore, if $f(t) = g(t) h(t)$ for some $g(t), h(t) \in R_m$, then $g(t) \circ h(t) = 0$ and so $S_f$ contains zero divisors.

When $S$ is commutative, i.e. an integral domain, it is well-known that we can associate to $S$ its field of fractions $D \supseteq S$, so that every element of $D$ has the form $rs^{-1}$ for some $r \in S$, $0 \neq s \in S$. On the other hand, if $S$ is a noncommutative domain then we cannot necessarily associate such a "right division ring of fractions" unless $S$ is a so-called right Ore domain:

\begin{definition}
A \textbf{right Ore domain} $S$ is a domain such that $aS \cap bS \neq \{ 0 \}$ for all $0 \neq a, b \in S$. The \textbf{ring of right fractions} of $S$ is a division ring $D$ containing $S$, such that every element of $D$ is of the form $rs^{-1}$ for some $r \in S$ and $0 \neq s \in S$.
\end{definition}

Any integral domain is a right Ore domain; its right ring of fractions is equal to its quotient field.

Let now $S$ be a right Ore domain with ring of right fractions $D$, $\sigma$ be an injective endomorphism of $S$ and $\delta$ be a $\sigma$-derivation of $S$. Then $\sigma$ and $\delta$ extend uniquely to $D$ by setting
\begin{equation} \label{eqn:extend sigma delta to right ring of fractions}
\sigma(rs^{-1}) = \sigma(r)\sigma(s)^{-1} \text{ and } \delta(rs^{-1}) = \delta(r)s^{-1} - \sigma(rs^{-1}) \delta(s) s^{-1},
\end{equation}
for all $r \in S, \ 0 \neq s \in S$ by \cite[Lemma 1.3]{goodearl1992prime}. We conclude:

\begin{theorem} \label{thm:right Ore domain, no zero divisors}
Let $f(t) \in S[t;\sigma,\delta]$ have invertible leading coefficient and extend $\sigma$ and $\delta$ to $D$ as in \eqref{eqn:extend sigma delta to right ring of fractions}. If $f(t)$ is irreducible in $D[t;\sigma,\delta]$ then $S_f = S[t;\sigma,\delta]/S[t;\sigma,\delta]f(t)$ contains no zero divisors.
\end{theorem}

\begin{proof}
If $f(t)$ is irreducible in $D[t;\sigma,\delta]$, then $D[t;\sigma,\delta]/D[t;\sigma,\delta]f$ is a right division algebra by Theorem \ref{thm:f(t) irreducible iff S_f right division}, therefore it contains no zero divisors. $S_f  = S[t;\sigma,\delta]/S[t;\sigma,\delta]f$ is contained in $D[t;\sigma,\delta]/D[t;\sigma,\delta]f$, so also contains no zero divisors.
\end{proof}


Let $K$ be a field and $y$ be an indeterminate. Then $K[y]$ is an integral domain which is not a division algebra. Given $a(y) \in K[y]$, denote $\mathrm{deg}_y(a(y))$ the degree of $a(y)$ as a polynomial in $y$. We obtain the following Corollaries of Theorem \ref{thm:right Ore domain, no zero divisors} using Corollaries \ref{cor:t^2-a(y) in K(y)[t;sigma] irreducibility} and \ref{cor:t^m-a(y) in K(y)[t;sigma] irreducible} in Chapter \ref{chapter:Irreducibility Criteria for Polynomials in a Skew Polynomial Ring}:

\begin{corollary}
Define the injective endomorphism $\sigma: K[y] \rightarrow K[y]$ by $\sigma \vert_K = \mathrm{id}$ and $\sigma(y) = y^2$. If $f(t) = t^2 - a(y) \in K[y][t;\sigma]$ where $0 \neq a(y) \in K[y]$ is such that $3 \nmid \mathrm{deg}_y(a(y))$, then $S_f$ contains no zero divisors. Here $S_f$ is an infinite-dimensional algebra over $S_0 = \mathrm{Fix}(\sigma) = K$.
\end{corollary}

\begin{proof}
Extend $\sigma$ to an endomorphism of $K(y)$, the field of fractions of $K[y]$ as in \eqref{eqn:extend sigma delta to right ring of fractions}. Then $f(t)$ is irreducible in $K(y)[t;\sigma]$ by Corollary \ref{cor:t^2-a(y) in K(y)[t;sigma] irreducibility} and hence $S_f$ contains no zero divisors by Theorem \ref{thm:right Ore domain, no zero divisors}.
\end{proof}

\begin{corollary} \label{cor:K[y] zero divisors generalised}
Let $\sigma$ be the automorphism of $K[y]$ defined by $\sigma \vert_K = \mathrm{id}$ and $\sigma(y) = qy$ for some $1 \neq q \in K^{\times}$.
Suppose $m$ is prime, $K$ contains a primitive $m^{\text{th}}$ root of unity and $f(t) = t^m - a(y) \in K[y][t;\sigma]$ where $0 \neq a(y) \in K[y]$ is such that $m \nmid \mathrm{deg}_y (a(y))$. Then $S_f$ contains no zero divisors.
\end{corollary}

\begin{proof}
Extend $\sigma$ to an automorphism of $K(y)$ as in \eqref{eqn:extend sigma delta to right ring of fractions}. Then $f(t)$ is irreducible in $K(y)[t;\sigma]$ by Corollary \ref{cor:t^m-a(y) in K(y)[t;sigma] irreducible}, which implies $S_f$ contains no zero divisors by Theorem \ref{thm:right Ore domain, no zero divisors}.
\end{proof}

\begin{remark}
Consider the set-up in Corollary \ref{cor:K[y] zero divisors generalised}. If $q$ is not a root of unity then $\sigma(y^i) = q^i y^i \neq y^i$ for all $i > 1$, therefore $S_f$ is an infinite-dimensional algebra over $S_0 = \mathrm{Fix}(\sigma) = K$.

Otherwise $q$ is a primitive $n^{\text{th}}$ root of unity for some $n \in \mathbb{N}$, then $\sigma(y^i) = q^i y^i = y^i$ if and only if $i = ln$ for some positive integer $l$. Thus $S_f$ is an algebra over $S_0 = \mathrm{Fix}(\sigma) = K[y^n]$, and since $K[y]$ is finite-dimensional over $K[y^n]$, $S_f$ is also finite-dimensional over $K[y^n]$.
\end{remark}

\section{The Nucleus}

In this Section we study the nuclei of $S_f$, generalising some of our results from Chapter \ref{chapter:The Structure of Petit Algebras}.

\begin{theorem} \label{thm:nucleus of S_f generalised}
Let $f(t) \in R = S[t;\sigma,\delta]$ be of degree $m$. Then
$$S \subseteq \mathrm{Nuc}_l(S_f) \ , \ S \subseteq \mathrm{Nuc}_m(S_f)$$
and
$$\mathrm{Nuc}_r(S_f) = \{ u \in R_m \ \vert \ fu \in Rf \} = E(f).$$
\end{theorem}

\begin{proof}
Let $b, c, d \in R_m$ and write $bc = q_1 f + r_1$ and $cd = q_2 f + r_2$ for some uniquely determined $q_1, q_2, r_1, r_2 \in R$ with $\mathrm{deg}(r_1(t)), \ \mathrm{deg}(r_2(t)) < m$. This means $b \circ c = bc - q_1 f$ and $c \circ d = cd - q_2 f$. We have
$$(b \circ c) \circ d = (bc - q_1 f) \circ d = (bc - q_1 f)d \ \mathrm{mod}_r f = (bcd - q_1fd) \ \mathrm{mod}_r f,$$
and
$$b \circ (c \circ d) = b \circ (cd - q_2f) = b(cd - q_2f) \ \mathrm{mod}_r f = bcd \ \mathrm{mod}_r f,$$
and hence $(b \circ c) \circ d = b \circ (c \circ d)$ if and only if $q_1 f d \ \mathrm{mod}_r f = 0$ if and only if $q_1 f d \in Rf$.
\begin{itemize}
\item[(i)] If $b \in S$ then 
\begin{equation} \label{eqn:generalised Nucleus1}
\mathrm{deg}(bc) \leq \mathrm{deg}(b) + \mathrm{deg}(c) \leq \mathrm{deg}(c) < m,
\end{equation}
but here $bc = q_1 f + r_1$ also means $q_1 = 0$, otherwise $\mathrm{deg}(q_1 f + r_1) \geq m$ as $f(t)$ has invertible leading coefficient, contradicting \eqref{eqn:generalised Nucleus1}.
Hence $q_1 f d \ \mathrm{mod}_r f = 0$ which implies $b \in \mathrm{Nuc}_l(S_f)$.
\item[(ii)] $S \subseteq \mathrm{Nuc}_m(S_f)$ is proven similarly to (i).
\item[(iii)] If $d \in E(f)$ then $fd \in Rf$ and thus $q_1 f d \in Rf$ for all $q_1 \in R$. This implies $d \in \mathrm{Nuc}_r(S_f)$ and $E(f) \subseteq \mathrm{Nuc}_r(S_f)$.

To prove the opposite inclusion, let now $d \in \mathrm{Nuc}_r(S_f)$ and choose $b(t), c(t) \in R$ with invertible leading coefficients $b_l$ and $c_n$ respectively such that $\mathrm{deg}(b(t)) + \mathrm{deg}(c(t)) = m$, so that $\mathrm{deg}(b(t)c(t)) = m$. We have $\mathrm{deg}(q_1(t) f(t)) = \mathrm{deg}(q_1(t)) + m$, but using that $b(t)c(t) = q_1(t)f(t)+r_1(t)$, we conclude $\mathrm{deg}(q_1(t)) = 0$, so $q_1(t) = q_1 \in S$ is non-zero. If $\mathrm{deg}(b(t)) = l$ then the leading coefficient of $b(t)c(t)$ is $b_l \sigma^l(c_n)$ and the leading coefficient of $q_1 f(t)$ is $q_1$. Therefore $q_1 = b_{l} \sigma^{l}(c_{n})$ is invertible in $S$, being a product of invertible elements of $S$. Since $d \in \mathrm{Nuc}_r(S_f)$ implies $q_1 f d \in Rf$, this yields $f d \in Rf$.
\end{itemize}
\end{proof}

By Theorem \ref{thm:nucleus of S_f generalised}(iii) we immediately conclude:

\begin{corollary}
$E(f) = S_f$ if and only if $S_f$ is associative.
\end{corollary}

\begin{proof}
$\mathrm{Nuc}_r(S_f) = S_f$ if and only if $S_f$ is associative so the result follows by Theorem \ref{thm:nucleus of S_f generalised}(iii).
\end{proof}

The following generalises \cite[(5)]{Petit1966-1967}:

\begin{proposition} \label{prop:Petit (5) generalised}
Let $f(t) \in R = S[t;\sigma,\delta]$ be of degree $m$. The powers of $t$ are associative if and only if $t \circ t^m = t^m \circ t$ if and only if $t \in \mathrm{Nuc}_r(S_f)$.
\end{proposition}

\begin{proof}
Let $t \in \mathrm{Nuc}_r(S_f)$. Then also $t, \ldots, t^{m-1} \in \mathrm{Nuc}_r(S_f)$ giving $[t^i,t^j,t^k] = 0$ for all $i,j,k < m$, i.e. the powers of $t$ are associative. In particular $(t \circ t^{m-1}) \circ t = t \circ (t^{m-1} \circ t)$, that is $t^m \circ t = t \circ t^m$. We are left to prove $t^m \circ t = t \circ t^m$ implies $t \in \mathrm{Nuc}_r(S_f)$. Suppose $t^m \circ t = t \circ t^m$, then
\begin{equation} \label{eqn:Petit (5) generalised t^m circ t}
t^m \circ t = (t^m - f(t)) \circ t = \big( t^{m+1} - f(t)t \big) \ \mathrm{mod}_r f,
\end{equation}
and
\begin{equation} \label{eqn:Petit (5) generalised t circ t^m}
t \circ t^m = t \circ (t^m - f(t)) = \big( t^{m+1} - t f(t) \big) \ \mathrm{mod}_r f = t^{m+1} \ \mathrm{mod}_r f,
\end{equation}
are equal. Comparing \eqref{eqn:Petit (5) generalised t^m circ t} and \eqref{eqn:Petit (5) generalised t circ t^m} yields $f(t)t \ \mathrm{mod}_r f = 0$, hence $f(t)t \in Rf$. This means $t \in \mathrm{Nuc}_r(S_f)$ by Theorem \ref{thm:nucleus of S_f generalised}.
\end{proof}

When $\delta = 0$ and $S$ is a domain, then either $S_f$ is associative or $S_f$ has left and middle nuclei equal to $S$:

\begin{theorem} \label{thm:generalised left, middle nucleus=S}
Let $S$ be a domain and $f(t) = t^m - \sum_{i=0}^{m-1} a_i t^i \in S[t;\sigma]$. If $S_f$ is not associative then $\mathrm{Nuc}_l(S_f) = \mathrm{Nuc}_m(S_f) = S$.
\end{theorem}

\begin{proof}
We have $S \subseteq \mathrm{Nuc}_l(S_f)$ and $S \subseteq \mathrm{Nuc}_m(S_f)$ by Theorem \ref{thm:nucleus of S_f generalised}. Suppose $S_f$ is not associative.
\begin{itemize}
\item[(i)] We prove $\mathrm{Nuc}_l(S_f) \subseteq S$: 

Suppose first that $\sigma(a_i) = a_i$ for all $i \in \{ 0, \ldots, m-1 \}$. Let $0 \neq p = \sum_{i=0}^{m-1} p_i t^i \in \mathrm{Nuc}_l(S_f)$ be arbitrary and $j \in \{ 0, \ldots, m-1 \}$ be maximal such that $p_j \neq 0$. Suppose towards a contradiction $j > 0$. Then
\begin{align*}
(p \circ t^{m-j}) \circ c &= \big( \sum_{i=0}^{j} p_i t^i \circ t^{m-j} \big) \circ c = \big( \sum_{i=0}^{j-1} p_i t^{i+m-j} + p_j \sum_{i=0}^{m-1} a_i t^i \big) \circ c \\
&= \sum_{i=0}^{j-1} p_i \sigma^{i+m-j}(c) t^{i+m-j} + \sum_{i=0}^{m-1} p_j a_i \sigma^i(c) t^i,
\end{align*}
and
\begin{align*}
p \circ (t^{m-j} \circ c) &= \sum_{i=0}^{j} p_i t^i \circ \sigma^{m-j}(c) t^{m-j} \\ &= \sum_{i=0}^{j-1} p_i \sigma^{i+m-j}(c) t^{i+m-j} + \sum_{i=0}^{m-1} p_j \sigma^m(c) a_i t^i,
\end{align*}
must be equal for all $c \in S$. Comparing the coefficients of $t^i$ yields
\begin{equation*}
p_j (a_i \sigma^i(c) - \sigma^m(c) a_i) = 0
\end{equation*}
for all $c \in S$, $i \in \{ 0, \ldots, m-1 \}$, thus $a_i \sigma^i(c) - \sigma^m(c) a_i = 0$ since $S$ is a domain and $p_j \neq 0$. This implies $S_f$ is associative by Proposition \ref{prop:f(t) two-sided delta=0 generalised}, a contradiction. Thus $p = p_0 \in S$.

Now suppose $a_l \notin \mathrm{Fix}(\sigma)$ for some $l \in \{ 0, \ldots, m-1 \}$. As before let $0 \neq p = \sum_{i=0}^{m-1} p_i t^i \in \mathrm{Nuc}_l(S_f)$ be arbitrary and $j \in \{ 0, \ldots, m-1 \}$ be maximal such that $p_j \neq 0$. Suppose towards a contradiction $j > 0$.

If $j = 1$ then
\begin{align*}
(p \circ t^{m-1}) \circ t &= \big( p_0 t^{m-1} + p_1 \sum_{i=0}^{m-1} a_i t^i \big) \circ t \\ &= p_0 \sum_{i=0}^{m-1} a_i t^i + \sum_{i=0}^{m-2} p_1 a_i t^{i+1} + p_1 a_{m-1} \sum_{i=0}^{m-1} a_i t^i
\end{align*}
and
\begin{align*}
p \circ (t^{m-1} \circ t) &= p_0 \sum_{i=0}^{m-1} a_i t^i + p_1 \sum_{i=0}^{m-2} \sigma(a_i) t^{i+1} + p_1 \sigma(a_{m-1}) \sum_{i=0}^{m-1} a_i t^i 
\end{align*}
must be equal. Comparing the coefficients of $t^0$ yields
\begin{equation} \label{eqn:generalised left, middle nucleus=S 1}
p_1 a_{m-1}a_0 = p_1 \sigma(a_{m-1})a_0,
\end{equation}
and comparing the coefficients of $t^i$ yields
\begin{equation} \label{eqn:generalised left, middle nucleus=S 2}
p_1 a_{i-1} + p_1 a_{m-1} a_i = p_1 \sigma(a_{i-1}) + p_1 \sigma(a_{m-1}) a_i,
\end{equation}
for all $i \in \{ 1, \ldots, m-1 \}$.

If $a_{m-1} \in \mathrm{Fix}(\sigma)$ then \eqref{eqn:generalised left, middle nucleus=S 2} implies $p_1(a_{i-1} - \sigma(a_{i-1})) = 0$, that is $a_{i} \in \mathrm{Fix}(\sigma)$ for all $i \in \{ 0, \ldots, m-2 \}$ because $S$ is a domain and $p_1 \neq 0$, a contradiction since $a_l \notin \mathrm{Fix}(\sigma)$. Therefore $a_{m-1} \notin \mathrm{Fix}(\sigma)$. Let $k \in \{ 0, \ldots, m-1 \}$ be minimal such that $a_k \neq 0$. If $k = 0$ then $p_1 = 0$ by \eqref{eqn:generalised left, middle nucleus=S 1} as $S$ is a domain, a contradiction. Otherwise \eqref{eqn:generalised left, middle nucleus=S 2} implies $p_1(a_{m-1}-\sigma(a_{m-1}))a_k = 0$, therefore $p_1=0$ as $S$ is a domain and $a_{m-1} \notin \mathrm{Fix}(\sigma)$, a contradiction.

Suppose now $j \geq 2$, then
\begin{align*}
(p &\circ t^{m-j}) \circ t = \big( \sum_{i=0}^{j-1} p_i t^{i+m-j} + p_j \sum_{i=0}^{m-1} a_i t^i \big) \circ t \\
&= \sum_{i=0}^{j-2} p_i t^{i+m-j+1} + p_{j-1} \sum_{i=0}^{m-1} a_i t^i + p_j \sum_{i=0}^{m-2} a_i t^{i+1} + p_j a_{m-1} \sum_{i=0}^{m-1} a_i t^i
\end{align*}
and
\begin{align*}
p &\circ (t^{m-j} \circ t) = \sum_{i=0}^{j} p_i t^i \circ t^{m-j+1} \\
&= \sum_{i=0}^{j-2} p_i t^{i+m-j+1} + p_{j-1} \sum_{i=0}^{m-1} a_i t^i + p_j \sum_{i=0}^{m-2} \sigma(a_i) t^{i+1} + p_j \sigma(a_{m-1}) \sum_{i=0}^{m-1} a_i t^i
\end{align*}
must be equal. Comparing them gives 
\begin{equation}
p_j a_{m-1} a_0 = p_j \sigma(a_{m-1}) a_0,
\end{equation}
and
\begin{equation}
p_j a_{i-1} + p_j a_{m-1} a_i = p_j \sigma(a_{i-1}) + p_j \sigma(a_{m-1}) a_i.
\end{equation}
This yields a contradiction similar to the $j=1$ case. Therefore $p = p_0 \in S$ and so $\mathrm{Nuc}_l(S_f) \subseteq S$.
\item[(ii)] The proof that $\mathrm{Nuc}_m(S_f) \subseteq S$ is similar to (i), but we look at $(t^{m-j} \circ p) \circ c = t^{m-j} \circ (p \circ c)$ and $(t^{m-j} \circ p) \circ t = t^{m-j} \circ (p \circ t)$ instead.
\end{itemize}
\end{proof}

When $S$ is a domain and $\delta$ is not necessarily $0$, we can prove a similar result to Theorem \ref{thm:generalised left, middle nucleus=S} for polynomials of degree $2$:

\begin{proposition} \label{prop:t^2 - a_1 t - a_0 left middle nucleus generalised}
Let $S$ be a domain and $f(t) = t^2 - a_1 t - a_0 \in S[t;\sigma,\delta]$ be such that one of the following holds:
\begin{itemize}
\item[(i)] $a_1 \in \mathrm{Fix}(\sigma)$ and $a_0 \notin \mathrm{Const}(\delta)$.
\item[(ii)] $a_1 \in \mathrm{Fix}(\sigma)$, $a_0 \in \mathrm{Fix}(\sigma)$ and $a_1 \notin \mathrm{Const}(\delta)$.
\item[(iii)] $a_1 \in \mathrm{Fix}(\sigma) \cap \mathrm{Const}(\delta)$ and $a_0 \notin \mathrm{Fix}(\sigma)$.
\item[(iv)] $a_1 \notin \mathrm{Fix}(\sigma)$ and $a_0 \in \mathrm{Const}(\delta)$.
\item[(v)] $a_1 \notin \mathrm{Fix}(\sigma)$, $a_0 \in \mathrm{Fix}(\sigma)$ and $a_1 \in \mathrm{Const}(\delta)$.
\end{itemize}
Then $\mathrm{Nuc}_l(S_f) = \mathrm{Nuc}_m(S_f) = S$.
\end{proposition}

\begin{proof}
Recall $S \subseteq \mathrm{Nuc}_l(S_f)$ by Theorem \ref{thm:nucleus of S_f generalised}. We prove the reverse inclusion: Suppose $p = p_0 + p_1 t \in \mathrm{Nuc}_l(S_f)$ for some $p_0, p_1 \in S$, then
\begin{align*}
(p \circ t) \circ t &= ( p_0 t + p_1(a_1 t + a_0)) \circ t \\
&= p_0(a_1 t + a_0) + p_1 a_1(a_1 t + a_0) + p_1 a_0 t \\
&= p_0 a_0 + p_1 a_1 a_0 + \big( p_0 a_1 + p_1 a_1^2 + p_1 a_0 \big) t,
\end{align*}
and
\begin{align*}
p &\circ (t \circ t) = (p_0 + p_1 t) \circ (a_1 t + a_0) \\
&= p_0 a_1 t + p_0 a_0 + p_1(\sigma(a_0)t + \delta(a_0)) + p_1(\sigma(a_1)t + \delta(a_1)) \circ t \\
&= p_0 a_1 t + p_0 a_0 + p_1 \sigma(a_0)t + p_1 \delta(a_0) + p_1 \sigma(a_1)(a_1t + a_0) + p_1 \delta(a_1) t \\
&= p_0 a_0 + p_1 \sigma(a_1)a_0 + p_1 \delta(a_0) + \big( p_0 a_1 + p_1 \sigma(a_1)a_1 + p_1 \delta(a_1) + p_1 \sigma(a_0) \big) t,
\end{align*}
must be equal. Comparing the coefficients of $t$ yields
\begin{equation} \label{eqn:generalised left Nucleus 3 delta not 0}
p_1 a_1 a_0 = p_1 \sigma(a_1) a_0 + p_1 \delta(a_0),
\end{equation}
and
\begin{equation} \label{eqn:generalised left Nucleus 4 delta not 0}
p_1 a_1^2 + p_1 a_0 = p_1 \sigma(a_1) a_1 + p_1 \sigma(a_0) + p_1 \delta(a_1).
\end{equation}
It is a straightforward exercise to check using \eqref{eqn:generalised left Nucleus 3 delta not 0} and \eqref{eqn:generalised left Nucleus 4 delta not 0} that in each of our five cases we must have $p_1 = 0$ so $p = p_0 \in S$. 

We prove $S \subseteq \mathrm{Nuc}_m(S_f)$ similarly but consider instead $(t \circ p) \circ t = t \circ (p \circ t)$.
\end{proof}

\chapter{Solvable Crossed Product Algebras and Applications to G-Admissible Groups} \label{chapter:G-Admissible Groups and Crossed Products}

In this Chapter, we show that for every finite-dimensional central simple algebra $A$ over a field $F$, which contains a maximal subfield $M$ with non-trivial $G = \mathrm{Aut}_F(M)$, then $G$ is solvable if and only if $A$ contains a finite chain of subalgebras, which are generalised cyclic algebras over their centers, satisfying certain conditions. This chain of subalgebras is closely related to a normal series of $G$ which exists when $G$ is solvable. In particular, we obtain that a crossed product algebra is solvable if and only if it has such a chain.

Recall from Section \ref{section:Generalised Cyclic Algebras}, that when $D$ is a finite-dimensional central division algebra over $C$, $\sigma \vert_C$ has finite order $m$ and $f(t) = t^m-a \in D[t;\sigma]$, $d \in \mathrm{Fix}(\sigma)^{\times}$ is right invariant, the associative quotient algebra $D[t;\sigma]/D[t;\sigma]f$ is also called a generalised cyclic algebra and denoted $(D,\sigma,a)$. For this Chapter, we extend the definition of a generalised cyclic algebra, to where we do not require $D$ be a division algebra:

\begin{definition}
Let $S$ be a finite dimensional central simple algebra of degree $n$ over $C$ and $\sigma \in \mathrm{Aut}(S)$ be such that $\sigma \vert_C$ has finite order $m$. We define the \textbf{generalised cyclic algebra} $(S,\sigma,a)$ to be the associative algebra of the form $S_f = S[t;\sigma]/S[t;\sigma]f$ where $f(t) = t^m - a \in S[t;\sigma]$, $a \in \mathrm{Fix}(\sigma)^{\times}$ is right invariant. $(S,\sigma,a)$ is a central simple algebra over $F = \mathrm{Fix}(\sigma) \cap C$ of degree $mn$ \cite[p.~4]{brown2017solvable}.
\end{definition}

\begin{definition}
Let $F$ be a field and $A$ be a central simple algebra over $F$ of degree $n$. $A$ is called a \textbf{$G$-crossed product algebra} or \textbf{crossed product algebra} if it contains a field extension $M/F$ which is Galois of degree $n$ with Galois group $G = \mathrm{Gal}(M/F)$.

Equivalently we can define a ($G$)-crossed product algebra $(M,G,\mathfrak{a})$ over $F$ via factor sets starting with a finite Galois field extension as follows: Suppose $M/F$ is a finite Galois field extension of degree $n$ with Galois group $G$ and $\{ a_{\sigma,\tau} \ \vert \ \sigma, \tau \in G \}$ is a set of elements of $M^{\times}$ such that 
\begin{equation} \label{eqn:Crossed product criteria 1}
a_{\sigma,\tau} a_{\sigma \tau, \rho} = a_{\sigma,\tau \rho} \sigma(a_{\tau,\rho}),
\end{equation}
for all $\sigma, \tau, \rho \in G$. Then a map $\mathfrak{a}: G \times G \rightarrow M^{\times}, \ (\sigma,\tau) \mapsto a_{\sigma,\tau}$, is called a \textbf{factor set} or \textbf{2-cocycle} of $G$.

An associative multiplication is defined on the $F$-vector space $\bigoplus_{\sigma \in G} M x_{\sigma}$ by
\begin{equation} \label{eqn:Crossed product criteria 2}
x_{\sigma} m = \sigma(m) x_{\sigma},
\end{equation}
\begin{equation} \label{eqn:Crossed product criteria 3}
x_{\sigma} x_{\tau} = a_{\sigma, \tau} x_{\sigma \tau},
\end{equation}
for all $m \in M$, $\sigma, \tau \in G$. This way $\bigoplus_{\sigma \in G} M x_{\sigma}$ becomes an associative central simple $F$-algebra that contains a maximal subfield isomorphic to $M$. This algebra is denoted $(M,G,\mathfrak{a})$ and is called a $G$-crossed product algebra over $F$. If $G$ is solvable then $A$ is also called a \textbf{solvable $G$-crossed product algebra} over $F$.
\end{definition}

\section{Crossed Product Subalgebras of Central Simple Algebras}

Let $M/F$ be a field extension of degree $n$ and $G = \mathrm{Aut}_{F}(M)$ be the group of automorphisms of $M$ which fix the elements of $F$. Let $A$ be a central simple algebra of degree $n$ over $F$ and suppose $M$ is contained in $A$, this makes $M$ a maximal subfield of $A$ \cite[Lemma 15.1]{berhuy2013introduction}. Such maximal subfields do not always exist for a general central simple algebra \cite[Remark 15.4]{berhuy2013introduction}, however they always exists for example when $A$ is a division algebra \cite[Corollary 15.6]{berhuy2013introduction}.
We denote by $A^{\times}$ the set of invertible elements of $A$, and for a subset $B$ in $ A$, denote $\mathrm{Cent}_{A}(B)$ the centralizer of $B$ in $A$.

Some of the results in this Chapter, namely Lemma \ref{lem:Petit (26)}, Corollary \ref{cor:Petit (28)} and Theorems \ref{thm:Petit (27)}, \ref{thm:Petit (29)}, are stated for central division algebras $A$ over $F$ by Petit in \cite[\S 7]{Petit1966-1967}, and none of them are proved there. The following  generalises \cite[(27)]{Petit1966-1967} to central simple algebras with a maximal subfield $M$ as above. The result was before  only stated for division algebras and also not in terms of crossed product algebras:


\begin{theorem} \label{thm:Petit (27)}
\begin{itemize}
\item[(i)] $A$ contains a subalgebra $M(G)$ which is a crossed product algebra $(M,G,\mathfrak{a})$ of degree $|G|$ over $\mathrm{Fix}(G)$ with maximal subfield $M$.
\item[(ii)] $A$ is equal to $M(G)$ if and only if $M$ is a Galois extension of $F$. In this case $A$ is a $G$-crossed product algebra over $F$.
\item[(iii)] For any subgroup $H$ of $G$, there is an $F$-subalgebra $M(H)$ of both $M(G)$ and $A$, which is a $H$-crossed product algebra of degree $|H|$ over $\mathrm{Fix}(H)$ with maximal subfield $M$.
\end{itemize}
\end{theorem}

\begin{proof}
\begin{itemize}
\item[(i)] Define $M(G) = \mathrm{Cent}_{A}(\mathrm{Fix}(G))$, then $M(G)$ is a central simple algebra over $\mathrm{Fix}(G)$ by the Centralizer Theorem for central simple algebras \cite[Theorem III.5.1]{berhuy2013introduction}. Furthermore, since $M$ is a maximal subfield of $M(G)$ and $M/\mathrm{Fix}(G)$ is a Galois field extension with Galois group $G$, we conclude $M(G)$ is a $G$-crossed product algebra.
\item[(ii)] Notice $[M:F] = n$, $A$ has dimension $n^2$ over $F$, and $M(G)$ has a basis $\{ x_{\sigma} | \ \sigma \in G \}$ as a vector space over $M$. If $M$ is not a Galois extension of $F$, then $\vert G \vert < n$ and thus $\{ x_{\sigma} \ \vert \ \sigma \in G \}$ cannot be a  set of generators for $A$ as a vector space over $M$. Conversely, if $M/F$ is a Galois extension, then $\vert G \vert = n$ and since $\{ x_{\sigma} \ \vert \ \sigma \in G \}$ is linearly independent over $M$, counting dimensions yields $M(G) = A$. The rest of the assertion is trivial.
\item[(iii)] For any subgroup $H$ of  $G$, let $M(H) = \mathrm{Cent}_A(\mathrm{Fix}(H))$. Since $\mathrm{Fix}(G) \subseteq \mathrm{Fix}(H)$, we have $M(H) = \mathrm{Cent}_A(\mathrm{Fix}(H))$ is contained in $M(G) = \mathrm{Cent}_A(\mathrm{Fix}(G))$ The proof now follows exactly as in (i).
\end{itemize}
\end{proof}

\begin{remark}
When $A$ has prime degree over $F$ and $M/F$ is not a Galois extension then $M(G) = M$: $M(G)$ is a simple subalgebra of $A$ by Theorem \ref{thm:Petit (27)}(i), therefore $\mathrm{dim}_F(M(G))$ divides $\mathrm{dim}_F(A) = n^2$ by the Centralizer Theorem for central simple algebras \cite[Theorem III.5.1]{berhuy2013introduction}. Now $M(G)$ contains $M$ and $M(G)$ is equal to $A$ if and only if $M/F$ is a Galois extension by Theorem \ref{thm:Petit (27)}(ii). As $n$ is prime and $M/F$ is not Galois, this means $M(G) = M$.
\end{remark}

A close look at the proof of Theorem \ref{thm:Petit (27)} yields the following observations:
\begin{lemma} \label{lem:Centralizer if Fix(H) in A}
\begin{itemize}
\item[(i)] Given any subgroup $H$ of $G$, $M(H)$ is a $H$-crossed product algebra over its center with $M(H) = (M,H,\mathfrak{a}_H)$, where $\mathfrak{a}_H$ denotes the factor set of the $G$-crossed product algebra $M(G) = (M,G,\mathfrak{a})$ restricted to the elements of $H$.
\item[(ii)] For any subgroup $H$ of $G$,
$M(H)$ is the centralizer of $\mathrm{Fix}(H)$ in $A$.
\end{itemize}
\end{lemma}

The following generalises \cite[(26)]{Petit1966-1967} to any central simple algebra $A$ with a maximal subfield $M$ as above. Again it was previously only stated and not proved for central division algebras:

\begin{lemma} \label{lem:Petit (26)}
\begin{itemize}
\item[(i)] For any $\sigma \in G$ there exists $x_{\sigma} \in A^{\times}$ such that the inner automorphism $$I_{x_{\sigma}}: A \rightarrow A, \ y \mapsto x_{\sigma} y x_{\sigma}^{-1}$$ restricted to $M$ is $\sigma$.
\item[(ii)] Given any $\sigma \in G$, we have $\{ x \in A^{\times} \ \vert \ I_{x} \vert_M = \sigma \} = M^{\times} x_{\sigma}.$
\item[(iii)] The set of cosets $\{ M^{\times} x_{\sigma} \ \vert \ \sigma \in G\}$ with multiplication given by
\begin{equation*} \label{eqn:M^X x_sigma M^times x_tau = M^times x_sigma tau}
M^{\times} x_{\sigma} M^{\times} x_{\tau} = M^{\times} x_{\sigma \tau},
\end{equation*}
is a group isomorphic to $G$, where $\sigma$ and $M^{\times} x_{\sigma}$ correspond under this isomorphism.
\end{itemize}
\end{lemma}

\begin{proof}
$A$ contains the $G$-crossed product algebra $M(G)$ by Theorem \ref{thm:Petit (27)}, thus (i) and (iii) follows from \eqref{eqn:Crossed product criteria 2} and \eqref{eqn:Crossed product criteria 3}.
For (ii) we have
\begin{align*}
I_{m x_{\sigma}}(y) &= (m x_{\sigma}) y (m x_{\sigma})^{-1} = (m x_{\sigma}) y (x_{\sigma}^{-1} m^{-1}) = m \sigma(y) m^{-1} = \sigma(y)
\end{align*}
for all $m, y \in M^{\times}$, and thus $M^{\times} x_{\sigma} \subset \{ x \in A^{\times} \ \vert \ I_{x} \vert_M = \sigma \}$.

Suppose $u \in \{ x \in A^{\times} \ \vert \ I_{x} \vert_M = \sigma \}$. Then as $u$ and $x_{\sigma}$ are invertible, we can write $u = v x_{\sigma}$ for some $v \in A^{\times}$. We are left to prove that $v \in M^{\times}$. We have
$$\sigma(y) = I_{u}(y) = (v x_{\sigma}) y (v x_{\sigma})^{-1} = v x_{\sigma} y x_{\sigma}^{-1} v^{-1} = v \sigma(y) v^{-1},$$ 
for all $y \in M$, and so $\sigma(y) v = v \sigma(y)$ for all $y \in M$, that is $m v = v m$ for all $m \in M$ since $\sigma$ is bijective. Therefore $v$ is contained in the centralizer of $M$ in $A$, which is equal to $M$ because $M$ is a maximal subfield of $A$.

\end{proof}

The following generalises \cite[(28)]{Petit1966-1967}:

\begin{corollary} \label{cor:Petit (28)}
If $H$ is a cyclic subgroup of $G$ of order $h > 1$ generated by $\sigma$, then there exists $c \in \mathrm{Fix}(\sigma)^{\times}$ such that
$$M(H) \cong (M/\mathrm{Fix}(\sigma), \sigma,c) = M[t;\sigma]/M[t;\sigma](t^h - c),$$
is a cyclic algebra of degree $h$ over $\mathrm{Fix}(\sigma)$.
\end{corollary}

\begin{proof}
$M(H)$ is a $H$-crossed product algebra of degree $h$ over $\mathrm{Fix}(\sigma)$ by Theorem \ref{thm:Petit (27)}. Moreover $H$ is a cyclic group and so $M(H)$ is a cyclic algebra of degree $h$ over $\mathrm{Fix}(\sigma)$ (see for example \cite[p.~49]{saltman1999lectures}). This means there exists $c \in \mathrm{Fix}(\sigma)^{\times}$ such that $M(H) \cong (M/\mathrm{Fix}(\sigma), \sigma,c)$.
\end{proof}

In particular, we conclude that if a central division algebra $A$ over $F$ contains a maximal subfield $M$ and non-trivial $\sigma \in \mathrm{Aut}_F(M)$ of order $h$, then it contains a cyclic division algebra of degree $h$, (though not necessarily with center $F$). This is the case even if $A$ is a noncrossed product (i.e. if $A$ is not a crossed product algebra):

\begin{theorem} \label{thm:central division algebra contains cyclic algebra}
(\cite[Theorem 4]{brown2017solvable}).
Let $A$ be a central division algebra of degree $n$ over $F$ with maximal subfield $M$ and non-trivial $\sigma \in \mathrm{Aut}_F(M)$ of order $h$. Then $A$ contains a cyclic division algebra $(M/\mathrm{Fix}(\sigma), \sigma,c)$ of degree $h$ over $\mathrm{Fix}(\sigma)$ as a subalgebra.
\end{theorem}

\begin{proof}
This follows immediately from Corollary \ref{cor:Petit (28)}.
\end{proof}

It is well-known that a central division algebra of prime degree over $F$ is a cyclic algebra if and only if it contains a cyclic subalgebra of prime degree (though not necessarily with center $F$) \cite[p.~2]{motiee2016note}. Together with Theorem \ref{thm:central division algebra contains cyclic algebra} this yields the following:

\begin{corollary} \label{cor:central division algebras of prime degree}
(\cite[Corollary 6]{brown2017solvable}).
Let $A$ be a central division algebra over $F$ of prime degree $p$. Then either $A$ is a cyclic algebra or each of its maximal subfields $M$ has trivial automorphism group $\mathrm{Aut}_F(M)$.
\end{corollary}

\begin{proof}
Suppose $G = \mathrm{Aut}_F(M)$ is non-trivial, then there exists a non-trivial $\sigma \in \mathrm{Aut}_F(M)$ of finite order $h$. Thus $A$ contains a cyclic division algebra of degree $h$ over $\mathrm{Fix}(\sigma)$ as a subalgebra. Looking at the possible intermediate field extensions of $M/F$ yields $[M:\mathrm{Fix}(\sigma)]$ is either $1$ or $p$. Since $G$ is not trivial, $[M:\mathrm{Fix}(\sigma)] = h = p$, thus $A$ contains a cyclic subalgebra of prime degree and so is itself a cyclic algebra.
\end{proof}

\section[CSA's Containing Max Subfield with Solvable Automorphism Group]{Central Simple Algebras Containing a Maximal Subfield \texorpdfstring{$M$}{M} with Solvable \texorpdfstring{$F$}{F}-Automorphism Group}

Suppose $G$ is a finite solvable group, then there exists a chain of subgroups
\begin{equation} \label{eqn:Subnormal series}
\{ 1 \} = G_0 < G_1 < \ldots < G_k = G,
\end{equation}
such that $G_{j}$ is normal in $G_{j+1}$ and $G_{j+1}/G_{j}$ is cyclic of prime order $q_{j}$ for all $j \in \{ 0, \ldots, k-1 \}$, i.e. 
\begin{equation}
G_{j+1}/G_{j} = \{ G_{j}, G_{j} \sigma_{j+1}, \ldots \},
\end{equation}
for some $\sigma_{j+1} \in G_{j+1}$. Theorem \ref{thm:Petit (27)}, Corollary \ref{cor:Petit (28)} and Lemma \ref{lem:Petit (26)} lead us to the following generalisation of \cite[(29)]{Petit1966-1967}. As before, the result was previously only stated (and not proved) for central division algebras over $F$, and also without the connection to crossed product algebras:

\begin{theorem} \label{thm:Petit (29)}
Let $M/F$ be a field extension of degree $n$ with non-trivial solvable $G = \mathrm{Aut}_{F}(M)$, and $A$ be a central simple algebra of degree $n$ over $F$ with maximal subfield $M$. Then there exists a chain of subalgebras
\begin{equation} \label{eqn:Petit (29) chain of subalgebras}
M = A_0 \subset A_1 \subset \ldots \subset A_k = M(G) \subseteq A,
\end{equation}
of $A$ which are $G_i$-crossed product algebras over $Z_i = \mathrm{Fix}(G_i)$, and where
\begin{equation} \label{eqn:Petit (29) chain of subalgebras 2}
A_{i+1} \cong A_i[t_i;\tau_i]/A_i[t_i;\tau_i](t_i^{q_i} - c_i),
\end{equation}
for all $i \in \{ 0, \ldots, k-1 \}$, such that
\begin{itemize}
\item[(i)] $q_i$ is the prime order of the factor group $G_{i+1}/G_i$ in the chain \eqref{eqn:Subnormal series},
\item[(ii)] $\tau_i$ is an $F$-automorphism of $A_i$ of inner order $q_i$ which restricts to $\sigma_{i+1} \in G_{i+1}$ which generates $G_{i+1}/G_i$, and
\item[(iii)] $c_i \in \mathrm{Fix}(\tau_i)$ is invertible.
\end{itemize}
\end{theorem}

Note that the inclusion $M(G) \subseteq A$ in \eqref{eqn:Petit (29) chain of subalgebras} is an equality if and only if $M/F$ is a  Galois extension by Theorem \ref{thm:Petit (27)}. In this case $A$ is a solvable $G$-crossed product algebra.

\begin{proof}
Define $A_i= M(G_i)$ for all $i \in \{ 1, \ldots, k \}$. $A_i$ is a $G_i$-crossed product algebra over $\mathrm{Fix}(G_i)$ by Theorem \ref{thm:Petit (27)}.

$G_1/G_0 \cong G_1$ is a cyclic subgroup of $G$ of prime order $q_0$ generated by some $\sigma_1 \in G$. Let $\tau_0 = \sigma_1$, then there exists $c_0 \in \mathrm{Fix}(\tau_0)^{\times}$ such that $A_1 = M(G_1)$ is $F$-isomorphic to
$$M[t_0;\tau_0]/M[t_0;\tau_0](t_0^{q_0} - c_0),$$
by Corollary \ref{cor:Petit (28)}, which is a cyclic algebra of prime degree $q_0$ over $\mathrm{Fix}(\tau_0)$.

Now $G_1 \triangleleft G_2$ and $G_2/G_1$ is cyclic of prime order $q_1$ with
\begin{equation} \label{eqn:proof petit (29) 1}
G_2/G_1 = \{ (G_1 \sigma_2)^i \ \vert \ i \in \mathbb{Z} \} = \{ G_1,
G_1 \sigma_2, \ldots, G_1 \sigma_2^{q_1-1} \},
\end{equation}
for some $\sigma_2 \in G_2$. Hence we can write $G_2 = \{ h \sigma_2^i \ \vert \ h \in G_1, 0 \leq i \leq q_1-1 \}$ and thus the crossed product algebra $A_2 = M(G_2)$ has a basis $$\{ x_{h \sigma_2^j} \ \vert \ h \in G_1, \ 0\leq j \leq q_1-1 \},$$
as an $M$-vector space. Recall $M^{\times} x_{h \sigma_2^j} = M^{\times} x_{h} x_{\sigma_2^j} = M^{\times} x_{h} x_{\sigma_2}^j$ for all $h \in G_1$ by Lemma \ref{lem:Petit (26)}, and $\{ 1, x_{\sigma_2}, \ldots, x_{\sigma_2}^{q_1-1} \}$ is a basis for $A_2$ as a left $A_1$-module, i.e.
\begin{equation} \label{eqn:proof petit (29) 2}
A_2 = A_1 + A_1 x_{\sigma_2} + \ldots + A_1 x_{\sigma_2}^{q_1-1}.
\end{equation}
We have $G_2 G_1 = G_1 G_2$ as $G_1$ is normal in $G_2$ and so for every $h \in G_1$, we get $\sigma_2 h = h' \sigma_2$ for some $h' \in G_1$. Choose the basis $\{ x_{h} \ \vert \ h \in G_1 \}$ of $A_1$ as a vector space over $M$. By
\eqref{eqn:Crossed product criteria 3} we obtain
\begin{equation} \label{eqn:proof petit (29) 3}
x_{\sigma_2} x_{h} = a_{\sigma_2,h} x_{\sigma_2 h} =
a_{\sigma_2,h} x_{h' \sigma_2} = a_{\sigma_2,h}
(a_{h',\sigma_2})^{-1} x_{h'} x_{\sigma_2}.
\end{equation}
Recall $x_{\sigma_2}\in A^\times$  by Lemma \ref{lem:Petit (26)}. The inner automorphism
 $$\tau_1: A \rightarrow A,\ z \mapsto x_{\sigma_2} z x_{\sigma_2}^{-1}$$
restricts to $\sigma_2$ on $M$. Moreover,
\begin{equation} \label{eqn:proof petit (29) 4}
\begin{split}
\tau_1(x_{h}) &= x_{\sigma_2} x_{h} x_{\sigma_2}^{-1} =
 a_{\sigma_2,h} (a_{h',\sigma_2})^{-1} x_{h'} x_{\sigma_2} x_{\sigma_2}^{-1} \\
 &= a_{\sigma_2,h} (a_{h',\sigma_2})^{-1} x_{h'} \in A_1,
\end{split}
\end{equation}
for all $h \in G_1$, i.e. $\tau_1 \vert_{A_1}(y) \in A_1$ for all $y \in A_1$ and so $\tau_1 \vert_{A_1}$ is an $F$-automorphism of $A_1$. Furthermore, $x_{\sigma_2} x_{h} = \tau_1 \vert_{A_1}(x_{h}) x_{\sigma_2},$ for all $h \in G_1$ by \eqref{eqn:proof petit (29) 3}, \eqref{eqn:proof petit (29) 4}, and $$x_{\sigma_2} m = \sigma_2(m) x_{\sigma_2} = \tau_1 \vert_{A_1}(m) x_{\sigma_2},$$ for all $m \in M$.
We conclude that
\begin{equation} \label{eqn:proof petit (29) 6}
x_{\sigma_2} y = \tau_1 \vert_{A_1}(y) x_{\sigma_2}
\end{equation}
for all $y \in A_1$ by distributivity. Define $c_1 = x_{\sigma_2}^{q_1}$, then $\sigma_2^{q_1} \in G_1$ by \eqref{eqn:proof petit (29) 1} which implies $c_1 \in A_1$.  Furthermore $c_1$ is invertible since $x_{\sigma_2}$ is invertible. Also, $\tau_1 \vert_{A_1}(c_1) = x_{\sigma_2} x_{\sigma_2}^{q_1} x_{\sigma_2}^{-1} = c_1$ which means $c_1 \in \mathrm{Fix}(\tau_1 \vert_{A_1})^\times$. Since $$x_{\sigma_2^{-q_1}} x_{\sigma_2^{q_1}} = a_{\sigma_2^{-q_1},\sigma_2^{q_1}} x_{\mathrm{id}} \in M^{\times},$$ it follows that $c_1^{-1} = x_{\sigma_2^{q_1}}^{-1} \in M^{\times}x_{\sigma_2^{-q_1}} \in A_1$ as $\sigma_2^{-q_1} \in G_1$. Hence $\tau_1 \vert_{A_1}$ has inner order $q_1$, as indeed
$(\tau_1 \vert_{A_1})^{q_1}: A_1 \rightarrow A_1, z \mapsto c_1 z c_1^{-1},$
is an inner automorphism.

Consider the  algebra $$B_2 = A_1[t_1;\tau_1 \vert_{A_1}] /A_1[t_1;\tau_1 \vert_{A_1}](t_1^{q_1} - c_1)$$ with center
\begin{align*}
\mathrm{Cent}(B_2) & \supset \{ b \in A_1 \ \vert \ bh = hb \text{ for
all } h \in B_2 \} = \mathrm{Cent}(A_1) \cap \mathrm{Fix}(\tau_1) \supset
F.
\end{align*}
Now, using \eqref{eqn:proof petit (29) 2} and \eqref{eqn:proof petit (29) 6}, the map
$$\phi: A_2 \rightarrow B_2, \ y x_{\sigma_2}^i \mapsto y t_1^i \qquad \qquad (y \in A_1),$$
can readily be seen to be an isomorphism between $F$-algebras. Indeed, it is clearly bijective and $F$-linear. In addition, we have
\begin{align*}
\phi \big( (y x_{\sigma_2}^i)(z x_{\sigma_2}^j) \big) &= \phi \big( y \tau_1 \vert_{A_1}^i(z) x_{\sigma_2}^i x_{\sigma_2}^j \big) \\
&= \begin{cases}
\phi \big( y \tau_1 \vert_{A_1}^i(z) x_{\sigma_2}^{i+j} \big) & \text{ if } i+j < q_1, \\
\phi \big( y \tau_1 \vert_{A_1}^i(z) x_{\sigma_2}^{q_1} x_{\sigma_2}^{i+j-q_1} \big) & \text{ if } i+j \geq q_1,
\end{cases} \\
&= \begin{cases}
y \tau_1 \vert_{A_1}^i(z) t^{i+j} & \text{ if } i+j < q_1, \\
y \tau_1 \vert_{A_1}^i(z) c_1 t^{i+j-q_1} & \text{ if } i+j \geq q_1,
\end{cases}
\end{align*}
by \eqref{eqn:proof petit (29) 6}, and
\begin{align*}
\phi(y x_{\sigma_2}^i) \circ \phi(z x_{\sigma_2}^j) &= (yt^i) \circ (zt^j) = \begin{cases}
y \tau_1 \vert_{A_1}^i(z) t^{i+j} & \text{if } i+j < q_1, \\
y \tau_1 \vert_{A_1}^i(z) t^{i+j-q_1} c_1 & \text{if } i+j \geq q_1,
\end{cases} \\
&= \begin{cases}
y \tau_1 \vert_{A_1}^i(z) t^{i+j} & \text{ if } i+j < q_1, \\
y \tau_1 \vert_{A_1}^i(z) c_1 t^{i+j-q_1} & \text{ if } i+j \geq q_1,
\end{cases}
\end{align*}
for all $y, z \in A_1$, $i, j \in \{ 0, \ldots, q_1-1 \}$. By distributivity we conclude $\phi$ is also multiplicative, thus an isomorphism between $F$-algebras. Continuing in this manner for $G_2 \triangleleft G_3$ etc. yields the assertion.
\end{proof}

\begin{remarks}
\begin{itemize}
\item[(i)] If $A$ is a division algebra, the algebras $A_i$ in Theorem \ref{thm:Petit (29)} are also division algebras, being subalgebras of the finite-dimensional algebra $A$.
\item[(iii)] The algebras $A_i$ in Theorem \ref{thm:Petit (29)} are associative being subalgebras of the associative algebra $A$. Therefore $t_i^{q_i}-c_i \in A_i[t_i;\tau_i]$ are right invariant for all $i \in \{ 0, \ldots, k-1 \}$ by \cite[Theorem 4]{pumplun2015finite}.
\end{itemize}
\end{remarks}

%

%
%

We obtain the following straightforward observations about Theorem \ref{thm:Petit (29)}:

\begin{corollary} \label{cor:Observations on Theorem Petit(29)}
Let $M/F$ be a field extension of degree $n$ with non-trivial solvable $G = \mathrm{Aut}_{F}(M)$, and $A$ be a central simple algebra of degree $n$ over $F$ with maximal subfield $M$. Consider the algebras $A_i = M(G_i)$ as in Theorem \ref{thm:Petit (29)}.
\begin{itemize}
\item[(i)] $A_i = \mathrm{Cent}_{A}(\mathrm{Fix}(G_i))$ for all $i \in \{ 0, \ldots, k-1 \}$.
\item[(ii)] $A_i$ is a crossed product algebra over $Z_i = \mathrm{Fix}(G_i)$ of degree $|G_i| = \prod_{l=0}^{i-1} q_l$ for all $i \in \{ 1, \ldots, k-1 \}$.
\item[(iii)] $M = Z_0 \supset \ldots \supset Z_{k-1} \supset Z_k \supset F$.
\item[(iv)] $Z_{i-1} / Z_i$ has prime degree $q_{i-1}$ for all $i \in \{ 1, \ldots, k \}$.
\item[(v)] $M/Z_i$ is a Galois field extension and $M$ is a maximal subfield of $A_i$ for all $i \in \{ 0, \ldots, k \}$.
\item[(vi)] \cite[Corollary 10]{brown2017solvable} $A_i$ is a generalised cyclic algebra over $Z_i$ for all $i \in \{ 0, \ldots, k \}$.
\item[(vii)] $A$ contains the cyclic algebra $$(M/\mathrm{Fix}(\sigma_1),\sigma_1,c_0) \cong M[t_0;\sigma_1]/M[t_0;\sigma_1](t_0^{q_0} - c_0),$$ of prime degree $q_0$ over $\mathrm{Fix}(\sigma_1)$.
\end{itemize}
\end{corollary}

\begin{proof}
\begin{itemize}
\item[(i)] This is Lemma \ref{lem:Centralizer if Fix(H) in A}.
\item[(ii)] $A_i$ has degree $[M:\mathrm{Fix}(G_i)]$ which is equal to $|G_i|$ by the Fundamental Theorem of Galois Theory.
\item[(iii)] Follows from the fact that $\{ 1 \} = G_0 \leq G_1 \leq \ldots \leq G_k = G$ and $Z_i = \mathrm{Fix}(G_i)$.
\item[(iv)] We have $n = [M:F] = [M:Z_i][Z_i:F] = |G_i| [Z_i:F]$ for all $i$ by the Fundamental Theorem of Galois Theory, therefore
$$[Z_{i-1}:Z_i] = \frac{[Z_{i-1}:F]}{[Z_i:F]} = \frac{|G_i|}{|G_{i-1}|} = q_{i-1},$$
for all $i \in \{1,\ldots, k\}$ as required.
\item[(v)] $M/\mathrm{Fix}(G_i)$ is a Galois field extension with Galois group $G_i$ by Galois Theory. The rest of the assertion is trivial by Theorem \ref{thm:Petit (29)}.
\item[(vii)] $G_1 = \langle \sigma_1 \rangle$ is a cyclic subgroup of $G$ of order $q_0$, therefore the result follows by Corollary \ref{cor:Petit (28)}.
\end{itemize}
\end{proof}

\begin{corollary} \label{cor:non-central elements}
(\cite[Corollaries 9, 11]{brown2017solvable}).
Let $A$ be a central division algebra over $F$ containing a maximal subfield $M$ with non-trivial solvable $G = \mathrm{Aut}_F(M)$.
\begin{itemize} 
\item[(i)] There is a non-central element $t_0 \in A$ such that $t_0^{q_0} \in \mathrm{Fix}(\sigma_1)^{\times}$ and $t_0^m \notin \mathrm{Fix}(\sigma_1)$ for all $m \in \{ 1, \ldots, q_0-1 \}$.
\item[(ii)] $A$ contains a chain of generalised cyclic division algebras $A_i$ over intermediate fields $Z_i = \mathrm{Fix}(G_i)$ of $M/F$ as in \eqref{eqn:Petit (29) chain of subalgebras}.
\end{itemize}
\end{corollary}

Our next result generalises \cite[(9)]{petit1968quasi} and characterises all the algebras with a maximal subfield $M/F$ that have a non-trivial solvable automorphism group $G = \mathrm{Aut}_{F}(M)$:

\begin{theorem} \label{thm:Petit (29) Generalised}
Let $M/F$ be a field extension of degree $n$ with non-trivial $G = \mathrm{Aut}_{F}(M)$, and $A$ be a central simple algebra of degree $n$ over $F$ containing $M$. Then $G$ is solvable if there exists a chain of  subalgebras
\begin{equation} \label{eqn:Petit (29) chain of subalgebras Galois}
M = A_0 \subset A_1 \subset \ldots \subset A_k  \subseteq A
\end{equation}
of $A$ which all have maximal subfield $M$,  where $A_k$ is a $G$-crossed product algebra over $\mathrm{Fix}(G)$, and where
\begin{equation} \label{eqn:Petit (29) chain of subalgebras Galois 2}
A_{i+1} \cong A_i[t_i;\tau_i]/A_i[t_i;\tau_i](t_i^{q_i} - c_i),
\end{equation}
for all $i \in \{ 0, \ldots, k-1 \}$,  with
\begin{itemize}
\item[(i)] $q_i$ a prime,
\item[(ii)] $\tau_i$ an $F$-automorphism of $A_i$ of inner order $q_i$ which restricts to an automorphism $\sigma_{i+1} \in G$, and
\item[(iii)] $c_i \in \mathrm{Fix}(\tau_i)^\times$.
\end{itemize}
\end{theorem}

\begin{proof}
Suppose there exists a  chain of algebras $A_i$, $i \in \{ 0, \ldots, k \}$ satisfying the above assumptions. Put  $G_k = G$. Each $A_i$ has center $Z_i = Z_{i-1} \cap \mathrm{Fix}(\tau_{i-1})$ by Corollary \ref{cor:Comm(S_f) = F generalised}, so that by induction
$$Z_i = \mathrm{Fix}(\tau_0) \cap \mathrm{Fix}(\tau_1) \cap \dots \cap \mathrm{Fix}(\tau_{i-1}) \supset F,$$
in particular $Z_i \subseteq M$.

$M/Z_i$ is a Galois extension contained in $A_i$: Let $H_i$ be the subgroup of $G$ generated by $\sigma_1, \ldots, \sigma_i$, $i \in \{ 1, \ldots, k \}$. Then
\begin{align*}
Z_i &= \mathrm{Fix}(\tau_0) \cap \mathrm{Fix}(\tau_1) \cap \dots \cap \mathrm{Fix}(\tau_{i-1}) \\ &= M \cap \mathrm{Fix}(\tau_0) \cap \mathrm{Fix}(\tau_1) \cap \dots \cap \mathrm{Fix}(\tau_{i-1}) \\ &= \mathrm{Fix}(\sigma_1) \cap \dots \cap \mathrm{Fix}(\sigma_{i}) = \mathrm{Fix}(H_i),
\end{align*}
so $M/Z_i$ is a Galois field extension by Galois theory.
Put $G_i = \mathrm{Gal}(M/Z_i)$, then each $A_i$ is a $G_i$-crossed product algebra. In particular, $G_i$ is a subgroup of $G_{i+1}$.

We use induction to prove that  each $G_i$, thus $G$, is a solvable group. For $i = 1$,
$$A_1 \cong M[t_0;\sigma_1]/M[t_0;\sigma_1](t_0^{q_0}-c_0)$$
is a cyclic algebra of degree $q_0$ over $\mathrm{Fix}(\sigma_1)$. $G_1 = \langle \sigma_1 \rangle$ is a cyclic group of prime order $q_0$ and therefore solvable.

We assume as induction hypothesis that if there exists a  chain
$$M = A_0 \subset \ldots \subset A_j$$
of  algebras such that \eqref{eqn:Petit (29) chain of subalgebras Galois 2} holds for all $i \in \{ 0, \ldots, j-1 \}$, $j \geq 1$, then $G_j$ is solvable. For the induction step we take a chain of  algebras $M = A_0 \subset \ldots \subset A_j\subset A_{j+1},$
$$A_{i+1} \cong A_i[t_i;\tau_i]/A_i[t_i;\tau_i](t_i^{q_i} -c_i)$$
where $\tau_i$ is an automorphism of $A_i$ of inner order $q_i$ which induces an automorphism  $\sigma_{i+1} \in G$, $c_i \in \mathrm{Fix}(\tau_i)$ is invertible and $q_i$ is prime, for all $i \in \{ 0, \ldots, j \}$. By the induction hypothesis, $G_j$ is a solvable group.

We show that $G_{j+1}$ is solvable: $t_j$ is an invertible element of
$$ A_{j+1}\cong A_j[t_j;\tau_j]/A_j[t_j;\tau_j](t_j^{q_j} - c_j),$$
with inverse $c_j^{-1}t_j^{q_j-1}$. $A_{j}$ is a $G_{j}$-crossed product algebra over $Z_{j}$ with maximal subfield $M$. The $F$-automorphism $\tau_j$ on $A_j$ satisfies $t_j l = \tau_j(l) t_j$ for all $l \in A_j$ which implies the inner automorphism $$I_{t_j}: A \rightarrow A, \ d \mapsto t_j d t_j^{-1}$$
restricts to $\tau_j$ on $A_j$ and so also restricts to $\sigma_{j+1}$ on $M$.

For any $\sigma \in G$ there exists an invertible $x_{\sigma} \in A $ such that the inner automorphism
$$I_{x_{\sigma}}: A \rightarrow A, \ y \mapsto x_{\sigma} y x_{\sigma}^{-1}$$
restricted to $M$ is $\sigma$ by Lemma \ref{lem:Petit (26)}. Hence we have $x_{\sigma_{j+1}} = t_j$ with $x_{\sigma_{j+1}}$ as defined in Lemma \ref{lem:Petit (26)}. We know that $\{ 1, t_j, \ldots, {t_j}^{q_j-1} \}$ is a basis for $A_{j+1}$ as a left $A_j$-module. By \eqref{eqn:Crossed product criteria 3}  we have  $x_{\sigma_{j+1}^2} = a_1{t_j}^2$, $x_{\sigma_{j+1}^3} = a_2 {t_j}^3, \dots$ for suitable $a_i \in M^\times$, so that w.l.o.g. $\{ 1, x_{\sigma_{j+1}}, \ldots, x_{\sigma_{j+1}^{q_j-1}} \}$ is a basis for $A_{j+1}$ as a left $A_j$-module.

Since $A_j$ is a $G_j$-crossed product algebra, it has $\{ x_{\rho} \ \vert \ \rho \in G_j \}$ as $M$-basis, and hence $A_{j+1}$ has  basis $$\{ x_{\rho} x_{\sigma_{j+1}^i} \ \vert \ \rho \in G_j, \ 0 \leq i \leq q_j - 1 \}$$ as a vector space over $M$.

Additionally, $x_{\rho} x_{\sigma_{j+1}^i} \in M^{\times} x_{\rho \sigma_{j+1}^i}$ by Lemma \ref{lem:Petit (26)} (iii) and thus $A_{j+1}$ has the $M$-basis
$$\{ x_{\rho \sigma_{j+1}^i} \ \vert \ \rho \in G_j, \ 0 \leq i \leq q_j-1 \}.$$
Now $A_{j+1}$ is a $G_{j+1}$-crossed product algebra and thus also has the $M$-basis $\{ x_{\sigma} \ \vert \ \sigma \in G_{j+1} \}$. We use these two basis to show that $G_{j+1} = G_{j} \langle \sigma_{j+1} \rangle$: Write
$$x_{\rho \sigma_{j+1}^i} = \sum_{\sigma \in G_{j+1}} m_{\sigma} x_{\sigma}$$
for some $m_{\sigma} \in M$, not all zero. Then
$$x_{\rho \sigma_{j+1}^i} m = \sum_{\sigma \in G_{j+1}} m_{\sigma} x_{\sigma} m = \sum_{\sigma \in G_{j+1}} m_{\sigma} \sigma(m) x_{\sigma},$$
and
$$x_{\rho \sigma_{j+1}^i} m = \rho \sigma_{j+1}^i(m) x_{\rho \sigma_{j+1}^i} = \rho \sigma_{j+1}^i(m) \sum_{\sigma \in G_{j+1}} m_{\sigma} x_{\sigma},$$
for all $m \in M$. Let $\sigma \in G_{j+1}$ be such that $m_{\sigma} \neq 0$, then in particular $$m_{\sigma} \sigma(m) x_{\sigma} = \rho \sigma_{j+1}^i(m)  m_{\sigma} x_{\sigma},$$ for all $m \in M$, that is $\sigma =  \rho \sigma_{j+1}^i$. This means that $\{  \rho \sigma_{j+1}^i \ | \ \rho \in G_j, \ 0 \leq i \leq q_j-1 \} \subseteq G_{j+1}$. Both sets have the same size so must be equal and we conclude $G_{j+1} = G_j \langle \sigma_{j+1} \rangle$.

Finally we prove $G_j$ is a normal subgroup of $G_{j+1}$:  the inner automorphism $I_{x_{\sigma_{j+1}}}$ restricts to the $F$-automorphism $\tau_j$ of $A_j$. In particular, this implies $x_{\sigma_{j+1}} x_{\rho} x_{\sigma_{j+1}}^{-1} \in A_j,$ for all $\rho \in G_j$. Furthermore, $$x_{\sigma_{j+1} \rho \sigma_{j+1}^{-1}} \in M^{\times} x_{\sigma_{j+1}} x_{\rho} x_{\sigma_{j+1}^{-1}} = M^{\times} x_{\sigma_{j+1}} x_{\rho} x_{\sigma_{j+1}}^{-1} \subset A_j,$$ for all $\rho \in G_j$ by Lemma \ref{lem:Petit (26)}.

Hence $\sigma_{j+1} \rho \sigma_{j+1}^{-1} \in G_j$ because $A_j$ is a $G_j$-crossed product algebra. Similarly, we see $\sigma_{j+1}^r \rho \sigma_{j+1}^{-r} \in G_j$ for all $r \in \mathbb{N}$. Let $g \in G_{j+1}$ be arbitrary and write $g = h \sigma_{j+1}^r$ for some $h \in G_j$, $r \in \{ 0, \ldots, q_j-1 \}$ which we can do because $G_{j+1} = G_j \langle \sigma_{j+1} \rangle$. Then
\begin{align*}
g \rho g^{-1} &= ( h \sigma_{j+1}^r) \rho ( h \sigma_{j+1}^r)^{-1} = h (\sigma_{j+1}^r \rho \sigma_{j+1}^{-r})h^{-1} \in G_j,
\end{align*}
for all $\rho \in G_j$ so $G_j$ is indeed normal.

It is well-known that a group $G$ is solvable if and only if given a normal subgroup $H$ of $G$, both $H$ and $G/H$ are solvable. It is clear  now that $G_{j+1}/G_j$ is cyclic and hence solvable, which implies $G_{j+1}$ is solvable as required.
\end{proof}

\section{Solvable Crossed Product Algebras} \label{section:Solvable Crossed Product Algebras}

We now focus on the case when $M/F$ is a Galois field extension:

Suppose $M/F$ is a finite Galois field extension of degree $n$ and $A$ is a central simple algebra of degree $n$ over $F$ with maximal subfield $M$. i.e. now $A$ is a $G$-crossed product algebra where $G = \mathrm{Gal}(M/F)$. We obtain the following as a special case of Theorem \ref{thm:Petit (29)} and Corollary \ref{cor:Observations on Theorem Petit(29)}:

\begin{theorem} \label{thm:Petit (29) solvable crossed product algebra version}
Let $A$ be a $G$-crossed product algebra of degree $n$ over $F$ with maximal subfield $M$ such that $M/F$is a Galois field extension of degree $n$ with non-trivial solvable $G = \mathrm{Gal}(M/F)$. Then there exists a chain of  subalgebras
\begin{equation*} 
M = A_0 \subset A_1 \subset \ldots \subset A_k = M(G) = A
\end{equation*}
of $A$ which are generalised cyclic algebras
\begin{equation*}
A_{i+1} \cong A_i[t_i;\tau_i]/A_i[t_i;\tau_i](t_i^{q_i} - c_i),
\end{equation*}
of degree $\prod_{l=0}^{i-1} q_l$ over $Z_i = \mathrm{Fix}(G_i)$ for all $i \in \{ 0, \ldots, k-1 \}$, such that
\begin{itemize}
\item[(i)] $q_i$ a prime,
\item[(ii)] $\tau_i$ an $F$-automorphism of $A_i$ of inner order $q_i$ which restricts to an automorphism $\sigma_{i+1} \in G$,
\item[(iii)] $c_i \in \mathrm{Fix}(\tau_i)^{\times}$, and
\item[(iv)] $Z_i/ Z_{i-1}$ has prime degree $q_{i-1}$ and $A_i$ is the centralizer of $\mathrm{Fix}(G_i)$ in $A$.
\end{itemize}
\end{theorem}

\begin{theorem} \label{thm:crossed product division condition}
In the set-up of Theorem \ref{thm:Petit (29) solvable crossed product algebra version}, $A$ is a division algebra if and only if
\begin{equation} \label{eqn:crossed product division condition}
b \tau_i(b) \cdots \tau_i^{q_i -1}(b) \neq c_i,
\end{equation}
for all $b \in A_i$, for all $i \in \{ 0, \ldots, k-1 \}$.
\end{theorem}

\begin{proof}
Suppose $A_k = A$ is a division algebra, then every subalgebra of $A$ must also be a division algebra as $A$ is finite-dimensional. Therefore $A_i$ are division algebras for all $i \in \{ 0, \ldots, k \}$. In particular this means $t_j^{q_j} - c_j \in A_j[t_j;\tau_j]$ are irreducible by Theorem \ref{thm:S_f_division_iff_irreducible}, thus \eqref{eqn:crossed product division condition} holds for all $b \in A_i$ by \cite[Theorem 1.3.16]{jacobson1996finite}.

Conversely suppose \eqref{eqn:crossed product division condition} holds for all $b \in A_i$, for all $i \in \{ 0, \ldots, k-1 \}$. We prove by induction that $A_j$ is a division algebra for all $j \in \{ 0, \ldots, k\}$, then in particular $A = A_k$ is a division algebra:

Clearly $A_0 = M$ is a field so in particular is a division algebra. Assume as induction hypothesis $A_j$ is a division algebra for some $j \in \{ 0, \ldots, k-1 \}$. By the proof of Theorem \ref{thm:Petit (29)}, $\tau_j^{q_j}$ is the inner automorphism $z \mapsto c_j z c_j^{-1}$ on $A_j$ and $\tau_j$ has inner order $q_j$. 
Then $A_{j+1} \cong A_j[t_j;\tau_j]/A_j[t_j;\tau_j](t_j^{q_j} - c_j)$ is a division algebra if and only if $t_j^{q_j} - c_j \in A_j[t_j;\tau_j]$ is irreducible, if and only if
$$b \tau_j(b) \cdots \tau_j^{q_j -1}(b) \neq c_j,$$
for all $b \in A_j$ by \cite[Theorem 1.3.16]{jacobson1996finite}. Thus $A_i$ is a division algebra for all $i \in \{ 0, \ldots, k \}$ by induction.
\end{proof}

The following result follows immediately from Theorem \ref{thm:Petit (29) Generalised}:

\begin{corollary} \label{cor:Characterisation of solvable crossed product algebras}
Let $A$ be a $G$-crossed product algebra of degree $n$ over $F$ with maximal subfield $M$ such that $M/F$is a Galois field extension of degree $n$ with non-trivial $G = \mathrm{Gal}(M/F)$. Then $G$ is solvable if there exists a chain of  subalgebras
\begin{equation*} 
M = A_0 \subset A_1 \subset \ldots \subset A_k  = A
\end{equation*}
of $A$ which all have maximal subfield $M$, and are generalised cyclic algebras
\begin{equation*}
A_{i+1} \cong A_i[t_i;\tau_i]/A_i[t_i;\tau_i](t_i^{q_i} - c_i),
\end{equation*}
over their centers for all $i \in \{ 0, \ldots, k-1 \}$, where $q_i$ a prime, $\tau_i$ an $F$-automorphism of $A_i$ of inner order $q_i$ which restricts to an automorphism $\sigma_{i+1} \in G$, and $c_i \in \mathrm{Fix}(\tau_i)^{\times}$.
\end{corollary}

\begin{remark}
Let $M/F$ be a finite Galois field extension with non-trivial solvable Galois group $G$ and $A$ be a solvable crossed product algebra over $F$ with maximal subfield $M$. Careful reading of \cite[p.~182-187]{albert1939structure} shows that Albert constructs the same chain of algebras
$$A_{i+1} = A_i[t_i;\tau_i]/A_i[t_i;\tau_i](t_i^{q_i}-c_i)$$
inside a solvable crossed product $A$ as we do in Theorem \ref{thm:Petit (29) solvable crossed product algebra version}. However they are not explicitly identified as quotient algebras of skew polynomial rings. We also obtain the converse of Albert's statement in Corollary \ref{cor:Characterisation of solvable crossed product algebras}. Furthermore, neither Theorems \ref{thm:Petit (29)}, \ref{thm:Petit (29) Generalised}, nor Corollaries \ref{cor:Observations on Theorem Petit(29)} and \ref{cor:non-central elements} require $M/F$ to be a Galois field extension, unlike Albert's result which requires $M/F$ to be Galois.
\end{remark}

\section{Some Applications to \texorpdfstring{$G$}{G}-Admissible Groups} \label{section:G-Admissible Groups}

The following definition is due to Schacher \cite{schacher1968subfields}:
\begin{definition}
A finite group $G$ is called \textbf{admissible} over a field $F$, if there exists a $G$-crossed product division algebra over $F$.
\end{definition}

Suppose $G$ is a finite solvable group, then we have a chain of normal subgroups $\{ 1 \} = G_0 \leq \ldots \leq G_k = G$, such that $G_j \triangleleft G_{j+1}$ and $G_{j+1}/G_j$ is cyclic of prime order $q_{j}$ for all $j \in \{ 0, \ldots, k-1 \}$. If $G$ is admissible over $F$, then Theorem \ref{thm:Petit (29)} implies that the subgroups $G_i$ of $G$ in the chain, are admissible over suitable intermediate fields of $M/F$:

\begin{theorem} \label{thm:subgroups of G are admissible}
Suppose $G$ is a finite solvable group which is admissible over a field $F$. Then each $G_i$ is admissible over the intermediate field $Z_i = \mathrm{Fix}(G_i)$ of $M/F$. Furthermore
$[Z_i:F] = \prod_{j=i}^{k-1} q_j$
for all $i \in \{ 1, \ldots, k \}$. In particular $G_{k-1}$ is admissible over $Z_{k-1} = \mathrm{Fix}(G_{k-1})$ which has prime degree $q_{k-1}$ over $F$. 
\end{theorem}

\begin{proof}
As $G$ is $F$-admissible there exists a $G$-crossed product division algebra $D$ over $F$. By Theorem \ref{thm:Petit (29) solvable crossed product algebra version}, there also exists a chain of $G_i$-crossed product division algebras $A_i$ over $Z_i$ with maximal subfield $M$, and $M/Z_i$ is a Galois field extension with $G_i = \mathrm{Gal}(M/Z_i)$. This means $G_i$ is $Z_i$-admissible.
\end{proof}

\begin{example}
Let $G = \mathbf{S_4}$, then $G$ is $\mathbb{Q}$-admissible \cite[Theorem 7.1]{schacher1968subfields}, so there exists a finite-dimensional associative central division algebra $D$ over $\mathbb{Q}$, with maximal subfield $M$ such that $M/\mathbb{Q}$ is a finite Galois field extension and $\mathrm{Gal}(M/\mathbb{Q}) = G$. Furthermore, $G$ is a finite solvable group, indeed we have the subnormal series
$$\{ \mathrm{id} \} \lhd \langle (12)(34) \rangle \lhd \mathbf{K} \lhd \mathbf{A_4} \lhd \mathbf{S_4},$$
where $\mathbf{K}$ is the Klein four-group and $\mathbf{A_4}$ is the alternating group. We have
$$\mathbf{S_4}/\mathbf{A_4} \cong \mathbb{Z}/2\mathbb{Z}, \ \mathbf{A_4}/\mathbf{K} \cong \mathbb{Z}/3\mathbb{Z},$$
$$\mathbf{K}/ \langle (12)(34) \rangle \cong \mathbb{Z}/2\mathbb{Z} \quad \text{ and} \quad \langle (12)(34) \rangle /\{ \mathrm{id} \} \cong \mathbb{Z}/2\mathbb{Z}.$$
By Theorem \ref{thm:Petit (29) solvable crossed product algebra version}, there exists a corresponding chain of division algebras
$$M = A_0 \subset A_1 \subset A_2 \subset A_3 \subset A_4 = D$$
over $\mathbb{Q}$, such that
$$A_{i+1} \cong A_i[t_i;\tau_i]/A_i[t_i;\tau_i](t_i^{q_i} - c_i),$$
for all $i \in \{ 0, 1, 2, 3 \}$, where $\tau_i$ is an automorphism of $A_i$, whose restriction to $M$ is $\sigma_{i+1} \in G$, $c_i \in \mathrm{Fix}(\tau_i)^{\times}$ and $\tau_i$ has inner order $2, 2, 3, 2$ for $i = 0, 1, 2, 3$ respectively. Moreover $q_0 = q_1 = q_3 = 2$ and $q_2 = 3$, and $A_i$ has degree $\prod_{l=0}^{i-1} q_l$ over its center $Z_i$ for all $i \in \{ 1, 2, 3, 4 \}$ by Corollary \ref{cor:Observations on Theorem Petit(29)}. In addition, by Theorem \ref{thm:subgroups of G are admissible} we conclude:
\begin{itemize}
\item[(i)] $\mathbf{A_4}$ is admissible over $Z_3$, where $Z_3 \subset M$ is a quadratic field extension of $\mathbb{Q}$.
\item[(ii)] $\mathbf{K}$ is admissible over $Z_2\subset M$, where $Z_2$ is a simple field extension of $Z_3$ of degree $q_2 = 3$. Therefore $Z_2$ is a simple field extension of $\mathbb{Q}$ of degree $q_2 q_3 = 6$ as any finite field extension of $\mathbb{Q}$ is simple.
\item[(iii)] $\langle (12)(34) \rangle$ is admissible over $Z_1 \subset M$, where $Z_1$ is a simple field extension of $Z_2$ of degree $q_1 = 2$. Therefore $Z_1$ is a simple extension of $\mathbb{Q}$ of degree $q_1 q_2 q_3 = 12$.
\end{itemize}
\end{example}

Schacher proved that for every finite group $G$, there exists an algebraic number field $F$ such that $G$ is admissible over $F$ \cite[Theorem 9.1]{schacher1968subfields}. Combining this with Theorem \ref{thm:Petit (29) solvable crossed product algebra version} we obtain:

\begin{corollary} 
Let $G$ be a finite solvable group. Then there exists an algebraic number field $F$ and a $G$-crossed product division algebra $D$ over $F$. Furthermore, there exists a chain of crossed product division algebras 
$$M = A_0 \subset A_1 \subset \ldots \subset A_k = D$$
over $F$, such that
$$A_{i+1} \cong A_i[t_i;\tau_i]/A_i[t_i;\tau_i](t_i^{q_i} - c_i),$$
for all $i \in \{ 0, \ldots, k-1 \}$, and satisfying
\begin{itemize}
\item[(i)] $q_i$ is the prime order of the factor group $G_{i+1}/G_i$ in the subnormal series \eqref{eqn:Subnormal series} which exists because $G$ is solvable,
\item[(ii)] $\tau_i$ is an automorphism of $A_i$ of inner order $q_i$ which restricts to $\sigma_i \in G$ on $M$,
\item[(iii)] $c_i \in \mathrm{Fix}(\tau_i)$ is invertible.
\end{itemize}
\end{corollary}

\begin{proof}
Such a field $F$ and division algebra $D$ exist by \cite[Theorem 9.1]{schacher1968subfields}. The assertion then follows by Theorem \ref{thm:Petit (29) solvable crossed product algebra version}.
\end{proof}

In \cite[Theorem 1]{sonn1983admissibility}, Sonn proved that a finite solvable group is admissible over $\mathbb{Q}$ if and only if all its Sylow subgroups are metacyclic, i.e. if every Sylow subgroup $H$ of $G$ has a cyclic normal subgroup $N$, such that $H/N$ is also cyclic. Combining this with Theorem \ref{thm:Petit (29) solvable crossed product algebra version} yields:

\begin{corollary} \label{cor:Q-admissible solvable petit (29) construction}
Let $G$ be a finite solvable group such that all its Sylow subgroups are metacyclic. Then there exists a $G$-crossed product division algebra $D$ over $\mathbb{Q}$, and a chain of crossed product division algebras 
$$M = A_0 \subset A_1 \subset \ldots \subset A_k = D$$
over $\mathbb{Q}$, such that
$$A_{i+1} \cong A_i[t_i;\tau_i]/A_i[t_i;\tau_i](t_i^{q_i} - c_i),$$
for all $i \in \{ 0, \ldots, k-1 \}$, such that
\begin{itemize}
\item[(i)] $q_i$ is the prime order of the factor group $G_{i+1}/G_i$ in the subnormal series \eqref{eqn:Subnormal series} which exists because $G$ is solvable,
\item[(ii)] $\tau_i$ is an automorphism of $A_i$ of inner order $q_i$ which restricts to $\sigma_i \in G$ on $M$,
\item[(iii)] $c_i \in \mathrm{Fix}(\tau_i)$ is invertible.
\end{itemize}
\end{corollary}

\begin{proof}
Such a division algebra $D$ exists by \cite[Theorem 1]{sonn1983admissibility}. The assertion then follows by Theorem \ref{thm:Petit (29) solvable crossed product algebra version}.
\end{proof}

\section[How to Construct Abelian Crossed Product Division Algebras]{How to Construct Crossed Product Division Algebras Containing a Given Abelian Galois Field Extension as a Maximal Subfield} \label{section:How to Construct Crossed Product Division Algebras Containing a Given Abelian Galois Field Extension as a Maximal Subfield}

Let $M/F$ be a Galois field extension of degree $n$ with abelian Galois group $G = \mathrm{Gal}(M/F)$. We now show how to canonically construct crossed product division algebras of degree $n$ over $F$ containing $M$ as a subfield. This generalises a result by Albert in which $n = 4$ and $G \cong \mathbb{Z}_2 \times \mathbb{Z}_2$ \cite[p.~186]{albert1939structure}, cf. also \cite[Theorem 2.9.55]{jacobson1996finite}: For $n = 4$ every central division algebra containing a quartic abelian extension $M$ with Galois group $\mathbb{Z}_2 \times \mathbb{Z}_2$ can be obtained this way \cite[p.~186]{albert1939structure}, that means as a generalised cyclic algebra $(D,\tau,c)$ with $D$ a quaternion algebra over its center.

Another way to construct such a crossed product algebra is via generic algebras, using a process going back to Amitsur and Saltman \cite{amitsur1978generic}, described also in \cite[\S 4.6]{jacobson1996finite}.

As $G$ is a finite abelian group, we have a chain of  subgroups
$$\{ 1 \} = G_0 \leq \ldots \leq G_k = G,$$
such that $G_{j} \triangleleft G_{j+1}$ and $G_{j+1} / G_{j}$ is cyclic of prime order $q_{j} > 1$ for all $j \in \{ 0, \ldots, k-1 \}$. We use this chain to construct the algebras we want:

$G_1 = \langle \sigma_1 \rangle$ is cyclic of prime order $q_0 > 1$ for some $\sigma_1 \in G$. Let $\tau_0 = \sigma_1$. Choose any $c_0 \in F^{\times}$ that satisfies $z \tau_0(z) \cdots \tau_0^{q_0-1}(z) \neq c_0$ for all $z \in M$ and define $f(t_0)=t_0^{q_0}-c_0\in M[t_0;\tau_0].$
Since  $\tau_0$ has order $q_0$, we have $\tau_0^{q_0}(z) c_0 = z c_0 = c_0 z$ for all $z \in M$, and so $f(t_0)$ is right invariant by Theorem \ref{thm:Properties of S_f petit}. Therefore we see that
$$A_1 = M[t_0;\tau_0] /M[t_0;\tau_0](t_0^{q_0}-c_0)$$
is an associative algebra which is cyclic of degree $q_0$ over $\mathrm{Fix}(\tau_0)$. Moreover, $f(t_0)$ is irreducible by \cite[Theorem 2.6.20(i)]{jacobson1996finite} and therefore $A_1$ is a division algebra.

Now $G_2/G_1$ is cyclic of prime order $q_1$, say $G_2/G_1 = \{\sigma_2^i G_1 \ \vert \ i \in \mathbb{Z} \}$ for some $\sigma_2 \in G_2$ where $\sigma_2^{q_1} \in G_1$. As $\sigma_2^{q_1} \in G_1$ we have $\sigma_2^{q_1} = \sigma_1^{\mu}$ for some $\mu \in \{ 0, \ldots, q_0-1 \}$. Define $c_1 = l_1 t_0^{\mu}$ for some $l_1 \in F^{\times}$ and define the map
$$\tau_1: A_1\rightarrow A_1, \ \sum_{i=0}^{q_0-1} m_i t_0^i \mapsto \sum_{i=0}^{q_0-1} \sigma_2(m_i) t_0^i,$$
which is an automorphism of $A_1$ by a straightforward calculation.

Denote the multiplication in $A_1$ by $\circ$. Then
$$\tau_1(c_1) = \sigma_2(l_1) t_0^{\mu} = l_1 t_0^{\mu} = c_1.$$
We have
\begin{align*}
\tau_1^{q_1} \Big( \sum_{i=0}^{q_0-1} m_i t_0^i \Big) \circ c_1
&= \sum_{i=0}^{q_0-1} \sigma_2^{q_1}(m_i) t_0^i \circ l_1
t_0^{\mu} = \sum_{i=0}^{q_0-1} l_1 \sigma_1^{\mu}(m_i) t_0^i
\circ t_0^{\mu}
\end{align*}
and
\begin{align*}
c_1 \circ \sum_{i=0}^{q_0-1} m_i t_0^i &= l_1 t_0^{\mu}
\circ \sum_{i=0}^{q_0-1} m_i t_0^i = \sum_{i=0}^{q_0-1} l_1
\sigma_1^{\mu}(m_i) t_0^{\mu} \circ t_0^i
\end{align*}
for all $m_i \in M$. Hence $\tau_1^{q_1}(z) \circ c_1 = c_1 \circ z$ for all $z \in A_1$ and $\tau_1(c_1) = c_1$, thus $f(t_1)=t_1^{q_1}-c_1\in A_1[t_1;\tau_1]$ is right invariant by Proposition \ref{prop:f(t) two-sided delta=0 generalised} and
$$A_2 = A_1[t_1;\tau_1]/A_1[t_1;\tau_1](t_1^{q_1}-c_1)$$
is a finite-dimensional associative algebra over $$\mathrm{Comm}(A_2) \cap A_1 = \mathrm{Fix}(\tau_1) \cap \mathrm{Cent}(A_1) = \mathrm{Fix}(\tau_1) \cap \mathrm{Fix}(\tau_0) \supset F$$ by Proposition \ref{prop:Comm(S_f) delta=0 generalised}(iii).

Again, $G_3/G_2$ is cyclic of prime order $q_2$, say $G_3/G_2 = \{\sigma_3^i G_2 \ \vert \ i \in \mathbb{Z} \}$ for some $\sigma_3 \in G$ with $\sigma_3^{q_2} \in G_2$. Write $\sigma_3^{q_2} = \sigma_2^{\lambda_1} \sigma_1^{\lambda_0}$ for some $\lambda_1 \in \{ 0, \ldots, q_1-1 \}$ and $\lambda_0 \in \{ 0, \ldots, q_0-1 \}$. The map
$$H_{\sigma_3}: A_1 \rightarrow A_1, \ \sum_{i=0}^{q_0-1} m_i t_0^i \mapsto \sum_{i=0}^{q_0-1} \sigma_3(m_i)t_0^i,$$
is an automorphism of $A_1$ by a straightforward calculation. Define
$$\tau_2: A_2 \rightarrow A_2, \ \sum_{i=0}^{q_1-1} x_i t_1^i \mapsto \sum_{i=0}^{q_1-1} H_{\sigma_3}(x_i) t_1^i \ (x_i \in A_1).$$
Then a straightforward calculation using that $H_{\sigma_3}$ commutes with $\tau_1$ and $H_{\sigma_3}(c_1) = c_1$ shows that $\tau_2$ is an automorphism of $A_2$. Define $c_2 = l_2 t_0^{\lambda_0} t_1^{\lambda_1}$ for some $l_2 \in F^{\times}$. Denote the multiplication in $A_i$ by $\circ_{A_i}$ and let $x_i = \sum_{j=0}^{q_0-1} y_{ij} t_0^j \in A_1$, $y_{ij} \in M$, $i \in \{ 0, \ldots, q_1-1 \}$. Then
$$\tau_2(c_2) = \tau_2(l_2 t_0^{\lambda_0} t_1^{\lambda_1}) = H_{\sigma_3}(l_2 t_0^{\lambda_0}) t_1^{\lambda_1} = l_2 t_0^{\lambda_0} t_1^{\lambda_1} = c_2.$$
Furthermore we have
\begin{align*}
\tau_2^{q_2} \Big( \sum_{i=0}^{q_1-1} x_i t_1^i \Big) \circ_{A_2} c_2
&= \sum_{i=0}^{q_1-1} H_{\sigma_3}^{q_2}(x_i) t_1^i \circ_{A_2} l_2
t_0^{\lambda_0} t_1^{\lambda_1} \\ &= \sum_{i=0}^{q_1-1}
\sum_{j=0}^{q_0-1} \sigma_3^{q_2}(y_{ij}) t_0^j t_1^i \circ_{A_2} l_2
t_0^{\lambda_0} t_1^{\lambda_1} \\ &= \sum_{i=0}^{q_1-1}
\sum_{j=0}^{q_0-1} \sigma_2^{\lambda_1}(\sigma_1^{\lambda_0}(y_{ij}))
t_0^j t_1^i \circ_{A_2} l_2 t_0^{\lambda_0} t_1^{\lambda_1} \\ &=
\sum_{i=0}^{q_1-1} \Big( \sum_{j=0}^{q_0-1}
\sigma_2^{\lambda_1}(\sigma_1^{\lambda_0}(y_{ij})) t_0^j \circ_{A_1}
\tau_1^i(l_2 t_0^{\lambda_0}) \Big) t_1^i \circ_{A_2} t_1^{\lambda_1}
\\ &= \sum_{i=0}^{q_1-1} \Big( \sum_{j=0}^{q_0-1} l_2
\sigma_2^{\lambda_1}(\sigma_1^{\lambda_0}(y_{ij})) t_0^j \circ_{A_1}
t_0^{\lambda_0} \Big) t_1^i \circ_{A_2} t_1^{\lambda_1},
\end{align*}
and
\begin{align*}
c_2 \circ_{A_2} \sum_{i=0}^{q_1-1} x_i t_1^i &= l_2 t_0^{\lambda_0}
t_1^{\lambda_1} \circ_{A_2} \sum_{i=0}^{q_1-1} x_i t_1^i =
\sum_{i=0}^{q_1-1} \big( l_2 t_0^{\lambda_0} \circ_{A_1}
\tau_1^{\lambda_1}(x_i) \big) t_1^{\lambda_1} \circ_{A_2} t_1^i \\ &=
\sum_{i=0}^{q_1-1} \sum_{j=0}^{q_0-1} \big( l_2 t_0^{\lambda_0}
\circ_{A_1} \sigma_2^{\lambda_1}(y_{ij}) t_0^j \big) t_1^{\lambda_1}
\circ_{A_2} t_1^i \\ &= \sum_{i=0}^{q_1-1} \Big( \sum_{j=0}^{q_0-1}
l_2 \sigma_1^{\lambda_0}(\sigma_2^{\lambda_1}(y_{ij}))
t_0^{\lambda_0} \circ_{A_1} t_0^j \Big) t_1^{\lambda_1} \circ_{A_2}
t_1^i.
\end{align*}
Hence $\tau_2^{q_2}(z) \circ_{A_2} c_2 = c_2 \circ_{A_2} z$ for all $z \in A_2$ and $\tau_2(c_2) = c_2$, therefore $f(t_2)=t_2^{q_2}-c_2\in A_2[t_2;\tau_2]$ is right invariant by Proposition \ref{prop:f(t) two-sided delta=0 generalised} and thus
$$A_3 = A_2[t_2;\tau_2]/A_2[t_2;\tau_2](t_2^{q_2}-c_2)$$
is a finite-dimensional associative  algebra over
$$\mathrm{Comm}(A_3) \cap A_2 = \mathrm{Fix}(\tau_2) \cap \mathrm{Cent}(A_2) = \mathrm{Fix}(\tau_0) \cap \mathrm{Fix}(\tau_1) \cap \mathrm{Fix}(\tau_2) \supset F$$
by Proposition \ref{prop:Comm(S_f) delta=0 generalised}(iii). Continuing in this manner we obtain a chain $M = A_0 \subset \ldots \subset A_k$ of finite-dimensional associative algebras
$$A_{i+1} = A_i[t_i;\tau_i]/A_i[t_i;\tau_i](t_i^{q_i}-c_i)$$
over
$$\mathrm{Fix}(\tau_i) \cap \mathrm{Cent}(A_i) = \mathrm{Fix}(\tau_0) \cap \mathrm{Fix}(\tau_1) \cap \dots \cap \mathrm{Fix}(\tau_i) \supset F,$$
for all $i \in \{ 0, \ldots, k-1 \}$, where $\tau_0 = \sigma_1$ and $\tau_i$ restricts to $\sigma_{i+1}$ on $M$ for all $i \in \{ 0, \ldots,k-1 \}$. Moreover,
$$[A_i:M] = [A_i: A_{i-1}] \cdots [A_1:M] =\prod_{l=0}^{i-1} q_l$$
hence
$$[A_k:F] = \big( \prod_{l=0}^{k-1} q_l\big) n = n^2,$$
and $A_k$ contains $M$ as a subfield.

Let us furthermore assume that each $c_i$ above, $i \in \{ 0, \ldots, k-1 \}$, is successively chosen such that
\begin{equation} \label{eqn:last}
z \tau_i(z) \cdots \tau_i^{q_i-1}(z)\neq c_i
\end{equation}
for all $z \in A_i$, then using that $\tau_i$ has inner order $q_i$ and $f(t_i) = t_i^{q_i} - c_i \in A_i[t_i;\tau_i]$ is an irreducible twisted polynomial, we conclude $A_{i+1}$ is a division algebra by \cite[Theorem 1.3.16]{jacobson1996finite}.

\begin{lemma} \label{lem:tau_i has inner order q_i}
For all $i \in \{ 0, \ldots, k-1 \}$, $\tau_i:A_i \rightarrow A_i$ has inner order $q_i$.
\end{lemma}

\begin{proof}
The automorphism $\tau_0 = \sigma_1: M \rightarrow M$ has inner order $q_0$.

Fix $i \in \{ 1, \ldots, k-1 \}$. $A_i$ is finite-dimensional over $F$, so it is also finite-dimensional over its center $\mathrm{Cent}(A_i) \supset F$. Recall that $\tau_i^{q_i}(z)  c_i = c_i  z$ for all $z \in A_i$, in particular $\tau_i^{q_i}\vert_{\mathrm{Cent}(A_i)} = id $. As $q_i$ is prime this means either $\tau_i \vert_{\mathrm{Cent}(A_i)} = id $ or $\tau_i \vert_{\mathrm{Cent}(A_i)}$ has order $q_i>1$.

Assume  that $\tau_i \vert_{\mathrm{Cent}(A_i)} = id $, then $\tau_i$ is an inner automorphism of $A_i$ by the Theorem of Skolem-Noether, say $\tau_i(z) = uzu^{-1}$ for some invertible $u \in A_i$, for all $z \in A_i$. In particular $\tau_i(m)= \sigma_{i+1}(m) = umu^{-1}$ for all $m \in M$. Write
$u = \sum_{j=0}^{q_{i-1}-1} u_j t_{i-1}^j$
for some $u_j \in A_{i-1}$, thus
\begin{align*}
\sigma_{i+1}(m)u &= \sigma_{i+1}(m) \sum_{j=0}^{q_{i-1}-1} u_j t_{i-1}^j =
\sum_{j=0}^{q_{i-1}-1} u_j t_{i-1}^j m \\ &= \sum_{j=0}^{q_{i-1}-1}
u_j \tau_{i-1}^j(m) t_{i-1}^j = \sum_{j=0}^{q_{i-1}-1} u_j
\sigma_{i}^j(m) t_{i-1}^j.
\end{align*}
for all $m \in M$. Choose $\eta_{i}$ with $u_{\eta_{i}} \neq 0$ then
\begin{equation} \label{eqn:tau_i is not inner 1}
\sigma_{i+1}(m) u_{\eta_{i}} = u_{\eta_{i}}
\sigma_{i}^{\eta_{i}}(m),
\end{equation}
for all $m \in M$.

If $i = 1$ we are done. If $i \geq 2$ then we can also write $u_{\eta_{i}} = \sum_{l=0}^{q_{i-2}-1} w_l t_{i-2}^l$ for some $w_l \in A_{i-2}$, therefore \eqref{eqn:tau_i is not inner 1} yields
\begin{align*}
\sigma_{i+1}(m) \sum_{l=0}^{q_{i-2}-1} w_l t_{i-2}^l &= \sum_{l=0}^{q_{i-2}-1} w_l t_{i-2}^l \sigma_{i}^{\eta_{i}}(m) = \sum_{l=0}^{q_{i-2}-1} w_l \tau_{i-2}^l(\sigma_{i}^{\eta_{i}}(m)) t_{i-2}^l \\ 
&= \sum_{l=0}^{q_{i-2}-1} w_l \sigma_{i-1}^l(\sigma_{i}^{\eta_{i}}(m)) t_{i-2}^l,
\end{align*}
for all $m \in M$. Choose $\eta_{i-1}$ with $w_{\eta_{i-1}} \neq 0$, then $$\sigma_{i+1}(m) w_{\eta_{i-1}} = w_{\eta_{i-1}} \sigma_{i-1}^{\eta_{i-1}}(\sigma_{i}^{\eta_{i}}(m)),$$ for all $m \in M$.

Continuing in this manner we see that there exists $s \in M^{\times}$ such that
$$\sigma_{i+1}(m) s = s \sigma_1^{\eta_1}(\sigma_2^{\eta_2}(\cdots (\sigma_{i}^{\eta_{i}}(m) \cdots ),$$
for all $m \in M$, hence
$$\sigma_{i+1}(m) = \sigma_1^{\eta_1}(\sigma_2^{\eta_2}(\cdots (\sigma_{i}^{\eta_{i}}(m) \cdots ),$$
for all $m \in M$ where $\eta_j \in \{ 0, \ldots, q_{j-1}-1 \}$ for all $j \in \{ 1, \ldots, i \}$. But $\sigma_{i+1} \notin G_i$ and thus
$$\sigma_{i+1} \neq \sigma_1^{\eta_1} \circ \sigma_2^{\eta_2} \circ \cdots \circ \sigma_{i}^{\eta_{i}},$$
a contradiction.

It follows that  $\tau_i \vert_{\mathrm{Cent}(A_i)}$ has order  $q_i>1$. By the Skolem-Noether Theorem the kernel of the restriction map $\mathrm{Aut}(A_i) \rightarrow \mathrm{Aut}(\mathrm{Cent}(A_i))$ is the group of inner automorphisms of $A_i$, and so $\tau_i$ has inner order $q_i$.
\end{proof}

\begin{proposition}
$\mathrm{Cent}(A_k) = F$.
\end{proposition}

\begin{proof}
$F \subset \mathrm{Cent}(A_k)$ by construction. Let now
$$z = z_0 + z_1 t_{k-1} + \ldots + z_{q_{k-1}-1} t_{k-1}^{q_{k-1}-1} \in \mathrm{Cent}(A_k)$$
where $z_i \in A_{k-1}$. Then $z$ commutes with all $l \in A_{k-1}$, hence $l z_i = z_i \tau_{k-1}^i(l)$ for all $i \in \{ 0, \ldots, q_{k-1}-1 \}$. This implies $z_0 \in \mathrm{Cent}(A_{k-1})$ and $z_i = 0$ for all $i \in \{ 1, \ldots, q_{k-1}-1 \}$, otherwise $z_i$ is invertible and $\tau_{k-1}^i$ is inner, a contradiction by Lemma \ref{lem:tau_i has inner order q_i}. Thus $z = z_0 \in \mathrm{Cent}(A_{k-1})$. A similar argument shows $z \in \mathrm{Cent}(A_{k-1})$ and continuing in this manner we conclude $z \in  M = \mathrm{Cent}(A_0)$.

Suppose, for a contradiction, that $z \notin F$. Then $\rho(z) \neq z$ for some $\rho \in G$. Since the $\sigma_{i+1}$ were chosen so that they generate the cyclic factor groups $ G_{i+1}/G_i,$ we can write $\rho = \sigma_1^{i_0} \circ\sigma_2^{i_1} \circ\cdots \circ \sigma_{k}^{i_{k-1}}$ for some $i_s \in \{ 0, \ldots, q_s-1 \}$. We have
\begin{align*}
t_0^{i_0} t_1^{i_1} \cdots t_{k-1}^{i_{k-1}} z &=
\sigma_1^{i_0}(\sigma_2^{i_1}(\cdots(\sigma_{k}^{i_{k-1}}(z) \cdots
) t_0^{i_0} t_1^{i_1} \cdots t_{k-1}^{i_{k-1}} \\ &= \rho(z)
t_0^{i_0} t_1^{i_1} \cdots t_{k-1}^{i_{k-1}} \neq z t_0^{i_0}
t_1^{i_1} \cdots t_{k-1}^{i_{k-1}},
\end{align*}
contradicting that $z \in \mathrm{Cent}(A_k)$. Therefore $\mathrm{Cent}(A_k) \subset F$.
\end{proof}

This yields a recipe for constructing a $G$-crossed product division algebra $A = A_k$ over $F$ with maximal subfield $M$ provided it is possible to find suitable $c_i$'s satisfying \eqref{eqn:last}.

\bibliographystyle{abbrv}
\bibliography{Resultssofar}
\end{document}